\documentclass[11pt,a4paper]{article}
\usepackage[T1]{fontenc}
\usepackage{geometry}
\usepackage{amsmath,amssymb,amsthm}
\usepackage{graphicx}
\usepackage{xcolor}
\definecolor{revisionred}{HTML}{8B1A1A}
\definecolor{revisionblue}{HTML}{173F7A}
\usepackage{booktabs,array}
\usepackage[hidelinks,hypertexnames=false]{hyperref}
\usepackage{enumitem}
\setlist[enumerate]{itemsep=2pt,topsep=5pt}
\setlist[itemize]{itemsep=2pt,topsep=5pt}
\allowdisplaybreaks[1]
\usepackage{fancyhdr}
\hypersetup{
  pdftitle={On the Geometry of Low Rank Correlation Matrices},
  pdfauthor={Olivier Bisson, David Loiseaux and Xavier Pennec}
}

\usepackage{microtype}
\usepackage{stmaryrd}
\usepackage{subcaption}
\usepackage{mathtools}
\usepackage{tikz-cd}
\usetikzlibrary{cd}

\usepackage{tabularx,longtable}
\usepackage{algorithm}
\usepackage{algpseudocode}
\usepackage[normalem]{ulem}
\usepackage{cancel}
\usepackage{aliascnt}

\DeclareMathOperator{\diag}{diag}
\DeclareMathOperator{\rank}{rank}
\DeclareMathOperator{\spanv}{span}
\DeclareMathOperator{\tr}{tr}

\DeclareMathOperator{\im}{im}
\DeclareMathOperator{\sgn}{sgn}

\DeclareMathOperator*{\argmin}{arg\, min}

\DeclareMathOperator{\Length}{{Length}}
\DeclareMathOperator{\Vol}{Vol}
\DeclareMathOperator{\pdet}{pdet}
\DeclareMathOperator{\diam}{diam}
\DeclareMathOperator{\Ric}{Ric}
\DeclareMathOperator{\Scal}{Scal}
\DeclareMathOperator{\Hol}{Hol}
\DeclareMathOperator{\Exp}{Exp}

\usepackage{dsfont}
\newcommand{\ind}[1]{\mathds{1}_{#1}}

\DeclarePairedDelimiterXPP{\Frob}[2]{}{\langle}{\rangle}{_F}{#1,#2}

\newcommand{\SPD}{\mathcal{S}^{++}}
\newcommand{\so}{\mathfrak{so}}
\newcommand{\Corr}{\mathrm{Corr}}
\newcommand{\Crk}[2]{\Corr^{(#1)}(#2)}
\newcommand{\Crle}[2]{\Corr^{\le #1}(#2)}

\newcommand{\R}{\mathbb{R}}
\newcommand{\grad}{\nabla}

\newcommand{\dProc}[1]{\widehat d_{#1}^{\mathrm{Proc}}}

\newcommand{\Sym}{\mathcal{S}}
\newcommand{\dd}{\,\mathrm d}
\newcommand{\padIota}[2]{\iota_{#1\to #2}}
\newcommand{\padXi}[2]{\Xi_{#1\to #2}}
\newcommand{\iks}{\padIota{s}{k}}
\newcommand{\Xks}{\padXi{s}{k}}
\newcommand{\Qrk}[2]{\mathord{{Q_{#1}^{(#2)}}}}
\newcommand{\Mrk}[2]{\mathord{{M_{#1}^{(#2)}}}}
\newcommand{\Qle}[2]{\mathord{{Q_{#1}^{\le #2}}}}
\newcommand{\Qks}{\Qrk{k}{s}}
\newcommand{\Qkj}{\Qrk{k}{j}}
\newcommand{\Qss}{\Qrk{s}{s}}
\newcommand{\Mks}{\Mrk{k}{s}}
\newcommand{\Mss}{\Mrk{s}{s}}
\newcommand{\Qkles}{\Qle{k}{s}}

\newcommand{\gQrk}[2]{\mathord{{g_{Q_{#1}^{(#2)}}}}}
\newcommand{\dM}[1]{\mathord{{d_{M_{#1}}}}}
\newcommand{\dQ}[1]{\mathord{{d_{Q_{#1}}}}}
\newcommand{\dg}[1]{\mathord{{d_{#1}}}}
\newcommand{\dSphere}[1]{\mathord{{d_{\mathbb S^{#1}}}}}
\newcommand{\dIntrinsic}[1]{\mathord{{d_{#1}}}}
\newcommand{\dInt}[2]{\mathord{{d_{\mathrm{int},#1}^{(#2)}}}}
\newcommand{\dCorr}[2]{\mathord{{d_{#1,#2}^{\Corr}}}}
\newcommand{\dCorrle}[2]{\mathord{{d_{#1,\le #2}^{\Corr}}}}
\newcommand{\dChord}[1]{\mathord{{\delta_{#1}}}}
\newcommand{\dQrk}[2]{\mathord{{d_{Q_{#1}^{(#2)}}}}}
\newcommand{\dle}[2]{\mathord{{d_{#1}^{\le #2}}}}
\newcommand{\gQks}{\gQrk{k}{s}}
\newcommand{\gQss}{\gQrk{s}{s}}
\newcommand{\dQks}{\dQrk{k}{s}}
\newcommand{\dQss}{\dQrk{s}{s}}
\newcommand{\dkles}{\dle{\mathrm{int}, k}{s}}
\newcommand{\Hbundle}{\mathcal H}
\newcommand{\Vbundle}{\mathcal V}
\newcommand{\Tbundle}{\mathcal T}
\newcommand{\Hz}[1]{\Hbundle_{#1}}
\newcommand{\Vz}[1]{\Vbundle_{#1}}
\newcommand{\Tz}[1]{\Tbundle_{#1}}

\newcommand\restr[2]{{\left.\kern-\nulldelimiterspace #1
\vphantom{\big|} \right|_{#2} }}

\newcommand{\I}{\mathrm I}

\theoremstyle{plain}

\newtheorem{theorem}{Theorem}[section]

\newaliascnt{lemma}{theorem}
\newtheorem{lemma}[lemma]{Lemma}
\aliascntresetthe{lemma}

\newaliascnt{proposition}{theorem}
\newtheorem{proposition}[proposition]{Proposition}
\aliascntresetthe{proposition}

\newaliascnt{corollary}{theorem}
\newtheorem{corollary}[corollary]{Corollary}
\aliascntresetthe{corollary}

\theoremstyle{definition}

\newaliascnt{definition}{theorem}
\newtheorem{definition}[definition]{Definition}
\aliascntresetthe{definition}

\newaliascnt{example}{theorem}
\newtheorem{example}[example]{Example}
\aliascntresetthe{example}

\newaliascnt{remark}{theorem}
\newtheorem{remark}[remark]{Remark}
\aliascntresetthe{remark}

\usepackage[capitalise,noabbrev]{cleveref}

\crefname{theorem}{Theorem}{Theorems}
\crefname{lemma}{Lemma}{Lemmas}
\crefname{proposition}{Proposition}{Propositions}
\crefname{corollary}{Corollary}{Corollaries}
\crefname{definition}{Definition}{Definitions}
\crefname{example}{Example}{Examples}
\crefname{remark}{Remark}{Remarks}

\numberwithin{equation}{section}

\begin{document}

\title{On the Geometry of Low Rank Correlation Matrices}
\author{Olivier Bisson\thanks{Inria Centre at Universit\'e
    C\^ote d'Azur.
  Email: \texttt{olivier.bisson@inria.fr}.}
  \and
  David Loiseaux\thanks{
    Lawrence Berkeley National Lab.
  Email: \texttt{djloiseaux@lbl.gov}.}
  \and
  Xavier Pennec\thanks{Inria Centre at Universit\'e
    C\^ote d'Azur.
Email: \texttt{xavier.pennec@inria.fr}.}}
\date{October 20, 2026}

\maketitle

\begin{abstract}
  Correlation matrices of bounded rank form a stratified space whose factor
  representations lie in a product of spheres. 
  Using the orbit-type description of the rank stratification, we prove that padding factors preserves the Riemannian geometry of each
  fixed-rank stratum and the intrinsic length metric of every rank-bounded
  subset. In the reduced model, a Sylvester equation describes the mechanical
  connection and its horizontal lifts. We derive sectional-curvature formulas
  and bounds, characterize zero-curvature directions, and determine the
  connection's holonomy groups.
  In particular, a flat rank-two connection can have nontrivial
  dihedral holonomy.
  For geodesics subject to a rank bound, rank
  cannot decrease in the interior of a minimizing segment with maximal-rank
  endpoints; the local excess over the unconstrained quotient distance is controlled by
  a dimension-independent cubic estimate.
  Finally, we compute the exact change of measure
  between quotient Riemannian volume and Euclidean surface volume in
  the off-diagonal entries. Coarea yields explicit factor densities
  for both normalized uniform laws.

\end{abstract}

\noindent\textbf{Keywords.}
Low-rank correlation matrices, Riemannian quotients, stratified spaces,
rank-constrained geodesics, sectional curvature, mechanical connection,
holonomy, volume comparison.

\smallskip
\noindent\textbf{MSC 2020.} Primary 53C20; Secondary 53C05, 53C22, 53C23, 53C29.

\clearpage
\setcounter{tocdepth}{2}

\tableofcontents
\clearpage

\section{Introduction}
\label{sec:introduction}

Correlation matrices describe dependence after the scale of each variable
has been removed. Their positive semidefiniteness and unit diagonal impose
coupled constraints on their entries, and a rank restriction adds a
singular geometric structure. Such restrictions arise both from the data
and from the model: a sample correlation matrix constructed from \(N\)
centered observations of \(n\) nonconstant variables has rank at most
\(\min(n,N-1)\), while a factor representation limits the number of
independent directions used to describe dependence. Low-rank correlation
matrices consequently occur in problems ranging from financial
approximation to the analysis of brain connectivity
\cite{duanGeneralizedLowRank2014,Chen2025}. Interpolating these matrices,
computing averages, optimizing over them, and generating random examples
all require a geometry that respects both the diagonal constraint and
the prescribed rank.

{
  The difficulty is not merely to represent a correlation matrix using fewer
  parameters, but to perform statistical operations without leaving the
  prescribed model: the set of correlation matrices satisfying a
  chosen rank bound. 
  For example, in functional connectivity, each subject or
  recording window yields a correlation matrix
  \cite{HutchisonEtAl2013,AllenEtAl2014}. Comparing groups or tracking
  temporal changes calls for representative matrices and interpolating paths
  that respect the chosen rank bound \cite{Chen2025}.
  Entrywise averages preserve positive semidefiniteness and
  the unit diagonal, but have higher rank in general. Conversely,
  averaging factors
  requires accounting for their orthogonal nonuniqueness. Intrinsic
  distances and rank-constrained Fr\'echet means provide a way to compare
  dependence structures and define representative matrices within a chosen
  rank bound, motivating applications to group summaries, clustering, and
  interpolation \cite{Chen2025}.
}

{
  For \(n \times n\) correlation matrices, the rank-\(s\) stratum has
  dimension \(n(s-1)-s(s-1)/2\) (\cref{cor:dimension-fixed-rank-stratum}),
  rather than \(n(n-1)/2\) at full rank. At fixed \(s\), the scaling is linear in \(n\) 
  and supports parsimonious dependence models with reduced
  storage and computational costs
  \cite{GrubisicPietersz2007}.
  Low-rank structure also allows partially observed Gram or
  correlation matrices to be completed from fewer entries 
  \cite[Section~3]{SingerCucuringu2010}. In particular, in a statistical setting, this translates to 
  rank-controlled concentration bounds of the approximation
  error 
  of covariance matrices on Gaussian data
  \cite[Corollary~2]{KoltchinskiiLounici2017}.
  In simulation, however, full-dimensional laws for random correlation
  matrices \cite{Joe2006,LewandowskiKurowickaJoe2009} assign zero
  measure to lower-rank strata, so sampling at a prescribed smaller rank
  requires a reference measure supported on the relevant stratum. Our
  results address the intrinsic consistency of these constructions
  across factor dimensions, shortest paths under rank constraints, and
  the reference measures needed to define uniform random examples.
}

A factorization of a correlation matrix \(C=XX^\top\), with \(X\in\R^{n\times k}\), gives a
natural way to impose these constraints: the rows of \(X\) have unit
norm, forming a product of spheres, and the number of its
columns bounds the rank of \(C\). Orthogonal changes of factor
{(\(X\mapsto XR\), with \(R\) orthogonal)} leave
\(C\) unchanged, so the restriction of the Frobenius metric on the tangent bundle of the product of spheres induces a natural
quotient metric on correlation matrices.
This construction raises questions that are essential for computation.
Can a shortest path between matrices of the same rank benefit from
passing through another rank? How does the choice of a horizontal factor
evolve around a closed path? Finally, which probability law is obtained
by drawing independent uniform rows on a sphere, and how does it relate
to uniformity on correlation matrices themselves when pushing it forward via the Gram map $X\mapsto XX^\top$?

\paragraph{Rank constraints and minimizing geodesics.}
The quotient geometry already restricts how ambient shortest paths
can meet the rank strata. It is well known that a minimizing segment in a
Riemannian orbit space has no interior point of more singular
orbit type than either endpoint (see for instance \cite[Lemma~3.5]{AKLM03}). For correlation matrices, the
interior rank is therefore at least the rank of each endpoint.
An ambient minimizer between rank-\(s\) matrices may nevertheless
pass through higher rank, which is inadmissible when the model
requires rank at most \(s\).

Restricting admissible curves changes the intrinsic metric problem,
so the ambient \emph{bumping} phenomenon of minimizing geodesic cannot directly apply. We prove that, for
\(2\leq s\leq\min(n,k)\), \emph{every intrinsic minimizing
segment under the rank bound \(s\) between rank-\(s\) endpoints
remains rank \(s\)}
(\cref{thm:no-interior-rank-drop-rank-bounded}). Thus allowing
intermediate rank loss does not shorten fixed-rank distances.
The proof uses a global isometry with a reduced quotient, in which
the endpoints lie in the principal stratum and the orbit-type
theorem applies. Fixed-rank reduction then provides a principal
bundle of full-column-rank sphere factors for computing these
minimizers.

\paragraph{The limitations of projection.}
A natural approach to rank constraints is to carry out a
construction in the ambient space and then project its output to
lower rank. Recovering intrinsic geometry in this way presents
several difficulties. Spectral truncation solves the unstructured
Frobenius low-rank approximation problem \cite{EckartYoung1936},
but generally destroys the unit diagonal. Subsequent diagonal
rescaling restores feasibility when all diagonal entries remain
positive, as in the initialization discussed in
\cite{GrubisicPietersz2007}. This procedure need not be uniquely
defined: truncation can be nonunique at an eigenvalue tie,
and zero diagonal entries prevent rescaling. For example, a
rank-\(s\) spectral truncation of \(\I_n\), \(s<n\), may be
\(\diag(\I_s,0)\).

\begin{figure}[!htbp]
  \centering
  \begin{subfigure}[t]{.49\linewidth}
    \centering
    \includegraphics[width=\linewidth]{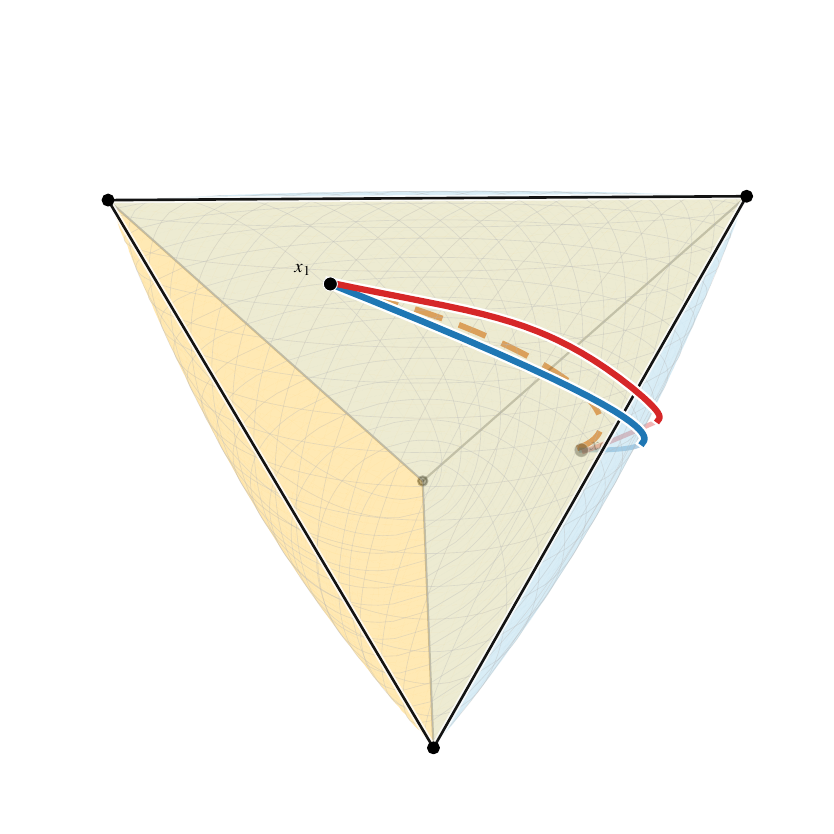}
    \caption{Continuous, non-minimizing nearest projection.}
    \label{fig:corr3-projected-ambient-geodesic-success}
  \end{subfigure}\hfill
  \begin{subfigure}[t]{.49\linewidth}
    \centering
    \includegraphics[width=\linewidth]{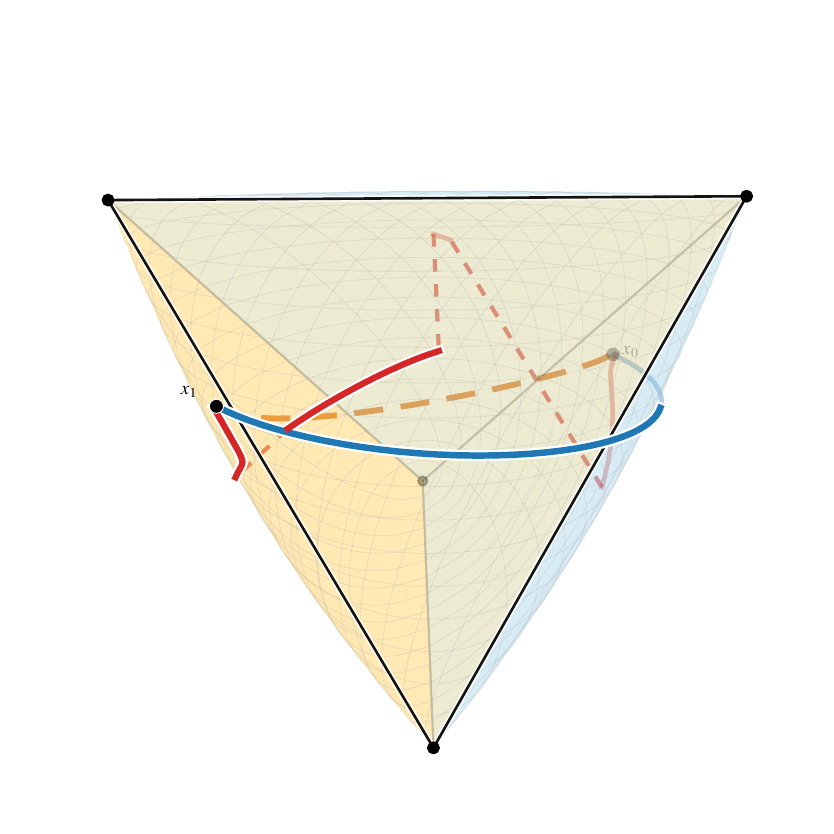}
    \caption{Numerical projection with face switches.}
    \label{fig:corr3-projected-ambient-geodesic-obstruction}
  \end{subfigure}
  \par\smallskip
  \includegraphics[width=.96\linewidth]{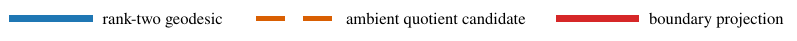}
  \caption{Projection of ambient candidates onto the rank-$2$ elliptope stratum $\Crk 23$ and intrinsic rank-two geodesics.
    }
  \label{fig:corr3-projected-ambient-geodesic-comparison}
\end{figure}

There is also a global topological obstruction. The full elliptope
is a compact convex space of dimension \(d\coloneq n(n-1)/2\)
whose relative boundary consists exactly of singular
correlation matrices, i.e., non-full rank matrices.
Radial parametrization about \(\I_n\)
identifies the body and its boundary with a closed \(d\)-ball and
its boundary sphere. Brouwer's no-retraction theorem
\cite[Corollary~2.15]{Hatcher2002} therefore rules out a continuous
map from the full elliptope onto its singular boundary that fixes
every singular matrix. Even for this boundary rank constraint, a
global projection cannot both be continuous and fix every
feasible point.

Smooth local projections near a fixed-rank matrix remain available
and yield valid retractions for manifold optimization
\cite[Lemma~4 and Proposition~5]{AbsilMalick2012}. Their local approximation properties do
not, by themselves, identify shortest paths under a rank
constraint. \Cref{fig:corr3-projected-ambient-geodesic-success}
illustrates this distinction for \(3\times3\) correlation matrices $\Corr (3)$:
a continuous nearest projection of a selected ambient candidate
can fail to minimize the intrinsic rank-two distance.
\Cref{fig:corr3-projected-ambient-geodesic-obstruction}
shows non-continuity of the projection.

To determine rank-constrained shortest paths, one therefore needs
comparisons of lengths along admissible curves. Similarly,
identifying the density of a projected probability law relative to
a chosen geometric volume requires Jacobian calculations.
The reduction, connection, and volume methods developed below
provide this information directly in the quotient geometry.

{
The main results follow four steps. Padding zero columns gives isometries for both
fixed-rank strata and their rank-bounded closures
(\cref{thm:fixed-rank-riemannian-reduction,thm:global-rank-bounded-metric-reduction}). More precisely, padding realizes the standard normalizer reduction of
each rank orbit type \cite{Michor2008}, providing a
principal $O(s)$-bundle over every rank-$s$ stratum.
Orthogonality to the fibers defines a principal connection, which we
call the \emph{reduced mechanical connection} and whose connection form
is obtained by solving a Sylvester equation.
This reduced connection then yields explicit curvature formulas and
a classification of horizontal-factor holonomy
(\cref{thm:sectional-curvature-fixed-rank-stratum,thm:holonomy-all-rank-strata}).
These reductions identify the intrinsic cost of a rank constraint and
give a uniform cubic bound for its local excess
(\cref{thm:no-interior-rank-drop-rank-bounded,thm:universal-cubic-rank-excess}).
Finally, an exact volume comparison provides factor densities for
sampling correlation matrices uniformly with respect to either
quotient or ambient surface volume
(\cref{thm:quotient-versus-ambient-volume,cor:ambient-euclidean-volume-lift}).
}

\subsection{Background and geometric setting}
\label{subsec:intro-background}

The approximation of correlation matrices has a substantial computational
literature. Higham's nearest-correlation problem \cite{Higham2002}
minimizes an entrywise discrepancy over the full convex set of
correlation matrices. A rank bound makes this feasible set nonconvex.
Grubi\v{s}i\'c and Pietersz \cite{GrubisicPietersz2007} develop
geometric optimization methods for the resulting approximation problem,
including weighted objectives. Their work already identifies the
product-of-spheres quotient and the smooth manifold of matrices of
maximal allowed rank.

Orthogonal quotients of matrix factors have an established role in the
geometry of positive semidefinite matrices. The Bures--Wasserstein
distance admits a factor formulation through orthogonal alignment;
its matrix analysis and its relation to geometric means are developed
in \cite[Theorems~1 and 2]{BhatiaJainLim2019}. 
Quotient geometry for positive semidefinite
matrices of fixed rank is studied in \cite[Section~2]{MassartAbsil2020}, while
\cite[Theorem~6.2]{Thanwerdas2023} treats minimizing geodesics between covariance
matrices of possibly different ranks. For correlation matrices, the
unit-diagonal condition replaces the Euclidean factor space by a
product of spheres. This changes the global distance problem:
alignment is measured by spherical lengths, and increasing the
factor dimension can decrease the quotient distance.

The direct starting point for our work is Chen's quotient geometry of
bounded- and fixed-rank correlation matrices \cite{Chen2025}.
Chen identifies the rank decomposition with the orbit-type
stratification, establishes the compatibility of the quotient and
matrix descriptions, and develops distances, minimizing geodesics,
and differential operations on the \emph{principal stratum}. We build on
this construction to study intrinsic reduction across the singular
rank decomposition, the finer curvature and holonomy of the
connection, and the comparison of natural volume measures.
Our arguments combine the geometry of proper isometric group actions
and their orbit spaces \cite{Michor2008,AlexandrinoBettiol2015}
with length-space methods for studying intrinsic distances under
rank constraints.

Let \(n,k\geq2\), put \(r\coloneq\min(n,k)\), and let \(1\leq s\leq r\).
In the following, we denote by \(\Corr(n)\) the elliptope of
\(n\times n\) correlation matrices,
\[
  \Corr(n)
  \coloneq
  \bigl\{
    C\in\R^{n\times n}:
    C=C^\top\succeq0,\ C_{ii}=1\ (1\leq i\leq n)
  \bigr\}.
\]
Its rank-\(s\) stratum and rank-\(s\) skeleton are denoted, respectively, by
\begin{align*}
\Crk{s}{n}
&\coloneq\{C\in\Corr(n):\rank C=s\}, \textnormal{ and }\\
\Crle{s}{n}
&\coloneq\{C\in\Corr(n):\rank C\leq s\}
  =\bigcup_{j=1}^{s}\Crk{j}{n}.
\end{align*}
The factor and quotient spaces, together with their rank pieces and
rank-bounded subsets, are
\begin{align*}
  M_k
  &\coloneq (\mathbb S^{k-1})^n
     =\{X\in\R^{n\times k}:\|X_{i\cdot}\|=1\ \text{for all }i\},
  & Q_k
  &=\pi_k(M_k)\coloneq M_k/O(k),\\
  \Mks
  &\coloneq \{X\in M_k:\rank X=s\},
  & \Qks
  &\coloneq \pi_k(\Mks),\\
  M_k^{\leq s}
  &\coloneq \{X\in M_k:\rank X\leq s\},
  & \Qkles
  &\coloneq \pi_k(M_k^{\leq s})
     =\bigcup_{j=1}^{s}\Qkj,
\end{align*}
where \(O(k)\) acts by right matrix multiplication and
\(\pi_k:M_k\to Q_k\) denotes the quotient projection.

The Gram map
\[
  \Gamma_k:M_k\longrightarrow\Crle{k}{n},
  \qquad X\longmapsto XX^\top,
\]
descends to a homeomorphism
\(\bar\Gamma_k:Q_k\to\Crle{k}{n}\), whose restriction to a fixed rank stratum
\(\bar\Gamma_k:\Qks\to\Crk{s}{n}\) is a diffeomorphism
{(\cref{prop:quotient-model-low-rank-correlations})}.
These maps form the commutative diagrams
\[
\begin{tikzcd}[column sep={4em,between origins},row sep=3em]
  M_k
    \arrow[rr,"\Gamma_k"]
    \arrow[dr,"\pi_k"']
  && \Crle{k}{n}\\
  & Q_k
    \arrow[ur,"\bar\Gamma_k"']
  &
\end{tikzcd}
\qquad
\begin{tikzcd}[column sep={4em,between origins},row sep=3em]
  \Mks
    \arrow[rr,"\Gamma_k"]
    \arrow[dr,"\pi_k"']
  && \Crk{s}{n}\\
  & \Qks
    \arrow[ur,"\bar\Gamma_k"']
  &
\end{tikzcd}
\]
where the maps in the second diagram denote the corresponding
restrictions, which we omit for clarity. Each quotient stratum \(\Qks\), and hence each correlation stratum
\(\Crk{s}{n}\), is a smooth manifold of dimension
\[
  m_s\coloneq n(s-1)-\frac{s(s-1)}2
\]
{(see \cref{cor:dimension-fixed-rank-stratum})}.
These strata are path connected for \(s\geq2\); at rank one,
each is a discrete set of \(2^{n-1}\) points
{(\cref{thm:path-connectedness-fixed-rank-strata,rem:rank-one-stratum})}.

The rank decomposition differs from the classical facial decomposition of the elliptope \(\Corr(n)\).
The facial structure of this convex body is studied by Laurent and
Poljak \cite{LaurentPoljak1996}, and its extreme points are
characterized by Li and Tam \cite{LiTam1994}. The minimal face
containing \(C\in\Corr(n)\) is
\[
  F(C)=\{A\in\Corr(n):\ker C\subseteq\ker A\},
\]
as described in \cite[Proposition~2.7]{LaurentPoljak1995}.
The matrix \(C\) lies in the relative interior of \(F(C)\),
throughout which the kernel, and hence the rank, is constant.
At lower rank, however, fixing the rank allows the kernel and the
minimal face to vary. The rank decomposition is therefore coarser
than the partition into relative interiors of faces.

The quotient representation provides a metric framework for studying
this rank decomposition. Equip \(M_k\) with the product of the round
sphere metrics,
\[
  (g_{M_k})_X(U,V)\coloneq \Frob{U}{V},
  \qquad U,V\in \Tz{X}M_k.
\]
This makes \(M_k\) a compact Riemannian manifold. Its Riemannian
distance \(\dM{k}\) induces the metric quotient distance \(\dQ{k}\) on
\(Q_k\), with
\begin{equation}
  \label{eq:intro-quotient-distance}
  \begin{aligned}
    \dM{k}(X,Y)^2
    =\sum_{i=1}^{n}
    \arccos^2\!\bigl(\langle X_{i\cdot},Y_{i\cdot}\rangle\bigr),\qquad
    \dQ{k}([X],[Y])
    =\min_{R\in O(k)}\dM{k}(X,YR).
  \end{aligned}
\end{equation}
The quotient \((Q_k,\dQ{k})\) is a compact geodesic length space:
its distance is the infimum of lengths of continuous rectifiable
curves {(\cref{def:curve-lengths})}, and this infimum is attained between every pair of points
{(\cref{prop:reduced-quotient-global-metric-geometry})}.
These curves may pass through any rank present in \(Q_k\).

On each smooth factor stratum \(\Mks\), the restricted product
metric induces a quotient Riemannian metric \(\gQks\), uniquely
characterized by requiring
\[
  \pi_k:
  \bigl(\Mks,\restr{g_{M_k}}{\Mks}\bigr)
  \longrightarrow
  \bigl(\Qks,\gQks\bigr)
\]
to be a Riemannian submersion. We denote its associated Riemannian
distance by \(\dQrk{k}{s}\). It is defined by minimizing
\(\gQks\)-lengths over continuous piecewise \(C^1\) curves
that remain in \(\Qks\).
This fixed-stratum Riemannian distance need not equal the pointwise restriction of the ambient quotient distance \(\dQ{k}\), which permits paths through other ranks.

A rank constraint can also be imposed directly on curves in the
ambient metric space \((Q_k,\dQ{k})\). For
\(2\leq s\leq\min(n,k)\), define
\[
  \begin{aligned}
    \dkles(p,q)
    \coloneq\inf_{\substack{
      \gamma:[0,1]\to \Qkles\\
      \gamma(0)=p,\ \gamma(1)=q}}
    \Length_{\dQ{k}}(\gamma),\qquad
    \dInt{k}{s}(p,q)
    \coloneq\inf_{\substack{
      \gamma:[0,1]\to \Qks\\
      \gamma(0)=p,\ \gamma(1)=q}}
    \Length_{\dQ{k}}(\gamma),
  \end{aligned}
\]
where both infima range over continuous \(\dQ{k}\)-rectifiable curves.
These are intrinsic length distances on the indicated subsets:
the first permits ranks at most \(s\), while the second requires
rank exactly \(s\). In both definitions, curve length is measured
using the same ambient quotient distance \(\dQ{k}\).
Their relation to the quotient Riemannian distance is established
by the reduction results below. In particular, the global reduction
will show that
\((\Qkles,\dkles)\) is compact and geodesic
{(\cref{cor:rank-bounded-global-alexandrov-geometry})}.

The Gram map transports these structures to correlation matrices.
We retain the notation \(\dQ{k}\) for the transported quotient distance
on \(\Crle{k}{n}\). On \(\Crk{s}{n}\), we write
\(g_Q\) and \(\dCorr{k}{s}\) for the transports of
\(\gQks\) and \(\dQrk{k}{s}\), respectively, and we write
\(\dCorrle{k}{s}\) for the transport of
\(\dkles\) to \(\Crle{s}{n}\).

For the later volume comparisons, we also use the embedded
Riemannian metrics on \(\Crk{s}{n}\) induced by off-diagonal
Euclidean coordinates and by the full Frobenius inner product.
For \(U,V\in \Tz{C}\Crk{s}{n}\), these are
\[
  (g_{\mathrm E})_C(U,V)
  \coloneq\sum_{i<j}U_{ij}V_{ij},
  \qquad
  (g_{\mathrm F})_C(U,V)
  \coloneq\Frob{U}{V}
  =2(g_{\mathrm E})_C(U,V).
\]
Their Riemannian distances therefore satisfy
\(\dg{g_{\mathrm F}}=\sqrt2\,\dg{g_{\mathrm E}}\).
These intrinsic distances generally differ from the corresponding
ambient chord distances \(\|C-D\|_{\mathrm E}\) and \(\|C-D\|_F\);
see \cref{sec:volume-measures-correlation-strata}.

The following table summarizes the conventions. All Riemannian
distances listed are length metrics; the full quotient and the
rank-bounded skeletons are treated as stratified length spaces.
\begin{center}
  \small
  \begin{tabularx}{\linewidth}{@{}l l X@{}}
    \toprule
    Space & Metric or distance & Geometry and admissible curves \\
    \midrule
    \(M_k\)
    & \(g_{M_k},\ \dM{k}\)
    & Riemannian manifold; all curves in the product of round spheres.
    \\

    \(Q_k\)
    & \(\dQ{k}\)
    & Stratified geodesic length space; all ranks present in \(Q_k\).
    \\

    \(\Crle{k}{n}\)
    & \(\dQ{k}\) via \(\bar\Gamma_k\)
    & Gram transport of the quotient length space; all ranks present.
    \\

    \(\Qks\)
    & \(\gQks,\ \dQrk{k}{s}\)
    & Riemannian manifold; rank exactly \(s\), with quotient
      Riemannian lengths.
    \\

    \(\Qks\)
    & \(\dInt{k}{s}\)
    & Intrinsic length space; rank exactly \(s\), with lengths
      measured using \(\dQ{k}\).
    \\

    \(\Qkles\)
    & \(\dkles\)
    & Stratified length space; rank at most \(s\), with lengths
      measured using \(\dQ{k}\).
    \\

    \(\Crk{s}{n}\)
    & \(g_Q,\ \dCorr{k}{s}\)
    & Riemannian manifold; rank exactly \(s\), with the transported
      quotient metric.
    \\

    \(\Crle{s}{n}\)
    & \(\dCorrle{k}{s}\)
    & Stratified length space; rank at most \(s\), with the
      transported intrinsic skeleton distance.
    \\

    \(\Crk{s}{n}\)
    & \(g_{\mathrm E},g_{\mathrm F};
        \ \dg{g_{\mathrm E}},\dg{g_{\mathrm F}}\)
    & Riemannian manifolds; rank exactly \(s\), with intrinsic
      lengths for the embedded metrics.
    \\
    \bottomrule
  \end{tabularx}
\end{center}

With these definitions in place, the reduction results give a precise
comparison between the distances. Let
\[
  \Xks:Q_s\longrightarrow \Qkles,
  \qquad [Z]\longmapsto[(Z,0)],
\]
be the map induced by zero-padding factors. For
\(p,q\in \Qks\), write
\(p_s\coloneq \Xks^{-1}(p)\) and
\(q_s\coloneq \Xks^{-1}(q)\). We establish
\[
  \begin{aligned}
    \dQ{k}(p,q)
    \leq \dkles(p,q)
      =\dQ{s}(p_s,q_s)
    =\dQrk{s}{s}(p_s,q_s)
      =\dQrk{k}{s}(p,q)
      =\dInt{k}{s}(p,q).
  \end{aligned}
\]
The inequality follows by restricting admissible curves.
The first equality is our global rank-bounded reduction
(\cref{thm:global-rank-bounded-metric-reduction}), and the second is
the full-rank quotient compatibility established by
Chen~\cite[Proposition~3.5]{Chen2025}.
The third is our fixed-rank Riemannian reduction
(\cref{thm:fixed-rank-riemannian-reduction}).
The last is deduced in
\cref{thm:ambient-quotient-length-on-strata} from these reductions
and the standard invariance of rectifiable curve lengths under
passage to the associated intrinsic length metric.

Thus allowing lower-rank points within the rank-\(s\) skeleton
does not decrease the intrinsic distance between rank-\(s\)
endpoints. If \(s<\min(n,k)\), the ambient distance \(\dQ{k}\) also
permits higher-rank paths and can be strictly smaller.
On correlation matrices, the reduction identities become
\[
  \begin{aligned}
    \dCorr{k}{s}=\dCorrle{k}{s}=\dCorr{s}{s}
    \quad\text{on }\Crk{s}{n},\qquad
    \dCorrle{k}{s}=\dCorrle{s}{s}
    \quad\text{on }\Crle{s}{n}.
  \end{aligned}
\]

\subsection{Rank reduction and constrained distances}
\label{subsec:intro-rank-constraints}

Appending zero columns defines a canonical map
\(\Xks:Q_s\to \Qkles\),
\([Z]\mapsto[(Z,0)]\). Our reduction theorems
{(\cref{thm:fixed-rank-riemannian-reduction,thm:global-rank-bounded-metric-reduction})}
show that this map
preserves both the Riemannian geometry of a fixed-rank stratum and
the intrinsic metric of its singular closure:
\begin{equation}
  \label{eq:intro-rank-reduction}
  \begin{aligned}
    \Xks:\
    (\Qss,\gQss)
    &\ \cong\ (\Qks,\gQks),\\
    \Xks:\ (Q_s,\dQ{s})
    &\ \cong\
    (\Qkles,\dkles),
    \qquad 2\leq s\leq r .
  \end{aligned}
\end{equation}
The first isometry provides a reduced principal \(O(s)\)-bundle
for differential calculations on every rank-\(s\) stratum. The
second identifies the full geometry of paths constrained by a
rank bound, including paths whose rank changes.
Thus the intrinsic geometry associated with the rank bound \(s\)
depends only on \(n\) and \(s\).
The smooth reduction realizes the standard normalizer reduction
\cite{GroveSearle2000} (\cref{prop:padding-normalizer-reduction}).

The singular reduction requires control of lengths across rank
changes. We combine length-preserving lifts for quotient submetries
\cite{KapovitchLytchak2022} with semialgebraic approximation of
intrinsic lengths \cite{KurdykaOrro1997}. On intervals of constant
rank, a moving orthonormal frame represents a lifted curve in
\(\R^{n\times s}\). The derivative of the frame contributes only
a vertical velocity, which disappears in the quotient. This
gives the length comparison needed to pass from the local
Riemannian reduction to the global isometry
{(\cref{thm:global-rank-bounded-metric-reduction})}.

{
Although \(\Crle{s}{n}\) is singular for $2\le s < n$, Alexandrov geometry
allows us to extend lower sectional-curvature bounds to such metric spaces.
On a complete, finite-dimensional geodesic metric space with distance
\(d\), nonnegative Alexandrov curvature is characterized by
\begin{equation}
\label{eq:alexandrov-midpoint-comparison}
  d(A,M)^2 \ge
  \frac{d(A,B)^2+d(A,C)^2}{2}-\frac{d(B,C)^2}{4}
\end{equation}
for every geodesic triangle \(ABC\), where \(M\) is the midpoint
of its chosen minimizing geodesic side \(BC\)
\cite{BuragoBuragoIvanov2001}.
Our reduction shows that \(\Crle{s}{n}\), equipped with its intrinsic
rank-bounded metric, is a compact geodesic Alexandrov space of
nonnegative curvature
(\cref{cor:rank-bounded-global-alexandrov-geometry}).
}
Combining our reduction with orbit-type monotonicity for minimizing
geodesics \cite{AKLM03}, we show that, between two rank-\(s\) endpoints,
every intrinsic minimizer under the rank bound \(s\) remains in
rank \(s\) throughout its interior
(\cref{thm:no-interior-rank-drop-rank-bounded}).
Allowing intermediate rank loss therefore does not shorten
fixed-rank distances. By fixed-rank reduction, these minimizers
are computed on \(\Qss\), using the principal bundle
\(\Mss\to \Qss\). Thus the global reduction controls
the intrinsic rank bound, while the fixed-rank reduction supplies
the principal-bundle geometry for its minimizing paths.
This property coexists with geodesic
incompleteness: we construct horizontal geodesics that reach
lower rank in finite time
(\cref{thm:finite-time-rank-degeneration}).

For a different constraint, in which a path between regular
endpoints is required to meet \(\Qle{k}{\ell}\), \(\ell<r\), we
show that each minimizer consists of two ambient minimizing
segments meeting at a single lower-rank point in its image. Its
length is determined by the finite-dimensional optimization
\[
  \min_{p\in \Qle{k}{\ell}}
  \bigl(\dQ{k}(x_0,p)+\dQ{k}(p,x_1)\bigr)
\]
(\cref{prop:one-contact-lower-rank-locus,prop:reduction-constrained-hit}).

Increasing the allowed rank can, however, shorten a path. For
\(p,q\in \Qks\), we quantify this effect by the excess
\[
  \Delta_s(p,q)
  \coloneq
  \dkles(p,q)-\dQ{k}(p,q).
\]
The chordal quotient distance
\(\dChord{k}([X],[Y])\coloneq \min_{O\in O(k)}\|X-YO\|_F\)
is computed by orthogonal Procrustes alignment
\cite{Schonemann1966} and is unchanged by padding
{(\cref{lem:chordal-comparison-padded-factors})}. Comparing
chords and arcs on the row spheres gives an explicit arcsine
bound for \(\Delta_s\) when \(\dChord{k}\leq2\), and the uniform
local estimate
\begin{equation}
  \label{eq:intro-cubic-excess}
  0\leq\Delta_s(p,q)
  \leq
  \left(\frac{\pi}{3}-1\right)\dChord{k}(p,q)^3
  \leq
  \left(\frac{\pi}{3}-1\right)\dQ{k}(p,q)^3,
  \qquad \dQ{k}(p,q)\leq1
\end{equation}
(\cref{thm:arcsine-bound-rank-excess,thm:universal-cubic-rank-excess}).
The constant is independent of \(n\), \(s\), and \(k\), and the
estimate requires no lower bound on the positive eigenvalues of
the endpoints. It therefore controls the local cost of a rank
constraint uniformly as the endpoints approach lower rank.

\subsection{Curvature and holonomy}
\label{subsec:intro-curvature-holonomy}

The fixed-rank reduction places the differential geometry on
the principal bundle \(\Mss\to \Qss\). Its mechanical
connection selects tangent factors orthogonal to the \(O(s)\)-orbits.
For \(Z\in \Mss\), let \(E_Z\coloneq Z^\top Z\), and consider the
Sylvester operator
\[
  \mathcal L_{E_Z}:\so(s)\longrightarrow\so(s),
  \qquad
  \mathcal L_{E_Z}(\Omega)\coloneq E_Z\Omega+\Omega E_Z.
\]
The established Sylvester horizontal projection \cite{MassartAbsil2020},
in its correlation-matrix formulation \cite[Section~5.1]{Chen2025},
is expressed through the connection form
{(\cref{prop:sylvester-mechanical-connection})}
\begin{equation}
  \label{eq:intro-mechanical-connection}
  \omega_Z(W)
  =
  \mathcal L_{E_Z}^{-1}(Z^\top W-W^\top Z),
  \qquad W\in \Tz{Z}\Mss.
\end{equation}
We use this formulation to compute the curvature of the
connection and, through O'Neill's equations \cite{ONeill1966},
the sectional curvature of every fixed-rank stratum.
The explicit formula separates the curvature of the row spheres
from a nonnegative term generated by the orthogonal quotient
(\cref{thm:sectional-curvature-fixed-rank-stratum}).
It also characterizes all zero-curvature planes: a plane is flat
exactly when, for horizontal lifts \(U,V\), the two tangent
vectors \(U_{i\cdot}\) and \(V_{i\cdot}\) are linearly dependent
in every row
(\cref{thm:zero-sectional-curvature-characterization}).

The spectrum of \(\mathcal L_{E_Z}\) on \(\so(s)\) consists of the pairwise sums
\(\lambda_a + \lambda_b\) (\(1\leq a<b\leq s\)) of eigenvalues of \(E_Z\) with distinct indices. This yields quantitative control
of curvature. In particular, on the region
\(\lambda_{\min}(E_Z)\geq\varepsilon>0\),
\begin{equation}
  \label{eq:intro-curvature-bound}
  0\leq\sec\leq1+\frac{3}{2\varepsilon}.
\end{equation}
We establish asymptotic sharpness of this bound, for fixed
\(s\geq3\) and \(\varepsilon>0\), as \(n\) increases
(\cref{cor:curvature-bounds-thick-parts,prop:asymptotic-sharpness-thick-part}).
The pairwise-sum structure also distinguishes different approaches
to the rank boundary: if only one positive eigenvalue tends to
zero while the others stay bounded below, these curvature
denominators remain bounded away from zero
{(\cref{rem:curvature-degeneration-lower-rank})}.

We further determine the restricted and full holonomy groups of
the mechanical connection
(\cref{thm:holonomy-all-rank-strata}).
For \(s\geq3\), the restricted holonomy is \(SO(s)\);
the full holonomy is \(SO(s)\) when \(n=s\) and \(O(s)\)
when \(n>s\). In rank two the connection is flat, with trivial restricted
holonomy, but its full holonomy is the dihedral group \(D_n\) of order \(2n\)
when \(n>2\); it is trivial when \(n=s=2\).
Rank one has trivial holonomy.
These groups describe the orthogonal transformations acquired
by a factor under horizontal lifting of closed curves in the
correlation stratum. The classification distinguishes local
curvature from global monodromy and determines the possible
path dependence of horizontal factorizations. In rank at least
three, the connection curvature also proves nonpolarity of the
action: there is no section meeting the orbits everywhere
orthogonally
(\cref{thm:non-polarity-rank-at-least-three}).

\subsection{Volume comparison and uniform probability laws}
\label{subsec:intro-volume}

A geometric description of random correlation matrices requires
a specified reference measure. For positive definite correlation
matrices, partial-correlation, vine, and extended onion constructions
\cite{Joe2006,LewandowskiKurowickaJoe2009}
provide prescribed distributions, including uniformity with respect
to Lebesgue measure on the off-diagonal entries. Lower-rank strata
have zero measure for this full-dimensional uniform law, so
uniformity at a prescribed smaller rank requires a reference
measure of the appropriate dimension.
For \(2\leq s\leq n\), two natural choices on \(\Crk{s}{n}\) are the Riemannian volume \(\Vol_Q\) of
the quotient metric and the surface volume
\(\Vol_{\mathrm E}\) induced by the off-diagonal
Euclidean inner product
\[
  g_{\mathrm E}(U,V)\coloneq \sum_{i<j}U_{ij}V_{ij}.
\]
The ambient Frobenius inner product is \(2g_{\mathrm E}\), and
therefore gives the same normalized volume law.
Independent uniform rows on \(\mathbb S^{s-1}\) define a third
natural law by the Gram map. Although this construction is
orthogonally invariant, its pushforward weights points according
to the volumes of their factor orbits.

We compute these orbit volumes and the exact density relating
the two geometric volumes. For \(C=XX^\top\in\Crk{s}{n}\)
with $X \in \Mss$,
write \(E_X=X^\top X\) and
\[
  \begin{aligned}
    P_C&\coloneq XE_X^{-1}X^\top=CC^\dagger,\\
    A_C&\coloneq 2\diag\bigl((P_C)_{11},\ldots,(P_C)_{nn}\bigr)
    -P_C\odot P_C,\\
    D_s(C)&\coloneq \prod_{1\leq a<b\leq s}
    (\lambda_a(C)+\lambda_b(C))^{1/2},
  \end{aligned}
\]
where \(C^\dagger\) is the Moore-Penrose inverse,
\(\odot\) denotes the Hadamard (entrywise) product, and
\(\lambda_1(C),\ldots,\lambda_s(C)\) are the positive eigenvalues
of \(C\). Thus \(P_C\) is the orthogonal projector onto the range
of \(C\), and
\(D_s(C)^2=\det\left(\restr{\mathcal L_{E_X}}{\so(s)}\right)\)
{(\cref{lem:volume-section-orbit-sylvester})}.
Our volume comparison theorem gives the Radon--Nikodym derivative
\begin{equation}
  \label{eq:intro-volume-density}
  \frac{\dd\Vol_{\mathrm E}}
  {\dd\Vol_Q}(C)
  =
  2^{(s-n)/2}\,
  \pdet(C)^{(n-s+1)/2}\,
  D_s(C)\sqrt{\det A_C},
\end{equation}
where \(\pdet(C)\coloneq \prod_{a=1}^s\lambda_a(C)\)
(\cref{thm:quotient-versus-ambient-volume}).
The formula accounts for both the spectrum and the position
of the range of \(C\) relative to the coordinate axes.
Its derivation combines the coarea formula
\cite{Federer1969,cibotaruNoteAreaCoarea2016} with an explicit
Gram-quotient Jacobian and the normal Jacobians of the
unit-row and unit-diagonal constraints.

Both volumes have finite positive total mass, so they define
normalized laws \(\mu_{n,s}^{Q}\) and \(\mu_{n,s}^{\mathrm E}\)
{(\cref{lem:finite-volume-comparison-section})}.
If \(\mathcal U_{M_s}\) denotes the uniform probability measure on the product of spheres $M_s$ (i.e., the mass-1-normalized volume measure),
these laws are Gram pushforwards of factor laws with densities
on \(\Mss\) given (up to renormalization constants) by
\begin{equation}
  \label{eq:intro-lifted-densities}
  \begin{aligned}
    \frac{\dd\nu_Q}{\dd\mathcal U_{M_s}}(X)
    &\propto
    D_s(XX^\top)^{-1},\\
    \frac{\dd\nu_{\mathrm E}}{\dd\mathcal U_{M_s}}(X)
    &\propto
    \pdet(XX^\top)^{(n-s+1)/2}\sqrt{\det A_{XX^\top}}.
  \end{aligned}
\end{equation}
These laws satisfy
\((\Gamma_s)_\#\nu_Q=\mu_{n,s}^{Q}\) and
\((\Gamma_s)_\#\nu_{\mathrm E}=\mu_{n,s}^{\mathrm E}\)
(\cref{thm:uniform-rank-r-sampler,cor:ambient-euclidean-volume-lift}).
In particular, for \(s\geq2\), the Gram matrix of independent
uniform sphere rows is uniform for quotient volume if and only
if \(s=2\)
(\cref{cor:raw-gram-not-uniform}).
These identities supply explicit densities on the factor space
for sampling the two geometric laws and for changing reference
measure in integrals on correlation strata. 

The comparison also quantifies behavior near rank loss.
When one positive eigenvalue \(\lambda\) tends to zero and the
remaining positive eigenvalues stay bounded below, the density given 
in \cref{eq:intro-volume-density} is comparable to
\(\lambda^{(n-s+1)/2}\)
(\cref{cor:normalized-ambient-quotient-density}).
Thus the two notions of uniformity have different densities
near the singular boundary, even when the curvature
denominators remain controlled. The volume and connection
calculations are linked by the same Sylvester operator:
its inverse determines horizontal correction and curvature,
while its determinant determines the volume of orthogonal
factor orbits.

\paragraph{Organization of the paper.}
\Cref{sec:quotient-model-rank-stratification} establishes the
quotient model and rank stratification, and
\cref{sec:rank-reduction-metric-reduction} proves the two
reduction theorems.
\Cref{sec:topology-fixed-rank-strata} gives explicit local
product descriptions and proves path connectivity.
\Cref{sec:quotient-metric-sylvester-connection} develops the
reduced connection expressed in terms of a Sylvester equation, metric compatibility, and finite-time
rank degeneration; \cref{sec:curvature-fixed-rank-strata}
derives the curvature formulas, quantitative bounds, and
nonpolarity; \cref{sec:holonomy-mechanical-connection} classifies the
holonomy groups.
\Cref{sec:intrinsic-geodesics-rank-constraints,sec:geodesics-constrained-hit-lower-rank}
treat intrinsic rank constraints and paths required to meet
lower rank, respectively.
\Cref{sec:excess-rank-constrained-geodesics} bounds the distance
excess caused by a rank constraint.
Finally, \cref{sec:volume-measures-correlation-strata} computes
the volume comparison and derives the corresponding uniform
laws and their factor densities.

\section{The quotient model and rank stratification}
\label{sec:quotient-model-rank-stratification}

The purpose of this section is to fix the basic quotient model and to show
that its orbit-type stratification is exactly the rank stratification of
low-rank correlation matrices.  This identification is the structural reason why rank constraints will be treated as stratified quotient constraints rather than as external algebraic constraints.

Throughout the paper, let $n\geq 2$ and $k\geq 2$ be fixed integers.  We set
\[
  r\coloneq \min(n,k),
\]
and use $s\in\{1,\ldots,r\}$ for an arbitrary rank.  The ambient space is the product of spheres
\[
  M_k\coloneq (\mathbb S^{k-1})^n
  =\Bigl\{X\in\R^{n\times k}: \|X_{i\cdot}\|_2=1,
  \ i=1,\ldots,n\Bigr\},
\]
where $X_{i\cdot}\in\R^k$ denotes the $i$-th row of $X$.  For later use, we record its tangent space:
\begin{equation}
\label{eq:tangent-space-product-spheres}
  \Tz{X}M_k
  =\Bigl\{V\in\R^{n\times k}:\;
  \langle V_{i\cdot},X_{i\cdot}\rangle_{\R^k}=0,
  \ i=1,\ldots,n\Bigr\}.
\end{equation}
The orthogonal group $O(k)$ acts on $M_k$ on the right by matrix multiplication
\begin{equation}
\label{eq:right-action}
  X\cdot O\coloneq XO,
  \qquad X\in M_k,\ O\in O(k).
\end{equation}
This action is smooth. Since $O(k)$ is compact, it is also proper. We denote the orbit map and the orbit space by
\[
  \pi_k:M_k\longrightarrow Q_k\coloneq M_k/O(k),
  \qquad \pi_k(X)=[X].
\]
We equip $M_k$ with the product of the standard round metrics on
$\mathbb S^{k-1}$, denoted by $g_{M_k}$. This is the restriction
of the Frobenius inner product to its tangent bundle. Explicitly, for
$X\in M_k$
and $U,V\in\Tz{X}M_k$,
\begin{equation}
\label{eq:product-sphere-metric}
  (g_{M_k})_X(U,V)
  \coloneq \sum_{i=1}^n
    \langle U_{i\cdot},V_{i\cdot}\rangle_{\R^k}
  =\tr(U^\top V)
  =\Frob{U}{V}.
\end{equation}
Thus $g_{M_k}$ is the metric induced by the Frobenius inner product
on $\R^{n\times k}$. The action in \cref{eq:right-action}
is isometric, since for every $O\in O(k)$,
\[
  (g_{M_k})_{XO}(UO,VO)
  =\tr(O^\top U^\top V O)
  =\tr(U^\top V)
  =(g_{M_k})_X(U,V).
\]
The metric consequences of this fact will be developed later;
in the present section we only use the smooth proper quotient
structure.

\subsection{Orbit-type conventions}
\label{subsec:orbit-type-conventions}

We recall the minimum amount of orbit-type terminology needed in the sequel.
Let a Lie group $G$ act smoothly and properly on a smooth manifold $M$.  For
$x\in M$, its isotropy subgroup is
\[
  G_x\coloneq \{g\in G:g\cdot x=x\}.
\]
If $H\leq G$ is a closed subgroup, we denote by $(H)$ its conjugacy class and set
\[
  M^{(H)}\coloneq \{x\in M:(G_x)=(H)\}.
\]
The slice theorem for proper actions implies that the connected components of the sets $M^{(H)}$ form the orbit-type stratification of $M$; the corresponding pieces in the quotient $M/G$ are the images of these sets under the orbit map. For compact Lie group actions, this stratification is Whitney; see, for instance, \cite{Palais1961,AKLM03,Pflaum2001,Sniatycki2013,AlexandrinoBettiol2015}.

In this paper, we shall use the notation $\Mks$ and $\Qks$ for the rank-$s$ orbit-type pieces. Strictly speaking, if one insists that a stratum be connected, then the strata are the connected components of the orbit-type pieces. This minor distinction will not matter until the topology of fixed-rank pieces is discussed in \cref{sec:topology-fixed-rank-strata}.

\subsection{The Gram map and the quotient model}
\label{subsec:gram-map-quotient-model}

Let
\[
  \Corr(n)
  \coloneq \{C\in\Sym(n): C\succeq 0,
  \ \diag(C)=\mathbf 1_n\}
\]
be the set of $n\times n$ \emph{correlation matrices}. We write the \emph{bounded-rank} and \emph{fixed-rank} subsets of correlation matrices respectively,
\[
  \Crle{k}{n}
  \coloneq \{C\in\Corr(n):\rank(C)\leq k\},
  \qquad
  \Crk{s}{n}
  \coloneq \{C\in\Corr(n):\rank(C)=s\}.
\]
The Gram map
\begin{equation}
\label{eq:gram-map}
  \Gamma_k:M_k\longrightarrow \Crle{k}{n},
  \qquad
  \Gamma_k(X)\coloneq XX^\top,
\end{equation}
is invariant under the right action of $O(k)$.  Hence it descends to a
continuous map
\begin{equation}
\label{eq:descended-gram-map}
  \overline\Gamma_k:Q_k\longrightarrow \Crle{k}{n},
  \qquad
  \overline\Gamma_k([X])=XX^\top.
\end{equation}
The following parametrization supplies both the local inverse of the Gram
map and the product charts used in the topological arguments of
\cref{sec:topology-fixed-rank-strata}. For \(C\in\R^{n\times n}\) and index sets
\(I,J\subseteq\{1,\ldots,n\}\), write
\(C_{IJ}\coloneq (C_{ij})_{i\in I,\,j\in J}\), with indices listed in
increasing order. In particular, \(C_{II}\) is the principal
submatrix indexed by \(I\).

\begin{lemma}[Cholesky--Schur parametrization]
\label{lem:cholesky-schur-charts}
Fix \(1\le s\le n\). For \(|I|=s\), set \(J\coloneq I^c\) and
\begin{equation}
\label{eq:VI-definition}
  V_I\coloneq \{C\in\Crk{s}{n}:\det(C_{II})>0\}.
\end{equation}
These sets form an open cover of \(\Crk{s}{n}\). Write
\(A=L(A)L(A)^\top\) for the Cholesky factorization, with \(L(A)\)
lower triangular and with positive diagonal. Then
\begin{equation}
\label{eq:VI-chart-map}
  \Phi_I:V_I\longrightarrow
  \Crk{s}{s}\times(\mathbb S^{s-1})^{\,n-s},
  \qquad
  \Phi_I(C)\coloneq \bigl(C_{II},C_{JI}L(C_{II})^{-\top}\bigr)
\end{equation}
is a diffeomorphism. In the block order \(I,J\), its inverse is
\begin{equation}
\label{eq:VI-chart-inverse}
  \Psi_I(A,Y)\coloneq
  \begin{pmatrix}
    A & L(A)Y^\top\\
    YL(A)^\top & YY^\top
  \end{pmatrix},
\end{equation}
where \(Y\) is an \((n-s)\times s\) matrix with unit rows.
In particular, the smooth map \(Z_I:V_I\to M_s\) given in this
block order by
\[
  Z_I(C)\coloneq
  \begin{pmatrix}
    L(C_{II})\\ C_{JI}L(C_{II})^{-\top}
  \end{pmatrix}
\]
has rank \(s\) and satisfies \(Z_I(C)Z_I(C)^\top=C\).
The smooth structure here is the embedded one in the space of symmetric
matrices; when \(n=s\), the empty spherical product is a point.
\end{lemma}

\begin{proof}
Each \(V_I\) is relatively open by continuity of the determinant.
For \(C=XX^\top\) of rank \(s\), choose \(s\) independent rows of
\(X\in\R^{n\times s}\); the corresponding principal block
\(C_{II}=X_I X_I^\top\) is positive definite. Thus the sets cover.

Fix \(C\in V_I\), and write \(A\coloneq C_{II}\), \(B\coloneq C_{JI}\), and \(D\coloneq C_{JJ}\)
in the block order \(I,J\). Block elimination gives
\[
  \begin{pmatrix}\I_s&0\\-BA^{-1}&\I_{n-s}\end{pmatrix}
  C
  \begin{pmatrix}\I_s&-A^{-1}B^\top\\0&\I_{n-s}\end{pmatrix}
  =
  \begin{pmatrix}A&0\\0&D-BA^{-1}B^\top\end{pmatrix}.
\]
Both multiplying matrices are invertible and
\(\rank(C)=\rank(A)=s\), so \(D=BA^{-1}B^\top\).
Consequently \(Y\coloneq BL(A)^{-\top}\) satisfies \(YY^\top=D\).
Since \(C\) has unit diagonal, \(A\in\Crk{s}{s}\) and every row
of \(Y\) has norm one. This proves that \(\Phi_I\) is well-defined
and that \(Z_I(C)Z_I(C)^\top=C\).

Conversely, for \(A\in\Crk{s}{s}\) and unit-row \(Y\), the matrix
\(Z\coloneq \bigl(L(A)^\top,Y^\top\bigr)^\top\) has unit rows and rank \(s\).
Hence \(ZZ^\top=\Psi_I(A,Y)\) belongs to \(V_I\).
The block formulas now give
\(\Phi_I\circ\Psi_I=\operatorname{Id}\) and
\(\Psi_I\circ\Phi_I=\operatorname{Id}\).
Cholesky factorization and inversion are smooth on the positive-definite
cone, so both formulas are smooth. More precisely, the formula for
\(\Phi_I\) extends smoothly to symmetric matrices with \(C_{II}\)
positive definite. Its composition with \(\Psi_I\) is the identity,
so \(\Psi_I\) is an immersion with continuous inverse onto \(V_I\),
hence an embedding. The cover therefore gives the asserted embedded
smooth structure without using the quotient identification.
\end{proof}

The following proposition makes the quotient interpretation exact.

\begin{proposition}[Quotient model for low-rank correlations]
\label{prop:quotient-model-low-rank-correlations}
The descended Gram map in \cref{eq:descended-gram-map} is a homeomorphism
\[
  \overline\Gamma_k:Q_k\xrightarrow{\ \cong\ }
  \Crle{k}{n}.
\]
Moreover, for every $s\in\{1,\ldots,r\}$, it restricts to a diffeomorphism
\[
  \overline\Gamma_k:
  \Qks
  \xrightarrow{\ \cong\ }
  \Crk{s}{n},
\]
where $\Qks$ is equipped with its orbit-type smooth structure and
$\Crk{s}{n}$ with its standard smooth structure as a fixed-rank correlation stratum. Consequently, $Q_k=M_k/O(k)$ may be identified up to homeomorphism with the space of correlation matrices of rank at most $k$, and the rank of a point $[X]\in Q_k$ is the rank of $X$, equivalently the rank of $XX^\top$.
\end{proposition}

\begin{proof}
The map is well-defined because \((XO)(XO)^\top=XX^\top\) for every $O\in O(k)$. It is surjective: if $C\in\Crle{k}{n}$, choose a factorization $C=XX^\top$ with $X\in\R^{n\times k}$, after padding by zero columns if necessary. Since $C_{ii}=1$, each row of $X$ has Euclidean norm one, hence $X\in M_k$. It remains to identify the fibers. Suppose $X,Y\in M_k$ satisfy $XX^\top=YY^\top$. Let
\[
  R_X\coloneq \spanv\{X_{1\cdot},\ldots,X_{n\cdot}\}\subset\R^k,
  \qquad
  R_Y\coloneq \spanv\{Y_{1\cdot},\ldots,Y_{n\cdot}\}\subset\R^k.
\]
The equality of Gram matrices says that the assignment
$X_{i\cdot}\mapsto Y_{i\cdot}$ preserves all pairwise inner products. In particular, \(\dim R_X=\dim R_Y=\rank(XX^\top)\). Hence the assignment extends to a linear isometry $A:R_X\to R_Y$. Extending $A$ to an element $O\in O(k)$ gives $XO=Y$. Therefore the fibers of $\Gamma_k$ are precisely the $O(k)$-orbits. Since $M_k$ is compact and $\Crle{k}{n}$ is Hausdorff, the continuous bijection \(\overline\Gamma_k:Q_k\to\Crle{k}{n}\) is a homeomorphism.

For the rankwise statement,
\(\Gamma_k:\Mks\to\Crk{s}{n}\), \(X\mapsto XX^\top\),
is smooth and \(O(k)\)-invariant, hence descends smoothly to \(\Qks\).
On each \(V_I\), \cref{lem:cholesky-schur-charts} provides the smooth
full-rank factor \(Z_I(C)\). Padding by \(k-s\) zero columns and
projecting to the quotient gives the smooth map
\[
  V_I\longrightarrow\Qks,\qquad C\longmapsto[(Z_I(C),0)].
\]
Its Gram matrix is \(C\), so it is the inverse of \(\overline\Gamma_k\)
on \(V_I\). These sets cover \(\Crk{s}{n}\), proving that the
rankwise Gram identification is a diffeomorphism.
\end{proof}

\begin{remark}[Smooth structure on the rank pieces]
\label{rem:smooth-structure-rank-pieces}
The quotient pieces $\Qks$ are smooth manifolds by the orbit-type theorem. Via \cref{prop:quotient-model-low-rank-correlations}, we shall freely regard $\Crk{s}{n}$ as the corresponding smooth quotient piece. In other words, the smooth structure used on a fixed-rank correlation stratum is the one transported via diffeomorphism from the orbit-type manifold $\Qks$.
\end{remark}

\subsection{Isotropy groups and rank orbit types}
\label{subsec:isotropy-rank-orbit-types}

We now compute the isotropy groups of the right action
\[
M_k\times O(k)\to M_k,\qquad (X,O)\mapsto XO,
\]
where \(M_k\coloneq (\mathbb S^{k-1})^n\). This computation is the key point: the conjugacy class of the isotropy subgroup depends only on the rank. When \(X\in M_k\) has full column rank \(k\), this action is free at \(X\). Indeed, if \(XO=X\), then \(X(O-\I_k)=0\) and since \(X\) has trivial kernel, it follows that \(O=\I_k\). In general,
however, matrices \(X\in M_k\) have nontrivial isotropy:
\[
O(k)_X\coloneq \{O\in O(k): XO=X\}.
\]

As observed in \cite{Thanwerdas2023,Chen2025}, the next proposition computes the isotropy subgroup explicitly and shows that it is determined, up to conjugacy, by the rank. 

\begin{proposition}[Isotropy group]
\label{prop:isotropy-group-rank}
Let \(X\in M_k\) and set \(s\coloneq \rank(X)\). Then \(O(k)_X\) is conjugate in \(O(k)\)
to the standard subgroup
\begin{equation}
\label{eq:standard-isotropy-rank-s}
H_s\coloneq \I_s\oplus O(k-s)\subset O(k).
\end{equation}
More precisely, if \(X=U\Sigma V^\top\), \(U\in O(n)\), \(V\in O(k)\) is a singular value decomposition of \(X\), with
\[
\Sigma=
\begin{pmatrix}
D & 0\\
0 & 0
\end{pmatrix},
\qquad
D=\diag(\sigma_1,\dots,\sigma_s),
\qquad
\sigma_i>0,
\]
then \(O(k)_X=V\,(\I_s\oplus O(k-s))\,V^\top\).
\end{proposition}

\begin{proof}
Let \(O\in O(k)\). Then
\[
O\in O(k)_X
\quad\Longleftrightarrow\quad
XO=X
\quad\Longleftrightarrow\quad
U\Sigma V^\top O=U\Sigma V^\top
\quad\Longleftrightarrow\quad
\Sigma W=\Sigma,
\]
where we set \(W\coloneq V^\top O V\in O(k)\). Writing \(W\) in blocks compatible with \(\Sigma\),
\[
W=
\begin{pmatrix}
A & B\\
C & E
\end{pmatrix},
\qquad
A\in\R^{s\times s},\quad
B\in\R^{s\times (k-s)},
\]
the identity \(\Sigma W=\Sigma\) becomes
\[
\begin{pmatrix}
D & 0\\
0 & 0
\end{pmatrix}
\begin{pmatrix}
A & B\\
C & E
\end{pmatrix}
=
\begin{pmatrix}
D & 0\\
0 & 0
\end{pmatrix}.
\]
Hence \(DA=D\), \(DB=0\). Since \(D\) is invertible, this forces \(A=\I_s\), \(B=0\). Therefore \(W\) has the form
\[
W=
\begin{pmatrix}
\I_s & 0\\
C & E
\end{pmatrix}.
\]
Using the orthogonality relation \(W^\top W=\I_k\), the upper-left block gives \(\I_s+C^\top C = \I_s\), hence \(C=0\), and the lower-right block gives \(E\in O(k-s)\). Thus
\[
W=\I_s\oplus E = \begin{pmatrix}
\I_s & 0\\
0 & E
\end{pmatrix}
\text{ with }E\in O(k-s).
\]
Undoing the change of basis \(W=V^\top O V\), we obtain
\[
O(k)_X
=
\{\,V(\I_s\oplus E)V^\top : E\in O(k-s)\,\}
=
V\,(\I_s\oplus O(k-s))\,V^\top.
\]
This proves the claim.
\end{proof}

In view of \cref{prop:isotropy-group-rank}, the isotropy subgroup of a point
\(X\in M_k\) is determined, up to conjugacy in \(O(k)\), by the rank of \(X\). This already suggests that the orbit type decomposition of \(M_k\) should coincide with its rank decomposition. The next theorem makes this precise, both upstairs on \(M_k\) and downstairs on the quotient \(Q_k\coloneq M_k/O(k)\). It also shows that, through the Gram-map identification, the same decomposition corresponds exactly to the fixed-rank subsets of the space of low-rank correlation matrices.

\begin{theorem}[Orbit type stratification equals rank stratification]
\label{thm:orbit-type-equals-rank-stratification}
For the right action of \(O(k)\) on \(M_k=(\mathbb S^{k-1})^n\), the
orbit-type pieces are exactly the rank level sets
\begin{equation}
\label{eq:rank-pieces-upstairs}
\Mks
\coloneq 
\{X\in M_k:\rank(X)=s\},
\qquad
s=1,\dots,r,
\end{equation}
where \(r=\min(n,k)\). The induced orbit-type pieces in the quotient \(Q_k\coloneq M_k/O(k)\) are therefore
\begin{equation}
\label{eq:rank-pieces-downstairs}
\Qks
\coloneq 
\pi_k(\Mks)
=
\{[X]\in Q_k:\rank(X)=s\},
\qquad
s=1,\dots,r.
\end{equation}
Under the identification \(\overline\Gamma_k:Q_k\longrightarrow \Crle{k}{n}\),  these pieces correspond exactly to the fixed-rank correlation strata \(\Qks\cong \Crk{s}{n}\). In particular,
\[
Q_k=\bigsqcup_{s=1}^r \Qks,
\qquad
\Crle{k}{n}=\bigsqcup_{s=1}^r \Crk{s}{n},
\]
and these two stratifications coincide through the quotient model.
\end{theorem}

\begin{proof}
Let \(X,Y\in M_k\), with \(\rank(X)=s\) and \(\rank(Y)=t\). By
\cref{prop:isotropy-group-rank}, the isotropy subgroups \(O(k)_X\) and \(O(k)_Y\)
are conjugate respectively to
\[
H_s\coloneq \I_s\oplus O(k-s),
\qquad
H_t\coloneq \I_t\oplus O(k-t).
\]
If \(s=t\), then both isotropy subgroups are conjugate to \(H_s\), hence \(X\) and \(Y\) have the same orbit type. Conversely, suppose that \(O(k)_X\) and \(O(k)_Y\) are conjugate. Then \(H_s\) and \(H_t\) are conjugate in \(O(k)\). Conjugate subgroups have fixed-point subspaces of the same dimension. However,
\[
  \operatorname{Fix}(H_s)
  \coloneq \{v\in\R^k:h v=v\text{ for every }h\in H_s\}
  =\R^s\oplus\{0\},
\]
so \(\dim\operatorname{Fix}(H_s)=s\). Similarly, \(\dim\operatorname{Fix}(H_t)=t\). Hence \(s=t\). Therefore two points of \(M_k\) have the same orbit type if and only if they have the same rank, which proves that the orbit-type pieces upstairs are
exactly the sets in \cref{eq:rank-pieces-upstairs}. Passing to the quotient gives \cref{eq:rank-pieces-downstairs}. The final assertion follows from \cref{prop:quotient-model-low-rank-correlations} together with \(\rank(XX^\top)=\rank(X)\).
\end{proof}

For simplicity, we shall systematically write \(\Mks\) and \(\Qks\) for these orbit-type strata.

\begin{definition}[Rank-bounded pieces]
\label{def:rank-bounded-pieces}
For \(1\le s\le r\), we define
\[
M_k^{\le s}\coloneq \{X\in M_k:\rank(X)\le s\},
\qquad
\Qkles\coloneq \pi_k(M_k^{\le s})
=\bigcup_{j=1}^s \Qkj.
\]
On the correlation side, this corresponds to
\[
\Crle{s}{n}
\coloneq 
\{C\in\Corr(n):\rank(C)\le s\}
=
\bigcup_{j=1}^s \Crk{j}{n}.
\]
When \(s\le k\), the identification \(Q_k\cong \Crle{k}{n}\) restricts to \(\Qkles\cong \Crle{s}{n}\).
\end{definition}

\begin{corollary}[Frontier relation]
\label{cor:rank-frontier-relation}
For every \(s\in\{1,\dots,r\}\),
\[
\overline{\Mks}=M_k^{\le s},
\qquad
\overline{\Qks}=\Qkles,
\qquad
\overline{\Crk{s}{n}}=\Crle{s}{n},
\]
where topological closures are taken respectively in \(M_k\), \(Q_k\), and
\(\Crle{k}{n}\). In particular, the principal orbit-type stratum is the maximal-rank stratum
\[
\Mrk{k}{r}\subset M_k,
\qquad
\Qrk{k}{r}\subset Q_k,
\]
and it is open and dense.
\end{corollary}

The proof is explicit.

\begin{proof}
The rank bound is closed, so
\(\overline{\Mks}\subset M_k^{\le s}\).
Conversely, fix \(X\in M_k^{\le s}\) and choose an
\(s\)-dimensional subspace \(W\subset\R^k\)
containing its row space. Since \(s\le n\), rank-\(s\)
matrices are dense in \(\R^{n\times s}\).
Approximate \(X\) by such matrices and normalize their
rows, which are nonzero for sufficiently close
approximations. Normalization preserves rank and is
continuous near \(X\), proving
\(\overline{\Mks}=M_k^{\le s}\).

The other closure identities follow because \(\pi_k\)
is continuous and closed and \(\overline\Gamma_k\) is
a homeomorphism. The maximal-rank strata are dense by
the case \(s=r\), and open because their complements
are closed lower-rank loci. By
\cref{thm:orbit-type-equals-rank-stratification},
they are the principal orbit-type strata.
\end{proof}

\section{Rank reduction and metric reduction}
\label{sec:rank-reduction-metric-reduction}

The quotient model of \cref{sec:quotient-model-rank-stratification}
comes with two reduction principles of different nature. The first
one is a \emph{smooth} statement: every fixed-rank stratum is a
Riemannian quotient manifold isometric to a full-rank quotient in the
smaller ambient dimension determined by the rank. The second one is a
\emph{singular} statement: the whole rank-bounded subset, endowed
with its intrinsic
path metric, is isometric to the ambient quotient space in that
smaller dimension. These two theorems will play different roles
throughout the paper. The fixed-rank reduction will be the basic
device for all local
differential-geometric computations, whereas the rank-bounded
reduction will be the basic device for the metric theory of rank constraints.

\subsection{Metric conventions and padding}
\label{subsec:metric-conventions}

Throughout this section, let
\[
  1\le s\le r=\min(n,k),
\]
and consider the padding map
\begin{equation}
  \label{eq:padding_map_section2}
  \iks:\R^{n\times s}\longrightarrow \R^{n\times k},
  \qquad
  \iks(Z)\coloneq(Z,0),
\end{equation}
obtained by appending $k-s$ zero columns. Since row norms are preserved,
\(\iks(M_s)\subset M_k\), and \cref{eq:padding_map_section2} induces
a map on orbit spaces
\begin{equation}
  \label{eq:Xi_ks_section2}
  \Xks:Q_s\longrightarrow Q_k,
  \qquad
  \Xks([Z])\coloneq[(Z,0)].
\end{equation}
Because padding does not change rank, one has
\[
  \Xks(\Qrk{s}{j})\subset \Qkj,
  \qquad
  \Xks(\Qle{s}{j})\subset \Qle{k}{j},
  \qquad 1\le j\le s.
\]

\begin{remark}[Geometric meaning of padding]
  \label{rem:padding-normalizer-interpretation}
  At rank $s$, padding realizes the standard normalizer reduction
  of the $O(k)$-action: its image is the locus with stabilizer
  exactly $H_s=\I_s\oplus O(k-s)$, and the residual group
  $N_{O(k)}(H_s)/H_s\simeq O(s)$ acts freely on this locus.
  The resulting principal bundle and its mechanical connection
  identify with the reduced full-rank bundle and the reduced
  mechanical connection; see \cref{subsec:padding-normalizer-reduction},
  where the normalizer and this construction are made explicit.
\end{remark}

\begin{definition}[Product metric and quotient distance]
  \label{def:factor-quotient-metrics}
  For \(m\ge2\), the product-of-spheres metric on \(M_m\) is
  \[
    (g_{M_m})_X(V,W)\coloneq \Frob{V}{W},
    \qquad V,W\in\Tz{X}M_m.
  \]
  We denote by $\dM{m}$ the product Riemannian distance
  and by $\dQ{m}$ the quotient distance on $Q_m=M_m/O(m)$. Concretely,
  \begin{equation}
    \label{eq:factor-and-quotient-distances}
    \begin{aligned}
      \dM{m}(X,Y)^2
      &=\sum_{i=1}^n
      \arccos^2\!\bigl(\langle X_{i\cdot},Y_{i\cdot}\rangle\bigr),\\
      \dQ{m}([X],[Y])&=\min_{R\in O(m)}\dM{m}(X,YR).
    \end{aligned}
  \end{equation}
\end{definition}
The minimum exists by compactness of $O(m)$; see also
\cref{prop:reduced-quotient-global-metric-geometry}.

\begin{definition}[Curve length and rectifiability]
  \label{def:curve-lengths}
  A continuous curve $\alpha:[0,1]\to(Y,d)$ in a metric
  space is called \emph{rectifiable} if its length
  \[
    \Length_d(\alpha)
    \coloneq\sup_{\mathcal P}
    \sum_{i=1}^{m}
    d\bigl(\alpha(t_{i-1}),\alpha(t_i)\bigr)
  \]
  is finite, where the supremum runs over all finite partitions
  $\mathcal P=\{0=t_0<t_1<\cdots<t_m=1\}$.
  On a connected Riemannian manifold $(M,g)$, we use the associated
  Riemannian distance $\dg{g}$ and write
  $\Length_M(\alpha)\coloneq\Length_{\dg{g}}(\alpha)$.
  For a continuous piecewise $C^1$ curve, this agrees with its
  Riemannian length:
  \[
    \Length_M(\alpha)=\Length_g(\alpha)
    \coloneq\int_0^1
    \sqrt{g_{\alpha(t)}
    \bigl(\dot\alpha(t),\dot\alpha(t)\bigr)}\dd t.
  \]
\end{definition}

\begin{lemma}[Padding is a totally geodesic isometric embedding]
  \label{lem:padding-isometric-embedding}
  Let $2\le s\le k$. The padding map
  \[
    \iks:(M_s,g_{M_s})\longrightarrow(M_k,g_{M_k}),
    \qquad
    \iks(Z)=(Z,0),
  \]
  is a totally geodesic Riemannian isometric embedding.
  Consequently, for every rectifiable curve
  $\alpha:[0,1]\to M_s$,
  \[
    \Length_{M_k}(\iks\circ\alpha)
    =
    \Length_{M_s}(\alpha).
  \]
\end{lemma}

\begin{proof}
  The padding map is a smooth embedding: its differential
  \[
    \dd\iks(Z)[V]=(V,0)
  \]
  is injective, and projection onto the first $s$ columns gives
  a continuous inverse on its image. Moreover, for
  $V,W\in\Tz{Z}M_s$,
  \[
    (g_{M_k})_{\iks(Z)}
    \bigl(\dd\iks(Z)[V],\dd\iks(Z)[W]\bigr)
    =\Frob{(V,0)}{(W,0)}
    =(g_{M_s})_Z(V,W).
  \]
  Thus $(\iks)^*g_{M_k}=g_{M_s}$. To prove total geodesicity,
  consider the orthogonal reflection
  \[
    J\coloneq\diag(\I_s,-\I_{k-s})\in O(k).
  \]
  The map $\sigma:M_k\to M_k$, $\sigma(X)\coloneq XJ$, is an isometry,
  and its fixed-point set is precisely $\iks(M_s)$.
  Let $\gamma$ be a geodesic in $M_k$ whose initial point belongs
  to $\iks(M_s)$ and whose initial velocity is tangent to this
  submanifold. Since $\sigma$ fixes $\iks(M_s)$ pointwise,
  $\gamma$ and $\sigma\circ\gamma$ have the same initial point
  and velocity. Uniqueness of geodesics (with prescribed initial
  point and velocity) therefore gives
  $\sigma\circ\gamma=\gamma$, so $\gamma$ remains in $\iks(M_s)$.
  Hence $\iks(M_s)$ is totally geodesic. Finally, the distance
  formula for a product of unit spheres gives,
  for every $Z,W\in M_s$,
  \begin{align*}
    \dM{k}\bigl(\iks(Z),\iks(W)\bigr)^2
    &=
    \sum_{i=1}^n
    \arccos^2\!\bigl(
      \langle(Z_{i\cdot},0),(W_{i\cdot},0)\rangle
    \bigr)\\
    &=
    \sum_{i=1}^n
    \arccos^2\!\bigl(
      \langle Z_{i\cdot},W_{i\cdot}\rangle
    \bigr)
    =\dM{s}(Z,W)^2.
  \end{align*}
  Thus padding preserves every partition sum in the definition
  of curve length. Taking the supremum over partitions yields
  \(\Length_{M_k}(\iks\circ\alpha) =\Length_{M_s}(\alpha)\) for every
  rectifiable curve $\alpha:[0,1]\to M_s$.
\end{proof}

\begin{figure}[h!]
  \centering
  \includegraphics[width=.72\linewidth]{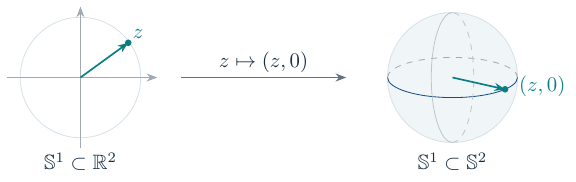}
  \caption{Padding one row for $s=2$ and $k=3$:
    the map $z\mapsto(z,0)$ embeds $\mathbb S^1$
  isometrically as an equator of $\mathbb S^2$.}
  \label{fig:padding-factors}
\end{figure}

\subsection{Fixed-rank Riemannian reduction}
\label{subsec:fixed-rank-riemannian-reduction}

Equip each rank piece $\Mks\subset M_k$ with the induced
Riemannian metric. By
\cref{thm:orbit-type-equals-rank-stratification}, the restricted
$O(k)$-action is smooth, proper, isometric, and has a single
orbit type. Consequently, $\Qks$ is a smooth manifold and
the quotient projection is a smooth submersion
\cite[\S2.2]{AKLM03}.

\begin{definition}[Quotient Riemannian metric and distance]
  \label{def:fixed-rank-quotient-metric-distance}
  The quotient Riemannian metric,
  denoted by $\gQks$, is the unique metric making
  \[
    \pi_k:
    \bigl(\Mks,\restr{g_{M_k}}{\Mks}\bigr)
    \longrightarrow
    \bigl(\Qks,\gQks\bigr)
  \]
  a Riemannian submersion. Explicitly, the metric-horizontal space at
  $X\in\Mks$ is
  \[
    \Hz{X}
    \coloneq\Tz{X}\Mks\cap
    \bigl(X\so(k)\bigr)^{\perp_F}.
  \]
  For $u,v\in\Tz{[X]}\Qks$, let
  $u_X^\#,v_X^\#\in\Hz{X}$ be their unique horizontal lifts
  under $\dd\pi_k(X)$. Then
  \begin{equation}
    \label{eq:fixed-rank-quotient-metric}
    (\gQks)_{[X]}(u,v)
    =(g_{M_k})_X(u_X^\#,v_X^\#)
    =\Frob{u_X^\#}{v_X^\#}.
  \end{equation}
  The associated intrinsic Riemannian distance is
  \begin{equation}
    \label{eq:fixed-rank-quotient-distance}
    \dQks(p,q)
    \coloneq\inf_{\gamma}
    \int_0^1
    \sqrt{(\gQks)_{\gamma(t)}
    \bigl(\dot\gamma(t),\dot\gamma(t)\bigr)}\dd t,
  \end{equation}
  where the infimum runs over all continuous piecewise $C^1$ curves
  $\gamma:[0,1]\to\Qks$ with $\gamma(0)=p$ and $\gamma(1)=q$,
  with the convention $\inf\varnothing=+\infty$. The same
  construction defines $\gQss$ and $\dQss$ on $\Qss$.
\end{definition}

As in \cref{def:curve-lengths}, we write
$\Length_{M_k}\coloneq\Length_{\dM{k}}$ for lengths
computed using the product Riemannian distance on $M_k$,
and $\Length_{\dQ{k}}$ for lengths computed using the quotient
distance $\dQ{k}$ on $Q_k$.
The first reduction theorem shows that the Riemannian geometry
of the rank-$s$ stratum depends only on the pair $(n,s)$.

\begin{theorem}[Fixed-rank Riemannian reduction]
  \label{thm:fixed-rank-riemannian-reduction}
  Let \(1\le s\le r\). The map \(\Xks\) takes \(\Qss\) onto \(\Qks\) and
  restricts to a Riemannian isometry
  \[
    \Xks:\bigl(\Qss,\gQss\bigr)
    \xrightarrow{\ \cong\ }
    \bigl(\Qks,\gQks\bigr).
  \]
  Equivalently, it is an isometry of the corresponding Riemannian
  distance spaces:
  \[
    \Xks:\bigl(\Qss,\dQss\bigr)
    \xrightarrow{\ \cong\ }
    \bigl(\Qks,\dQks\bigr).
  \]
\end{theorem}

\begin{proof}
  We proceed in three steps.

  \smallskip
  \noindent\emph{Step 1: \(\Xks\) is well-defined and bijective from
    \(\Qss\) onto
  \(\Qks\).}

  Consider the map
  \begin{equation*}
    \phi \coloneq \pi_k\circ \iks : X \in \Mss \longmapsto [\iks(X)]\in \Qks.
  \end{equation*}
  Since $O(s)\oplus \left\{ \I_{k-s} \right\} \leq O(k)$,
  the function $\phi$ is $O(s)$-invariant. In particular it descends
  to a well-defined quotient map $\bar \phi \colon \Mss /  O(s) \to
  \Qks$. Let \([X]\in \Qks\) and choose a representative \(X\in \Mks\).
  Its row span
  \(R_X\coloneq\spanv\{X_{1\cdot},\ldots,X_{n\cdot}\}\subset
  \R^k\) has dimension \(s\). Choose an orthonormal basis of
  \(R_X\), extend it to an orthonormal basis of \(\R^k\), and
  let \(U\in O(k)\) be the orthogonal matrix whose first \(s\)
  columns are the chosen basis of \(R_X\); equivalently,
  \(\spanv(U_{\cdot\{1,\ldots,s\}})=R_X\). Then \(XU\) has last
  \(k-s\) columns equal to zero, so we may write
  \(XU=(Z,0)\) with \(Z\in \R^{n\times s}\). Since right
  multiplication by \(U\) preserves row norms, \((Z,0)\in
  M_k\), and since \(\rank(XU)=\rank(X)=s\)
  we have \(\rank(Z)=s\), i.e., \(Z\in \Mss\).
  Therefore every class in \(\Qks\) is of the form \(\Xks([Z])\), so
  \(\Xks\) is surjective onto \(\Qks\).

  For injectivity, assume \(\Xks([Z])=\Xks([Z'])\) with \(Z,Z'\in
  \Mss\). Then there exists \(O\in O(k)\) such that
  \((Z,0)O=(Z',0)\). Write \(O\) in blocks relative to \(\R^k
  =\R^s\oplus \R^{k-s}\):
  \[
    O=
    \begin{pmatrix}
      A & B\\
      C & D
    \end{pmatrix}.
  \]
  Then \((Z,0)O=(ZA,ZB)=(Z',0)\), so \(ZB=0\). Since \(Z\) has full
  column rank \(s\), this forces \(B=0\). Orthogonality of \(O\) then
  implies \(A\in O(s)\), and therefore \(Z'=ZA\). Thus \([Z']=[Z]\)
  in \(\Qss=\Mss/O(s)\). Hence \(\Xks\) is bijective from \(\Qss\)
  onto \(\Qks\).

  Finally, \(\iks\) is smooth and equivariant in the sense that for every
  \(A\in O(s)\),
  \[
    \iks(ZA)=\iks(Z)(A\oplus \I_{k-s}).
  \]
  Therefore \(\Xks\) is the induced smooth map on quotient manifolds.

  \smallskip
  \noindent\emph{Step 2: preservation of the quotient Riemannian metrics.}
  Let \(Z\in \Mss\) and set \(\widehat Z\coloneq\iks(Z)=(Z,0)\in \Mks\).
  Since we have:
  \[
    \pi_k\circ \iks=\Xks\circ \pi_s,
  \]
  differentiation yields
  \begin{equation}
    \label{eq:chain_rule_fixed_rank_reduction}
    \dd \pi_k(\widehat Z)\circ \dd \iks(Z)
    =
    \dd \Xks([Z])\circ \dd \pi_s(Z).
  \end{equation}
  Let \(u\in \Tz{[Z]}\Qss\), and let
  \(u_Z^\#\in \Tz{Z}\Mss\) be its unique horizontal lift at \(Z\) for
  \(\pi_s\). Define \(\widehat u\coloneq\dd \iks(Z)(u_Z^\#)=(u_Z^\#,0)\).
  Then \(\widehat u\in \Tz{\widehat Z}\Mks\), and by
  \cref{eq:chain_rule_fixed_rank_reduction}, \(\dd \pi_k(\widehat
  Z)(\widehat u)=\dd \Xks([Z])(u)\). Thus \(\widehat u\) is a lift of
  \(\dd \Xks({[Z]})(u)\). We now
  check that \(\widehat u\) is horizontal. A vector \(W\in
  \Tz{\widehat Z}\Mks\) is horizontal for the quotient map
  \(\pi_k:\Mks\to \Qks\) if and only if it is Frobenius orthogonal to
  \(\widehat Z\,\so(k)\), equivalently if and only if
  \(W^\top \widehat Z\) is symmetric. Writing \(W=(W_1,W_2)\) with
  \(W_1\in \R^{n\times s}\) and \(W_2\in \R^{n\times
  (k-s)}\), one has
  \[
    W^\top\widehat Z=
    \begin{pmatrix}
      W_1^\top Z & 0\\
      W_2^\top Z & 0
    \end{pmatrix}.
  \]
  Hence symmetry is equivalent to
  \[
    W_1^\top Z=Z^\top W_1,
    \qquad
    W_2^\top Z=0.
  \]
  For \(W=\widehat u=(u_Z^\#,0)\), the second condition is automatic,
  and the first is
  exactly the horizontality condition of \(u_Z^\#\) for the principal
  $O(s)$-bundle \( \pi_s : \Mss\to \Qss\). Thus
  \(\widehat u\) is horizontal at \(\widehat Z\), and by
  uniqueness of horizontal lifts it is the horizontal lift of
  \(\dd\Xks([Z])(u)\).

  Finally, \(\iks\) is an isometric embedding for the Frobenius
  metric, hence \(\Frob{\widehat u}{\widehat u} = \Frob{
  u_Z^\#}{u_Z^\#}\). By the defining property of the quotient metric,
  \[
    (\gQks)_{\Xks([Z])}\bigl(\dd \Xks([Z])(u),\dd \Xks([Z])(u)\bigr)
    =
    \Frob{\widehat u}{\widehat u}
    =
    \Frob{u_Z^\#}{u_Z^\#}
    =
    \gQss(u,u).
  \]
  Therefore \(\dd \Xks\) is an isometry on each tangent space, so
  \(\Xks\) is a Riemannian isometry.

  \smallskip
  \noindent\emph{Step 3: conclusion at the level of distances.}
  A Riemannian isometry preserves lengths of smooth curves and
  therefore preserves the
  induced Riemannian distances. In particular,
  \[
    \Xks:\bigl(\Qss,\dQss\bigr)
    \xrightarrow{\ \cong\ }
    \bigl(\Qks,\dQks\bigr)
  \]
  is an isometry of metric spaces.
\end{proof}

The fixed-rank reduction immediately transfers to correlation
matrices of fixed rank through
\cref{prop:quotient-model-low-rank-correlations}.

\begin{corollary}[Fixed-rank reduction on correlation strata]
  \label{cor:fixed-rank-reduction-correlation}
  Let \(1\le s\le r\). Via the diffeomorphisms
  \[
    \Qss\cong \Crk{s}{n}
    \qquad\text{and}\qquad
    \Qks\cong \Crk{s}{n},
  \]
  the quotient metrics induced from \(\Qss\) and \(\Qks\) coincide.
  In particular, every fixed-rank geometric computation on
  \(\Crk{s}{n}\) may be performed isometrically in the reduced
  full-rank model \(\Qss\).
\end{corollary}

\begin{proof}
  This is just \cref{thm:fixed-rank-riemannian-reduction} transported
  through the
  rankwise Gram diffeomorphisms of
  \cref{prop:quotient-model-low-rank-correlations}.\qedhere
\end{proof}

We also record the following stratumwise compatibility, which is an
immediate consequence of
\cref{thm:fixed-rank-riemannian-reduction} and will be useful
throughout the sequel.

\begin{lemma}[Stratumwise compatibility of the padding map]
  \label{lem:stratumwise-compatibility-padding}
  Let \(1\le j\le s\le r\). Then
  \[
    \restr{\Xks}{\Qrk{s}{j}}
    =
    \padXi{j}{k}\circ \padXi{j}{s}^{-1}:
    \Qrk{s}{j}\longrightarrow \Qkj.
  \]
  In particular, the restriction of \(\Xks\) to \(\Qrk{s}{j}\) is a Riemannian
  isometry.
\end{lemma}

\begin{proof}
  Let \([Y]\in \Qrk{s}{j}\). By
  \cref{thm:fixed-rank-riemannian-reduction}, there exists a
  unique \([Z]\in \Qrk{j}{j}\) such that
  \[
    [Y]=\padXi{j}{s}([Z])=[(Z,0)]\in \Qrk{s}{j}.
  \]
  Then
  \[
    \Xks([Y])
    =
    \Xks(\padXi{j}{s}([Z]))
    =
    [(Z,0,0)]
    =
    \padXi{j}{k}([Z]).
  \]
  This proves the identity. The Riemannian isometry statement follows because
  \(\padXi{j}{s}\) and \(\padXi{j}{k}\) are Riemannian isometries.
\end{proof}

\subsection{Global rank-bounded metric reduction}
\label{subsec:global-rank-bounded-metric-reduction}

We now pass from the smooth fixed-rank reduction to its singular
metric counterpart.
\begin{definition}[Intrinsic distances under rank constraints]
  \label{def:intrinsic-rank-constraint-distances}
  Let \(1\le s\le r\). The intrinsic length distances induced by
  \(\dQ{k}\) on \(\Qkles\) and \(\Qks\) are, respectively,
  \begin{equation}
    \label{eq:intrinsic-rank-bounded-distance}
    \dkles(\bar p,\bar q)
    \coloneq\inf_{\substack{\bar\gamma:[0,1]\to\Qkles\\
    \bar\gamma(0)=\bar p,\ \bar\gamma(1)=\bar q}}
    \Length_{\dQ{k}}(\bar\gamma)
  \end{equation}
  and
  \begin{equation}
    \label{eq:intrinsic-stratum-distance-induced-by-dk}
    \dInt{k}{s}(\bar p,\bar q)
    \coloneq\inf_{\substack{\bar\gamma:[0,1]\to\Qks\\
    \bar\gamma(0)=\bar p,\ \bar\gamma(1)=\bar q}}
    \Length_{\dQ{k}}(\bar\gamma).
  \end{equation}
  Both infima range over continuous \(\dQ{k}\)-rectifiable curves with
  the indicated endpoints and images, with the convention
  \(\inf\varnothing=+\infty\). In each case, length is measured using
  the ambient quotient distance \(\dQ{k}\); the constraint specifies
  which curves are admissible.
\end{definition}

The second reduction theorem identifies \((\Qkles,\dkles)\)
with the quotient space \((Q_s,\dQ{s})\). The proof uses two standard
external metric inputs on submetries and semialgebraic curves, which we state
explicitly. We recall that a map is called a \emph{local submetry}
if, around every point, it maps sufficiently small metric balls onto
metric balls of the same radius centered at the image point.

\begin{theorem}[Length-preserving lifts for quotient submetries]
  \label{thm:length-preserving-lifts-submetry}
  Let \(P:X\to Y\) be a local submetry between proper metric spaces.
  Then every rectifiable curve \(\eta:[0,1]\to Y\) admits a lift with
  prescribed initial point and the same length: for every \(x_0\in
  X\) with \(P(x_0)=\eta(0)\), there exists a rectifiable curve
  \(\widetilde\eta:[0,1]\to X\) such that
  \[
    \widetilde\eta(0)=x_0,
    \qquad
    P\circ\widetilde\eta=\eta,
  \]
  and
  \[
    \Length_X(\widetilde\eta)
    =
    \Length_Y(\eta).
  \]
  In particular, if a compact Lie group \(G\) acts isometrically on a
  compact Riemannian manifold \(M\), then the quotient projection
  \[
    \pi:M\to M/G
  \]
  has this lifting property.
\end{theorem}

\begin{proof}
  This is \cite[Lemma~2.8]{KapovitchLytchak2022}. The last assertion
  follows from the fact that the metric quotient map of an isometric
  action with closed orbits is a submetry; see
  \cite[Example~2.5]{KapovitchLytchak2022}.
\end{proof}

\begin{theorem}[Semialgebraic curves compute inner distance]
  \label{thm:semialgebraic-curves-compute-inner-distance}
  Let \(A\subset \R^N\) be a connected semialgebraic set,
  endowed with its inner distance
  \[
    \dIntrinsic{A}(p,q)
    \coloneq
    \inf\bigl\{\Length(\gamma):
      \gamma:[0,1]\to A \text{ rectifiable},\
    \gamma(0)=p,\ \gamma(1)=q\bigr\}.
  \]
  Then, for every \(p,q\in A\),
  \[
    \dIntrinsic{A}(p,q)
    =
    \inf\bigl\{\Length(\gamma):
      \gamma:[0,1]\to A \text{ continuous semialgebraic},\
    \gamma(0)=p,\ \gamma(1)=q\bigr\}.
  \]
  Equivalently, for every rectifiable curve \(\gamma:[0,1]\to A\) and every
  \(\varepsilon>0\), there exists a continuous semialgebraic curve
  \(\eta:[0,1]\to A\) with the same endpoints and
  \[
    \Length(\eta)
    \le
    \Length(\gamma)+\varepsilon.
  \]
  Moreover, there is a finite subdivision
  \[
    0=t_0<t_1<\cdots<t_m=1
  \]
  such that $\restr{\eta}{(t_{j-1},t_j)}$ is $C^1$ for every
  $j=1,\ldots,m$. The curve $\eta$ is continuous on $[0,1]$.
\end{theorem}

\begin{proof}
  The equality between the inner distance and the infimum over
  continuous semialgebraic curves is due to Kurdyka--Orro
  \cite{KurdykaOrro1997}; see also the formulation recalled in
  \cite[\S5.1] {CampesatoFichouKurdykaParusinski2017}. The final
  assertion follows from the piecewise $C^1$
  regularity of semialgebraic maps; see for instance
  \cite{BochnakCosteRoy}. Applied to
  $\eta:[0,1]\to\R^N$, this gives a finite
  semialgebraic partition of $[0,1]$ into points and open
  intervals such that $\eta$ is $C^1$ on each interval.
  The point cells and the interval endpoints determine the
  required subdivision. Continuity at the subdivision points
  follows from the continuity of $\eta$ on $[0,1]$.
\end{proof}

\begin{proposition}[Global metric geometry of the reduced quotient]
  \label{prop:reduced-quotient-global-metric-geometry}
  Let \(m\ge2\).  The quotient
  \(Q_m=M_m/O(m)\) is a compact geodesic Alexandrov space with curvature
  bounded below by zero
  in the sense of \cref{eq:alexandrov-midpoint-comparison}.
  Its distance satisfies
  \begin{equation}
    \label{eq:reduced-quotient-orbital-distance}
    \dQ{m}([X],[Y])
    =
    \min_{R\in O(m)}\dM{m}(X,YR),
    \qquad X,Y\in M_m,
  \end{equation}
  and
  \[
    \diam(Q_m)
    \le
    \diam(M_m)
    =
    \pi\sqrt n.
  \]
  If \(R_*\) realizes the minimum in
  \cref{eq:reduced-quotient-orbital-distance}, every minimizing geodesic in
  \(M_m\) from \(X\) to \(YR_*\) projects to a minimizing geodesic in
  \(Q_m\).
\end{proposition}

\begin{proof}
  The orbital distance formula is classical
  \cite[Lemma~3.3.6]{BuragoBuragoIvanov2001}; compactness of \(O(m)\)
  turns the infimum there into a minimum.  The orbit projection is a
  submetry, and submetries preserve the length-space property
  \cite[Lemmas~2.1 and~2.3]{KapovitchLytchak2022}.  Since \(M_m\) is
  compact, so is \(Q_m\), and a compact length space is geodesic
  \cite[Theorem~2.5.23]{BuragoBuragoIvanov2001}.  Moreover, \(M_m\) has
  sectional curvature bounded below by zero, and a submetry preserves an
  Alexandrov lower-curvature bound
  \cite[Proposition~3.1]{KapovitchLytchak2022}; see also
  \cite[Corollary~8.34]{AlexanderKapovitchPetrunin2024}.  The quotient map
  is \(1\)-Lipschitz, which gives the diameter bound.  Finally, if \(\alpha\)
  is a minimizing geodesic from \(X\) to \(YR_*\), then
  \[
    \dQ{m}([X],[Y])
    \le
    \Length_{\dQ{m}}(\pi_m\circ\alpha)
    \le
    \Length_{M_m}(\alpha)
    =
    \dQ{m}([X],[Y]),
  \]
  so \(\pi_m\circ\alpha\) is minimizing.
\end{proof}

\begin{theorem}[Global rank-bounded metric reduction]
  \label{thm:global-rank-bounded-metric-reduction}
  Let \(1\le s\le r\). The map \(\Xks\) is a bijection from \(Q_s\) onto
  \(\Qkles\). Moreover, for \(s\ge2\) it is an isometry of length spaces:
  \[
    \Xks:\bigl(Q_s,\dQ{s}\bigr)
    \xrightarrow{\ \cong\ }
    \bigl(\Qkles,\dkles\bigr).
  \]
  In particular, when \(s\ge2\), minimizing geodesic segments in
  \(\bigl(\Qkles,\dkles\bigr)\) correspond
  via \(\Xks\) exactly to minimizing geodesic segments in \((Q_s,\dQ{s})\).
\end{theorem}

Before giving the proof, let us describe the main idea.
Padding preserves product-of-spheres lengths, but ambient quotient
distances may decrease after padding. We must show that the extra
\(k-s\) coordinates provide no shorter intrinsic path once curves
are constrained to remain in \(\Qkles\).

The proof is carried out upstairs. A rectifiable curve
\(\bar\eta:[0,1]\to\Qkles\) admits a length-preserving lift
\(X:[0,1]\to M_k^{\le s}\).
For \(\varepsilon>0\), replace \(X\) by a continuous semialgebraic
curve \(\widetilde X\) in the same set, with the same endpoints and after a finite subdivision, \(\widetilde X\) is smooth, has
constant rank \(m\le s\) on each open subinterval, and
\[
  \Length_{M_k}(\widetilde X)
  \le \Length_{M_k}(X)+\varepsilon, \qquad
\Gamma_k(\widetilde X)
\in\Crle{s}{n}
\subseteq\Crle{k}{n}.
\]
The reduced quotient curve is defined globally by
\[
  \bar z(t):=\bar\Gamma_s^{-1}
  \bigl(\widetilde X(t)\widetilde X(t)^\top\bigr)\in Q_s.
\]
There we choose a \(C^1\) orthonormal \(s\)-frame
\(U(t)\in\R^{k\times s}\) containing its row span and set
\[
  Z:=\widetilde XU\in \Mrk{s}{m},
  \qquad
  ZZ^\top=\widetilde X\widetilde X^\top,
  \qquad [Z]=\bar z.
\]
Since \(\widetilde X=ZU^\top\), differentiation gives
\[
  Z'=\widetilde X'U+Z\Omega,
  \qquad
  \Omega:=U^\top U'\in\so(s),
\]
where skew-symmetry follows by differentiating
\(U^\top U=\I_s\).
The term \(Z\Omega\) is vertical and therefore disappears under
the quotient differential. Consequently,
\[
  \|\bar z'(t)\|_{\gQrk{s}{m}}
  \le \|\widetilde X'(t)U(t)\|_F
  \le \|\widetilde X'(t)\|_F,
\]
because \(\|U(t)\|_{\mathrm{op}}=1\).

Integrating on compact subintervals, extending to their endpoints
by continuity, and summing the finitely many pieces yields
\[
  \Length_{\dQ{s}}(\bar z)
  \le \Length_{M_k}(\widetilde X)
  \le \Length_{M_k}(X)+\varepsilon
  = \Length_{\dQ{k}}(\bar\eta)+\varepsilon.
\]
Taking infima over admissible curves and letting
\(\varepsilon\downarrow0\) proves the nontrivial distance inequality.
The opposite inequality follows by padding curves in \(M_s\),
which preserves their product-of-spheres length.

\smallskip
The proof of the global rank-bounded reduction theorem makes use of
the following lemma.

\begin{lemma}[Rank-constant subdivision of semialgebraic curves]
  \label{lem:rank-constant-subdivision-semialgebraic-curve}
  Let $X:[0,1]\to \R^{n\times k}$ be a continuous
  semialgebraic curve. Then there exists a finite subdivision of $[0,1]$
  \[
    0=t_0<t_1<\cdots<t_N=1
  \]
  such that \(X\) is smooth on each
  open interval \((t_{j-1},t_j)\), and
  $\rank(X(t))$ is constant on each such open interval.
\end{lemma}

\begin{proof}
  By smooth semialgebraic cell decomposition
  \cite[Proposition~2.9.10]{BochnakCosteRoy}, there exists a finite
  set \(S\subset[0,1]\) such that \(X\)
  is smooth on every connected component of \([0,1]\setminus S\).

  For each \(0\le j\le\min(n,k)\), let
  \[
    \mathcal R^{(j)}
    \coloneq
    \{A\in\R^{n\times k}:\rank(A)=j\}.
  \]
  The set \(\mathcal R^{(j)}\) is semialgebraic, since it is
  characterized by the vanishing of all \((j+1)\times(j+1)\) minors
  and the nonvanishing of at least one \(j\times j\) minor. Since
  \(X\) is semialgebraic, its graph
  \(\Gamma_X\coloneq\{(t,X(t)):t\in[0,1]\}\) is semialgebraic. Moreover,
  \[
    X^{-1}\bigl(\mathcal R^{(j)}\bigr)
    =
    \operatorname{pr}_1\!\left(
      \Gamma_X\cap
      \bigl([0,1]\times\mathcal R^{(j)}\bigr)
    \right).
  \]
  The Tarski--Seidenberg theorem
  \cite[Theorem~2.2.1]{BochnakCosteRoy} therefore implies that
  \(X^{-1}(\mathcal R^{(j)})\) is a semialgebraic subset of
  \([0,1]\). Hence it is a finite union of points and intervals.
  Let \(S_j\) be the finite set consisting of its isolated points and
  the endpoints of its interval components. The ordered elements of
  the finite set
  \[
    \{0,1\}\cup S\cup
    \bigcup_{j=0}^{\min(n,k)}S_j
  \]
  give a subdivision \(0=t_0<t_1<\cdots<t_N=1\). On every interval
  \((t_{j-1},t_j)\), the map \(X\) is smooth and the rank of \(X(t)\)
  is constant.
\end{proof}

Let us now prove \cref{thm:global-rank-bounded-metric-reduction}.

\begin{proof}
  For \(1\le s\le r\), every \(X\in M_k^{\le s}\) admits
  \(O\in O(k)\) such that \(XO=(Z,0)\) with \(Z\in M_s\);
  hence \(\Xks\) is surjective onto \(\Qkles\).
  If \(\Xks([Z_1])=\Xks([Z_2])\), then
  \(Z_1Z_1^\top=Z_2Z_2^\top\), so
  \cref{prop:quotient-model-low-rank-correlations} gives
  \([Z_1]=[Z_2]\) in \(Q_s\). Thus
  \(\Xks:Q_s\to\Qkles\) is bijective.

  Assume now that \(s\ge2\). Fix
  \(\bar p=[Z_0]\), \(\bar q=[Z_1]\in Q_s\), and
  \(\varepsilon>0\).

  \smallskip
  \noindent\emph{First inequality.}
  Choose a rectifiable curve \(\bar\alpha:[0,1]\to Q_s\)
  joining \(\bar p\) to \(\bar q\), with
  \[
    \Length_{\dQ{s}}(\bar\alpha)
    \le \dQ{s}(\bar p,\bar q)+\varepsilon.
  \]
  By \cref{thm:length-preserving-lifts-submetry}, it has a
  length-preserving lift \(\alpha:[0,1]\to M_s\) starting
  at \(Z_0\). Set \(\widehat\alpha\coloneq (\alpha,0)\).
  Its projection \(\pi_k\circ\widehat\alpha\) lies in
  \(\Qkles\) and joins \(\Xks(\bar p)\) to \(\Xks(\bar q)\).
  Since \(\pi_k\) is \(1\)-Lipschitz and padding preserves
  lengths by \cref{lem:padding-isometric-embedding},
  \[
    \begin{aligned}
      \dkles\bigl(\Xks(\bar p),\Xks(\bar q)\bigr)
      &\le \Length_{\dQ{k}}(\pi_k\circ\widehat\alpha)
      \le \Length_{M_k}(\widehat\alpha)\\
      &= \Length_{M_s}(\alpha)
      = \Length_{\dQ{s}}(\bar\alpha)
      \le \dQ{s}(\bar p,\bar q)+\varepsilon.
    \end{aligned}
  \]
  Letting \(\varepsilon\to0\) proves the first inequality.
  In particular, \(\dkles\) is finite between these endpoints.

  \smallskip
  \noindent\emph{Reverse inequality.}
  Choose a rectifiable curve
  \(\bar\eta:[0,1]\to\Qkles\) joining
  \(\Xks(\bar p)\) to \(\Xks(\bar q)\), with
  \[
    \Length_{\dQ{k}}(\bar\eta)
    \le
    \dkles\bigl(\Xks(\bar p),\Xks(\bar q)\bigr)
    +\varepsilon.
  \]
  By \cref{thm:length-preserving-lifts-submetry}, viewed
  as a curve in \(Q_k\), it has a lift
  \(X:[0,1]\to M_k\) satisfying
  \(X(0)=(Z_0,0)\), \(\pi_k\circ X=\bar\eta\), and
  \begin{equation}
    \label{eq:length_preserving_lift}
    \Length_{M_k}(X)=\Length_{\dQ{k}}(\bar\eta).
  \end{equation}
  Rank is constant along each \(O(k)\)-orbit, so
  \(X(t)\in M_k^{\le s}\) for every \(t\).
  Moreover, \(X(1)=(Z_1,0)O\) for some \(O\in O(k)\). The set
  \(M_k^{\le s}\) is compact and semialgebraic.
  By \cref{thm:semialgebraic-curves-compute-inner-distance},
  there is a continuous semialgebraic curve
  \(\widetilde X:[0,1]\to M_k^{\le s}\) with the same
  endpoints as \(X\), such that
  \[
    \Length_{M_k}(\widetilde X)
    =
    \Length_F(\widetilde X)
    \le \Length_{M_k}(X)+\varepsilon,
  \]
  where the equality of lengths follows from the restriction
  of the Frobenius metric to \(M_k\). Observe that the curve
  \(\widetilde X\) need not lift
  \(\bar\eta\); only its endpoints and this length bound
  will be used. By \cref{lem:rank-constant-subdivision-semialgebraic-curve},
  choose a finite subdivision
  \[
    0=t_0<t_1<\cdots<t_N=1
  \]
  such that \(\widetilde X\) is \(C^1\) and has constant
  rank on each \(I_i=(t_{i-1},t_i)\). Using the Gram homeomorphism from
  \cref{prop:quotient-model-low-rank-correlations}, define
  the global curve
  \[
    \bar z(t)
    \coloneq
    \overline\Gamma_s^{-1}
    \bigl(\widetilde X(t)\widetilde X(t)^\top\bigr)
    \in Q_s.
  \]
  This curve is continuous, including at the subdivision
  points, and satisfies
  \[
    \bar z(0)=\bar p,\qquad \bar z(1)=\bar q.
  \]
  Indeed, its endpoint Gram matrices are
  \(Z_0Z_0^\top\) and \(Z_1Z_1^\top\).
  We will prove that
  \(\Length_{\dQ{s}}(\bar z)\le
  \Length_{M_k}(\widetilde X)\). To obtain this estimate, we work
  separately on each constant-rank
  interval and choose a \(C^1\) lift \(Z\) of \(\bar z\) into
  \(\Mrk{s}{m}\), where \(m\) is the rank on that interval.
  The quotient metric \(\gQrk{s}{m}\) then allows us to compute
  the speed of \(\bar z\) from the velocity of this lift and
  bound its \(\dQ{s}\)-length on compact subintervals.

  Fix \(I=(a,b)\) an open interval on which
  \(\widetilde X\) is \(C^1\) and has constant rank and set
  \(m\coloneq\rank \widetilde X(t)\le s\), \(t\in I\). At each $t$, let
  \(E(t)\coloneq\spanv\{\widetilde X_{1\cdot}(t),\ldots,
  \widetilde X_{n\cdot}(t)\}\subset\R^k\). Since \(\widetilde
  X\) is \(C^1\) and has constant rank \(m\) on \(I\), the row spaces
  \(E(t)\) form a \(C^1\) rank-\(m\) subbundle of \(I\times\R^k
  \). Indeed, locally in \(t\), one may choose \(m\) rows of
  \(\widetilde X(t)\) which remain linearly independent and hence
  give a \(C^1\) frame of \(E(t)\). Since \(I\) is an interval, this
  bundle is trivial; after applying Gram--Schmidt and extending the resulting
  orthonormal \(m\)-frame inside the trivial bundle, we obtain a
  \(C^1\) orthonormal \(s\)-frame
  \[
    U(t)\in\R^{k\times s},
    \qquad
    U(t)^\top U(t)=\I_s,
  \]
  such that \(E(t)\subset \im U(t)\), \(t\in I\). Note
  that the continuity of $t \mapsto U(t)$ follows because the chosen
  row frame is \(C^1\), and the Gram--Schmidt process is smooth on
  the open set of linearly independent frames. Define the curve
  \[
    Z(t)\coloneq\widetilde X(t)U(t)\in\R^{n\times s}.
  \]
  Since every row of \(\widetilde X(t)\) lies in
  \(\im U(t)\), multiplication
  by \(U(t)\) preserves the norm of each row, hence, $Z(t)\in M_s$. Moreover,
  \[
    Z(t)Z(t)^\top
    =
    \widetilde X(t)U(t)U(t)^\top\widetilde X(t)^\top
    =
    \widetilde X(t)\widetilde X(t)^\top,
  \]
  because \(U(t)U(t)^\top\) is the orthogonal projector onto
  \(\im U(t)\), which contains the row span of \(\widetilde X(t)\).
  Therefore, on \(I\), $\bar z(t)=[Z(t)]$ and \(\bar z(I)\subset
  \Qrk{s}{m}\). Thus \(Z:I\to \Mrk{s}{m}\) is a \(C^1\) lift of
  \(\restr{\bar z}{I}\).

  Set \(A\coloneq U^\top U'\). Differentiating
  \(U^\top U=\I_s\) gives \(A^\top=-A\), while
  differentiating \(Z=\widetilde XU\) yields
  \begin{equation}
    \label{eq:reduced-velocity-identity}
    Z'=\widetilde X'U+ZA,
    \qquad A^\top=-A.
  \end{equation}
  The vector \(ZA\) is vertical for the \(O(s)\)-action.
  By \cref{prop:quotient-speed-formula},
  \[
    \begin{aligned}
      \|\dot{\bar z}\|_{\gQrk{s}{m}}
      =
      \min_{\Omega\in\so(s)}
      \|Z'-Z\Omega\|_F
      \le \|Z'-ZA\|_F
      =\|\widetilde X'U\|_F
      \le\|\widetilde X'\|_F.
    \end{aligned}
  \]

  For every compact subinterval \([u,v]\subset I\),
  the length upper bound in
  \cref{prop:quotient-speed-formula} therefore gives
  \[
    \begin{aligned}
      \Length_{\dQ{s}}(\restr{\bar z}{[u,v]})
      \le
      \Length_{\gQrk{s}{m}}(\restr{\bar z}{[u,v]})
      =\int_u^v
      \|\dot{\bar z}(t)\|_{\gQrk{s}{m}}\dd t
      \le\int_u^v\|\widetilde X'(t)\|_F\dd t.
    \end{aligned}
  \]
  In particular, the same integral bounds
  \(\dQ{s}(\bar z(u),\bar z(v))\).
  By continuity of \(\bar z\), this endpoint-distance
  bound extends to \(a\le u<v\le b\).
  Applying it to consecutive points of an arbitrary
  finite partition of \([a,b]\), and taking the supremum
  over partitions, gives
  \[
    \Length_{\dQ{s}}(\restr{\bar z}{[a,b]})
    \le
    \int_a^b\|\widetilde X'(t)\|_F\dd t
    =
    \Length_{M_k}(\restr{\widetilde X}{[a,b]}),
  \]
  with the integral understood as an improper integral
  at the endpoints.
  No compatibility between the frames on adjacent
  intervals is required, since \(\bar z\) is defined
  globally and is continuous across rank changes.
  Finite additivity of length now yields
  \begin{equation}
    \label{eq:vertical_part_vanishes}
    \Length_{\dQ{s}}(\bar z)
    \le
    \sum_{i=1}^N
    \Length_{M_k}
    \bigl(\restr{\widetilde X}{[t_{i-1},t_i]}\bigr)
    =
    \Length_{M_k}(\widetilde X).
  \end{equation}

  Consequently,
  \[
    \begin{aligned}
      \dQ{s}(\bar p,\bar q)
      &\le \Length_{\dQ{s}}(\bar z)
      &&\text{(by definition of \(\dQ{s}\) )}
      \\
      &\le \Length_{M_k}(\widetilde X)
      &&\text{(by \cref{eq:vertical_part_vanishes})}
      \\
      &\le \Length_{M_k}(X)+\varepsilon
      &&\text{(by the choice of }\widetilde X\text{)}
      \\
      &= \Length_{\dQ{k}}(\bar\eta)+\varepsilon
      &&\text{(by \cref{eq:length_preserving_lift})}
      \\
      &\le
      \dkles\bigl(\Xks(\bar p),\Xks(\bar q)\bigr)
      +2\varepsilon
      &&\text{(by the choice of }\bar\eta\text{).}
    \end{aligned}
  \]
  Letting \(\varepsilon\to0\) proves the reverse inequality.
  Together with the first inequality and bijectivity,
  this shows that \(\Xks\) is an isometry from
  \((Q_s,\dQ{s})\) onto \((\Qkles,\dkles)\).
\end{proof}

\begin{corollary}[Global Alexandrov geometry of rank-bounded spaces]
  \label{cor:rank-bounded-global-alexandrov-geometry}
  Let \(2\le s\le\min(n,k)\).  Then
  \((\Qkles,\dkles)\) is a compact geodesic Alexandrov space with curvature
  bounded below by zero, and
  \[
    \diam(\Qkles,\dkles)
    \le
    \pi\sqrt n.
  \]
  Through the Gram homeomorphism, the same statement holds for
  \(\Crle{s}{n}\) equipped with the transported intrinsic
  rank-bounded metric.
\end{corollary}

\begin{proof}
  Transport \cref{prop:reduced-quotient-global-metric-geometry} through the
  isometry of
  \cref{thm:global-rank-bounded-metric-reduction}, and then through the Gram
  homeomorphism of
  \cref{prop:quotient-model-low-rank-correlations}.
\end{proof}

As before, the quotient-side statement immediately translates to the
correlation side.

\begin{corollary}[Global rank-bounded reduction on correlation matrices]
  \label{cor:global-rank-bounded-reduction-correlation}
  Let \(2\le s\le r\). Via the homeomorphisms
  \[
    Q_s\cong \Crle{s}{n}
    \qquad\text{and}\qquad
    \Qkles\cong \Crle{s}{n},
  \]
  the metric induced from \((Q_s,\dQ{s})\) coincides with the intrinsic
  rank-bounded metric
  induced from \((Q_k,\dQ{k})\). In particular, every metric problem on
  \(\Crle{s}{n}\) depending only on the intrinsic rank-bounded
  distance may be
  reduced canonically to the smaller quotient space \((Q_s,\dQ{s})\).
\end{corollary}

\begin{proof}
  This is \cref{thm:global-rank-bounded-metric-reduction} transported
  through the Gram homeomorphisms of
  \cref{prop:quotient-model-low-rank-correlations}. \qedhere
\end{proof}

\section{Topology of fixed-rank correlation strata}
\label{sec:topology-fixed-rank-strata}

The rankwise Gram identification of
\cref{prop:quotient-model-low-rank-correlations} and the metric reduction
of \cref{sec:rank-reduction-metric-reduction} let us study the topology
of the fixed rank locus \(\Crk{s}{n}\).
For \(1\le s\le n\), we will use the principal-minor cover
\(V_I\) of \cref{lem:cholesky-schur-charts}, with its product model
\[
  V_I\simeq\Crk{s}{s}\times(\mathbb S^{s-1})^{n-s}.
\]
We derive the dimension and path connectedness of each stratum from this model.

\begin{remark}
  Since \(\Crk{s}{s}\) is the full-rank correlation manifold in
  dimension \(s\), the Cholesky decomposition gives a particularly
  simple parametrization. Namely, each \(C\in\Crk{s}{s}\) has a
  unique Cholesky factorization \(C=LL^\top\) where \(L\) is lower
  triangular with positive diagonal entries. The condition
  \(\diag(C)=\mathbf 1\) means that each row of \(L\)
  has Euclidean norm one. Thus the \(i\)-th row of \(L\) lies on
  \(\mathbb S^{i-1}\), and the Cholesky condition \(L_{ii}>0\) yields
  the open hemisphere
  \[
    \mathbb S^{i-1}_+
    \coloneq
    \{x=(x_1,\ldots,x_i)\in\mathbb S^{i-1}:x_i>0\}.
  \]
  Therefore \(\Crk{s}{s} \simeq \prod_{i=1}^s \mathbb S^{i-1}_+\)
  and each open hemisphere \(\mathbb S^{i-1}_+\) is diffeomorphic to
  the open ball \(\mathbb B^{i-1}\subset\R^{i-1}\). Consequently,
  \[
    \Crk{s}{s}
    \simeq
    \prod_{i=1}^s \mathbb B^{i-1}
    \simeq
    \R^{s(s-1)/2}.
  \]
  Indeed, stereographic projection from the opposite pole identifies each
  \(\mathbb S^{i-1}_{+}\) diffeomorphically with \(\mathbb B^{i-1}\),
  while the radial map \(x \longmapsto \frac{x}{\sqrt{1-\lVert
  x\rVert^{2}}}\) identifies \(\mathbb B^{i-1}\) diffeomorphically
  with \(\R^{i-1}\).
  Thus each principal-minor chart may also be written as \(V_I \simeq
  \R^{s(s-1)/2}\times(\mathbb S^{s-1})^{n-s}\).
\end{remark}

\begin{corollary}[Dimension of the fixed-rank stratum]
  \label{cor:dimension-fixed-rank-stratum}
  For every \(1\le s\le n\), the smooth manifold \(\Crk{s}{n}\)
  has dimension
  \[
    \dim \Crk{s}{n}
    =
    \frac{s(s-1)}{2}+(n-s)(s-1)
    =
    n(s-1)-\frac{s(s-1)}{2}.
  \]
  Equivalently, for every \(k\ge s\), the quotient stratum \(\Qks\) has the same
  dimension.
\end{corollary}

\begin{proof}
  By \cref{lem:cholesky-schur-charts}, the dimension is
  \(\dim\Crk{s}{s}+(n-s)\dim\mathbb S^{s-1}
  =s(s-1)/2+(n-s)(s-1)\).
  The quotient statement follows from
  \cref{prop:quotient-model-low-rank-correlations}.
\end{proof}

To pass from the product charts to connectedness of the whole stratum,
we first exhibit a point common to every \(V_I\).

\begin{lemma}[A common point of all principal-minor charts]
  \label{lem:common-point-all-VI}
  Assume \(2\le s\le n\).  There exists a matrix
  \(C_\ast\in\Crk{s}{n}\) such that \(C_\ast\in \bigcap_{|I|=s}
  V_I\), equivalently
  \[
    \det((C_\ast)_{II})>0
    \qquad
    \text{for every } I\subset\{1,\ldots,n\},\ |I|=s.
  \]
\end{lemma}

\begin{proof}
  Choose pairwise distinct real numbers $a_1,\ldots,a_n\in\R,$ and define
  \[
    v_i
    \coloneq
    \frac{(1,a_i,a_i^2,\ldots,a_i^{s-1})}
    {\|(1,a_i,a_i^2,\ldots,a_i^{s-1})\|_{\R^s}}
    \in\mathbb S^{s-1},
    \qquad i=1,\ldots,n.
  \]
  Let \(X_\ast\in\R^{n\times s}\) be the matrix with rows
  \(v_i\), and set \(C_\ast\coloneq X_\ast X_\ast^\top\). Then
  \(C_\ast\in\Crle{s}{n}\). Now let \(I=\{i_1,\ldots,i_s\}\) be
  any subset of cardinality \(s\). The submatrix \(X_{\ast,I}\) is
  obtained from the Vandermonde matrix
  \(\bigl(a_{i_\alpha}^{\,\beta-1}\bigr)_{1\le\alpha,\beta\le s}\) by
  nonzero row rescalings. Then, computing the Vandermonde determinant:
  \[
    \det(X_{\ast,I})
    \overset{\mathrm{rescaling}}{\propto}
    \det\bigl(a_{i_\alpha}^{\,\beta-1}\bigr)_{1\le\alpha,\beta\le s}
    =
    \prod_{1\le\alpha<\beta\le s}(a_{i_\beta}-a_{i_\alpha})
    \overset{\mathrm{pairwise\ distinct}}{\neq} 0.
  \]
  Therefore
  \[
    \det((C_\ast)_{II})
    =\det(X_{\ast,I}X_{\ast,I}^\top)
    =\det(X_{\ast,I})^2
    >0,
  \]
  and in particular $C_\ast \in V_I$.
  Since \(I\) was arbitrary, \(C_\ast\in\bigcap_{|I|=s}V_I\).
\end{proof}

\begin{theorem}[Path connectedness of fixed-rank strata]
  \label{thm:path-connectedness-fixed-rank-strata}
  For every \(2\le s\le n\), the fixed-rank correlation stratum
  \(\Crk{s}{n}\) is path connected. Consequently, for every
  \(k\ge s\), the quotient stratum $\Qks$ is path connected.
\end{theorem}

\begin{proof}
  By \cref{lem:cholesky-schur-charts}, the sets \(V_I\) cover
  \(\Crk{s}{n}\) and are diffeomorphic to
  \(\Crk{s}{s}\times(\mathbb S^{s-1})^{n-s}\).
  The first factor is convex, and the spheres
  are path connected for \(s\ge2\). Thus each \(V_I\) is path connected.
  They have a common point by \cref{lem:common-point-all-VI}, so their union
  is path connected. The same holds for \(\Qks\) by
  \cref{prop:quotient-model-low-rank-correlations}.
\end{proof}

\begin{remark}[The rank-one stratum]
  \label{rem:rank-one-stratum}
  The restriction \(s\ge2\) in
  \cref{thm:path-connectedness-fixed-rank-strata} is necessary.
  Indeed, if \(C\in\Crk{1}{n}\), then \(C=vv^\top\) for some
  \(v\in\{\pm1\}^n\), and the choices \(v\) and \(-v\) give the same
  correlation matrix. Thus
  \[
    \Crk{1}{n}
    =
    \{vv^\top:v\in\{\pm1\}^n\}
  \]
  has \(2^{n-1}\) points. In particular, for \(n\ge2\), the rank-one
  stratum is not path connected.
\end{remark}

\begin{remark}[Scope of the topological model]
  \label{rem:scope-topological-model}
  The product charts \(V_I\) provide an explicit local model for the
  fixed-rank correlation strata:
  \[
    \Crk{s}{n}
    =
    \bigcup_{|I|=s} V_I,
    \qquad
    V_I\simeq \Crk{s}{s}\times(\mathbb S^{s-1})^{n-s}.
  \]
  This local description is sufficient for the arguments used later: it gives
  path connectedness for \(s\ge2\), and it supplies explicit principal-minor
  coordinates around every point of \(\Crk{s}{n}\). A full
  topological classification of \(\Crk{s}{n}\), however, depends
  on the global gluing pattern of the charts \(V_I\) along their
  overlaps. This lies beyond the scope of the present paper.
\end{remark}

\section{Intrinsic stratum lengths and the Sylvester mechanical connection}
\label{sec:quotient-metric-sylvester-connection}

The preceding sections introduced the quotient model for low-rank
correlation matrices, its rank strata and the two reduction theorems,
with padding isometry
\[
  Q_k=M_k/O(k), \quad \Qks=\Mks/O(k), \quad \Xks:\Qss\longrightarrow \Qks.
\]
The goal of this section is twofold. First, we give explicit formulas
for the horizontal distribution on the reduced full-rank model
\[
  \pi_s:\Mss\longrightarrow \Qss=\Mss/O(s),
\]
where the action is free. These formulas are written in terms of a
nonsingular Sylvester operator and identify the mechanical connection
associated with the product-of-spheres metric. Second, we identify
the intrinsic length geometry of a fixed-rank stratum as a
consequence of the fixed-rank Riemannian reduction, the global
rank-bounded metric reduction, and the compatibility of the quotient
distance with the Riemannian distance on the reduced full-rank stratum.

The metric point is the following. There are two different distances
naturally associated with \(\Qks\). The first is the pointwise
restriction of the ambient quotient distance \(\dQ{k}\) to pairs of
points in \(\Qks\). The second is the intrinsic length distance
obtained by measuring only curves that remain inside \(\Qks\). These
two distances need not coincide in general, because the ambient
quotient distance is computed in the whole quotient \(Q_k\), and
paths realizing or approximating it need not stay in the fixed-rank stratum.

We prove below that the intrinsic length distance induced by \(\dQ{k}\) on
\(\Qks\) is exactly the Riemannian geodesic distance associated with
the quotient Riemannian metric \(\gQks\). We recall this quotient
metric is characterized by the fact that
\[
  \pi_k:\Mks\longrightarrow \Qks
\]
is a Riemannian submersion: if \(u,v\in\Tz{[X]}\Qks\), and if
\(u_X^\#,v_X^\#\in\Hz{X}\coloneq \Vz{X}^\perp\cap\Tz{X}\Mks\) denote their
unique horizontal lifts at \(X\), then
\[
  \gQks(u,v)
  \coloneq 
  \Frob{u_X^\#}{v_X^\#}.
\]
We denote by \(\dQks\) the corresponding Riemannian geodesic distance.

\subsection{The reduced principal bundle and horizontal distribution}
\label{subsec:reduced-principal-bundle-horizontal-distribution}

Fix \(1\le s\le r=\min(n,k)\). The full-rank piece
\[
  \Mss=\{Z\in M_s:\rank(Z)=s\}
\]
is an open submanifold of \(M_s\), since its complement is the
rank-deficient locus cut out by the vanishing of all \(s\times s\)
minors. On \(\Mss\), the right action of \(O(s)\) is free. Hence
\[
  \pi_s:\Mss\longrightarrow \Qss
\]
is a principal \(O(s)\)-bundle. The case \(s=1\) is zero-dimensional;
all formulas below remain formally valid with
\(\so(1)=\{0\}\), while the subsequent
differential-geometric applications concern \(s\ge2\).

\begin{definition}[Vertical and horizontal spaces]
  \label{def:vertical-horizontal-spaces}
  For \(Z\in\Mss\), the \emph{vertical space} at $Z$ is the tangent
  space to the orbit \(Z\cdot O(s)\) (see, e.g., \cite[Chap.~II,
  \S1]{KobayashiNomizu1963})
  \begin{equation}
    \label{eq:vertical-space-reduced-model}
    \Vz{Z}
    \coloneq 
    \ker \dd\pi_s(Z)
    =
    \{Z\Omega:\Omega\in\so(s)\}.
  \end{equation}
  The metric-horizontal space at $Z$ is the Frobenius-orthogonal
  complement of \(\Vz{Z}\) inside \(\Tz{Z}\Mss\):
  \begin{equation}
    \label{eq:horizontal-space-reduced-model}
    \Hz{Z}
    \coloneq 
    \Vz{Z}^\perp
    \subset \Tz{Z}\Mss.
  \end{equation}
\end{definition}

The horizontal space is characterized by the following lemma.

\begin{lemma}[Horizontality criterion]
  \label{lem:horizontality-criterion}
  Let \(Z\in \Mss\) and \(W\in\Tz{Z}\Mss\).  Then
  \[
    W\in\Hz{Z}
    \qquad\Longleftrightarrow\qquad
    Z^\top W \text{ is symmetric}
    \qquad\Longleftrightarrow\qquad
    Z^\top W - W^\top Z = 0.
  \]
\end{lemma}

\begin{proof}
  For \(\Omega\in\so(s)\),
  \[
    \Frob{W}{Z\Omega}
    =
    \tr(W^\top Z\Omega)
    =
    \tr(\Omega W^\top Z).
  \]
  Thus \(W\) is orthogonal to every vertical vector \(Z\Omega\) if and only if
  \(W^\top Z\) is Frobenius-orthogonal to \(\so(s)\).  Since the
  Frobenius-orthogonal complement of \(\so(s)\) in
  \(\R^{s\times s}\) is \(\Sym(s)\), this is equivalent to
  \(W^\top Z\) being symmetric, or equivalently to \(Z^\top W\) being symmetric.
\end{proof}

\subsection{The Sylvester operator and the mechanical connection}
\label{subsec:sylvester-mechanical-connection}

For \(Z\in \Mss\), set the symmetric, positive definite matrix
\begin{equation}
  \label{eq:gram-factor-EZ}
  E_Z\coloneq Z^\top Z\in\SPD(s).
\end{equation}
We will use the Sylvester operator
\begin{equation}
  \label{eq:sylvester-operator}
  \mathcal L_{E_Z}:\so(s)\longrightarrow\so(s),
  \qquad
  \mathcal L_{E_Z}(\Omega)\coloneq E_Z\Omega+\Omega E_Z.
\end{equation}

\begin{lemma}[Invertibility of the Sylvester operator]
  \label{lem:sylvester-invertible}
  Let \(E\in\SPD(s)\).  Then
  \[
    \mathcal L_E:\so(s)\to\so(s),
    \qquad
    \mathcal L_E(\Omega)=E\Omega+\Omega E,
  \]
  is a linear automorphism.
\end{lemma}

\begin{proof}
  If \(\Omega\in\so(s)\), then
  \[
    (E\Omega+\Omega E)^\top
    =
    \Omega^\top E+E\Omega^\top
    =
    -\Omega E-E\Omega
    =
    -(E\Omega+\Omega E),
  \]
  so \(\mathcal L_E\) maps \(\so(s)\) into itself.
  Diagonalize \(E=RDR^\top\), with \(R\in O(s)\) and
  \(D=\diag(\lambda_1,\ldots,\lambda_s)\) with
  \(\lambda_i>0\). Conjugation by \(R\) identifies \(\mathcal L_E\)
  with \(\mathcal L_D\) as follows. If \(\widetilde\Omega=R^\top\Omega R\), then
  \[
    R^\top\mathcal L_E(\Omega)R
    =
    D\widetilde\Omega+\widetilde\Omega D
    =
    \mathcal L_D(\widetilde\Omega).
  \]
  Since \(\Omega\mapsto R^\top\Omega R\) is an automorphism of
  \(\so(s)\), \(\mathcal L_E\) is injective if and only if
  \(\mathcal L_D\) is injective. For
  \(\Omega\in\so(s)\),
  \[
    \bigl(\mathcal L_D(\Omega)\bigr)_{ij}
    =
    (\lambda_i+\lambda_j)\Omega_{ij}.
  \]
  If $\Omega$ is such that \(\mathcal L_D(\Omega)=0\), then for every
  \(i,j\), we have
  \((\lambda_i+\lambda_j)\Omega_{ij}=0\), i.e., \(\ker\mathcal L_D=\{0\}\).
  Hence \(\mathcal L_D\), and therefore \(\mathcal L_E\), is injective.
  Since \(\so(s)\) is finite-dimensional, the rank-nullity
  theorem ensures that
  \(\mathcal L_E\) is an automorphism.
\end{proof}

We use the following convention for connection forms on right
principal bundles, following
\cite[Chap.~II, Proposition~1.1]{KobayashiNomizu1963}. Let \(P\) be a right
principal bundle, with structure group \(G\), and let \(\mathfrak g\)
be the Lie algebra of \(G\). For \(\xi\in\mathfrak g\), the
associated fundamental vector field on \(P\)
is defined by
\[
  \xi_P(p)
  \coloneq
  \frac{\dd}{\dd t}\Big|_{t=0} p\cdot \exp(t\xi),
  \qquad p\in P.
\]
A principal connection \(1\)-form is a \(\mathfrak g\)-valued \(1\)-form
\[
  \omega\in\Omega^1(P;\mathfrak g)
\]
satisfying, for every \(p\in P\), every \(\xi\in\mathfrak g\), every
\(g\in G\), and every \(v\in\Tz{p}P\), the following two conditions:
\[
  \begin{aligned}
    \omega_p(\xi_P(p))
    &=
    \xi,
    &&\text{(reproduction)}
    \\
    \omega_{p\cdot g}\bigl(\dd R_g(p)[v]\bigr)
    &=
    \operatorname{Ad}_{g^{-1}}\bigl(\omega_p(v)\bigr),
    &&\text{(equivariance)}
  \end{aligned}
\]
where \(R_g:P\to P\) denotes right multiplication by \(g\).
Its kernel defines the
corresponding horizontal distribution,
\[
  \Hz{p}\coloneq \ker\omega_p\subset\Tz{p}P.
\]
For instance, in the case \(G=O(s)\), with Lie algebra
\(\so(s)\), the equivariance
condition becomes
\begin{equation}\label{eq:equivariance-condition}
  \omega_{ZR}(WR)
  =
  R^\top\omega_Z(W)R,
  \qquad R\in O(s).
\end{equation}

\begin{definition}[Reduced mechanical connection]
  \label{def:reduced-mechanical-connection}
  The \emph{reduced mechanical connection} is the principal
  \(O(s)\)-connection \(\omega\) on
  \[
    \pi_s:
    (\Mss,\restr{g_{M_s}}{\Mss})
    \longrightarrow
    (\Qss,\gQss)
  \]
  whose horizontal spaces are the metric-horizontal spaces
  \[
    \ker\omega_Z=\Hz{Z}
    =\bigl\{W\in\Tz{Z}\Mss:
      \Frob{W}{Z\Omega}=0
    \text{ for every }\Omega\in\so(s)\bigr\}.
  \]
  Equivalently, for \(W\in\Tz{Z}\Mss\), the matrix
  \(\omega_Z(W)\in\so(s)\) is uniquely characterized by
  \[
    W-Z\omega_Z(W)\in\Hz{Z}.
  \]
  Here \emph{reduced} refers to the full-column-rank factor model,
  whose factor dimension equals the rank \(s\).
\end{definition}

\begin{proposition}[Sylvester formula for the mechanical connection]
  \label{prop:sylvester-mechanical-connection}
  Consider the principal \(O(s)\)-bundle \(\pi_s:\Mss\longrightarrow
  \Qss\). For \(Z\in \Mss\), set \(E_Z\coloneq Z^\top Z\in\SPD(s)\), and define
  \[
    J_Z : W\in\Tz{Z}\Mss\longmapsto J_Z(W)\coloneq Z^\top W-W^\top Z\in\so(s).
  \]
  Then the \(\so(s)\)-valued \(1\)-form
  \[
    \omega_Z(W)
    \coloneq 
    \mathcal L_{E_Z}^{-1}\bigl(J_Z(W)\bigr)
  \]
  is a principal connection \(1\)-form on \(\pi_s:\Mss\to \Qss\).  Its
  horizontal distribution is
  \[
    \ker\omega_Z
    =
    \{W\in\Tz{Z}\Mss:Z^\top W=W^\top Z\}
    =
    \Vz{Z}^\perp.
  \]
  Thus \(\omega\) is the reduced mechanical connection of
  \cref{def:reduced-mechanical-connection}.  Moreover, for every
  \(W\in\Tz{Z}\Mss\), the vertical and horizontal components of \(W\)
  are written, respectively
  \[
    W^{\Vbundle}=Z\omega_Z(W),
    \qquad
    W^{\Hbundle}=W-Z\omega_Z(W).
  \]
\end{proposition}

\begin{proof}
  Since \(Z\) has rank \(s\), the matrix \(E_Z=Z^\top Z\) is positive
  definite. Hence \(\mathcal L_{E_Z}\) is invertible by
  \cref{lem:sylvester-invertible}, so \(\omega\) is well-defined. Let
  \(\Omega\in\so(s)\). The corresponding fundamental
  vertical vector at \(Z\) is \(Z\Omega\), and
  \[
    J_Z(Z\Omega)
    =
    Z^\top Z\Omega-(Z\Omega)^\top Z
    =
    E_Z\Omega+\Omega E_Z
    =
    \mathcal L_{E_Z}(\Omega).
  \]
  Thus the reproduction condition is satisfied:
  \(\omega_Z(Z\Omega)=\Omega\). Next let \(R\in O(s)\).  For
  \(\widehat Z\coloneq ZR\) and \(\widehat W\coloneq WR\), we have
  \[
    E_{\widehat Z}=R^\top E_ZR,
    \qquad
    J_{\widehat Z}(\widehat W)=R^\top J_Z(W)R.
  \]
  Moreover, \(\mathcal L_{R^\top E_ZR}^{-1}(R^\top A R) =
  R^\top\mathcal L_{E_Z}^{-1}(A)R\) for every
  \(A\in\so(s)\). Therefore
  \[
    \omega_{ZR}(WR)
    =
    R^\top\omega_Z(W)R,
  \]
  which is the required equivariance condition for a right principal
  \(O(s)\)-bundle,  cf.
  \cref{eq:equivariance-condition}. It remains to identify the
  kernel. Since \(\mathcal L_{E_Z}\) is invertible,
  \(\ker\omega_Z=\ker J_Z\). Hence,
  \[
    W\in\ker\omega_Z
    \quad\Longleftrightarrow\quad
    Z^\top W-W^\top Z=0
    \quad\Longleftrightarrow\quad
    Z^\top W \text{ is symmetric}.
  \]
  This concludes \(\ker\omega_Z = \Vz{Z}^\perp\), by
  \cref{lem:horizontality-criterion}. Finally, the vertical component
  of \(W\) is the unique vector \(Z\Omega\in\Vz{Z}\) such that
  \(W-Z\Omega\) is horizontal.  The equation above gives
  \(\Omega=\omega_Z(W)\), and therefore
  \begin{equation}
    W^{\Vbundle}=Z\omega_Z(W),
    \qquad
    W^{\Hbundle}=W-Z\omega_Z(W).
    \label{eq:vertical-decomposition}
  \end{equation}
\end{proof}

\begin{remark}[Why the Sylvester formula is rankwise]
  \label{rem:sylvester-rankwise}
  On the reduced full-rank model \(\Mss\), the matrix
  \(E_Z=Z^\top Z\) is positive definite, which guarantees
  invertibility of \(\mathcal L_{E_Z}\) on
  \(\so(s)\). Positive definiteness is a
  sufficient condition, but is not necessary. Indeed, let \(E\) be
  symmetric positive semidefinite,
  with eigenvalues $0\le\lambda_1\le\cdots\le\lambda_s$. In an
  orthonormal eigenbasis, the Sylvester operator
  acts on the skew-symmetric matrix entries by
  \[
    \bigl(\mathcal L_E(\Omega)\bigr)_{ij}
    =
    (\lambda_i+\lambda_j)\Omega_{ij},
    \qquad i<j.
  \]
  Consequently,
  \[
    \restr{\mathcal L_E}{\so(s)}
    \text{ is invertible}
    \quad\Longleftrightarrow\quad
    \dim\ker E\le1.
  \]
  In particular, the operator remains invertible when
  exactly one eigenvalue vanishes. For \(s\ge2\), whenever
  it is invertible, its operator norm with respect to the
  Frobenius norm satisfies
  \[
    \|\mathcal L_E^{-1}\|_{F\to F}
    =
    \frac{1}{\lambda_1+\lambda_2}.
  \]
  Thus a single eigenvalue tending to zero does not cause
  this inverse to become unbounded if the remaining
  eigenvalues stay bounded away from zero. The principal-bundle
  interpretation imposes a separate
  condition. At a factor \(Z\in M_s\) of rank \(j<s\),
  the stabilizer of the \(O(s)\)-action is conjugate to
  \(O(s-j)\), and the action is therefore no longer free.
  This already occurs for \(j=s-1\), when the stabilizer
  is the nontrivial discrete group \(O(1)\), even though
  the Sylvester operator remains invertible. Accordingly, the formula
  for \(\omega\) remains
  algebraically meaningful at corank one, but the free
  principal \(O(s)\)-bundle description does not extend
  across that rank change. On a rank-\(j\) stratum, we
  instead use the reduced principal bundle $\Mrk{j}{j}\longrightarrow
  \Qrk{j}{j}$ and transport its geometry by the fixed-rank reduction
  isometry.
\end{remark}

\subsection{Padding as normalizer reduction}
\label{subsec:padding-normalizer-reduction}

Padding realizes the standard normalizer reduction of the rank-$s$
orbit type. This identifies the reduced full-rank model and its
mechanical connection with the corresponding constructions for
proper group actions.

\begin{definition}[Normalizer reduction]
  \label{def:normalizer-reduction}
  Let a Lie group $G$ act smoothly and properly on the right on a
  manifold $P$, with all stabilizers conjugate to a closed subgroup $H$.
  The \emph{normalizer} of $H$ and the \emph{exact-stabilizer locus} are
  \[
    N_G(H)\coloneq \{g\in G:gHg^{-1}=H\},
    \qquad P_H\coloneq \{p\in P:G_p=H\}.
  \]
  The \emph{normalizer reduction} is the principal $N_G(H)/H$-bundle
  \[
    \pi_H:P_H\longrightarrow P/G,
    \qquad p\longmapsto[p]_G,
  \]
  with \emph{residual action}
  \[
    p\cdot(gH)\coloneq p\cdot g,
    \qquad p\in P_H,\quad g\in N_G(H).
  \]
\end{definition}

The slice theorem implies that $P_H$ is a smooth submanifold.
The residual action is proper and free, and the canonical map
\[
  P_H/(N_G(H)/H)\longrightarrow P/G,
  \qquad [p]_{N_G(H)/H}\longmapsto[p]_G,
\]
is a diffeomorphism; see \cite[\S29.21]{Michor2008}.
Indeed, every $G$-orbit meets $P_H$, and two points of $P_H$
belong to the same $G$-orbit exactly when they belong to the same
$N_G(H)$-orbit, since $G_{p\cdot g}=g^{-1}G_pg$.
The two quotients have the same local slice coordinates.
For an isometric action, the restricted metric on $P_H$ defines
a mechanical connection by orthogonality to the residual orbits.

\begin{definition}[Reduced principal bundle over a rank stratum]
  \label{def:normalizer-reduced-principal-bundle}
  For $1\le s\le r$, set $H_s\coloneq \I_s\oplus O(k-s)$ and
  \begin{equation}
    \label{eq:normalizer-reduced-principal-bundle}
    \begin{gathered}
      P_{k,s}\coloneq \{X\in\Mks:O(k)_X=H_s\},\\
      \pi_{k,s}^{\mathrm{red}}
      \coloneq \restr{\pi_k}{P_{k,s}}:P_{k,s}\longrightarrow\Qks.
    \end{gathered}
  \end{equation}
  This is the normalizer reduction of the $O(k)$-action on $\Mks$,
  which has a single orbit type. We call it the \emph{reduced
  principal bundle}, with structure group $N_{O(k)}(H_s)/H_s$.
  We equip $P_{k,s}$ with the induced metric $\restr{g_{M_k}}{P_{k,s}}$
  and $\Qks$ with its quotient Riemannian metric $\gQks$.
  We use the convention $O(0)=\{e\}$.
\end{definition}

This is the standard construction on an orbit-type stratum;
see \cite[\S29.23]{Michor2008}. The following proposition identifies
it explicitly with padding.

\begin{proposition}[Padding realizes normalizer reduction]
  \label{prop:padding-normalizer-reduction}
  For $1\le s\le r$,
  \begin{equation}
    \label{eq:padding-normalizer-identifications}
    P_{k,s}=\iks(\Mss),
    \qquad N_{O(k)}(H_s)=O(s)\oplus O(k-s).
  \end{equation}
  The map
  \[
    N_{O(k)}(H_s)/H_s\longrightarrow O(s),
    \qquad \diag(R,S)H_s\longmapsto R,
  \]
  is a Lie-group isomorphism. Under this identification, the
  residual action is $(Z,0)\cdot R=(ZR,0)$.

  The following diagram commutes:
  \[
    \begin{tikzcd}[column sep=1.5em, row sep=large]
      {\bigl(\Mss,\restr{g_{M_s}}{\Mss}\bigr)}
      \arrow[r, "\iks"]
      \arrow[d, "\pi_s"']
      &
      {\bigl(P_{k,s},\restr{g_{M_k}}{P_{k,s}}\bigr)}
      \arrow[r, hook]
      \arrow[d, "\pi_{k,s}^{\mathrm{red}}"]
      &
      {\bigl(\Mks,\restr{g_{M_k}}{\Mks}\bigr)}
      \arrow[d, "\pi_k"]
      \\
      {\bigl(\Qss,\gQss\bigr)}
      \arrow[r, "\Xks"']
      &
      {\bigl(\Qks,\gQks\bigr)}
      \arrow[r, equal]
      &
      {\bigl(\Qks,\gQks\bigr)} .
    \end{tikzcd}
  \]
  Here $\iks$ and $\Xks$ are Riemannian isometries, the upper
  right arrow is the inclusion, and all vertical arrows are
  Riemannian submersions. After the above identification of structure
  groups, the left square is an isometric principal-bundle
  isomorphism covering $\Xks$.
  In particular, the canonical map
  \[
    P_{k,s}/(N_{O(k)}(H_s)/H_s)\longrightarrow\Qks,
    \qquad [X]_{N_{O(k)}(H_s)/H_s}\longmapsto[X]_{O(k)},
  \]
  is a Riemannian isometry for the quotient metric induced by
  $\restr{g_{M_k}}{P_{k,s}}$ and the metric $\gQks$.

  Using the induced Lie-algebra isomorphism, regard the mechanical
  connection $\omega^{\mathrm{red}}$ of
  $\pi_{k,s}^{\mathrm{red}}$ as an $\so(s)$-valued form.
  Then, for $Z\in\Mss$ and $W\in\Tz{Z}\Mss$,
  \begin{equation}
    \label{eq:normalizer-mechanical-connection}
    \omega^{\mathrm{red}}_{(Z,0)}\bigl((W,0)\bigr)
    =\omega_Z(W)
    =\mathcal L_{E_Z}^{-1}\bigl(Z^\top W-W^\top Z\bigr).
  \end{equation}
  Consequently, $(\iks)^*\omega^{\mathrm{red}}=\omega$.
\end{proposition}

\begin{proof}
  The fixed subspace of $H_s$ in $\R^k$ is
  $\R^s\oplus\{0\}$. Thus every $X\in P_{k,s}$ has the
  form $(Z,0)$ with $Z\in\Mss$. Conversely, the rows of a rank-$s$
  factor $(Z,0)$ span this subspace, so its stabilizer is exactly
  $H_s$. Hence $P_{k,s}=\iks(\Mss)$: it consists of the factors
  whose row space is this fixed coordinate $s$-plane.
  An orthogonal matrix normalizing $H_s$ preserves its fixed subspace
  and its orthogonal complement, and is therefore block diagonal.
  Conversely, every element of $O(s)\oplus O(k-s)$ normalizes $H_s$.
  This proves \cref{eq:padding-normalizer-identifications}.

  The smooth homomorphism $\diag(R,S)\mapsto R$
  has kernel $H_s$. Its induced map on the quotient has smooth
  inverse $R\mapsto\diag(R,\I_{k-s})H_s$, so it is
  a Lie-group isomorphism. Moreover,
  \[
    \iks(ZR)
    =(ZR,0)
    =(Z,0)\diag(R,\I_{k-s}),
    \qquad
    \pi_{k,s}^{\mathrm{red}}\circ\iks=\Xks\circ\pi_s.
  \]
  These identities prove equivariance and commutativity of the left
  square; the right square commutes by the definition of
  $\pi_{k,s}^{\mathrm{red}}$.

  Projection onto the first $s$ columns is the smooth inverse of
  $\iks:\Mss\to P_{k,s}$, and padding preserves the induced
  Frobenius metrics. The base map $\Xks$ is the Riemannian
  isometry of \cref{thm:fixed-rank-riemannian-reduction}.
  Since $\pi_s$ is a Riemannian submersion, the commuting left
  square shows that $\pi_{k,s}^{\mathrm{red}}$ is also a Riemannian
  submersion for the displayed metrics. Its quotient metric is
  therefore $\gQks$, proving the assertion about the canonical
  quotient identification.

  Finally, the residual vertical space at $(Z,0)$ is
  \[
    \{(Z\Omega,0):\Omega\in\so(s)\}.
  \]
  The differential $\dd\iks(Z)[W]=(W,0)$ preserves the metrics
  and maps vertical spaces onto vertical spaces. It therefore
  identifies their orthogonal complements in the respective tangent
  spaces, and hence the mechanical horizontal distributions.
  Together with equivariance and the reproduction property, this
  gives $(\iks)^*\omega^{\mathrm{red}}=\omega$.
  \cref{eq:normalizer-mechanical-connection} follows from
  \cref{prop:sylvester-mechanical-connection}.
\end{proof}

\begin{remark}[Role of the reduction]
  The normalizer reduction represents the same quotient Riemannian
  geometry by a free $O(s)$-action on the open submanifold $\Mss$
  of $(\mathbb S^{s-1})^n$. Its Levi--Civita connection and curvature
  are restrictions of those of the product of spheres, and its
  mechanical connection is expressed through an invertible
  $s\times s$ Sylvester operator.
  Horizontal lifting and O'Neill's formulas already apply to the
  Riemannian submersion $\pi_k:\Mks\to\Qks$; the reduction
  provides the simpler total space and an ordinary principal
  connection. The distinction between freeness and invertibility
  of the Sylvester operator is explained in
  \cref{rem:sylvester-rankwise}.
\end{remark}

\subsection{Quotient speeds and horizontal lifts}
\label{subsec:quotient-speeds-horizontal-lifts}

The connection form gives an explicit formula for the Riemannian
speed of a curve in the reduced quotient.

\begin{proposition}[Quotient speed on a fixed-rank stratum]
  \label{prop:quotient-speed-formula}
  Let \(1\le m\le s\), let \(I\subset\R\) be an interval, and
  let \(Z:t\in I\longmapsto Z(t)\in \Mrk{s}{m}\) be a continuous
  piecewise \(C^1\) curve. Set \(\bar z(t)\coloneq [Z(t)]\in \Qrk{s}{m}\). At
  every time where \(Z'(t)\) exists, one has
  \[
    \|\dot{\bar z}(t)\|_{\gQrk{s}{m}}
    =
    \|Z'(t)^{\Hbundle}\|_F,
  \]
  where \(Z'(t)^{\Hbundle}\) denotes the Frobenius-orthogonal
  projection of \(Z'(t)\) onto the horizontal space
  \[
    \Hz{Z(t)}
    \coloneq 
    \Vz{Z(t)}^\perp
    \cap \Tz{Z(t)}\Mrk{s}{m}.
  \]
  Equivalently,
  \[
    \|\dot{\bar z}(t)\|_{\gQrk{s}{m}}
    =
    \min_{\Omega\in\so(s)}
    \|Z'(t)-Z(t)\Omega\|_F.
  \]
  The minimizing vertical vector \(Z(t)\Omega\) is unique. The matrix
  \(\Omega\) itself is determined only modulo the Lie algebra of the
  isotropy group of \(Z\), that is
  \[
    \so(s)_{Z(t)}
    \coloneq 
    \{\Lambda\in\so(s):Z(t)\Lambda=0\}.
  \]
  In particular, when \(m=s\), the action is free and the minimizing
  matrix is unique. It is given by
  \[
    \Omega=\omega_{Z(t)}(Z'(t)), \quad \textnormal{so that} \quad
    Z'(t)^{\Hbundle} = Z'(t)-Z(t)\omega_{Z(t)}(Z'(t)).
  \]
  If \(I=[a,b]\), then
  \begin{equation}
    \label{eq:quotient-length-upper-bound}
    \Length_{\dQ{s}}(\bar z)
    \le
    \Length_{\gQrk{s}{m}}(\bar z)
    =
    \int_a^b\|Z'(t)^{\Hbundle}\|_F\dd t.
  \end{equation}
\end{proposition}

\begin{proof}
  The rank-\(m\) locus \(\Mrk{s}{m}\) is a single orbit-type
  manifold for the \(O(s)\)-action and the restricted quotient map
  \[
    \pi_s^{(m)}
    \coloneq 
    \restr{\pi_s}{\Mrk{s}{m}}:
    \Mrk{s}{m}
    \longrightarrow
    \Qrk{s}{m}
  \]
  is a Riemannian submersion for the quotient metric
  \(\gQrk{s}{m}\), hence,
  \[
    \|\dot{\bar z}(t)\|_{\gQrk{s}{m}}
    =
    \|Z'(t)^{\Hbundle}\|_F.
  \]
  Write the orthogonal decomposition \(Z'(t)=Z'(t)^{\Hbundle}+Z'(t)^{\Vbundle}\). For every \(\Omega\in\so(s)\),
  the vector \(Z(t)\Omega\) is vertical, and hence
  \[
    Z'(t)-Z(t)\Omega
    =
    Z'(t)^{\Hbundle}
    +
    \bigl(Z'(t)^{\Vbundle}-Z(t)\Omega\bigr)
  \]
  is an orthogonal decomposition. Consequently,
  \[
    \|Z'(t)-Z(t)\Omega\|_F^2
    =
    \|Z'(t)^{\Hbundle}\|_F^2
    +
    \|Z'(t)^{\Vbundle}-Z(t)\Omega\|_F^2.
  \]
  Since \(Z'(t)^{\Vbundle}\in\Vz{Z(t)}\), there exists
  \(\Omega_0\in\so(s)\) such that \(Z'(t)^{\Vbundle}=Z(t)\Omega_0\). It follows that
  \[
    \min_{\Omega\in\so(s)}
    \|Z'(t)-Z(t)\Omega\|_F
    =
    \|Z'(t)^{\Hbundle}\|_F.
  \]
  The minimizing vertical vector is necessarily
  \(Z'(t)^{\Vbundle}\), and is therefore unique (as orthogonal
  projections onto a linear subspace are unique). Two minimizing
  matrices \(\Omega_1\) and \(\Omega_2\) represent the same vertical
  vector if and only if \(Z(t)(\Omega_1-\Omega_2)=0\), that is, if
  and only if \(\Omega_1-\Omega_2\in\so(s)_{Z(t)}\), and
  all minimizers are of the form \(\Omega = \Omega_1 + \Lambda \)
  with \( \Lambda \in \so(s)_{Z(t)}\).  When \(m=s\), the
  matrix \(Z(t)\) has full column rank, so its isotropy algebra is
  trivial. The minimizing matrix is therefore unique, and
  \cref{prop:sylvester-mechanical-connection} gives
  \(\Omega_0=\omega_{Z(t)}(Z'(t))\). Hence
  \[
    Z'(t)^{\Hbundle}
    =
    Z'(t)-Z(t)\omega_{Z(t)}(Z'(t)).
  \]

  Finally, the Riemannian submersion \(\pi_s^{(m)}\) admits local
  horizontal lifts of \(\bar z\), each having the same Riemannian
  length as the corresponding curve segment. Since
  \(\pi_s:(M_s,\dM{s})\to(Q_s,\dQ{s})\) is \(1\)-Lipschitz, the
  \(\dQ{s}\)-length of each segment is at most the length of its
  horizontal lift in \(M_s\). Subdividing \([a,b]\) into finitely many
  such segments and adding their lengths proves
  \cref{eq:quotient-length-upper-bound}.
\end{proof}

\begin{definition}[Horizontal lifts and parallel transport]
  \label{def:horizontal-lifts-parallel-transport}
  Let \(\bar z:[a,b]\to\Qss\) be a continuous piecewise \(C^1\)
  curve. A \emph{horizontal lift} of \(\bar z\) is a continuous
  piecewise \(C^1\) curve \(Z:[a,b]\to\Mss\) such that
  \[
    \pi_s(Z(t))=\bar z(t),
    \qquad Z'(t)\in\Hz{Z(t)}
    \quad\text{at every differentiability time}.
  \]
  Equivalently, \(\omega_{Z(t)}(Z'(t))=0\). The endpoint of the
  unique horizontal lift with prescribed initial point defines
  \emph{parallel transport} along \(\bar z\):
  \[
    \operatorname{PT}_{\bar z}:
    \pi_s^{-1}(\bar z(a))\longrightarrow\pi_s^{-1}(\bar z(b)),
    \qquad Z_a\longmapsto Z(b),\quad Z(a)=Z_a.
  \]
  Existence and uniqueness are established in
  \cref{lem:horizontal-lifts-exist} below.
\end{definition}

\begin{lemma}[Horizontal lifts]
  \label{lem:horizontal-lifts-exist}
  Let \(\bar z:[a,b]\to\Qss\) be a continuous piecewise
  \(C^1\) curve. For every prescribed initial point
  \(Z_a\in\pi_s^{-1}(\bar z(a))\), there exists a unique
  continuous piecewise \(C^1\) horizontal lift
  \(Z:[a,b]\to\Mss\) satisfying \(Z(a)=Z_a\).
  Moreover,
  \[
    \Length_{\gQss}(\bar z)
    =
    \Length_{M_s}(Z)
    =
    \int_a^b \|Z'(t)\|_F\,\dd t.
  \]
\end{lemma}

\begin{proof}
  Choose a partition \(a=t_0<\cdots<t_N=b\) such that
  \(\restr{\bar z}{[t_{\ell-1},t_\ell]}\) is \(C^1\) for every
  \(\ell\). Since
  \(\pi_s:\Mss\to\Qss\) is a principal \(O(s)\)-bundle
  endowed with the mechanical connection, each restriction
  admits a unique horizontal lift with prescribed initial
  point. Starting at \(Z_a\) and using the endpoint of
  each lifted piece as the initial point of the next
  produces a continuous piecewise \(C^1\) horizontal lift
  \(Z:[a,b]\to\Mss\). Uniqueness on each piece implies
  uniqueness on the whole interval.

  Since \(\pi_s\) is a Riemannian submersion and \(Z'\)
  is horizontal, on each \(C^1\) piece we have
  \[
    \|\dot{\bar z}(t)\|_{\gQss}
    =
    \|Z'(t)\|_F.
  \]
  Integrating over the pieces gives
  \[
    \Length_{\gQss}(\bar z)
    =
    \sum_{\ell=1}^N
    \int_{t_{\ell-1}}^{t_\ell}\|Z'(t)\|_F\,\dd t
    =
    \Length_{M_s}(Z)
    =
    \int_a^b\|Z'(t)\|_F\,\dd t.
  \]
\end{proof}

\begin{lemma}[Padding preserves horizontality]
  \label{lem:padding-preserves-horizontal-lifts}
  Let \(Z\in \Mss\), and write
  \[
    X\coloneq \iks(Z)=(Z,0)\in \Mks.
  \]
  If \(W\in\Tz{Z}\Mss\) is horizontal for the \(O(s)\)-action, then
  $\iks(W) = \dd\iks(Z)[W]
  =(W,0)\in\Tz{X}\Mks$ is horizontal for the \(O(k)\)-action.
  Moreover, \(\|(W,0)\|_F=\|W\|_F\).
\end{lemma}

\begin{proof}
  Let \(\Omega\in\so(k)\), and write it in block form according to
  \(\R^k=\R^s\oplus\R^{k-s}\):
  \[
    \Omega=
    \begin{pmatrix}
      A & B \\
      -B^\top & C
    \end{pmatrix},
    \qquad
    A\in\so(s).
  \]
  Then
  \[
    X\Omega
    =
    (Z,0)
    \begin{pmatrix}
      A & B \\
      -B^\top & C
    \end{pmatrix}
    =
    (ZA,ZB).
  \]
  Therefore
  \[
    \Frob{(W,0)}{X\Omega}
    =
    \Frob{W}{ZA}
    =
    \tr(W^\top ZA)
    =
    \tr(A W^\top Z).
  \]
  Since \(W\) is horizontal in the reduced model, \(Z^\top W\) is
  symmetric, equivalently \(W^\top Z\) is symmetric. Hence
  \(\tr(A W^\top Z)=0\) for every
  \(A\in\so(s)\). Thus \((W,0)\) is orthogonal to every
  vertical vector \(X\Omega\), and hence is horizontal for the
  \(O(k)\)-action. The equality of norms is immediate.
\end{proof}

\subsection{Intrinsic length on a fixed-rank stratum}
\label{subsec:quotient-distance}

We use the metric length $\Length_d$ and the term ``rectifiable'' as
defined in \cref{def:curve-lengths}. In particular,
$\Length_{M_k}$ is length for the product-sphere metric, while
$\Length_{\dQ{k}}$ is length in the ambient quotient. In particular,
$\dQ{k}$ permits paths through other ranks even when both endpoints lie
in $\Qks$.
The intrinsic distance \(\dInt{k}{s}\) is defined in
\cref{def:intrinsic-rank-constraint-distances}, by
\cref{eq:intrinsic-stratum-distance-induced-by-dk}.
As there, the infimum over an empty curve class is \(+\infty\);
Riemannian distances between distinct connected components are
interpreted in the same way. For \(s\ge2\), the strata are path connected by
\cref{thm:path-connectedness-fixed-rank-strata}.

The following is a consequence of the two reduction theorems and the
compatibility of the ambient quotient distance with the Riemannian
distance on the reduced full-rank stratum
\cite[Proposition~3.5]{Chen2025}.

\begin{corollary}[Intrinsic distance on a fixed-rank stratum]
  \label{thm:ambient-quotient-length-on-strata}
  Let \(1\le s\le r=\min(n,k)\). The intrinsic length distance induced
  by \(\dQ{k}\) on \(\Qks\) is the Riemannian distance of
  \((\Qks,\gQks)\):
  \[
    \dInt{k}{s}=\dQks.
  \]
  More precisely, every continuous \(\dQ{k}\)-rectifiable curve
  \(\bar\gamma:[a,b]\to\Qks\) satisfies
  \[
    \Length_{\dQ{k}}(\bar\gamma)
    =\Length_{\dQks}(\bar\gamma).
  \]
  In particular, every continuous piecewise \(C^1\) curve in \(\Qks\)
  satisfies
  \[
    \Length_{\dQ{k}}(\bar\gamma)
    =\Length_{\gQks}(\bar\gamma).
  \]
\end{corollary}

\begin{proof}
  Assume first that \(s\ge2\). Given \(\bar p,\bar q\in\Qks\), set
  \(\bar p_s\coloneq \Xks^{-1}(\bar p)\), \(\bar q_s\coloneq \Xks^{-1}(\bar q)\).
  We first verify $\dQ{s}(\bar p_s,\bar q_s)=\dQss(\bar p_s,\bar q_s)$.
  By compactness of \(O(s)\), we may choose representatives
  \(Z_0,Z_1\in\Mss\) of \(\bar p_s,\bar q_s\) such that
  \(\dM{s}(Z_0,Z_1)=\dQ{s}(\bar p_s,\bar q_s)\). Let \(Z:[0,1]\to M_s\)
  be a minimizing product-sphere geodesic
  joining \(Z_0\) to \(Z_1\). Since it minimizes the distance between
  these two orbits, it is normal to every orbit it meets, and its
  projection is a \(\dQ{s}\)-minimizing segment
  \cite[Proposition~3.1]{AKLM03}.
  By \cite[Lemma~3.5]{AKLM03}, the stabilizer of every interior
  point of \(Z\) is contained in the endpoint stabilizers.
  The latter are trivial because \(Z_0\) and \(Z_1\) have full
  column rank. Hence \(Z(t)\in\Mss\) for every \(t\in[0,1]\).
  The Riemannian-submersion property therefore gives
  \[
    \dQss(\bar p_s,\bar q_s)
    \le
    \Length_{\gQss}(\pi_s\circ Z)
    =
    \Length_{M_s}(Z)
    =
    \dQ{s}(\bar p_s,\bar q_s).
  \]
  Conversely, horizontal lifting
  (\cref{lem:horizontal-lifts-exist}) and the \(1\)-Lipschitz
  property of \(\pi_s:(M_s,\dM{s})\to(Q_s,\dQ{s})\) give
  \[
    \dQ{s}(\bar p_s,\bar q_s)
    \le \Length_{\gQss}(\bar\eta)
  \]
  for every continuous piecewise \(C^1\) curve \(\bar\eta\) in
  \(\Qss\) joining \(\bar p_s\) to \(\bar q_s\).
  Taking the infimum over these curves proves the reverse
  inequality and establishes full-rank compatibility. Combining this
  identity with
  \cref{thm:global-rank-bounded-metric-reduction,thm:fixed-rank-riemannian-reduction}
  gives
  \begin{equation}
    \label{eq:rank-bounded-distance-on-top-stratum}
    \begin{aligned}
      \dkles(\bar p,\bar q)
      =\dQ{s}(\bar p_s,\bar q_s)
      =\dQss(\bar p_s,\bar q_s)
      =\dQks(\bar p,\bar q).
    \end{aligned}
  \end{equation}
  Passing from a metric to its intrinsic length distance does not
  change the length of a rectifiable curve. Indeed, if
  \(\bar\gamma\) is a continuous \(\dQ{k}\)-rectifiable curve in
  \(\Qkles\), then for \(u<v\),
  \[
    \dQ{k}\bigl(\bar\gamma(u),\bar\gamma(v)\bigr)
    \le
    \dkles\bigl(\bar\gamma(u),\bar\gamma(v)\bigr)
    \le
    \Length_{\dQ{k}}\bigl(\restr{\bar\gamma}{[u,v]}\bigr).
  \]
  The right-hand side also shows that \(\bar\gamma\) is continuous
  for \(\dkles\). Summing over a partition and taking the supremum
  gives
  \[
    \Length_{\dQ{k}}(\bar\gamma)
    =
    \Length_{\dkles}(\bar\gamma).
  \]
  For curves in \(\Qks\),
  \cref{eq:rank-bounded-distance-on-top-stratum} therefore gives
  the claimed equality of metric lengths.
  Conversely, since \(\dQ{k}\le\dkles\), every continuous
  \(\dQks\)-rectifiable curve in \(\Qks\) is
  \(\dQ{k}\)-rectifiable.
  Taking infima over rectifiable curves with fixed endpoints proves
  \[
    \dInt{k}{s}=\dQks.
  \]
  For piecewise \(C^1\) curves, metric length for \(\dQks\) is the
  usual Riemannian length for \(\gQks\), which proves the final
  assertion.

  If \(s=1\), the stratum is finite and discrete. Every continuous
  curve in it is constant, and both distances are zero on the
  diagonal and \(+\infty\) between distinct points.
  All the length identities are then immediate.
\end{proof}

\begin{remark}[Restricted distance versus intrinsic stratum distance]
  \label{rem:restricted-versus-intrinsic-stratum-distance}
  For \(s\ge2\) and \(\bar p,\bar q\in\Qks\), the proof gives
  \[
    \dQ{k}(\bar p,\bar q)
    \le\dkles(\bar p,\bar q)
    =\dQks(\bar p,\bar q).
  \]
  Thus allowing curves to pass through lower-rank strata within
  \(\Qkles\) does not decrease the distance between rank-\(s\) endpoints.
  By contrast, the unrestricted ambient quotient distance \(\dQ{k}\) can
  use higher-rank points, so its pointwise restriction to \(\Qks\) need
  not equal \(\dQks\); compare \cite[Proposition~3.6]{Chen2025}.
  When \(s=k\le n\), one has \(\Qkles=Q_k\), and the equality reduces
  to the full-rank compatibility used above.
\end{remark}

\subsection{Chordal quotient comparisons}
\label{subsec:chordal-quotient-comparisons}

For use in \cref{sec:excess-rank-constrained-geodesics}, we record
two comparisons involving the following distance.

\begin{definition}[Chordal quotient distance]
  \label{def:chordal-quotient-distance}
  For \(m\ge1\), the \emph{chordal quotient distance} on \(Q_m\) is
  \[
    \dChord{m}([A],[B])
    \coloneq 
    \min_{R\in O(m)}\|A-BR\|_F,
    \qquad A,B\in M_m.
  \]
  It minimizes Frobenius chord lengths between factor representatives.
  For \(m\ge2\), the angular quotient distance \(\dQ{m}\) instead
  minimizes product-of-spheres distances, as in
  \cref{def:factor-quotient-metrics}.
\end{definition}

The next lemma relates the Procrustes distance, the chordal quotient
distance and the nuclear norm.

\begin{lemma}[Procrustes and chordal distance]
  \label{lem:procrustes-chordal}
  For \(A,B\in M_m\), one has \cite{Schonemann1966}
  \[
    \dChord{m}([A],[B])^2
    =
    \min_{R\in O(m)}\|A-BR\|_F^2
    =
    \|A\|_F^2+\|B\|_F^2-2\|A^\top B\|_*.
  \]
  Moreover, if $A^\top B=U\Sigma V^\top$ is a singular value
  decomposition, with \(U,V\in O(m)\) and \(\Sigma\) diagonal with
  nonnegative entries, then the minimum is attained at \(R=VU^\top\).
  If furthermore, \(A^\top B\) is invertible, this minimizer is
  unique and can be written as
  \[
    R
    =
    \big((A^\top B)^\top(A^\top B)\big)^{-1/2}(A^\top B)^\top.
  \]
  In particular, on the open set where \(A^\top B\) is invertible,
  the minimizing orthogonal factor depends smoothly, indeed
  analytically, on the entries of \(A^\top B\).
\end{lemma}

\begin{proof}
  For any \(R\in O(m)\),
  \[
    \|A-BR\|_F^2
    =
    \|A\|_F^2+\|B\|_F^2
    -2\tr(A^\top B R).
  \]
  By von Neumann's trace inequality, $\tr(A^\top B
  R)\le \|A^\top B\|_*.$ For \(R=VU^\top\), we have
  \[
    \tr(A^\top B R)
    =
    \tr(U\Sigma V^\top VU^\top)
    =
    \tr(U\Sigma U^\top)
    =
    \tr\Sigma
    =
    \|A^\top B\|_*.
  \]
  Thus the upper bound is attained, and hence
  \[
    \min_{R\in O(m)}\|A-BR\|_F^2
    =
    \|A\|_F^2+\|B\|_F^2-2\|A^\top B\|_*.
  \]
  When \(A^\top B\) is invertible, the minimizer is unique and equals
  the orthogonal
  polar factor
  \[
    R
    =
    \big((A^\top B)^\top(A^\top B)\big)^{-1/2}(A^\top B)^\top.
  \]
  This formula is analytic on the invertible locus, hence smooth.
\end{proof}

\begin{lemma}[Chordal comparison for padded factors]
  \label{lem:chordal-comparison-padded-factors}
  Let \(Z_0,Z_1\in M_s\), and set \(X_i\coloneq \iks(Z_i)=(Z_i,0)\in M_k\),
  \(i=0,1\). Then \(\dQ{k}([X_0],[X_1])\ge \dChord{s}([Z_0],[Z_1])\).
\end{lemma}

\begin{proof}
  For all \(P,Q\in M_k\), the Riemannian distance on \(M_k\) satisfies
  \[
    \dM{k}(P,Q)
    =
    \inf_{\substack{\gamma\subset M_k\\ \gamma(0)=P,\,\gamma(1)=Q}}
    \Length_{M_k}(\gamma).
  \]
  Since the product metric on \(M_k\) is induced by the ambient
  Frobenius inner product, every piecewise \(C^1\) curve
  \(\gamma\subset M_k\) has length
  \(\Length_{M_k}(\gamma)=\Length_{F}(\gamma)\). Therefore
  \[
    \begin{aligned}
      \dM{k}(P,Q)
      =
      \inf_{\substack{\gamma\subset M_k\\ \gamma(0)=P,\,\gamma(1)=Q}}
      \Length_{F}(\gamma)
      \ge
      \inf_{\substack{\gamma\subset \R^{n\times
      k}\\ \gamma(0)=P,\,\gamma(1)=Q}}
      \Length_{F}(\gamma)
      =
      \|P-Q\|_F.
    \end{aligned}
  \]
  Applying this with \(P=X_0\) and \(Q=X_1R\), and then taking the infimum over
  \(R\in O(k)\), gives
  \[
    \begin{aligned}
      \dQ{k}([X_0],[X_1])
      =
      \inf_{R\in O(k)} \dM{k}(X_0,X_1R)
      \ge
      \inf_{R\in O(k)} \|X_0-X_1R\|_F
      =
      \dChord{k}([X_0],[X_1]).
    \end{aligned}
  \]
  It remains to compute the chordal quotient distance of the padded
  factors. By \cref{lem:procrustes-chordal},
  \(\dChord{m}([A],[B])^2=\|A\|_F^2+\|B\|_F^2-2\|A^\top B\|_*\). For
  \(A,B\in M_m=(\mathbb S^{m-1})^n\), one has $\|A\|_F^2=\|B\|_F^2=n$
  and therefore $  \dChord{m}([A],[B])^2 = 2n-2\|A^\top B\|_*$. Now since
  \[
    X_0^\top X_1
    =
    \begin{pmatrix}
      Z_0^\top Z_1 & 0 \\
      0 & 0
    \end{pmatrix},
  \]
  we have \(\|X_0^\top X_1\|_*= \|Z_0^\top Z_1\|_*\). Hence
  \(\dChord{k}([X_0],[X_1])= \dChord{s}([Z_0],[Z_1])\). Combining this
  with the previous inequality gives \(\dQ{k}([X_0],[X_1]) \ge
  \dChord{s}([Z_0],[Z_1])\) as claimed.
\end{proof}

\subsection{Computing fixed-rank geodesics by horizontal lifts}
\label{subsec:fixed-rank-geodesics-horizontal-lifts}

On \(\Crk sn\), write
\[
  g_Q\coloneq (\bar\Gamma_k)_*\gQks
  \]
for the quotient Riemannian metric from $\Qks$ transported by the
Gram diffeomorphism
$\bar \Gamma_k \colon [X] \in \Qks \longmapsto XX^\top \in\Crk sn$.
We now explain how geodesics of
the fixed-rank Riemannian stratum
\((\Qks,\gQks)\) are computed in
practice, using horizontal lifts in the
reduced full-rank model, or equivalently, by
\cref{thm:ambient-quotient-length-on-strata},
the geodesics for the intrinsic fixed-rank
distance \(\dInt{k}{s}\). As
stated in
\cref{rem:restricted-versus-intrinsic-stratum-distance},
they should not be confused with geodesics
for the ambient quotient distance \(\dQ{k}\),
even when both endpoints lie in
\(\Qks\). Let us first recall the
ambient quotient geodesics. The product
metric on \(M_k=(\mathbb S^{k-1})^n\) is the
metric induced by the Frobenius inner product
on \(\R^{n\times k}\). Thus geodesics
in \(M_k\) are computed row by row as
great-circle geodesics on \(\mathbb
S^{k-1}\). Given \([X_0],[X_1]\in Q_k\),
choose an optimal alignment \(R_*\in O(k)\)
such that
\[
  \dQ{k}([X_0],[X_1])
  =
  \dM{k}(X_0,X_1R_*).
\]
Writing row-wise
\[
  x_i^0\coloneq (X_0)_{i\cdot},
  \qquad
  x_i^1\coloneq (X_1R_*)_{i\cdot},
\]
and setting \(\theta_i\coloneq \arccos\langle x_i^0,x_i^1\rangle\), a
minimizing geodesic \(X(t)\) from \(X_0\) to \(X_1R_*\) is obtained
row-wise via \(x_i(t)\coloneq X(t)_{i\cdot}\) as follows:
\[
  x_i(t)
  =
  \frac{\sin((1-t)\theta_i)}{\sin\theta_i}\,x_i^0
  +
  \frac{\sin(t\theta_i)}{\sin\theta_i}\,x_i^1,
  \qquad 0<\theta_i<\pi.
\]
If \(\theta_i=0\), the row is constant. If \(\theta_i=\pi\), one
chooses any unit vector \(w_i\in \Tz{x_i^0}\mathbb S^{k-1}\) and takes
\[
  x_i(t)=\cos(\pi t)x_i^0+\sin(\pi t)w_i.
\]
The corresponding ambient quotient geodesic is \(t\longmapsto
[X(t)]\in Q_k\). Even if \([X_0]\) and \([X_1]\) belong to
\(\Qks\), this ambient geodesic need not remain in the
fixed-rank stratum. Fixed-rank geodesics are instead computed in the
reduced full-rank model, where the Sylvester operator is nonsingular. Let
\[
  x_0\in \Qks,
  \qquad
  \bar v_0\in\Tz{x_0}\Qks.
\]
By the fixed-rank reduction isometry, write
\[
  x_0=\Xks([Z_0]),
  \qquad
  Z_0\in \Mss,
\]
and define
\[
  \bar u_0
  \coloneq 
  \dd\Xks^{-1}(x_0)[\bar v_0]
  \in\Tz{[Z_0]}\Qss.
\]
Choose any lift
\[
  W_0\in\Tz{Z_0}\Mss
  \qquad\text{such that}\qquad
  \dd\pi_s(Z_0)[W_0]=\bar u_0.
\]
The horizontal representative of \(\bar u_0\) is
\[
  W_0^{\Hbundle}\coloneq W_0-Z_0\Omega_0,
\]
where \(\Omega_0\in\so(s)\) is the unique solution of the
Sylvester equation
\[
  E_{Z_0}\Omega_0+\Omega_0E_{Z_0}
  =
  Z_0^\top W_0-W_0^\top Z_0,
  \qquad
  E_{Z_0}\coloneq Z_0^\top Z_0\in\SPD(s).
\]
Thus \(W_0^{\Hbundle}\) is the unique horizontal lift at \(Z_0\) of
\(\bar u_0\):
\[
  W_0^{\Hbundle}\in\Hz{Z_0},
  \qquad
  \dd\pi_s(Z_0)[W_0^{\Hbundle}]=\bar u_0.
\]
The geodesic in the total space \(M_s=(\mathbb S^{s-1})^n\) with initial data
\[
  Z(0)=Z_0,
  \qquad
  Z'(0)=W_0^{\Hbundle}
\]
is explicit row by row. Writing \(z_i\) and \(w_i\) for the \(i\)-th
rows of \(Z_0\) and \(W_0^{\Hbundle}\), and setting
\(\alpha_i\coloneq \|w_i\|\), one has
\begin{equation}
  \label{eq:fixed-rank-rowwise-exponential}
  \begin{aligned}
    & z_i(t)
    =
    \begin{cases}
      \displaystyle
      \cos(\alpha_i t)\,z_i
      +
      \sin(\alpha_i t)\,\dfrac{w_i}{\alpha_i},
      & \alpha_i\ne0,\\[2mm]
      z_i,
      & \alpha_i=0,
    \end{cases}\\
    & Z(t)_{i\cdot}=z_i(t),\qquad i=1,\ldots,n.
  \end{aligned}
\end{equation}
Since \(\Mss\subset M_s\) is open, this curve remains in
\(\Mss\) for \(|t|\) sufficiently small.  Moreover, because
\(\pi_s:\Mss\to \Qss\) is a Riemannian submersion, a
geodesic with horizontal initial velocity remains horizontal for as
long as it is defined in \(\Mss\), and its projection
\(\gamma_s(t)\coloneq \pi_s(Z(t))\) is the geodesic of
\((\Qss,\gQss)\) with initial data
\(\gamma_s(0)=[Z_0],\;\dot\gamma_s(0)=\bar u_0\). Transporting by the
fixed-rank reduction isometry gives the corresponding local geodesic
in \(\Qks\):
\[
  \gamma_k(t)\coloneq \Xks(\gamma_s(t)).
\]
\begin{theorem}[Fixed-rank correlation geodesics by horizontal lifts]
  \label{thm:fixed-rank-correlation-geodesic}
  Let \(2\le s\le\min(n,k)\), \(Z_0\in \Mss\), and
  \(W_0\in\Tz{Z_0}\Mss\). Set
  \[
    E_0\coloneq Z_0^\top Z_0,\qquad
    \Omega_0\coloneq \mathcal L_{E_0}^{-1}
    \bigl(Z_0^\top W_0-W_0^\top Z_0\bigr),
    \qquad H_0\coloneq W_0-Z_0\Omega_0.
  \]
  Consider \(Z(t)\coloneq \Exp^{M_s}_{Z_0}(tH_0)\)
  using \cref{eq:fixed-rank-rowwise-exponential}, and let \(I\) be
  the connected component containing zero
  of \(\{t\in\R:\rank Z(t)=s\}\).
  Then
  \[
    C(t)\coloneq Z(t)Z(t)^\top,\qquad t\in I,
  \]
  is the unique geodesic in \((\Crk{s}{n},g_Q)\) with initial condition
  \[
    C(0)=Z_0Z_0^\top,\qquad
    \dot C(0)=H_0Z_0^\top+Z_0H_0^\top
    =W_0Z_0^\top+Z_0W_0^\top.
  \]
\end{theorem}

\begin{proof}
  By \cref{prop:sylvester-mechanical-connection}, \(H_0\) is horizontal.
  The rowwise formula in \cref{eq:fixed-rank-rowwise-exponential}
  therefore gives the
  quotient geodesic and its Gram image in \((\Crk{s}{n},g_Q)\).
\end{proof}
Thus the fixed-rank geodesic computation is carried out in the
reduced full-rank model \(\Mss\), using the nonsingular
\(s\times s\) Sylvester equation. This keeps the computation on the
free \(O(s)\)-bundle with \(Z^\top Z\) positive definite. For padded
factors in \(\Mks\), the \(k\times k\) Sylvester operator on
\(\so(k)\) is singular precisely when \(k-s\ge2\); see
\cref{rem:sylvester-rankwise}.

\smallskip
We conclude this section by providing an elementary example that
shows that a fixed-rank stratum is not geodesically complete.

\begin{theorem}[Finite-time rank degeneration]
  \label{thm:finite-time-rank-degeneration}
  Let \(2\le s\le \min(n,k)\). Then the fixed-rank Riemannian stratum
  $(\Qks,\gQks)$ is not geodesically complete.
\end{theorem}

\begin{proof}
  By the fixed-rank reduction isometry, it is enough to work in the
  reduced model \(\Qss\).  We construct a horizontal geodesic in
  \(M_s=(\mathbb S^{s-1})^n\) which reaches \(\Mrk{s}{s-1}\) in finite time.

  \smallskip
  First assume \(s=2\). Let \((e_1,e_2)\) be the standard basis of
  \(\R^2\) and set
  \[
    u\coloneq \frac{e_1+e_2}{\sqrt2},
  \]
  and take
  \[
    Z_0\coloneq
    \begin{pmatrix}
      e_1\\
      e_2\\
      u\\
      \vdots\\
      u
    \end{pmatrix}
    \in \Mrk{2}{2},
    \qquad
    W_0\coloneq
    \begin{pmatrix}
      e_2\\
      e_1\\
      0\\
      \vdots\\
      0
    \end{pmatrix}
    \in\Tz{Z_0}M_2.
  \]
  The vector $W_0$ is tangent because each row is orthogonal to the
  corresponding row of $Z_0$. It is horizontal because \(Z_0^\top
  W_0=W_0^\top Z_0=E_{12}+E_{21}\), where
  \(E_{ij}\coloneq(\delta_{ik}\delta_{jl})_{k,l}\), so $Z_0^\top W_0$
  is symmetric.
  The row-wise geodesic with initial data \((Z_0,W_0)\) satisfies
  \(z_1(t)=\cos t\,e_1+\sin t\,e_2\), \(z_2(t)=\sin t\,e_1+\cos
  t\,e_2\) and all remaining rows are constant equal to \(u\). At
  \(t=\pi/4\), all rows are equal to \(u\), hence \(\rank Z(\pi/4)=1\).

  \smallskip
  Now assume \(s\ge3\).  Take
  \[
    Z_0\coloneq
    \begin{pmatrix}
      e_1\\
      e_2\\
      e_3\\
      \vdots\\
      e_s\\
      e_3\\
      \vdots\\
      e_3
    \end{pmatrix}
    \in \Mss,
    \qquad
    W_0\coloneq
    \begin{pmatrix}
      e_2\\
      e_1\\
      0\\
      \vdots\\
      0\\
      0\\
      \vdots\\
      0
    \end{pmatrix}.
  \]
  Again \(W_0\) is tangent and horizontal, because each row is
  tangent to the corresponding sphere factor and \(Z_0^\top
  W_0=W_0^\top Z_0=E_{12}+E_{21}\). The corresponding row-wise
  geodesic has the same first two rows as above, while all other rows
  remain fixed. At \(t=\pi/4\), the first two rows coincide, so the row span is
  \[
    \spanv\left\{
      \frac{e_1+e_2}{\sqrt2},e_3,\ldots,e_s
    \right\},
  \]
  which has dimension \(s-1\). Hence \(\rank Z(\pi/4)=s-1\).

  \smallskip
  In both cases, a horizontal geodesic in \(\Mss\) reaches the
  lower-rank locus in finite time. Its projection is therefore a
  geodesic of \(\Qss\) and reaches lower rank strata in finite
  time as well. Transporting the example by $\Xks$ gives the same
  conclusion for \(\Qks\). Thus \((\Qks,\gQks)\) is
  not geodesically complete.
\end{proof}

\section{Curvature of fixed-rank strata via Sylvester equations}
\label{sec:curvature-fixed-rank-strata}

The previous section introduced the Sylvester mechanical connection
$\omega$ on the reduced full-rank model
\[
  \pi_s:\bigl(\Mss,\restr{g_{M_s}}{\Mss}\bigr)
  \longrightarrow\bigl(\Qss,\gQss\bigr).
\]
In this section we use that connection $1$-form to compute the
curvature $2$-form and \emph{intrinsic} sectional,
Ricci, scalar curvatures of each fixed-rank stratum. By the
fixed-rank reduction isometry
\(\Xks:(\Qss,\gQss)\longrightarrow(\Qks,\gQks)\),
it is enough to work on the reduced full-rank model \(\Mss\),
where the \(O(s)\)-action is free, turning
\(\pi_s:\Mss\longrightarrow\Qss\) into a principal bundle
with structure group \(O(s)\). Throughout this section we fix
\[
  2\le s\le \min(n,k),
\]
take \(Z\in \Mss\), \(E_Z\coloneq Z^\top Z\in\SPD(s)\) and denote as before
the vertical and (metric-, mechanical-) horizontal spaces
\(\Vz{Z}=\{Z\Omega:\Omega\in\so(s)\}\),
\(\Hz{Z}=\Vz{Z}^\perp\subset\Tz{Z}\Mss\), respectively. We recall
from \cref{prop:sylvester-mechanical-connection} that the mechanical
connection form is given pointwise by solving a Sylvester equation
\[
  \omega_Z(W)
  =
  \mathcal L_{E_Z}^{-1}\bigl(Z^\top W-W^\top Z\bigr),
  \qquad
  \mathcal L_{E_Z}(\Omega)=E_Z\Omega+\Omega E_Z.
\]
All curvature formulas below are written in this reduced model and
are then transported to \(\Qks\) by the Riemannian isometry \(\Xks\).
The resulting nonpolarity criterion closes this section; the holonomy
classification follows in \cref{sec:holonomy-mechanical-connection}.

\subsection{Curvature form and O'Neill tensor}
\label{subsec:curvature-oneill-tensor}

We first introduce the curvature \(2\)-form of the mechanical
connection and record the corresponding O'Neill tensor formula. We
follow O'Neill's convention for Riemannian submersions \cite{ONeill1966}.

\begin{proposition}[Curvature form, integrability, and O'Neill tensor]
  \label{prop:curvature-oneill-tensor}
  Let
  \[
    \Omega^\omega
    \coloneq 
    \dd \omega+\frac12[\omega\wedge\omega]
    \in\Omega^2(\Mss,\so(s))
  \]
  be the curvature \(2\)-form of the mechanical connection $\omega$.
  If \(U,V\) are horizontal vector fields on \(\Mss\), then
  \[
    \Omega^\omega(U,V)
    =
    -\omega([U,V]).
  \]
  Consequently, the horizontal distribution is integrable if and only if
  \(\Omega^\omega\) vanishes on horizontal pairs. Moreover, let \(A\)
  denote O'Neill's tensor $A$ of the Riemannian submersion
  \(\pi_s:\Mss\to \Qss\). For horizontal vector fields \(U,V\), one
  has (see \cite[Lemma~2]{ONeill1966})
  \[
    A_UV
    =
    \frac12[U,V]^{\Vbundle}
    =
    -\frac12 Z\,\Omega^\omega(U,V).
  \]
  In particular, if \(\Theta\coloneq \Omega^\omega(U,V)\in\so(s)\), then
  \[
    \|A_UV\|_F^2
    =
    \frac14
    \tr\bigl(\Theta^\top E_Z\Theta\bigr).
  \]
\end{proposition}

\begin{proof}
  The first part of the proof is classical, nevertheless we write it
  for completeness. For horizontal vector fields \(U,V\), one has
  \(\omega(U)=\omega(V)=0\). Therefore
  \[
    \dd \omega(U,V)
    =
    U(\omega(V))-V(\omega(U))-\omega([U,V])
    =
    -\omega([U,V]),
  \]
  and the term \([\omega\wedge\omega](U,V)\) vanishes. Hence
  \(\Omega^\omega(U,V)=-\omega([U,V])\). By Frobenius' theorem (see
  for instance \cite[Chap.~I, Proposition~1.2]{KobayashiNomizu1963}),
  the horizontal distribution is integrable
  if and only if the bracket of any two horizontal vector fields is
  again horizontal. This is equivalent to \([U,V]^{\Vbundle}=0\)
  for all horizontal \(U,V\). Since the vertical projection is
  \(W^{\Vbundle}=Z\omega(W)\), we get
  \[
    [U,V]^{\Vbundle}
    =
    Z\omega([U,V])
    =
    -Z\Omega^\omega(U,V).
  \]
  Thus integrability is equivalent to the vanishing of
  \(\Omega^\omega\) on horizontal pairs. For horizontal vector
  fields, O'Neill's tensor $A$ satisfies
  \(A_UV=\frac12[U,V]^{\Vbundle}\). Combining this with the
  previous identity gives
  \[
    A_UV=-\frac12Z\Omega^\omega(U,V).
  \]
  Finally, for \(\Theta\in\so(s)\),
  \[
    \|Z\Theta\|_F^2
    =
    \tr\bigl((Z\Theta)^\top Z\Theta\bigr)
    =
    \tr\bigl(\Theta^\top Z^\top Z\Theta\bigr)
    =
    \tr\bigl(\Theta^\top E_Z\Theta\bigr),
  \]
  which gives the norm formula.
\end{proof}

\subsection{The Sylvester sectional-curvature formula}
\label{subsec:sylvester-sectional-curvature}

We now compute the curvature $2$-form and the sectional curvature
explicitly. Define the $\so(s)$-valued $1$-form
\[
  J_Z(W)\coloneq Z^\top W-W^\top Z,
  \qquad Z\in\Mss,\quad W\in\Tz{Z}\Mss,
\]
so that \(\omega_Z=\mathcal L_{E_Z}^{-1}\circ J_Z\).
The Frobenius-orthogonal decomposition into horizontal and vertical
components is
\[
  W=W^{\Hbundle}+W^{\Vbundle},
  \qquad
  W^{\Vbundle}=Z\Omega,
  \qquad
  \Omega\coloneq \omega_Z(W)
  =\mathcal L_{E_Z}^{-1}\bigl(J_Z(W)\bigr).
\]
Here \(\Omega\in\so(s)\) represents the common infinitesimal
rotation generating the vertical component: \(J_Z(W)\)
determines this rotation through the Sylvester equation. We also
introduce the alternating $\so(s)$-valued
bilinear map
\[
  S_Z(U,V)\coloneq U^\top V-V^\top U,
  \qquad U,V\in\Tz{Z}\Mss.
\]
The next lemma shows that
\[
  S_Z(U,V)=\frac12(\dd J)_Z(U,V).
\]

\begin{lemma}[Exterior derivative of \(J\)]
  \label{lem:dJ-explicit}
  For all \(Z\in \Mss\) and all \(U,V\in\Tz{Z}\Mss\),
  \[
    (\dd J)_Z(U,V)
    =
    2\bigl(U^\top V-V^\top U\bigr).
  \]
  Consequently, for horizontal \(U,V\in\Hz{Z}\),
  \begin{equation}
    \label{eq:curvature-form-sylvester-horizontal}
    \Theta_Z(U,V)
    \coloneq 
    (\Omega^\omega)_Z(U,V)
    =
    2\mathcal L_{E_Z}^{-1}\bigl(U^\top V-V^\top U\bigr) = 2\mathcal
    L_{E_Z}^{-1}\bigl(S_Z(U,V)\bigr).
  \end{equation}
\end{lemma}

\begin{proof}
  The form \(J\) is the restriction of the ambient \(1\)-form
  \(Y^\top\dd Y-(\dd Y)^\top Y\). Since exterior differentiation
  commutes with restriction, we may compute using constant ambient
  extensions of \(U,V\), whose bracket vanishes. Thus
  \[
    \begin{aligned}
      (\dd J)_Z(U,V)
      &=U[J(V)]_Z-V[J(U)]_Z\\
      &=(U^\top V-V^\top U)-(V^\top U-U^\top V)
      =2(U^\top V-V^\top U).
    \end{aligned}
  \]

  For the curvature formula, apply the product rule to
  \(\omega_Z=\mathcal L_{E_Z}^{-1}\circ J_Z\):
  \[
    \begin{aligned}
      (\dd\omega)_Z(U,V)
      &=\mathcal L_{E_Z}^{-1}\bigl((\dd J)_Z(U,V)\bigr)\\
      &\quad+
      \bigl(\dd(\mathcal L_{E_\bullet}^{-1})\bigr)_Z[U]
      \bigl(J_Z(V)\bigr)\\
      &\quad-
      \bigl(\dd(\mathcal L_{E_\bullet}^{-1})\bigr)_Z[V]
      \bigl(J_Z(U)\bigr).
    \end{aligned}
  \]
  Here \(\dd(\mathcal L_{E_\bullet}^{-1})\) denotes the differential
  of the operator-valued map \(Y\mapsto\mathcal L_{E_Y}^{-1}\).
  If \(U,V\) are horizontal, then \(J_Z(U)=J_Z(V)=0\), so the
  last two terms vanish. Moreover,
  \(\omega_Z(U)=\omega_Z(V)=0\), hence
  \([\omega\wedge\omega]_Z(U,V)=0\). Consequently,
  \[
    \begin{aligned}
      (\Omega^\omega)_Z(U,V)
      =(\dd\omega)_Z(U,V)
      =\mathcal L_{E_Z}^{-1}\bigl((\dd J)_Z(U,V)\bigr)
      =2\mathcal L_{E_Z}^{-1}\bigl(S_Z(U,V)\bigr).
    \end{aligned}
  \]
\end{proof}
We now apply O'Neill's formula to the Riemannian submersion
\[
  \pi_s:
  \bigl(\Mss,\restr{g_{M_s}}{\Mss}\bigr)
  \longrightarrow
  \bigl(\Qss,\gQss\bigr).
\]
where we recall \(g_{M_s}\) is the restriction of the ambient
Frobenius metric to the tangent bundle \(\Tbundle M_s\), and
\(\gQss\) is the quotient metric. The curvature computed
below is the sectional curvature of the fixed-rank Riemannian stratum
\((\Qks,\gQks)\), or equivalently, via the fixed-rank reduction
isometry \(\Xks\), that of the reduced model $(\Qss,\gQss)$.

\begin{theorem}[Sectional curvature of a fixed-rank stratum]
  \label{thm:sectional-curvature-fixed-rank-stratum}
  Let \(x\in \Qks\), and choose \(Z\in \Mss\) such that
  $\Xks([Z])=x$. Let \(\bar\sigma\subset\Tz{x}\Qks\) be a \(2\)-plane, and let
  \(U,V\in\Hz{Z}\) be horizontal lifts, through the isometry \(\Xks\), of a
  basis of \(\bar\sigma\).  Write row-wise
  \(z_i\coloneq Z_{i\cdot}\in\R^{1\times s}\) and
  \[
    u_i\coloneq U_{i\cdot},
    \qquad
    v_i\coloneq V_{i\cdot},
    \qquad
    u_i,v_i\in\Tz{z_i}\mathbb S^{s-1}\subset\R^{1\times s}.
  \]
  Then the sectional curvature of $\gQks$ is
  \begin{equation}
    \label{eq:sectional-curvature-fixed-rank-general}
    \begin{aligned}
      \sec_{\Qks}(\bar\sigma)
      =
      \frac{
        \displaystyle
        \sum_{i=1}^n
        \bigl(
          \|u_i\|^2\|v_i\|^2-\langle u_i,v_i\rangle^2
        \bigr)
      }{
        \|U\wedge V\|_F^2
      }
      +
      \frac34\,
      \frac{
        \tr\bigl(\Theta_Z(U,V)^\top E_Z\Theta_Z(U,V)\bigr)
      }{
        \|U\wedge V\|_F^2
      },
    \end{aligned}
  \end{equation}
  where
  \[
    \|U\wedge V\|_F^2
    \coloneq 
    \|U\|_F^2\|V\|_F^2-\Frob{U}{V}^2
  \]
  and
  \[
    \Theta_Z(U,V)
    =
    2\mathcal L_{E_Z}^{-1}\bigl(U^\top V-V^\top U\bigr).
  \]
  In particular, if \(U,V\) are horizontal lifts of an orthonormal
  basis of \(\bar\sigma\), then
  \begin{equation}
    \label{eq:sectional-curvature-fixed-rank-orthonormal}
    \sec_{\Qks}(\bar\sigma)
    =
    \sum_{i=1}^n
    \bigl(
      \|u_i\|^2\|v_i\|^2-\langle u_i,v_i\rangle^2
    \bigr)
    +
    \frac34\,
    \tr\bigl(\Theta_Z(U,V)^\top E_Z\Theta_Z(U,V)\bigr).
  \end{equation}
\end{theorem}

\begin{proof}
  It is enough to compute the sectional curvature on the reduced
  quotient \(\Qss\), because \(\Xks:\Qss\to \Qks\) is a Riemannian
  isometry. Denote \(\bar U\coloneq \dd \pi_s(Z)[U]\), \(\bar V\coloneq \dd
  \pi_s(Z)[V]\). O'Neill's sectional curvature formula for a
  Riemannian submersion
  (see \cite[Thm.~1]{ONeill1966}) gives
  \[
    \sec_{\Qss}(\bar U,\bar V)
    =
    \sec_{\Mss}(U,V)
    +
    \frac{3\|A_UV\|_F^2}{\|U\wedge V\|_F^2}.
  \]
  Since \(\Mss\) is an open submanifold of \(M_s=(\mathbb S^{s-1})^n\), its
  Levi--Civita connection and curvature tensor are the restrictions
  of those of \(M_s\). Thus the upstairs sectional curvature is the
  usual product-of-spheres curvature. Hence
  \[
    \sec_{\Mss}(U,V)
    =
    \frac{
      \displaystyle\sum_{i=1}^n
      \bigl(
        \|u_i\|^2\|v_i\|^2-\langle u_i,v_i\rangle^2
      \bigr)
    }{\|U\wedge V\|_F^2}.
  \]
  By \cref{prop:curvature-oneill-tensor},
  \[
    \|A_UV\|_F^2
    =
    \frac14
    \tr\bigl(\Theta_Z(U,V)^\top E_Z\Theta_Z(U,V)\bigr).
  \]
  Substituting this into O'Neill's formula proves
  \cref{eq:sectional-curvature-fixed-rank-general}. The orthonormal
  formula follows immediately.
\end{proof}

\begin{corollary}[Nonnegative sectional curvature and independence of \(k\)]
  \label{cor:nonnegative-sectional-curvature}
  Every fixed-rank stratum \((\Qks,\gQks)\) has nonnegative sectional
  curvature. Moreover, its curvature tensor is transported by the
  fixed-rank reduction
  isometry from the reduced model \((\Qss,\gQss)\). In particular, all
  intrinsic curvature invariants of \(\Qks\) depend only on
  \((n,s)\), not on \(k\).
\end{corollary}

\begin{proof}
  Both terms in \cref{eq:sectional-curvature-fixed-rank-general} are
  nonnegative. The first is a sum of Cauchy--Schwarz defects in the
  sphere factors. The second is nonnegative because
  \(E_Z\in\SPD(s)\). The statement about independence of \(k\)
  follows from the Riemannian isometry \(\Xks\).
\end{proof}

\begin{remark}[Alexandrov curvature of the rank-bounded quotient]
  \label{rem:rank-bounded-alexandrov-curvature}
  The statement in
  \cref{prop:reduced-quotient-global-metric-geometry} is more global than the
  sign assertion in \cref{cor:nonnegative-sectional-curvature}: it gives an
  Alexandrov lower-curvature bound on the whole singular quotient
  \(M_s/O(s)\), {in the sense of
  \cref{eq:alexandrov-midpoint-comparison}}. On its smooth principal
  stratum this already implies
  nonnegative sectional curvature, and fixed-rank reduction transports that
  sign to every \(\Qks\).  The contribution of
  \cref{thm:sectional-curvature-fixed-rank-stratum} is the finer explicit
  formula, including the two nonnegative terms that support the
  zero-curvature characterization and quantitative curvature bounds below.
\end{remark}

\subsection{Zero sectional curvature directions}
\label{subsec:zero-sectional-curvature}

The sectional curvature formula in
\cref{eq:sectional-curvature-fixed-rank-general} has a remarkable
consequence: the quotient introduces no cancellation: a plane has
zero curvature downstairs if and only if its horizontal lift already
has zero curvature upstairs. Thus flat directions in the quotient are
completely detected by the row-wise geometry of the horizontal lift.

\begin{theorem}[Zero-curvature planes]
  \label{thm:zero-sectional-curvature-characterization}
  Let \(x\in \Qks\), choose \(Z\in \Mss\) such that \(\Xks([Z])=x\),
  and let \(\bar\sigma\subset\Tz{x}\Qks\) be a \(2\)-plane. Let
  \(U,V\in\Hz{Z}\) be horizontal lifts of a basis of \(\bar\sigma\), with rows
  \[
    z_i\coloneq Z_{i\cdot},\qquad u_i\coloneq U_{i\cdot},\qquad v_i\coloneq V_{i\cdot}.
  \]
  Then the following conditions are equivalent:
  \begin{enumerate}
    \item \(\sec_{\Qks}(\bar\sigma)=0\);
    \item \(\sec_{M_s}(U,V)=0\);
    \item for every \(i=1,\ldots,n\), the vectors \(u_i\) and \(v_i\)
      are linearly dependent in \(\Tz{z_i}\mathbb S^{s-1}\).
  \end{enumerate}
  When these equivalent conditions hold, the O'Neill tensor vanishes
  automatically:
  \[
    U^\top V=V^\top U,
    \qquad
    \Theta_Z(U,V)=0,
    \qquad
    A_UV=0.
  \]
\end{theorem}

\begin{proof}
  By \cref{thm:sectional-curvature-fixed-rank-stratum}, the quotient
  sectional curvature is the sum of two nonnegative terms:
  \[
    \sec_{\Qks}(\bar\sigma)
    =
    \sec_{M_s}(U,V)
    +
    \frac{3\|A_UV\|_F^2}{\|U\wedge V\|_F^2}.
  \]
  Thus \(\sec_{\Qks}(\bar\sigma)=0\) if and only if
  \(\sec_{M_s}(U,V)=0\) and \(A_UV=0\). The product-of-spheres
  sectional curvature formula gives
  \[
    \sec_{M_s}(U,V)
    =
    \frac{
      \displaystyle\sum_{i=1}^n
      \bigl(\|u_i\|^2\|v_i\|^2-\langle u_i,v_i\rangle^2\bigr)
    }{
      \|U\wedge V\|_F^2
    }.
  \]
  Each summand is nonnegative, and it vanishes exactly when \(u_i\)
  and \(v_i\) are linearly dependent. Hence conditions \((2)\) and
  \((3)\) are equivalent. It remains to show that, under condition
  \((3)\), the O'Neill term vanishes automatically. If \(u_i\) and
  \(v_i\) are linearly dependent, then $u_i^\top v_i=v_i^\top u_i$.
  Summing over \(i\) gives
  \[
    U^\top V=V^\top U.
  \]
  Therefore
  \[
    \Theta_Z(U,V)
    =
    2\mathcal L_{E_Z}^{-1}(U^\top V-V^\top U)
    =0.
  \]
  By \cref{prop:curvature-oneill-tensor}, this is equivalent to
  \(A_UV=0\).  Hence
  \((2)\) implies \((1)\), and the theorem follows.
\end{proof}

\subsection{Ricci and scalar curvature}
\label{subsec:ricci-scalar-curvature}

The Ricci and scalar curvatures are obtained by contracting the
sectional curvature formula. Let
\[
  m_s\coloneq \dim \Qss=\dim \Qks.
\]

\begin{corollary}[Ricci curvature]
  \label{cor:ricci-curvature-fixed-rank-stratum}
  Let \(x\in \Qks\), and choose \(Z\in \Mss\) such that
  \(\Xks([Z])=x\). Let \(\bar u\in\Tz{x}\Qks\) be a unit vector, and
  let \(U\in\Hz{Z}\) be the unique horizontal lift of \(\dd
  \Xks^{-1}(\bar u)\). Choose horizontal vectors
  \[
    V_1,\ldots,V_{m_s-1}\in\Hz{Z}\cap U^\perp
  \]
  such that their projections form an orthonormal basis of \(\bar
  u^\perp\). Write
  \[
    U=(u_1,\ldots,u_n),
    \qquad
    V_\alpha=(v_{\alpha,1},\ldots,v_{\alpha,n}),
    \qquad \alpha=1, \ldots, m_s-1.
  \]
  Then the Ricci curvature of $\gQks$ is
  \[
    \begin{aligned}
      \Ric_{\Qks}(\bar u,\bar u)
      =
      \sum_{\alpha=1}^{m_s-1}
      \Bigg[
        \sum_{i=1}^n
        \bigl(
          \|u_i\|^2\|v_{\alpha,i}\|^2
          -\langle u_i,v_{\alpha,i}\rangle^2
        \bigr)
        +\frac34
        \tr\bigl(
          \Theta_Z(U,V_\alpha)^\top E_Z\Theta_Z(U,V_\alpha)
        \bigr)
      \Bigg].
    \end{aligned}
  \]
  In particular,
  \[
    \Ric_{\Qks}(\bar u,\bar u)\ge0.
  \]
\end{corollary}

\begin{proof}
  By definition,
  \[
    \Ric(\bar u,\bar u)
    =
    \sum_{\alpha=1}^{m_s-1}\sec(\bar u,\bar v_\alpha),
  \]
  where \((\bar v_\alpha)\) is any orthonormal basis of \(\bar
  u^\perp\).  Applying
  \cref{thm:sectional-curvature-fixed-rank-stratum} to each plane
  \(\spanv\{\bar u,\bar v_\alpha\}\) gives the formula.
  Nonnegativity follows from \cref{cor:nonnegative-sectional-curvature}.
\end{proof}

\begin{corollary}[Scalar curvature]
  \label{cor:scalar-curvature-fixed-rank-stratum}
  Let \(x\in \Qks\), and choose \(Z\in \Mss\) such that
  \(\Xks([Z])=x\). Let $U_1,\ldots,U_{m_s}\in\Hz{Z}$ be a horizontal
  orthonormal basis projecting, through \(\Xks\), to an orthonormal
  basis of \(\Tz{x}\Qks\). Write $U_a=(u_{a,1},\ldots,u_{a,n})$. Then
  the scalar curvature of $\gQks$ is
  \[
    \begin{aligned}
      \Scal_{\Qks}(x)
      =
      2\sum_{1\le a<b\le m_s}
      \Bigg[
        \sum_{i=1}^n
        \bigl(
          \|u_{a,i}\|^2\|u_{b,i}\|^2
          -\langle u_{a,i},u_{b,i}\rangle^2
        \bigr)\\
        +\frac34
        \tr\bigl(
          \Theta_Z(U_a,U_b)^\top E_Z\Theta_Z(U_a,U_b)
        \bigr)
      \Bigg].
    \end{aligned}
  \]
  In particular,
  \[
    \Scal_{\Qks}(x)\ge0.
  \]
\end{corollary}

\begin{proof}
  The scalar curvature is
  \[
    \Scal(x)
    =
    2\sum_{1\le a<b\le m_s}\sec(\bar e_a,\bar e_b)
  \]
  for any orthonormal basis \((\bar e_a)\) of \(\Tz{x}\Qks\).  Applying
  \cref{thm:sectional-curvature-fixed-rank-stratum} to each pair
  gives the formula.
\end{proof}

\subsection{Spectral and thick-part estimates}
\label{subsec:curvature-thick-part-estimates}

The thick parts are compact subsets of the full-rank stratum \(\Qss\)
on which the smallest eigenvalue is uniformly bounded below; they
exhaust \(\Qss\) as \(\varepsilon\downarrow0\).

\begin{definition}[Thick parts]
  For \(\varepsilon>0\), the \(\varepsilon\)-thick part of
  \(\Qss\) is
  \[
    \Qrk{s,\varepsilon}{s}
    \coloneq 
    \bigl\{[Z]\in\Qss:
    \lambda_{\min}(Z^\top Z)\ge\varepsilon\bigr\}.
  \]
  This definition is independent of the representative, since
  changing \(Z\) within its orbit conjugates \(Z^\top Z\) by an
  orthogonal matrix.
\end{definition}

The Sylvester formula gives quantitative curvature estimates
on these thick parts. The exact spectral expression below shows that
the relevant
denominators are the pairwise sums of eigenvalues of
\(E_Z=Z^\top Z\). In particular, if
\[
  0<\lambda_1\le\lambda_2\le\cdots\le\lambda_s,
\]
then the norm of the inverse Sylvester operator on
\(\so(s)\) is governed by
\((\lambda_1+\lambda_2)^{-1}\). Thus it can blow up only
when at least two eigenvalues of \(E_Z\) become small. The following
coarser estimates in terms of \(\lambda_{\min}(E_Z)\) nevertheless
provide convenient uniform curvature bounds on thick parts.

\begin{lemma}[Spectral form of the Sylvester curvature term]
  \label{lem:spectral-sylvester-curvature-term}
  Let \(E\in\SPD(s)\), let \(Y\in\so(s)\), and set \(\Theta\coloneq \mathcal
  L_E^{-1}(Y)\). Let
  \[
    E=P\diag(\lambda_1,\ldots,\lambda_s)P^\top,
    \qquad
    \lambda_i>0,
  \]
  and write \(\widetilde Y\coloneq P^\top YP\),
  \(\widetilde\Theta\coloneq P^\top\Theta P\). Then, for all \(1\le i,j\le
  s\), \(\widetilde\Theta_{ij}=\frac{\widetilde
  Y_{ij}}{\lambda_i+\lambda_j}\). In particular,
  \begin{equation}
    \label{eq:spectral-sylvester-curvature-term}
    \tr(\Theta^\top E\Theta)
    =
    \sum_{1\le i<j\le s}
    \frac{\widetilde Y_{ij}^{\,2}}{\lambda_i+\lambda_j}.
  \end{equation}
  Consequently,
  \begin{equation}
    \label{eq:spectral-sylvester-curvature-bounds}
    \frac{1}{4\lambda_{\max}(E)}\|Y\|_F^2
    \le
    \tr(\Theta^\top E\Theta)
    \le
    \frac{1}{4\lambda_{\min}(E)}\|Y\|_F^2.
  \end{equation}
\end{lemma}

\begin{proof}
  After conjugating by \(P\), the Sylvester equation $E\Theta+\Theta
  E=Y$ becomes
  \[
    \Lambda\widetilde\Theta+\widetilde\Theta\Lambda=\widetilde Y,
    \qquad
    \Lambda=\diag(\lambda_1,\ldots,\lambda_s),
  \]
  where \(\widetilde\Theta=P^\top\Theta P\). Thus
  $(\lambda_i+\lambda_j)\widetilde\Theta_{ij}=\widetilde Y_{ij}$. Furthermore,
  \[
    \tr(\Theta^\top E\Theta)
    =
    \tr(\widetilde\Theta^\top\Lambda\widetilde\Theta)
    =
    \sum_{i,j=1}^s \lambda_i\widetilde\Theta_{ij}^{\,2}.
  \]
  Pairing the terms \((i,j)\) and \((j,i)\), and using skew-symmetry
  of $\widetilde \Theta$, gives
  \[
    \sum_{i,j=1}^s \lambda_i\widetilde\Theta_{ij}^{\,2}
    =
    \sum_{1\le i<j\le s}
    (\lambda_i+\lambda_j)\widetilde\Theta_{ij}^{\,2}
    =
    \sum_{1\le i<j\le s}
    \frac{\widetilde Y_{ij}^{\,2}}{\lambda_i+\lambda_j}.
  \]
  Since \(2\lambda_{\min}(E)
    \le
    \lambda_i+\lambda_j
    \le
  2\lambda_{\max}(E)\) and \(\|Y\|_F^2
    =
    \|\widetilde Y\|_F^2
    =
  2\sum_{1\le i<j\le s}\widetilde Y_{ij}^{\,2}\), the bounds follow.
\end{proof}

We now apply the spectral estimates to the thick parts defined above.

\begin{corollary}[Curvature bounds on thick parts]
  \label{cor:curvature-bounds-thick-parts}
  Let \(\varepsilon>0\), let \([Z]\in \Qrk{s,\varepsilon}{s}\),
  and let \(U,V\in\Hz{Z}\) be linearly independent horizontal
  vectors. Set as before $S_Z(U,V)\coloneq U^\top V-V^\top U$. Then
  \begin{equation}
    \label{eq:curvature-bounds-thick-parts}
    \begin{aligned}
      \sec_{M_s}(U,V)
      +
      \frac{3}{2\bigl(n-(s-2)\varepsilon\bigr)}
      \frac{\|S_Z(U,V)\|_F^2}{\|U\wedge V\|_F^2}
      &\le
      \sec_{\Qss}\bigl(\dd\pi_s(Z)[U],\dd\pi_s(Z)[V]\bigr) \\
      &\le
      \sec_{M_s}(U,V)
      +
      \frac{3}{4\varepsilon}
      \frac{\|S_Z(U,V)\|_F^2}{\|U\wedge V\|_F^2}.
    \end{aligned}
  \end{equation}
  In particular, every sectional curvature on
  \(\Qrk{s,\varepsilon}{s}\) satisfies
  \[
    0\le \sec_{\Qss}\le 1+\frac{3}{2\varepsilon}.
  \]
  The same estimates hold on the corresponding thick part of \(\Qks\)
  through the isometry \(\Xks\).
\end{corollary}

The proof makes use of the following lemma.

\begin{lemma}[Skew-symmetric Frobenius bound]
  \label{lem:skew-frobenius-estimate}
  Let \(U,V\in\R^{n\times s}\) be linearly independent, then
  \[
    \|U^\top V-V^\top U\|_F^2
    \le
    2\|U\wedge V\|_F^2.
  \]
\end{lemma}

\begin{proof}
  We first recall that, if \(\Theta\in\so(s)\), then $
  \|\Theta\|_{\mathrm{op}} \le \frac{1}{\sqrt2}\|\Theta\|_F$.
  Indeed, if \(\sigma_1\ge\cdots\ge\sigma_s\ge0\) are the singular values of
  \(\Theta\), then
  \[
    \|\Theta\|_{\mathrm{op}}=\sigma_1,
    \qquad
    \|\Theta\|_F^2=\sum_{a=1}^s\sigma_a^2.
  \]
  The nonzero singular values of a real skew-symmetric matrix occur
  with even multiplicity, hence \(\sigma_1\) appears at least twice
  when \(\Theta\neq0\), and therefore \(\|\Theta\|_F^2\ge
  2\|\Theta\|_{\mathrm{op}}^2\). Now set
  \[
    V_\perp
    \coloneq 
    V-\frac{\Frob{U}{V}}{\|U\|_F^2}U.
  \]
  Then \(\Frob{U}{V_\perp}=0\), \(\|U\wedge
  V\|_F^2=\|U\|_F^2\|V_\perp\|_F^2\) and, since \(U^\top U\) is symmetric,
  \[
    U^\top V_\perp-V_\perp^\top U
    =
    U^\top V-V^\top U.
  \]
  Let \(S\coloneq U^\top V_\perp-V_\perp^\top U\in\so(s)\). For every
  \(\Theta\in\so(s)\), from skew-symmetry and cyclicity of the trace,
  \(\Frob{S}{\Theta} = 2\Frob{V_\perp}{U\Theta}\). Thus, by
  Cauchy--Schwarz and the estimate above,
  \[
    |\Frob{S}{\Theta}|
    \le
    2\|V_\perp\|_F\|U\Theta\|_F
    \le
    2\|V_\perp\|_F\|U\|_F\|\Theta\|_{\mathrm{op}}
    \le
    \sqrt2\,\|U\|_F\|V_\perp\|_F\|\Theta\|_F.
  \]
  Recall that $\Vert S\|_F = \sup_{\|\Theta\|_F=1} \vert \Frob{S}{
  \Theta}\vert $. Indeed, Cauchy--Schwarz gives the
  inequality \(\le\), and equality is attained by
  \(\Theta=S/\|S\|_F\) if \(S\neq0\), which still belongs to
  \(\so(s)\). Thus, taking the supremum over
  \(\Theta\in\so(s)\) with \(\|\Theta\|_F=1\) gives \(\|S\|_F \le
  \sqrt2\,\|U\|_F\|V_\perp\|_F\). Therefore
  \[
    \|U^\top V-V^\top U\|_F^2
    =
    \|S\|_F^2
    \le
    2\|U\|_F^2\|V_\perp\|_F^2
    =
    2\|U\wedge V\|_F^2.
  \]
\end{proof}

We now prove \cref{cor:curvature-bounds-thick-parts}.

\begin{proof}
  Apply \cref{lem:spectral-sylvester-curvature-term} with
  \(Y=2S_Z(U,V)\). Since each row of \(Z\) has unit norm,
  \[
    \tr(E_Z)=\tr(Z^\top Z)=\|Z\|_F^2=n.
  \]
  On the \(\varepsilon\)-thick part,
  \(\lambda_{\min}(E_Z)\ge\varepsilon\). Hence,
  if \(\lambda_1,\ldots,\lambda_s\) are the eigenvalues of \(E_Z\), then
  \[
    2\varepsilon
    \le
    \lambda_i+\lambda_j
    \le
    n-(s-2)\varepsilon
    \qquad
    (i\ne j).
  \]
  Using the spectral formula in
  \cref{eq:spectral-sylvester-curvature-term} with \(Y=2S_Z(U,V)\), we obtain
  \[
    \frac{2}{n-(s-2)\varepsilon}\|S_Z(U,V)\|_F^2
    \le
    \tr\bigl(\Theta_Z(U,V)^\top E_Z\Theta_Z(U,V)\bigr)
    \le
    \frac{1}{\varepsilon}\|S_Z(U,V)\|_F^2.
  \]
  Substituting these inequalities into the sectional-curvature
  formula in \cref{eq:sectional-curvature-fixed-rank-general} gives
  \cref{eq:curvature-bounds-thick-parts}. It remains to justify the
  coarse bound. The product-of-spheres sectional curvature satisfies
  \(0\le \sec_{M_s}\le 1\). Moreover, by
  \cref{lem:skew-frobenius-estimate}, for any two linearly
  independent matrices \(U,V\),
  \[
    \|U^\top V-V^\top U\|_F^2
    \le
    2\|U\wedge V\|_F^2.
  \]
  Hence,
  \[
    \frac{3}{4\varepsilon}
    \frac{\|S_Z(U,V)\|_F^2}{\|U\wedge V\|_F^2}
    \le
    \frac{3}{2\varepsilon}.
  \]
  Together with \(\sec_{M_s}\le1\), this gives the coarse estimate.
\end{proof}

\begin{remark}[Degeneration near lower rank]
  \label{rem:curvature-degeneration-lower-rank}
  The thick-part estimates give a convenient sufficient control of
  the quotient curvature, but the factor
  \(\lambda_{\min}(Z^\top Z)^{-1}\) is not the exact degeneration
  scale. If
  \[
    0<\lambda_1\le\lambda_2\le\cdots\le\lambda_s
  \]
  are the eigenvalues of \(Z^\top Z\), then the exact spectral
  formula shows that the relevant scale is
  \[
    (\lambda_1+\lambda_2)^{-1}.
  \]
  Consequently, approaching a rank-\((s-1)\) configuration need not
  cause the quotient curvature to become unbounded. Possible
  curvature blow-up can occur only when at least two eigenvalues tend
  to zero, that is, when the factor approaches the locus of rank at
  most \(s-2\).
\end{remark}

\begin{proposition}[Asymptotic sharpness of the thick-part bound]
  \label{prop:asymptotic-sharpness-thick-part}
  Let \(s\ge 3\) and \(\varepsilon>0\). As the number \(n\) of sphere
  factors tends to infinity, there exist points
  \([Z_n]\in \Qrk{s,\varepsilon}{s}\) and tangent \(2\)-planes
  \(\bar\sigma_n\subset \Tz{[Z_n]}\Qss\) such that
  \[
    \sec_{\Qss}(\bar\sigma_n)
    \longrightarrow
    1+\frac{3}{2\varepsilon}.
  \]
  Hence the curvature bound on the \(\varepsilon\)-thick part is
  asymptotically sharp uniformly in \(n\).
\end{proposition}

\begin{proof}
  Choose integers \(m>\max\{2\varepsilon,2\}\),
  \(\ell\ge\varepsilon\) and let \(\theta\in(0,\pi/2)\) be determined
  by \(\sin^2\theta=\frac{2\varepsilon}{m}\). Set
  \[
    \phi_j\coloneq \frac{2\pi j}{m},
    \qquad
    c_m\coloneq
    \begin{pmatrix}
      \cos\phi_1\\
      \vdots\\
      \cos\phi_m
    \end{pmatrix},
    \qquad
    s_m\coloneq
    \begin{pmatrix}
      \sin\phi_1\\
      \vdots\\
      \sin\phi_m
    \end{pmatrix}.
  \]
  Denote by \(\mathbf 1_r\in\R^{r\times 1}\) the column vector
  of \(r\) ones and by
  \(\mathbf 0_{p\times q}\in\R^{p\times q}\) the zero matrix.
  For \(N\ge 2\), set
  \[
    n_N\coloneq m+N+\ell(s-3).
  \]
  Define
  \[
    B_{\ell,s-3}
    \coloneq 
    \begin{pmatrix}
      \mathbf 1_\ell
      & \mathbf 0_{\ell\times1}
      & \cdots
      & \mathbf 0_{\ell\times1}
      \\
      \mathbf 0_{\ell\times1}
      & \mathbf 1_\ell
      & \ddots
      & \vdots
      \\
      \vdots
      & \ddots
      & \ddots
      & \mathbf 0_{\ell\times1}
      \\
      \mathbf 0_{\ell\times1}
      & \cdots
      & \mathbf 0_{\ell\times1}
      & \mathbf 1_\ell
    \end{pmatrix}
    \in\R^{\ell(s-3)\times(s-3)}.
  \]
  Thus \(B_{\ell,s-3}\) is obtained by repeating each row of
  \(\I_{s-3}\) exactly \(\ell\) times. We then define
  \[
    Z_N
    \coloneq 
    \begin{pmatrix}
      \sin\theta\,c_m
      & \sin\theta\,s_m
      & \cos\theta\,\mathbf 1_m
      & \mathbf 0_{m\times(s-3)}
      \\
      \mathbf 0_{N\times1}
      & \mathbf 0_{N\times1}
      & \mathbf 1_N
      & \mathbf 0_{N\times(s-3)}
      \\
      \mathbf 0_{\ell(s-3)\times1}
      & \mathbf 0_{\ell(s-3)\times1}
      & \mathbf 0_{\ell(s-3)\times1}
      & B_{\ell,s-3}
    \end{pmatrix}
    \in\R^{n_N\times s}.
  \]
  When \(s=3\), the last block row and the last block column are
  omitted.
  Every row of \(Z_N\) has unit norm.

  Since \(m\ge3\), the root-of-unity identities
  \[
    \sum_{j=1}^m e^{\mathrm i\phi_j}=0,
    \qquad
    \sum_{j=1}^m e^{2\mathrm i\phi_j}=0
  \]
  imply
  \[
    c_m^\top\mathbf 1_m
    =
    s_m^\top\mathbf 1_m
    =
    c_m^\top s_m
    =0,
    \qquad
    \|c_m\|^2=\|s_m\|^2=\frac m2.
  \]
  Direct block multiplication therefore gives
  \[
    E_{Z_N}=Z_N^\top Z_N
    =
    \begin{pmatrix} \varepsilon & 0 & 0 & \mathbf 0_{1\times(s-3)}
      \\ 0 & \varepsilon & 0 & \mathbf 0_{1\times(s-3)} \\ 0 & 0 &
      N+m-2\varepsilon & \mathbf 0_{1\times(s-3)} \\ \mathbf
      0_{(s-3)\times1} & \mathbf 0_{(s-3)\times1} & \mathbf
      0_{(s-3)\times1} & \ell \I_{s-3}
    \end{pmatrix}
  \]
  Since all diagonal entries of \(E_{Z_N}\) are positive, \(Z_N\) has
  full column rank \(s\). Moreover, because \(\ell\ge\varepsilon\),
  \(\lambda_{\min}(E_{Z_N})=\varepsilon\) for every integer \(N\ge2\)
  satisfying \(N+m\ge3\varepsilon\). Thus \([Z_N]\in
  \Qrk{s,\varepsilon}{s}\) for every such \(N\). Define
  \(a_N=(a_1,\ldots,a_N)^\top\in\R^N\) by
  \[
    a_1\coloneq \sqrt{\frac{N-1}{N}},
    \qquad
    a_r\coloneq -\frac{1}{\sqrt{N(N-1)}},
    \quad 2\le r\le N.
  \]
  Then
  \[
    \mathbf 1_N^\top a_N=0,
    \qquad
    \|a_N\|^2=1.
  \]
  Define \(U_N,V_N\in \Tz{Z_N}\Mss\) by
  \[
    U_N\coloneq
    \begin{pmatrix}
      \mathbf 0_{m\times s}\\
      a_N
      & \mathbf 0_{N\times1}
      & \mathbf 0_{N\times1}
      & \mathbf 0_{N\times(s-3)}\\
      \mathbf 0_{\ell(s-3)\times s}
    \end{pmatrix},
    \quad
    V_N\coloneq
    \begin{pmatrix}
      \mathbf 0_{m\times s}\\
      \mathbf 0_{N\times1}
      & a_N
      & \mathbf 0_{N\times1}
      & \mathbf 0_{N\times(s-3)}\\
      \mathbf 0_{\ell(s-3)\times s}
    \end{pmatrix},
  \]
  We have that \(U_N\) and \(V_N\) are orthonormal
  \[
    \|U_N\|_F^2=\|V_N\|_F^2=\|a_N\|^2=1,
    \qquad
    \Frob{U_N}{V_N}=0.
  \]
  For a tangent vector \(W\in \Tz{Z}\Mss\), the horizontality
  criterion is
  \[
    W\in\Hz{Z}
    \quad\Longleftrightarrow\quad
    J_Z(W)\coloneq Z^\top W-W^\top Z=0.
  \]
  Applying this criterion to \(U_N\) gives
  \[
    \begin{aligned}
      J_{Z_N}(U_N)
      =
      \sum_{r=1}^N
      a_r\bigl(e_3e_1^\top-e_1e_3^\top\bigr)
      =
      \bigl(\mathbf 1_N^\top a_N\bigr)
      \bigl(e_3e_1^\top-e_1e_3^\top\bigr)
      =0.
    \end{aligned}
  \]
  Likewise,
  \[
    J_{Z_N}(V_N)
    =
    \bigl(\mathbf 1_N^\top a_N\bigr)
    \bigl(e_3e_2^\top-e_2e_3^\top\bigr)
    =0.
  \]
  Thus \(U_N\) and \(V_N\) are horizontal at \(Z_N\). Furthermore,
  \[
    S_{Z_N}(U_N,V_N)
    \coloneq 
    U_N^\top V_N-V_N^\top U_N
    =
    e_1e_2^\top-e_2e_1^\top.
  \]
  Let
  \[
    \mathcal L_{E_{Z_N}}(\Omega)
    =
    E_{Z_N}\Omega+\Omega E_{Z_N},
    \qquad
    \Theta_N
    \coloneq
    2\mathcal L_{E_{Z_N}}^{-1}
    \bigl(S_{Z_N}(U_N,V_N)\bigr).
  \]
  Since the first two eigenvalues of \(E_{Z_N}\) are both equal to
  \(\varepsilon\), and since the Sylvester operator acts entrywise:
  \(\bigl(\mathcal
  L_{E_{Z_N}}(\Omega)\bigr)_{ij}=(\lambda_i+\lambda_j)\Omega_{ij}\)
  for \(i,j = 1,\ldots, s\), we have
  \[
    \Theta_N
    =
    \frac{1}{\varepsilon}
    \bigl(e_1e_2^\top-e_2e_1^\top\bigr),
  \]
  and therefore \(\tr
    \bigl(
      \Theta_N^\top E_{Z_N}\Theta_N
    \bigr)
    =
  \frac{2}{\varepsilon}\). The product-of-spheres sectional curvature
  of the plane
  \(\sigma_N=\spanv\{U_N,V_N\}\) is
  \[
    \sec_{M_s}(\sigma_N)
    =
    \sum_{r=1}^N a_r^4
    =
    \frac{(N-1)^3+1}{N^2(N-1)}
    \longrightarrow 1.
  \]
  Consequently, for the quotient plane \(\bar\sigma_N =
  \dd\pi_s(Z_N)\sigma_N\), the sectional-curvature formula
  (\cref{eq:sectional-curvature-fixed-rank-orthonormal}) gives
  \[
    \sec_{\Qss}(\bar\sigma_N)
    =
    \sec_{M_s}(\sigma_N)
    +
    \frac34
    \tr
    \bigl(
      \Theta_N^\top E_{Z_N}\Theta_N
    \bigr)
    =
    \frac{(N-1)^3+1}{N^2(N-1)}
    +
    \frac{3}{2\varepsilon}.
  \]
  Letting \(N\to\infty\) proves that \(\sec_{\Qss}(\bar\sigma_N)
    \longrightarrow
  1+\frac{3}{2\varepsilon}\).
\end{proof}

\subsection{Non-polarity of the quotient action}
\label{subsec:non-polarity}

We now use the curvature $2$-form of the mechanical connection to
show that the horizontal distribution is not integrable in ranks
\(s\ge3\).  This is the differential-geometric obstruction behind the
non-polarity of the quotient action.

Recall that an isometric action of a compact Lie group \(G\) on a
Riemannian manifold \(M\) is called \emph{polar} if there exists an
immersed submanifold \(\Sigma\subset M\), called a section, which
meets every \(G\)-orbit and such that
\[
  \Tz{x}\Sigma \perp \Tz{x}(G\cdot x)
\]
at every point of intersection. On the principal part of the action,
this means that the horizontal distribution is tangent to local submanifolds.

If the \(O(s)\)-action on \(\Mss\) were polar, quotient computations would
simplify drastically.  One could work on a section \(\Sigma\), modulo
a finite Weyl group, instead of working with horizontal lifts in the
whole factor space (see for instance
  \cite{PalaisTerng1988,GroveZiller2012} or
\cite{AlexandrinoBettiol2015}). On the principal part, the O'Neill
tensor $A$ associated with horizontal fields tangent to the local
sections would vanish, horizontal lifts would stay tangent to a fixed
submanifold, and quotient geodesic computations would reduce to
geodesic computations along \(\Sigma\). The result below says that
this simplification is not available in ranks \(s\ge3\).

\begin{proposition}[Curvature obstruction to polarity]
  \label{prop:curvature-obstruction-polarity}
  On the principal bundle
  \[
    \pi_s:\Mss\longrightarrow \Qss,
  \]
   the horizontal distribution $\Hbundle=\ker\omega$ is integrable
  if and only if the curvature $2$-form $\Omega^\omega=0$. If the
  action is polar on the principal stratum, then $\Omega^\omega=0$.
  Consequently, if $\Omega^\omega$ is not identically zero, the
  action is not polar.
\end{proposition}

\begin{proof}
  For horizontal vector fields \(U,V\),
  \cref{prop:curvature-oneill-tensor} ensures
  \(
    \Omega^\omega(U,V)=-\omega([U,V]).
  \)
  Therefore $\Omega^\omega(U,V)=0$ for all horizontal \(U,V\) if and
  only if $\omega([U,V])=0$ for all horizontal \(U,V\), that is, if
  and only if \([U,V]\in\Hbundle\) for all horizontal vector fields
  \(U,V\). By Frobenius' theorem, this is precisely the integrability
  of the horizontal distribution. If the action were polar on the
  principal part, the local sections would be orthogonal to the
  orbits. Hence their tangent spaces would coincide with \(\Hbundle\), because both have dimension equal to the quotient dimension.
  Thus \(\Hbundle\) would be integrable. Therefore non-vanishing of
  \(\Omega^\omega\) rules out polarity.
\end{proof}

\begin{theorem}[Non-polarity in rank at least three]
  \label{thm:non-polarity-rank-at-least-three}
  Assume \(s\ge3\). Then the mechanical horizontal distribution on
  \[
    \pi_s:\Mss\to \Qss
  \]
  is not integrable. Consequently, the \(O(s)\)-action on \(\Mss\) is
  not polar on its principal full-rank stratum. By fixed-rank reduction
  {(\cref{thm:fixed-rank-riemannian-reduction})}, the same conclusion
  holds for the \(O(k)\)-action on \(\Mks\), for every \(k\ge s\).
\end{theorem}

\begin{proof}
  It is enough to find one point \(Z_0\in \Mss\) and two horizontal tangent
  vectors \(U,V\in\Hz{Z_0}\) such that \((\Omega^\omega)_{Z_0}(U,V)\ne0\). Let
  \[
    Z_0\coloneq 
    \begin{pmatrix}
      e_1\\
      e_2\\
      \vdots\\
      e_s\\
      e_1\\
      \vdots\\
      e_1
    \end{pmatrix}
    \in \Mss,
  \]
  where the first \(s\) rows are the canonical basis vectors of
  \(\R^s\), and the remaining \(n-s\) rows, if any, are equal
  to \(e_1\).  Then \(Z_0\) has rank \(s\). Choose three distinct
  indices \(a,b,c\in\{1,\ldots,s\}\). Let
  \[
    A\coloneq E_{ac}+E_{ca},
    \qquad
    B\coloneq E_{cb}+E_{bc},
  \]
  which are symmetric zero-diagonal matrices. Define tangent vectors
  \(U,V\in\Tz{Z_0}\Mss\) by taking their first \(s\) rows to be the rows
  of \(A\) and \(B\), respectively, and their remaining rows to be zero:
  \[
    U\coloneq
    \begin{pmatrix}
      A\\
      0
    \end{pmatrix},
    \qquad
    V\coloneq
    \begin{pmatrix}
      B\\
      0
    \end{pmatrix}.
  \]
  The zero-diagonal condition ensures tangency to the product of
  spheres.  Moreover,
  \[
    Z_0^\top U=A,
    \qquad
    Z_0^\top V=B.
  \]
  Since \(A\) and \(B\) are symmetric, the horizontality criterion gives
  $U,V\in\Hz{Z_0}$. Now
  \[
    S_{Z_0}(U,V)=U^\top V-V^\top U=AB-BA.
  \]
  Using \(E_{ij}E_{k\ell}=\delta_{jk}E_{i\ell}\), one obtains
  \[
    AB=E_{ab},
    \qquad
    BA=E_{ba}.
  \]
  Thus
  \[
    S_{Z_0}(U,V)=E_{ab}-E_{ba}\ne0.
  \]
  The curvature formula for horizontal vectors gives
  \[
    (\Omega^\omega)_{Z_0}(U,V)
    =
    2\mathcal L_{E_{Z_0}}^{-1}\bigl(S_{Z_0}(U,V)\bigr).
  \]
  Since \(\mathcal L_{E_{Z_0}}\) is an isomorphism on
  \(\so(s)\) (\cref{lem:sylvester-invertible}), this
  curvature value is nonzero. Hence the mechanical horizontal
  distribution is not integrable. By
  \cref{prop:curvature-obstruction-polarity}, the action is not polar
  on the principal stratum.

  Finally, let
  \[
    \widehat Z_0\coloneq \iks(Z_0),
    \qquad
    \widehat U\coloneq \dd\iks(Z_0)[U]=(U,0),
    \qquad
    \widehat V\coloneq \dd\iks(Z_0)[V]=(V,0).
  \]
  By \cref{lem:padding-preserves-horizontal-lifts},
  \(\widehat U\) and \(\widehat V\) are horizontal for the
  \(O(k)\)-action on \(\Mks\).

  By \cref{lem:padding-isometric-embedding}, the padding map
  \(\iks:M_s\to M_k\) is a totally geodesic Riemannian
  isometric embedding. Since \(\Mss\) is open in \(M_s\), any
  geodesic \(\gamma\) of \(\iks(\Mss)\) is an ambient
  geodesic in \(M_k\) whose image lies in \(\Mks\).
  Locally, \(\gamma\) minimizes length among all curves in \(M_k\),
  and hence also among curves constrained to \(\Mks\).
  Thus \(\gamma\) is a geodesic of \(\Mks\), so the restricted
  padding map to $\Mss$ is totally geodesic in $\Mks$ and thus
  preserves sectional curvature:
  \[
    \sec_{\Mks}(\widehat U,\widehat V)
    =
    \sec_{\Mss}(U,V).
  \]
  Moreover, by \cref{thm:fixed-rank-riemannian-reduction}, the induced
  map on the quotient strata is a Riemannian isometry, and hence
  \[
    \sec_{\Qks}
    \bigl(
      \dd\pi_k(\widehat Z_0)[\widehat U],
      \dd\pi_k(\widehat Z_0)[\widehat V]
    \bigr)
    =
    \sec_{\Qss}
    \bigl(
      \dd\pi_s(Z_0)[U],
      \dd\pi_s(Z_0)[V]
    \bigr).
  \]

  Let \(A^{(s)}\) and \(A^{(k)}\) denote the corresponding O'Neill
  tensors. Applying O'Neill's formula in the reduced and padded models
  and using \(\|\widehat U\wedge\widehat V\|_F = \|U\wedge V\|_F\), we obtain
  \[
    \|A^{(k)}_{\widehat U}\widehat V\|_F
    =
    \|A^{(s)}_UV\|_F.
  \]
  Since \(A^{(s)}_UV\neq0\), it follows that \(A^{(k)}_{\widehat
  U}\widehat V\neq0\). Thus the horizontal distribution on \(\Mks\)
  is not integrable.
  By \cref{prop:curvature-obstruction-polarity}, the \(O(k)\)-action on
  the rank-\(s\) stratum is not polar.

\end{proof}

\begin{remark}[Low-rank behaviour]
  \label{rem:low-rank-polarity}
  For \(s=1\), the structure group \(O(1)\) is discrete and the Lie
  algebra is zero. For \(s=2\), every row tangent space
  \(\Tz{z_i}\mathbb S^1\) is
  one-dimensional. Hence, for any tangent vectors \(U,V\in\Tz{Z}\Mrk{2}{2}\),
  the row vectors \(u_i\) and \(v_i\) are linearly dependent for
  every \(i\), and therefore
  \[
    u_i^\top v_i-v_i^\top u_i=0.
  \]
  Summing over \(i\), one obtains
  \[
    S_Z(U,V)=U^\top V-V^\top U=0.
  \]
  Thus the curvature formula gives $\Omega^\omega=0$ in rank \(2\).
  The curvature obstruction to polarity starts precisely in rank \(s\ge3\).
\end{remark}

\section{Holonomy of the mechanical connection}
\label{sec:holonomy-mechanical-connection}

\subsection{Classification of the holonomy groups}
\label{subsec:holonomy-mechanical-connection}

We now compute the holonomy group of the reduced mechanical
connection \(\omega\) of \cref{def:reduced-mechanical-connection}.
By fixed-rank reduction, the
answer depends only on the pair \((n,s)\), and not on the ambient
factor dimension \(k\). Throughout this section, paths and loops are
assumed to be continuous and piecewise \(C^1\).

\begin{definition}[Holonomy, restricted holonomy, and monodromy]
\label{def:holonomy-restricted-monodromy}
For \(Z\in\Mss\),
the \emph{holonomy group}
\[
  \Hol_Z(\omega)\subset O(s)
\]
is the subgroup consisting of all elements \(R\in O(s)\) such that
there exists a loop
\[
  \bar\gamma:[0,1]\to \Qss,
  \qquad
  \bar\gamma(0)=\bar\gamma(1)=\pi_s(Z),
\]
whose horizontal lift starting at \(Z\) ends at \(ZR\).
The \emph{restricted holonomy group}, denoted by
\(\Hol_Z(\omega)^0\), is obtained by restricting to
loops homotopic to the constant loop through loops with the same
base point. It is the identity component of the holonomy Lie group
\cite[Chap.~II, Theorem~4.2]{KobayashiNomizu1963}. We denote its Lie algebra by
\[
  \mathfrak{hol}_Z(\omega)\subset\so(s).
\]
The discrete quotient
\[
  \Hol_Z(\omega)/\Hol_Z(\omega)^0
\]
is called the \emph{monodromy group} of the connection. It records
holonomy modulo the contribution of contractible loops. For a flat
connection, restricted holonomy is trivial, so the monodromy group
is the full holonomy group. The holonomy element of a loop then
depends only on its based homotopy class.
\end{definition}

\begin{definition}[Oriented quotient and double cover]
\label{def:oriented-quotient-cover}
For \(1\le s\le n\), the \emph{oriented quotient} is
\[
  \widetilde Q_s\coloneq \Mss/SO(s),
  \qquad
  \widetilde\pi_s:\Mss\longrightarrow\widetilde Q_s,
  \quad Z\longmapsto[Z]_{SO(s)}.
\]
Its natural projection to the unoriented quotient is
\[
  \rho:\widetilde Q_s\longrightarrow\Qss,
  \qquad [Z]_{SO(s)}\longmapsto[Z]_{O(s)},
  \qquad \pi_s=\rho\circ\widetilde\pi_s.
\]
Since the \(O(s)\)-action is free and \(SO(s)\) is normal of
index two, \(\rho\) is a two-sheeted covering. Its fibres distinguish
the two classes in
\(O(s)/SO(s)\simeq\{\pm1\}\).
The \emph{monodromy of this covering} assigns to a based loop the
permutation of the two points of its fibre obtained by lifting the
loop from each of them.
\end{definition}

The main holonomy theorem is the following.

\begin{theorem}[Holonomy of the mechanical connection]
  \label{thm:holonomy-all-rank-strata}
  Let \(1\le s\le n\), and consider the reduced principal \(O(s)\)-bundle
  \[
    \pi_s : \Mss \coloneq  \{Z\in(\mathbb S^{s-1})^n:\rank
    Z=s\}\longrightarrow \Qss\coloneq \Mss/O(s)
  \]
  endowed with the mechanical connection \(\omega\). Then, for every
  \(Z\in \Mss\), one has
  \[
    \begin{array}{c|c|c|c}
      \text{Case} &
      \mathfrak{hol}_Z(\omega) &
      \Hol_Z(\omega)^0 &
      \Hol_Z(\omega)
      \\
      \hline
      s=1 & \{0\} & {\{\I_s\}} & {\{\I_s\}}
      \\[1mm]
      s=2,\ n=2 & \{0\} & \{\I_s\} & \{\I_s\}
      \\[1mm]
      s=2,\ n>2
        &
        \{0\}
        &
        \{\I_s\}
        &
        \simeq D_n
      \\[1mm]
      s\ge3,\ n=s & \so(s) & SO(s) & SO(s)
      \\[1mm]
      s\ge3,\ n>s & \so(s) & SO(s) & O(s)
    \end{array}
  \]
  For \(s=2\), the connection is flat. Here \(D_n\) denotes the
dihedral group of order \(2n\), isomorphic to the subgroup of
\(O(2)\) generated by
\[
  C_n\coloneq \{R_{2\pi j/n}:j=0,\ldots,n-1\}\subset SO(2)
\]
and any reflection, where
\[
  R_\alpha\coloneq 
  \begin{pmatrix}
    \cos\alpha&\sin\alpha\\
    -\sin\alpha&\cos\alpha
  \end{pmatrix}
\]
adds \(\alpha\) to the angle of a row vector under the right action.
\end{theorem}

The proof separates the connected holonomy from the possible
orientation-reversing elements. In ranks \(s\ge3\), curvature and
the Ambrose--Singer theorem determine the identity component; an
oriented double cover then determines the full group. In rank two,
the connection is flat, and relative-angle coordinates allow us to
compute its finite holonomy explicitly. The three propositions below
establish these claims.

\begin{proposition}[Restricted holonomy in rank at least three]
  \label{prop:restricted-holonomy-rank-at-least-three}
  Assume \(s\ge3\). Then, for every \(Z\in \Mss\),
  \[
    \mathfrak{hol}_Z(\omega)=\so(s), \quad \textnormal{and therefore} \quad \Hol_Z(\omega)^0=SO(s).
  \]
\end{proposition}

\begin{proof}
  We first prove the statement at the base point
  \[
    Z_0=
    \begin{pmatrix}
    1 & 0 & \cdots & 0\\
    0 & 1 & \ddots & \vdots\\
    \vdots & \ddots & \ddots & 0\\
    0 & \cdots & 0 & 1\\[2pt]
    1 & 0 & \cdots & 0\\
    \vdots & \vdots & & \vdots\\
    1 & 0 & \cdots & 0
  \end{pmatrix}=
    \begin{pmatrix}
      e_1\\
      e_2\\
      \vdots\\
      e_s\\
      e_1\\
      \vdots\\
      e_1
    \end{pmatrix}
    \in \Mss.
  \]
  The first \(s\) rows of \(Z_0\) form \(\I_s\), and its remaining
  \(n-s\) rows are equal to \(e_1\); these last rows are absent when
  \(n=s\). As above, write
  \[
    E_{Z_0}=Z_0^\top Z_0
    =
    \diag(n-s+1,1,\ldots,1),
  \]
  so \(\mathcal L_{E_{Z_0}}\) is an automorphism of
  \(\so(s)\) (see \cref{lem:sylvester-invertible}).

  Let \(A,B\) range over symmetric zero-diagonal matrices, and define horizontal
  vectors
  \[
    U_A\coloneq
    \begin{pmatrix}
      A\\
      \mathbf{0}_{(n-s)\times s}
    \end{pmatrix}
    \in \Hz{Z_0},
    \qquad
    U_B\coloneq
    \begin{pmatrix}
      B\\
      \mathbf{0}_{(n-s)\times s}
    \end{pmatrix}
    \in \Hz{Z_0}.
  \]
  The zero diagonal ensures tangency to the product of spheres.
  Horizontality follows from
  \[
    Z_0^\top U_A=A=A^\top,
    \qquad
    Z_0^\top U_B=B=B^\top.
  \]
  For such vectors,
  \[
    S_{Z_0}(U_A,U_B)=U_A^\top U_B-U_B^\top U_A=AB-BA=[A,B].
  \]
  The curvature formula in \cref{eq:curvature-form-sylvester-horizontal} gives
  \begin{equation}
  \label{eq:holonomy-curvature-basepoint}
    (\Omega^\omega)_{Z_0}(U_A,U_B)
    =
    2\mathcal L_{E_{Z_0}}^{-1}([A,B]).
\end{equation}

  We claim that the commutators \([A,B]\), with \(A,B\) symmetric and
  zero diagonal, span \(\so(s)\). Indeed, for distinct
  indices \(a,b,c\), set
  \[
    A\coloneq E_{ac}+E_{ca},
    \qquad
    B\coloneq E_{cb}+E_{bc}.
  \]
  Here \(E_{ij}\) denotes the elementary matrix with a single nonzero
  entry, equal to \(1\), in position \((i,j)\). Then
  \[
    [A,B]=E_{ab}-E_{ba}.
  \]
  Since, for distinct \(a, b\), the matrices \(E_{ab}-E_{ba}\) span
  \(\so(s)\), the claim follows. Because \(\mathcal
  L_{E_{Z_0}}^{-1}\) is an automorphism, the curvature values at
  \(Z_0\) span all of \(\so(s)\).

  By the Ambrose--Singer holonomy theorem
  (see for instance \cite[Chap.~II, Theorem~8.1]{KobayashiNomizu1963}),
  the holonomy Lie algebra is generated by curvature values
  transported along horizontal curves. In particular, it contains the
  span of the curvature values at \(Z_0\). Hence
  \[
    \so(s) \subseteq \mathfrak{hol}_{Z_0}(\omega) \implies \mathfrak{hol}_{Z_0}(\omega)=\so(s).
  \]

  Now let \(Z\in \Mss\) be arbitrary. Since \(\Qss\) is path-connected for
  \(s\ge2\) (\cref{thm:path-connectedness-fixed-rank-strata}), choose
  a path in the base joining \(\pi_s(Z_0)\) to \(\pi_s(Z)\), and lift
  it horizontally from \(Z_0\). Its endpoint lies in the fibre over
  \(\pi_s(Z)\), hence is of the form \(ZR\) for some \(R\in O(s)\).
  Horizontal transport identifies the holonomy groups at \(Z_0\) and
  \(ZR\), that is,
  \(\Hol_{Z_0}(\omega)=\Hol_{ZR}(\omega)\). On the other hand, points in the same fibre related by the
  right action have conjugate
  holonomy groups:
  \[
    \Hol_{ZR}(\omega)
    =
    R^{-1}\Hol_{Z}(\omega)R.
  \]
  Therefore
  \[
    \mathfrak{hol}_{Z}(\omega)
    =
    R\,\mathfrak{hol}_{Z_0}(\omega)\,R^{-1}
    =
    R\,\so(s)\,R^{-1}
    =
    \so(s).
  \]
  Therefore, the identity component of
  \(\Hol_Z(\omega)\) is the connected Lie subgroup of
  \(O(s)\) with Lie algebra \(\so(s)\), namely \(\Hol_{Z}(\omega)^0 = SO(s)\).
\end{proof}

The following path construction will be used to detect
orientation-reversing holonomy in every rank \(s\ge2\).

\begin{lemma}[A path to an orientation-reversed factor]
  \label{lem:orientation-reversing-factor-path}
  Let \(2\le s<n\), and let \(e_1,\ldots,e_s\) denote the
  standard row vectors of \(\R^s\). Set
  \[
    Z_0\coloneq
    \begin{pmatrix}
      e_1\\
      \vdots\\
      e_s\\
      e_1\\
      \vdots\\
      e_1
    \end{pmatrix}
    \in\Mss,
    \qquad
    R_-\coloneq \diag(-1,1,\ldots,1),
  \]
  where the first \(s\) rows of \(Z_0\) form \(\I_s\) and the
  remaining \(n-s\) rows equal \(e_1\). There exists a continuous
  piecewise smooth path \(\beta:[0,1]\to\Mss\) such that
  \[
    \beta(0)=Z_0,
    \qquad
    \beta(1)=Z_0R_-.
  \]
\end{lemma}

\begin{proof}
Write \(\mathbf 1_m\) for the column vector of \(m\) ones and
\(\mathbf{0}_{a\times b}\) for the \(a\times b\) zero matrix.
A block with zero rows contributes no actual row, and a block with
zero columns contributes no actual column. First rotate the first row
along \(u(t)\coloneq \cos(\pi t)e_1+\sin(\pi t)e_2\), keeping all other
rows fixed:
\begin{equation}
  \label{eq:orientation-reversing-first-row}
  \beta_0(t)\coloneq 
  \begin{pmatrix}
    \cos(\pi t) & \sin(\pi t) & \mathbf{0}_{1\times(s-2)}\\
    0 & 1 & \mathbf{0}_{1\times(s-2)}\\
    \mathbf{0}_{(s-2)\times1} & \mathbf{0}_{(s-2)\times1} & \I_{s-2}\\[2pt]
    \mathbf 1_{n-s} & \mathbf{0}_{(n-s)\times1}
      & \mathbf{0}_{(n-s)\times(s-2)}
  \end{pmatrix},
  \qquad t\in[0,1].
\end{equation}
Thus \(\beta_0(0)=Z_0\). The rows \(2,\ldots,s\) of
\(\beta_0(t)\) remain equal to \(e_2,\ldots,e_s\), while
its last \(n-s\) rows remain equal to \(e_1\).
Since \(n>s\), these fixed rows contain a basis of
\(\R^s\), so \(\rank\beta_0(t)=s\) for every
\(t\in[0,1]\).

Next, for \(j=1,\ldots,n-s\), rotate the \(j\)-th extra row along \( u(t) \),
keeping the previously rotated rows equal to \(-e_1\):
\begin{equation}
  \label{eq:orientation-reversing-extra-rows}
  \beta_j(t)\coloneq 
  \begin{pmatrix}
    -1 & 0 & \mathbf{0}_{1\times(s-2)}\\
    0 & 1 & \mathbf{0}_{1\times(s-2)}\\
    \mathbf{0}_{(s-2)\times1} & \mathbf{0}_{(s-2)\times1} & \I_{s-2}\\[2pt]
    -\mathbf 1_{j-1} & \mathbf{0}_{(j-1)\times1}
      & \mathbf{0}_{(j-1)\times(s-2)}\\
    \cos(\pi t) & \sin(\pi t) & \mathbf{0}_{1\times(s-2)}\\
    \mathbf 1_{n-s-j} & \mathbf{0}_{(n-s-j)\times1}
      & \mathbf{0}_{(n-s-j)\times(s-2)}
  \end{pmatrix},
  \qquad t\in[0,1].
\end{equation}
The first \(s\) rows remain equal to \(R_-\), so the rank
remains \(s\) throughout. Moreover,
\[
  \beta_j(0)=\beta_{j-1}(1)
  \quad (1\le j\le n-s),
  \qquad
  \beta_{n-s}(1)=Z_0R_-.
\]
Concatenating these paths and rescaling the parameter interval
gives a continuous piecewise smooth path from \(Z_0\) to \(Z_0R_-\):
\[
  \beta:[0,1]\to\Mss,
  \qquad
  \beta(0)=Z_0,
  \qquad
  \beta(1)=Z_0R_-.
\]
(Formally, we set \(N\coloneq n-s+1\) and define \(\beta(t)\coloneq \beta_j(Nt-j)\) for \(t\in[j/N,(j+1)/N]\) and \(j=0,\ldots,N-1\)).
\end{proof}

\begin{proposition}[Full holonomy in rank at least three]
  \label{prop:full-holonomy-rank-at-least-three}
  Assume \(s\ge3\).  Then, for every \(Z\in \Mss\),
  \[
    \Hol_Z(\omega)
    =
    \begin{cases}
      SO(s), & n=s,\\[1mm]
      O(s), & n>s.
    \end{cases}
  \]
\end{proposition}

\begin{proof}
  By \cref{prop:restricted-holonomy-rank-at-least-three}, \(\Hol_Z(\omega)^0=SO(s)\). Since \(\Hol_Z(\omega)\subset O(s)\), the only possibilities are
  \[
    \Hol_Z(\omega)=SO(s)
    \qquad\text{or}\qquad
    \Hol_Z(\omega)=O(s).
  \]
  It remains to decide whether orientation-reversing holonomy occurs.

  Consider the double cover
  \(\rho:\widetilde Q_s\to\Qss\) of
  \cref{def:oriented-quotient-cover}.

  Let \(Z\in \Mss\), let \(\bar\gamma\) be a loop in \(\Qss\) based at
  \(\pi_s(Z)\), and let \(\widehat\gamma\) be its horizontal lift
  starting at \(Z\). If $\widehat\gamma(1)=ZR$, then
  \(R\in\Hol_Z(\omega)\). Projecting \(\widehat\gamma\)
  to \(\widetilde Q_s\) gives a lift of \(\bar\gamma\) through the
  covering \(\rho\),
  starting at \([Z]_{SO(s)}\) and ending at \([ZR]_{SO(s)}\).  Since
  \(Z\) has full
  column rank,
  \begin{equation}
  \label{eq:holonomy-oriented-cover-criterion}
    [ZR]_{SO(s)}=[Z]_{SO(s)}
    \quad\Longleftrightarrow\quad
    R\in SO(s).
\end{equation}
  Thus the determinant of each holonomy element records the
monodromy of the corresponding loop. More explicitly, since \(\Hol_Z(\omega)^0=SO(s)\), the map
\(R\,SO(s)\mapsto\det R\) gives the isomorphism
\begin{equation}
  \label{eq:holonomy-component-determinant}
  \Hol_Z(\omega)/
  \Hol_Z(\omega)^0
  \simeq
  \det\bigl(\Hol_Z(\omega)\bigr)
  \subset\{\pm1\}.
\end{equation}

  We now distinguish the two cases.

  \smallskip
  \noindent\emph{Case \(n=s\).}
  Here \(\Mss\) consists of invertible \(s\times s\) matrices with
  unit row norms.
  The function
  \[
    [Z]_{SO(s)}\longmapsto \sgn(\det Z)
  \]
  is well defined on \(\widetilde Q_s\), because right multiplication
  by an element of \(SO(s)\) preserves \(\det Z\).  It is continuous
  and takes the values \(\pm1\). Hence \(\widetilde Q_s\) is the disjoint union of the two open and closed subsets corresponding to positive and negative determinant. The sign of the determinant is constant along every continuous path in \(\widetilde Q_s\). Therefore a lift of a loop in \(\Qss\) cannot join points having opposite determinant signs. Thus every loop has trivial monodromy in the double cover. Therefore
  every holonomy element has determinant \(+1\), and we have \(\Hol_Z(\omega)\subset SO(s)\). Since the identity component is already \(SO(s)\), we get
  \[
    \Hol_Z(\omega)=SO(s).
  \]

  \smallskip
  \noindent\emph{Case \(n>s\).}
  Let \(Z_0\), \(R_-\), and \(\beta\) be as in
  \cref{lem:orientation-reversing-factor-path}. Then
  \(\pi_s(Z_0R_-)=\pi_s(Z_0)\), whereas
  \([Z_0R_-]_{SO(s)}\ne[Z_0]_{SO(s)}\), since
  \(\det R_-=-1\).

  Projecting \(\beta\) to the oriented quotient \(\widetilde Q_s\) gives a path \([\beta]_{SO(s)}\) from \([Z_0]_{SO(s)}\) to \([Z_0R_-]_{SO(s)}\).
  Projecting further by \(\rho:\widetilde Q_s\longrightarrow \Qss\) gives a loop
  \[
    \bar\gamma(t)\coloneq \rho([\beta(t)]_{SO(s)}), \qquad t \in [0,1],
  \]
  based at \(\pi_s(Z_0)\), because \(Z_0\) and \(Z_0R_-\) lie in the same
  \(O(s)\)-orbit. Since the lift of this loop starting at
  \([Z_0]_{SO(s)}\) ends at \([Z_0R_-]_{SO(s)}\ne[Z_0]_{SO(s)}\), the
  loop has nontrivial monodromy in the double cover \(\rho\). Let
  \(\widehat\gamma\) be the horizontal lift of \(\bar\gamma\) to
  \(\Mss\) starting at \(Z_0\), and write
  \[
    \widehat\gamma(1)=Z_0H,
    \qquad H\in O(s).
  \]
  Then \(H\in\Hol_{Z_0}(\omega)\). Moreover, the two curves
\[
  t\longmapsto[\beta(t)]_{SO(s)}
  \qquad\text{and}\qquad
  t\longmapsto[\widehat\gamma(t)]_{SO(s)}
\]
both project through \(\rho\) to the same loop \(\bar\gamma\),
and both start at \([Z_0]_{SO(s)}\). By uniqueness of path lifts
with prescribed initial point, they coincide:
\begin{equation}
  \label{eq:holonomy-cover-lift-agreement}
  [\beta(t)]_{SO(s)}
  =
  [\widehat\gamma(t)]_{SO(s)}
  \qquad\text{for every }t\in[0,1].
\end{equation}
In particular, they have the same endpoint. Hence
  \[
    [Z_0H]_{SO(s)}=[Z_0R_-]_{SO(s)}.
  \]
  Since \(Z_0\) has full column rank, this implies \(HR_-^{-1}\in SO(s)\). As \(R_-\notin SO(s)\), it follows that \(H\notin SO(s)\). Therefore
  \[
    \Hol_{Z_0}(\omega)\not\subset SO(s).
  \]
  Since its identity component is \(SO(s)\), and \(SO(s)\) has index
  two in \(O(s)\), we get
  \[
    \Hol_{Z_0}(\omega)=O(s).
  \]
  Finally, holonomy groups over the path-connected base \(\Qss\) are
  conjugate, so the same conclusion holds for every \(Z\in \Mss\).
\end{proof}

Before treating the low-rank holonomy cases, we use relative-angle
coordinates to exhibit a diffeomorphism between the oriented
rank-two quotient and the \((n-1)\)-torus punctured at the
\(2^{n-1}\) points of \(\{0,\pi\}^{n-1}\).

\begin{lemma}[Relative-angle model of the oriented rank-two quotient]
\label{lem:rank-two-oriented-quotient}
Let \(n\ge2\), \(\mathbb T\coloneq \R/(2\pi\mathbb Z)\), and
\(\widetilde Q_2\coloneq \Mrk{2}{2}/SO(2)\). Write the rows of
\(Z\in M_2=(\mathbb S^1)^n\) as
\(z_i=(\cos\theta_i,\sin\theta_i)\), \(i=1,\ldots,n\). The relative-angle map
\[
  \Phi:\widetilde Q_2\longrightarrow
  \mathbb T^{n-1}\setminus\{0,\pi\}^{n-1},
  \qquad
  [Z]_{SO(2)}\longmapsto
  (\theta_1-\theta_n,\ldots,\theta_{n-1}-\theta_n)
\]
is a diffeomorphism. Without the rank restriction, the same map
identifies \(M_2/SO(2)\) with \(\mathbb T^{n-1}\).
\end{lemma}

\begin{proof}
The \(SO(2)\)-action on \(M_2\) is free, since a planar rotation
fixing a unit row vector must be the identity. It is also proper,
because \(SO(2)\) is compact. Hence \(M_2/SO(2)\) is a smooth
manifold. The diagonal \(SO(2)\)-action adds the same angle to all \(\theta_i\). Thus the relative angles \(\bar\delta_i\coloneq \theta_i-\theta_n\in\mathbb T\) are invariant.
Every orbit has a unique representative with \(\theta_n=0\), obtained
by rotating all rows by \(-\theta_n\). Conversely,
\(\bar\delta=(\bar\delta_1,\ldots,\bar\delta_{n-1})\in\mathbb T^{n-1}\) determines
the orbit of the factor with first \(n-1\) rows
\((\cos\bar\delta_i,\sin\bar\delta_i)\) and last row \((1,0)\).
These constructions are smooth inverses on \(M_2 / SO(2)\).

The representative has rank one precisely when every one of its first
\(n-1\) rows equals \((1,0)\) or \((-1,0)\). Equivalently,
\(\bar\delta_i\in\{0,\pi\}\) for every \(i<n\). Removing these points
gives the claimed restriction of \(\Phi\) to \(\widetilde Q_2\).
\end{proof}

For the rank-two argument and the example below, write
\(\widetilde\pi_2(Z)\coloneq [Z]_{SO(2)}\) for the oriented quotient map.
The quotient maps fit into the commutative diagram
\[
  \begin{tikzcd}[column sep=large, row sep=large]
    \Mrk{2}{2}
      \arrow[rr,"\pi_2"]
      \arrow[dr,"\widetilde\pi_2"']
    & & \Qrk{2}{2} \\
    & \widetilde Q_2
      \arrow[ur,"\rho"']
    &
  \end{tikzcd}
  \qquad
  \rho([Z]_{SO(2)})\coloneq [Z]_{O(2)}.
\]
Here \(\rho\) is the oriented double cover, and
\(\pi_2=\rho\circ\widetilde\pi_2\).

\begin{proposition}[Low-rank holonomy]
  \label{prop:low-rank-holonomy}
  For the reduced mechanical connection on
  \(\pi_s:\Mss\to\Qss\), the following hold.
  \begin{enumerate}
    \item If \(s=1\), then \(\Hol_Z(\omega)=\{\I_s\}\).
    \item If \(s=2\), then
      \[
        \Omega^\omega=0,
        \qquad
        \Hol_Z(\omega)^0=\{\I_s\}.
      \]
      Moreover,
      \[
        \Hol_Z(\omega)=\{\I_s\}\quad\text{if }n=2,
        \qquad
        \Hol_Z(\omega)\simeq D_n\quad\text{if }n>2,
      \]
      where \(D_n\) is as in \cref{thm:holonomy-all-rank-strata}.
  \end{enumerate}
\end{proposition}

The rank-two argument proceeds as follows. After establishing
flatness, we describe a loop in the oriented quotient by its relative
angles. Real lifts of these coordinates record the winding numbers,
while horizontality determines the common rotation of a factor
above the loop. This computes the orientation-preserving holonomy;
the double cover \(\rho\) then detects whether a reflection occurs.

\begin{proof}
\emph{Case \(s=1\).}
The spaces \(\Mrk{1}{1}=\{\pm1\}^n\) and
\(\Qrk{1}{1}=\Mrk{1}{1}/O(1)\) are discrete. Every loop in the base
and its horizontal lift are therefore constant, so
\(\Hol_Z(\omega)=\{\I_s\}\).

\smallskip
\noindent\emph{Case \(s=2\): flatness.}
Write the rows of \(Z\in \Mrk{2}{2}\) as
\(z_i=(\cos\theta_i,\sin\theta_i)\), with \(\theta_i\in\mathbb T\),
and set
\[
  K\coloneq \begin{pmatrix}0&1\\-1&0\end{pmatrix},
  \qquad R_\alpha\coloneq \exp(\alpha K).
\]
Since \(\Tz{z_i}\mathbb S^1=\spanv\{z_iK\}\), every
pair \(U,V\in\Tz{Z}\Mrk{2}{2}\) has rows
\[
  u_i=a_i z_iK,
  \qquad v_i=b_i z_iK,
  \qquad a_i,b_i\in\R.
\]
Thus \(u_i^\top v_i=a_ib_i(z_iK)^\top(z_iK)=v_i^\top u_i\), and
\[
  U^\top V-V^\top U
  =\sum_{i=1}^n(u_i^\top v_i-v_i^\top u_i)=0.
\]
For horizontal \(U,V\), the curvature formula in
\cref{eq:curvature-form-sylvester-horizontal} gives
\[
  (\Omega^\omega)_Z(U,V)
  =2\mathcal L_{E_Z}^{-1}(U^\top V-V^\top U)=0.
\]
Since the curvature form is horizontal, \(\Omega^\omega\equiv0\).
By Ambrose--Singer,
\[
  \mathfrak{hol}_Z(\omega)=\{0\},
  \qquad \Hol_Z(\omega)^0=\{\I_s\}.
\]
Indeed, the identity component is a connected zero-dimensional Lie
group, hence consists only of the identity. The full holonomy is
therefore the monodromy of this flat connection: horizontal transport
around a loop depends only on its homotopy class with fixed base point.

When \(n=2\), \cref{lem:rank-two-oriented-quotient} identifies
\(\widetilde Q_2\) with \(\mathbb T\setminus\{0,\pi\}\).
The remaining \(O(2)/SO(2)\)-action sends
\(\bar\delta\) to \(-\bar\delta\) and exchanges its two interval
components. Hence \(\Qrk{2}{2}\simeq(0,\pi)\). This base is
contractible, so the flat connection has trivial holonomy.
We assume \(n>2\) for the rest of the proof.

\smallskip
\noindent\emph{Orientation-preserving holonomy.}
Since \(SO(2)\) and \(O(2)\) have the same infinitesimal orbit
directions, the same horizontal distribution defines a principal
\(SO(2)\)-connection on
\(\widetilde\pi_2:\Mrk{2}{2}\to\widetilde Q_2\).
Its holonomy group is \(\Hol_Z(\omega)\cap SO(2)\).
Indeed, a horizontal lift of a loop in \(\Qrk{2}{2}\) with endpoint
\(ZR\) projects through \(\widetilde\pi_2\) to a closed loop exactly
when
\[
  [ZR]_{SO(2)}=[Z]_{SO(2)}
  \quad\Longleftrightarrow\quad R\in SO(2),
\]
because \(Z\) has full column rank. Thus the double cover detects
precisely the determinant of the holonomy element.

The product metric, equivalently the restriction of the Frobenius
inner product to \(\Tbundle M_2\), has the angular expression
\[
  g_{M_2}=\sum_{i=1}^n(\dd\theta_i)^2.
\]
The vertical direction is spanned by \(ZK\). Hence a tangent vector
\(\dot Z\), with rows \(\dot z_i=\dot\theta_i z_iK\), is horizontal
if and only if
\begin{equation}
  \label{eq:rank-two-angular-horizontality}
  0=(g_{M_2})_Z(\dot Z,ZK)=\sum_{i=1}^n\dot\theta_i,
\end{equation}
since \(\|z_iK\|_2=1\). Under the identification
\(\so(2)\simeq\R\), \(aK\leftrightarrow a\), the
connection form therefore has the local expression
\[
  \omega=\frac1n\sum_{i=1}^n\dd\theta_i.
\]
This form vanishes on horizontal vectors and has the required
normalization \(\omega_Z(aZK)=a\), in particular
\(\omega_Z(ZK)=1\), under this identification.

Let \(\gamma:[0,1]\to\widetilde Q_2\) be a loop based at
\(\widetilde\pi_2(Z)\), and set \(\bar\gamma\coloneq \rho\circ\gamma\).
Using \cref{lem:rank-two-oriented-quotient}, write
\[
  \Phi(\gamma(t))
  =(\bar\delta_1(t),\ldots,\bar\delta_{n-1}(t))
  \in\mathbb T^{n-1}\setminus\{0,\pi\}^{n-1}.
\]
For the covering map \(q:\R\to\mathbb T\),
\(q(x)=x\bmod 2\pi\), choose continuous real lifts
\(\delta_i:[0,1]\to\R\) satisfying
\(q(\delta_i(t))=\bar\delta_i(t)\). These lifts are piecewise
\(C^1\), and, since \(\gamma\) is closed,
\begin{equation}
  \label{eq:rank-two-winding-increments}
  \Delta\delta_i\coloneq \delta_i(1)-\delta_i(0)=2\pi m_i,
  \qquad m_i\in\mathbb Z,\quad i=1,\ldots,n-1.
\end{equation}

Let \(\widehat\gamma:[0,1]\to \Mrk{2}{2}\) be the horizontal lift
of \(\gamma\) through \(\widetilde\pi_2\), starting at \(Z\).
It is also the horizontal lift of \(\bar\gamma\) through \(\pi_2\),
since
\begin{equation}
  \label{eq:rank-two-lift-projections}
  \widetilde\pi_2\circ\widehat\gamma=\gamma,
  \qquad \pi_2\circ\widehat\gamma=\rho\circ\gamma=\bar\gamma.
\end{equation}
To specify the row angles, let
\(c:\mathbb T\to\mathbb S^1\),
\(c(\theta)\coloneq (\cos\theta,\sin\theta)\), be the usual diffeomorphism.
The \emph{circle-valued row angles of \(\widehat\gamma\)} are the
uniquely defined curves
\begin{equation}
  \label{eq:rank-two-circle-row-angles}
  \theta_i:[0,1]\to\mathbb T,
  \qquad
  \theta_i(t)\coloneq c^{-1}\bigl(\widehat\gamma(t)_{i\cdot}\bigr),
  \qquad i=1,\ldots,n.
\end{equation}
Since \(\widetilde\pi_2\circ\widehat\gamma=\gamma\), their
relative angles satisfy
\begin{equation}
  \label{eq:rank-two-circle-relative-angles}
  \bar\delta_i(t)=\theta_i(t)-\theta_n(t)
  \quad\text{in }\mathbb T,
  \qquad i=1,\ldots,n-1.
\end{equation}
Choose a continuous real lift \(\vartheta_n:[0,1]\to\R\)
of \(\theta_n\), so that \(q\circ\vartheta_n=\theta_n\), and set
\begin{equation}
  \label{eq:rank-two-compatible-real-angles}
  \vartheta_i(t)\coloneq \vartheta_n(t)+\delta_i(t),
  \qquad i=1,\ldots,n-1.
\end{equation}
These are real lifts of the other row angles because
\begin{equation}
  \label{eq:rank-two-row-angle-projection}
  q(\vartheta_i(t))
  =q(\vartheta_n(t))+q(\delta_i(t))
  =\theta_n(t)+\bar\delta_i(t)
  =\theta_i(t).
\end{equation}
Consequently, for every \(i=1,\ldots,n\),
\[
  \widehat\gamma(t)_{i\cdot}
  =c(\theta_i(t))
  =c(q(\vartheta_i(t)))
  =(\cos\vartheta_i(t),\sin\vartheta_i(t)).
\]
Thus \(\theta_i\) is an angle of a row of \(\widehat\gamma\),
\(\vartheta_i\) is its real lift, and \(\delta_i\) is a real lift
of the relative angle \(\bar\delta_i=\theta_i-\theta_n\).
The maps and their compatibility are summarized in
\cref{fig:rank-two-lifts-and-angles}.

On each smooth piece, \cref{eq:rank-two-angular-horizontality} and
\cref{eq:rank-two-compatible-real-angles} give
\begin{equation}
  \label{eq:rank-two-horizontal-angle-velocity}
  0=\sum_{i=1}^n\dot\vartheta_i
   =n\dot\vartheta_n+\sum_{i=1}^{n-1}\dot\delta_i,
  \qquad
  \dot\vartheta_n=-\frac1n\sum_{i=1}^{n-1}\dot\delta_i.
\end{equation}
Integrating \cref{eq:rank-two-horizontal-angle-velocity} over the
pieces, summing, and using \cref{eq:rank-two-winding-increments} yields
\begin{equation}
  \label{eq:rank-two-horizontal-angle-increment}
  \Delta\vartheta_n
  \coloneq \vartheta_n(1)-\vartheta_n(0)
  =-\frac1n\sum_{i=1}^{n-1}\Delta\delta_i
  =-\frac{2\pi}{n}\sum_{i=1}^{n-1}m_i.
\end{equation}
For \(i<n\), the increment of the \(i\)-th row angle is
\(\Delta\vartheta_i=\Delta\vartheta_n+2\pi m_i\).
All rows therefore undergo the same rotation modulo \(2\pi\), so
\begin{equation}
  \label{eq:rank-two-horizontal-endpoint}
  \widehat\gamma(1)=ZR_{\Delta\vartheta_n}.
\end{equation}
By \cref{eq:rank-two-horizontal-angle-increment}, the corresponding
orientation-preserving holonomy element is
\begin{equation}
  \label{eq:hol_formula}
  R_{\Delta\vartheta_n}
  =R_{-\frac{2\pi}{n}\sum_{i=1}^{n-1}m_i}.
\end{equation}

\begin{figure}[htbp]
  \centering
  \definecolor{holrowblue}{RGB}{42,83,145}
  \definecolor{holrelgreen}{RGB}{0,116,105}
  \definecolor{holpuncture}{RGB}{172,69,53}
  \resizebox{.90\linewidth}{!}{\begin{tikzpicture}[>=stealth,font=\small,
    rowbox/.style={draw=holrowblue!45,fill=holrowblue!4,
      rounded corners=2pt,align=center,text width=4.8cm,
      minimum height=.85cm,inner sep=5pt},
    relbox/.style={draw=holrelgreen!45,fill=holrelgreen!4,
      rounded corners=2pt,align=center,text width=4.8cm,
      minimum height=.85cm,inner sep=5pt},
    map/.style={->,draw=black!65,semithick}]
    \node[anchor=west,font=\small\bfseries] at (-2.58,.95)
      {(a) The quotient curves, row angles, and real lifts};
    \node[rowbox] (realrow) at (0,0)
      {$\vartheta(t)\in\R^n$\\[-1pt]
       {\scriptsize real row-angle lifts}};
    \node[relbox] (realrel) at (8.6,0)
      {$\delta(t)\in\R^{n-1}$\\[-1pt]
       {\scriptsize real relative-angle lifts}};
    \node[rowbox] (rowangle) at (0,-1.95)
      {$\theta(t)\in\mathbb T^n$\\[-1pt]
       {\scriptsize circle-valued row angles of $\widehat\gamma$}};
    \node[relbox] (relangle) at (8.6,-1.95)
      {$\bar\delta(t)\in\mathbb T^{n-1}$\\[-1pt]
       {\scriptsize relative-angle coordinates $\Phi(\gamma(t))$}};
    \node[rowbox] (factor) at (0,-3.9)
      {$\widehat\gamma(t)\in \Mrk{2}{2}$\\[-1pt]
       {\scriptsize horizontal factor curve}};
    \node[relbox] (oriented) at (8.6,-3.9)
      {$\gamma(t)\in\widetilde Q_2$\\[-1pt]
       {\scriptsize loop in the oriented quotient}};
    \node[draw=black!35,fill=black!2,rounded corners=2pt,
      text width=4.4cm,align=center,inner sep=5pt] (base) at (4.3,-5.55)
      {$\bar\gamma(t)\in \Qrk{2}{2}$\\[-1pt]
       {\scriptsize loop in the unoriented quotient}};
    \draw[map] (realrow) -- node[above,font=\small]
      {$\delta_i=\vartheta_i-\vartheta_n$} (realrel);
    \draw[map] (rowangle) -- node[above,font=\small]
      {$\bar\delta_i=\theta_i-\theta_n$} (relangle);
    \draw[map] (factor) -- node[above] {$\widetilde\pi_2$} (oriented);
    \draw[map] (realrow) -- node[right] {$q^{\times n}$} (rowangle);
    \draw[map] (realrel) -- node[right] {$q^{\times(n-1)}$} (relangle);
    \draw[map] (rowangle) -- node[right] {$c^{\times n}$} (factor);
    \draw[map] (relangle) -- node[right] {$\Phi^{-1}$} (oriented);
    \draw[map] (factor.south) -- node[below left] {$\pi_2$} (base.north west);
    \draw[map] (oriented.south) -- node[below right] {$\rho$} (base.north east);
  \end{tikzpicture}}
  \par\medskip
  \resizebox{.94\linewidth}{!}{\begin{tikzpicture}[>=stealth,font=\small]
    \node[anchor=west,font=\small\bfseries] at (-.45,3.8)
      {(b) A real lift in $\R^2$};
    \node[anchor=west,font=\small\bfseries] at (8.55,3.8)
      {(c) The punctured torus};
    \fill[holrelgreen!3] (0,0) rectangle (3,3);
    \fill[holrelgreen!7] (3,0) rectangle (6,3);
    \draw[black!40] (0,0) rectangle (6,3);
    \draw[dashed,black!35] (3,0) -- (3,3);
    \draw[black!12] (0,1.5) -- (6,1.5);
    \draw[->,black!65] (0,0) -- (6.35,0) node[right] {$\delta_1$};
    \draw[->,black!65] (0,0) -- (0,3.35);
    \node[anchor=west] at (.09,3.20) {$\delta_2$};
    \foreach \x/\lab in {0/0,1.5/\pi,3/2\pi,4.5/3\pi,6/4\pi}
      \node[below,font=\scriptsize] at (\x,-.05) {$\lab$};
    \foreach \y/\lab in {1.5/\pi,3/2\pi}
      \node[left,font=\scriptsize] at (-.05,\y) {$\lab$};
    \foreach \x in {0,1.5,3,4.5,6} {
      \foreach \y in {0,1.5,3} {
        \draw[holpuncture,semithick]
          ({\x-.05},{\y-.05}) -- ({\x+.05},{\y+.05})
          ({\x-.05},{\y+.05}) -- ({\x+.05},{\y-.05});
      }
    }
    \draw[holrelgreen,line width=1.2pt] (.375,.375) -- (3.375,.375);
    \draw[->,holrelgreen,line width=1.2pt] (1.4,.375) -- (1.9,.375);
    \fill[holrelgreen] (.375,.375) circle (2pt);
    \draw[holrelgreen,fill=white,line width=.9pt] (3.375,.375) circle (2.5pt);
    \node[above,holrelgreen,font=\scriptsize] at (.375,.48) {$\delta(0)$};
    \node[above,holrelgreen,font=\scriptsize] at (3.375,.48) {$\delta(1)$};
    \draw[<->,holrelgreen!80!black] (.375,1.05) --
      node[above,font=\scriptsize] {$2\pi$} (3.375,1.05);
    \node[align=center] at (3,-.78)
      {$\delta(1)-\delta(0)=(2\pi,0)$};
    \draw[->,black!65,semithick] (6.55,1.65) --
      node[above] {$q^{\times2}$} (8.45,1.65);
    \fill[holrelgreen!3] (9,0) rectangle (12,3);
    \draw[black!40] (9,0) rectangle (12,3);
    \draw[black!12] (9,1.5) -- (12,1.5);
    \draw[black!12] (10.5,0) -- (10.5,3);
    \draw[->,black!65] (9,0) -- (12.35,0) node[right] {$\bar\delta_1$};
    \draw[->,black!65] (9,0) -- (9,3.35);
    \node[anchor=west] at (9.09,3.20) {$\bar\delta_2$};
    \foreach \x/\lab in {9/0,10.5/\pi,12/2\pi}
      \node[below,font=\scriptsize] at (\x,-.05) {$\lab$};
    \foreach \y/\lab in {1.5/\pi,3/2\pi}
      \node[left,font=\scriptsize] at (8.95,\y) {$\lab$};
    \foreach \x in {9,10.5,12} {
      \foreach \y in {0,1.5,3} {
        \draw[holpuncture,semithick]
          ({\x-.05},{\y-.05}) -- ({\x+.05},{\y+.05})
          ({\x-.05},{\y+.05}) -- ({\x+.05},{\y-.05});
      }
    }
    \foreach \y in {0,3}
      \draw[->,holrowblue,semithick] (10.0,\y) -- (10.35,\y);
    \foreach \x in {9,12} {
      \draw[->,holrowblue,semithick] (\x,1.05) -- (\x,1.25);
      \draw[->,holrowblue,semithick] (\x,1.25) -- (\x,1.45);
    }
    \draw[holrelgreen,line width=1.2pt] (9,.375) -- (12,.375);
    \draw[->,holrelgreen,line width=1.2pt] (10.25,.375) -- (10.75,.375);
    \draw[->,holrelgreen,line width=1.2pt] (9.03,.375) -- (9.25,.375);
    \draw[holrelgreen,fill=white,line width=.9pt] (9.375,.375) circle (2.8pt);
    \fill[holrelgreen] (9.375,.375) circle (1.35pt);
    \node[above,holrelgreen,font=\scriptsize] at (10.6,.48)
      {$\bar\delta(t)=\Phi(\gamma(t))$};
    \node[align=center] at (10.5,-.78)
      {$\bar\delta(0)=\bar\delta(1)$};
  \end{tikzpicture}}
  \caption{The lifts used in rank-two holonomy.
    (a) The real row-angle lifts satisfy $q(\vartheta_i)=\theta_i$,
    where $\theta_i$ is the circle-valued angle of the $i$-th row of
    $\widehat\gamma$. All maps are evaluated along the indicated
    curves; $q^{\times m}$ and $c^{\times n}$ act componentwise.
    (b--c) For $n=3$, the real relative-angle lift
    $\delta(t)=(\pi/4+2\pi t,\pi/4)$ has distinct endpoints, while
    its projection closes after identifying opposite edges of the
    square. Crosses mark the removed rank-one classes and their
    periodic copies.}
  \label{fig:rank-two-lifts-and-angles}
\end{figure}

In particular,
\[
  \Hol_Z(\omega)\cap SO(2)
  \subset C_n\coloneq \{R_{2\pi j/n}:j=0,\ldots,n-1\}.
\]

To obtain a generator, choose \(\varepsilon\in(0,\pi)\) and take
the loop \(\gamma_\varepsilon\) defined by
\[
  \Phi(\gamma_\varepsilon(t))
  =\bigl(q(\varepsilon+2\pi t),
          q(\varepsilon),\ldots,q(\varepsilon)\bigr).
\]
Since \(n>2\), its second coordinate is present and never belongs
to \(\{0,\pi\}\); hence the loop avoids the removed rank-one set.
The real lifts \(\delta_1(t)=\varepsilon+2\pi t\) and
\(\delta_j(t)=\varepsilon\) for \(j\ge2\) have winding numbers
\((m_1,\ldots,m_{n-1})=(1,0,\ldots,0)\).
By \cref{eq:hol_formula}, the holonomy at any factor
\(Z_\varepsilon\) above \(\gamma_\varepsilon(0)\) is
\(R_{-2\pi/n}\), which generates \(C_n\). Therefore
\(\Hol_{Z_\varepsilon}(\omega)\cap SO(2)=C_n\).
For \(n>2\), the punctured torus \(\widetilde Q_2\) is
path-connected: removing finitely many points from a connected
manifold of dimension \(n-1\ge2\) preserves path-connectedness.
The holonomy groups of its principal \(SO(2)\)-connection are
therefore conjugate in \(SO(2)\). Since this group is abelian,
\begin{equation}
  \label{eq:rank-two-cyclic-holonomy}
  \Hol_Z(\omega)\cap SO(2)=C_n
  \qquad\text{for every }Z\in \Mrk{2}{2}.
\end{equation}

\smallskip
\noindent\emph{Orientation-reversing holonomy.}
Apply \cref{lem:orientation-reversing-factor-path} with \(s=2\).
It gives a path \(\beta\) from \(Z_0\) to \(Z_0R_-\), where
\(R_-\coloneq \diag(-1,1)\). Its projection
\(\pi_2\circ\beta\) is a loop, whose lift through \(\rho\)
starting at \([Z_0]_{SO(2)}\) is
\(\widetilde\pi_2\circ\beta\) and ends at the distinct point
\([Z_0R_-]_{SO(2)}\). By uniqueness of covering lifts, the
projection of the horizontal lift of this loop has the same endpoint.
The determinant criterion in \cref{eq:holonomy-oriented-cover-criterion}
therefore gives an element
\(\tau_0\in\Hol_{Z_0}(\omega)\) with
\(\det\tau_0=-1\).

Since the determinant kernel in \(\Hol_{Z_0}(\omega)\)
is \(C_n\),
\begin{equation}
  \label{eq:rank-two-dihedral-cosets}
  \Hol_{Z_0}(\omega)=C_n\cup\tau_0C_n.
\end{equation}
Indeed, any orientation-reversing element \(h\) satisfies
\(\tau_0^{-1}h\in C_n\), and both cosets belong to the holonomy
group. Every orientation-reversing element of \(O(2)\) is a
reflection, so
\[
  \tau_0^2=\I_2,
  \qquad \tau_0R_\alpha\tau_0^{-1}=R_{-\alpha}.
\]
Thus \(C_n\cup\tau_0C_n\) is a dihedral group of order \(2n\).
Finally, \(\Qrk{2}{2}\) is path-connected, so the full holonomy
groups at different base points are conjugate in \(O(2)\). Hence
\(\Hol_Z(\omega)\simeq D_n\) for every
\(Z\in \Mrk{2}{2}\).
\end{proof}

\subsection{Holonomy over the rank-\texorpdfstring{\(s\)}{s}
correlation stratum}
\label{subsec:holonomy-correlation-stratum}

We now translate the holonomy classification of the mechanical
connection to the rank-\(s\) correlation stratum. Recall that the
Gram map $\Gamma_s$ factors through the quotient as
\[
  \begin{tikzcd}
    \Mss \arrow[r,"\pi_s"] \arrow[dr,swap,"\Gamma_s"]
    & \Qss \arrow[d,"\bar\Gamma_s"] \\
    & \Crk{s}{n},
  \end{tikzcd}
\]
Therefore, the Gram map
\[
  \Gamma_s:\Mss\longrightarrow \Crk{s}{n}
\]
is a principal \(O(s)\)-bundle, whose fibres are precisely the factor
representations of a given rank-\(s\) correlation matrix. More explicitly, if
\[
  C=ZZ^\top\in\Crk{s}{n}, \qquad Z\in \Mss,
\]
then
\[
  \Gamma_s^{-1}(C) = \{Y\in \Mss:YY^\top=C\} = \{ZR:R\in O(s)\},
\]
where the equality follows from the full-column-rank condition: if
\(YY^\top=ZZ^\top\), then there exists a unique \(R\in O(s)\) such
that \(Y=ZR\). Through the Riemannian isometry
\[
  \bar\Gamma_s:(\Qss,\gQss)
  \longrightarrow
  \bigl(\Crk{s}{n},(\bar\Gamma_s)_*\gQss\bigr),
\]
the reduced mechanical connection \(\omega\) on
\(\pi_s:\Mss\to \Qss\), given by
\cref{prop:sylvester-mechanical-connection} as
\[
  \omega_Z(W)
  =
  \mathcal L_{E_Z}^{-1}\bigl(Z^\top W-W^\top Z\bigr),
  \qquad
  E_Z\coloneq Z^\top Z,
\]
may be viewed as a connection on the principal $O(s)$-bundle
$\Gamma_s$. Its holonomy measures the path dependence of horizontally
transported factors \(Z\) along loops in the rank-\(s\) correlation stratum.

\begin{definition}[Factor holonomy]
\label{def:factor-holonomy}
Fix \(1\le s\le n\) and \(Z\in\Mss\). Let
\(C:[0,1]\to\Crk{s}{n}\) be a continuous piecewise
\(C^1\) loop based at \(ZZ^\top\), and let
\(\widehat C:[0,1]\to\Mss\) be its \(\omega\)-horizontal
factor lift starting at \(Z\). Thus
\[
  \widehat C(t)\widehat C(t)^\top=C(t),
  \qquad \widehat C(0)=Z,
  \qquad
  \omega_{\widehat C(t)}\bigl(\widehat C'(t)\bigr)=0
\]
at every differentiability time. The \emph{factor holonomy} of
\(C\) based at \(Z\) is the unique matrix \(R_C\in O(s)\) such that
\[
  \widehat C(1)=ZR_C.
\]
The \emph{factor holonomy group} at \(Z\) is the subgroup
\[
  \Hol^{\Corr^{(s)}}_Z(\omega)
  \coloneq \bigl\{R_C:\ C\text{ is a loop in }\Crk{s}{n}
    \text{ based at }ZZ^\top\bigr\}
  \subset O(s).
\]
\end{definition}

\begin{corollary}[Holonomy over the rank-\(s\) correlation stratum]
  \label{thm:holonomy-correlation-strata}
  Let \(1\le s\le n\) and \(Z\in\Mss\). Then the factor
  holonomy group satisfies
  \[
    \Hol^{\Corr^{(s)}}_Z(\omega)
    =
    \Hol_Z(\omega).
  \]
  In particular, the classification in
  \cref{thm:holonomy-all-rank-strata} applies to factor holonomy
  over \(\Crk{s}{n}\).
\end{corollary}

\begin{proof}
  For a loop \(C\) in \(\Crk{s}{n}\), set
  \(\bar\gamma\coloneq \bar\Gamma_s^{-1}\circ C\). The factorization
  \(\Gamma_s=\bar\Gamma_s\circ\pi_s\) identifies lifts of \(C\)
  through \(\Gamma_s\) with lifts of \(\bar\gamma\) through \(\pi_s\).
  Horizontality is determined in both cases by the same connection
  form \(\omega\) on \(\Mss\). Thus the horizontal lifts with
  prescribed initial point coincide and have the same endpoint,
  yielding the same holonomy element. Since the descended Gram
  diffeomorphism identifies all based loops, the holonomy groups
  are equal.
\end{proof}

\subsection{An example of holonomy in low dimension}
We illustrate the rank-two classification for \(3\times3\)
correlation matrices in \cref{fig:holonomy-s2-dihedral}.
For \((s,n)=(2,3)\), write the rows of
\(Z\in \Mrk{2}{2}\subset(\mathbb S^1)^3\) as
\(z_i=(\cos\theta_i,\sin\theta_i)\), with \(\theta_i\in\mathbb T\).
By \cref{lem:rank-two-oriented-quotient}, the oriented quotient
\(\widetilde Q_2\) is \(\mathbb T^2\setminus\{0,\pi\}^2\), with
circle-valued relative angles
\[
  \bar\delta_1\coloneq \theta_1-\theta_3,
  \qquad \bar\delta_2\coloneq \theta_2-\theta_3.
\]

Using the notation of the proof of \cref{prop:low-rank-holonomy},
define a loop \(\gamma:[0,1]\to\widetilde Q_2\) by
\[
  \Phi(\gamma(t))=(q(\delta_1(t)),q(\delta_2(t))),
  \qquad
  \delta_1(t)\coloneq \frac\pi4+2\pi t,
  \qquad \delta_2(t)\coloneq \frac\pi4.
\]
The relative-angle loop and its real lift are shown in the lower
panels of \cref{fig:rank-two-lifts-and-angles}. Its projection
\(\bar\gamma\coloneq \rho\circ\gamma\) is a loop in
\(\Qrk{2}{2}\). Set
\[
  Z_*\coloneq 
  \begin{pmatrix}
    \cos(\pi/4)&\sin(\pi/4)\\
    \cos(\pi/4)&\sin(\pi/4)\\
    1&0
  \end{pmatrix}\in \Mrk{2}{2}.
\]
Then \(\widetilde\pi_2(Z_*)=\gamma(0)\) and
\(\pi_2(Z_*)=\bar\gamma(0)\). Let \(\widehat\gamma\) be the
horizontal lift of \(\gamma\) through \(\widetilde\pi_2\),
starting at \(Z_*\). Since \(\pi_2=\rho\circ\widetilde\pi_2\),
it is also the horizontal lift of \(\bar\gamma\) through \(\pi_2\).

Let \(\theta_i(t)\coloneq c^{-1}(\widehat\gamma(t)_{i\cdot})\in\mathbb T\)
be its circle-valued row angles, as in
\cref{eq:rank-two-circle-row-angles}. Choose compatible real lifts
\(\vartheta=(\vartheta_1,\vartheta_2,\vartheta_3)\) satisfying
\[
  q(\vartheta_i(t))=\theta_i(t)\quad(i=1,2,3),
  \qquad
  \vartheta(0)=(\pi/4,\pi/4,0),
\]
with \(\vartheta_i-\vartheta_3=\delta_i\) for \(i=1,2\), as in
\cref{eq:rank-two-compatible-real-angles}.
The prescribed relative angles and
\cref{eq:rank-two-angular-horizontality} give
\[
  \begin{cases}
    \dot\vartheta_1-\dot\vartheta_3=2\pi,\\
    \dot\vartheta_2-\dot\vartheta_3=0,\\
    \dot\vartheta_1+\dot\vartheta_2+\dot\vartheta_3=0.
  \end{cases}
\]
Consequently,
\begin{equation}
  \label{eq:rank-two-example-angle-lift}
  \vartheta(t)
  =\left(\frac\pi4,\frac\pi4,0\right)
    +t\left(\frac{4\pi}3,-\frac{2\pi}3,-\frac{2\pi}3\right),
\end{equation}
and the horizontal factor curve is
\[
  \widehat\gamma(t)=
  \begin{pmatrix}
    \cos\vartheta_1(t)&\sin\vartheta_1(t)\\
    \cos\vartheta_2(t)&\sin\vartheta_2(t)\\
    \cos\vartheta_3(t)&\sin\vartheta_3(t)
  \end{pmatrix}\in \Mrk{2}{2}.
\]
After one turn of the first relative angle,
\begin{equation}
  \label{eq:rank-two-example-holonomy}
  \widehat\gamma(1)=Z_*R_{-2\pi/3}.
\end{equation}
Thus one traversal of \(\bar\gamma\) gives the cyclic holonomy
generator \(a=R_{-2\pi/3}\), and repeated traversals give
\(a^2\) and \(a^3=\I_2\).

For orientation-reversing holonomy, take the path \(\beta\) from
\cref{lem:orientation-reversing-factor-path}, with
\[
  \beta(0)=Z_0=
  \begin{pmatrix}
    1&0\\0&1\\1&0
  \end{pmatrix},
  \qquad
  \beta(1)=Z_0\diag(-1,1).
\]
The curve \(\pi_2\circ\beta\) is a loop, while its lift
\(\widetilde\pi_2\circ\beta\) to the oriented cover has distinct
endpoints. The corresponding horizontal lift therefore yields an
orientation-reversing holonomy element at \(Z_0\). The two loops
illustrate the two types of holonomy: the first produces a generator
of \(C_3\), and the second has determinant \(-1\). They are based
at different points, but their holonomy groups are conjugate; by
\cref{prop:low-rank-holonomy}, each is isomorphic to \(D_3\).

\begin{figure}[htpb]
  \centering
  \begin{subfigure}[t]{.48\linewidth}
    \centering
    \includegraphics[width=\linewidth]{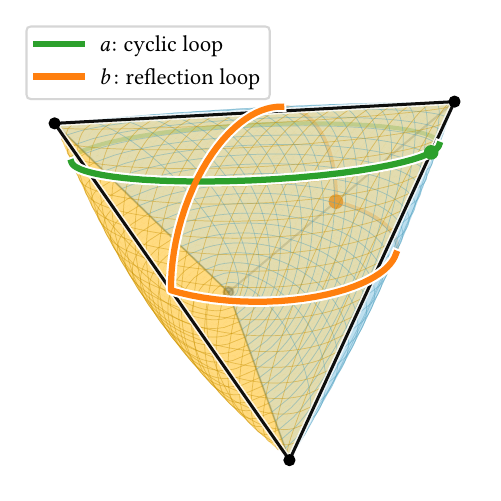}
    \caption{Loops on \(\Crk{2}{3}\).}
  \end{subfigure}\hfill
  \begin{subfigure}[t]{.48\linewidth}
    \centering
    \includegraphics[width=\linewidth]{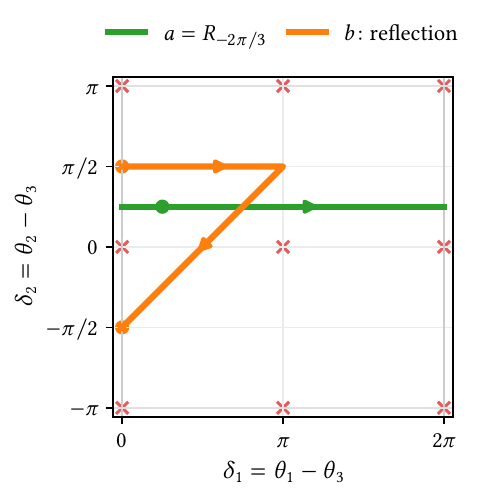}
    \caption{Relative-angle cover \(\mathbb T^2\setminus\{0,\pi\}^2\).}
  \end{subfigure}
  \caption{Rank-two mechanical holonomy for \((s,n)=(2,3)\).
    The green loop has holonomy \(a=R_{-2\pi/3}\); the orange loop $b$
    has orientation-reversing holonomy. In relative-angle coordinates
    modulo \(2\pi\), the green lift closes on the torus, whereas
    the orange lift joins distinct points exchanged by reflection.}
  \label{fig:holonomy-s2-dihedral}
\end{figure}

\section{Intrinsic geodesics under rank constraints}
\label{sec:intrinsic-geodesics-rank-constraints}

We now turn from fixed-rank geometry to minimizing geodesics subject
to a rank bound. We use the ambient distance $\dQ{k}$, the stratum distance
$\dQks$, and the intrinsic rank-bounded distance $\dkles$ as defined in
\cref{subsec:metric-conventions,subsec:fixed-rank-riemannian-reduction,subsec:global-rank-bounded-metric-reduction}.
For rank-$s$ endpoints, the equality
\[
  \dkles(\bar p,\bar q)=\dQks(\bar p,\bar q)
\]
already follows from \cref{eq:rank-bounded-distance-on-top-stratum}.
We establish the additional fact that every rank-bounded minimizer
between these endpoints has rank $s$ at all interior times.

We shall use the following metric fact about minimizing segments in
orbit spaces. It is a standard consequence of the orbit-type
structure of a Riemannian quotient; see \cite[Lemma~3.5]{AKLM03}. If
a compact Lie group acts isometrically on a complete Riemannian
manifold and \(\bar\gamma\) is a minimizing geodesic segment in the
orbit space, then the orbit type along the interior of \(\bar\gamma\)
is no more singular
than the orbit types at the endpoints. Equivalently, for a horizontal lift
\(\gamma\), the stabilizer of \(\gamma(t)\), \(t\in(0,1)\), is
conjugate to a subgroup of the stabilizers of the endpoints.

\begin{theorem}[No interior rank drop for rank-bounded minimizers]
  \label{thm:no-interior-rank-drop-rank-bounded}
  Let $2\le s\le\min(n,k)$, and let $\bar p,\bar q\in\Qks$.
  Every $\dkles$-minimizing geodesic segment from $\bar p$ to
  $\bar q$ has rank $s$ at all interior times.
\end{theorem}

Existence of the minimizing segment used below follows from
\cref{cor:rank-bounded-global-alexandrov-geometry}: the rank-bounded space is
compact and geodesic.  The orbit-type argument is still essential for the
stronger conclusion that every such segment remains in the top-rank stratum
at interior times.

\begin{proof}
  Let \(\bar\gamma:[0,1]\to \Qkles\) be a \(\dkles\)-minimizing
  segment from \(\bar p\) to \(\bar q\).  By the global rank-bounded metric
  reduction theorem, \cref{thm:global-rank-bounded-metric-reduction}, the map
  \[
    \Xks:(Q_s,\dQ{s})
    \longrightarrow
    (\Qkles,\dkles)
  \]
  is an isometry of length spaces. Hence
  \[
    \bar\gamma_s\coloneq\Xks^{-1}\circ\bar\gamma
  \]
  is a minimizing segment in \((Q_s,\dQ{s})\) joining $\bar
  p_s\coloneq\Xks^{-1}(\bar p)$ to $\bar q_s\coloneq\Xks^{-1}(\bar q)$. Since
  \(\bar p\) and \(\bar q\) have rank \(s\), their preimages \(\bar
  p_s\) and \(\bar q_s\) lie in the full-rank stratum \(\Qss\). On
  \(\Mss\), the right \(O(s)\)-action is free. Therefore the
  stabilizers of full-rank representatives of \(\bar p_s\) and \(\bar
  q_s\) are trivial.

  By \cite[Lemma~3.5]{AKLM03}, the stabilizer of an interior point of
  a horizontal lift of \(\bar\gamma_s\) is conjugate to a subgroup of
  the stabilizers of the endpoints. Since both endpoint stabilizers
  are trivial, every interior stabilizer is trivial. Thus
  \[
    \bar\gamma_s((0,1))\subset \Qss,
    \qquad
    \bar\gamma((0,1))
    =
    \Xks\bigl(\bar\gamma_s((0,1))\bigr)
    \subset \Qks.
  \]
  This proves the no-rank-drop statement.
\end{proof}

We now translate \cref{thm:no-interior-rank-drop-rank-bounded} to the
correlation-matrix side. For \(2\le s\le \min(n,k)\), define the
intrinsic rank-bounded distance on \(\Crle{s}{n}\) by
\[
  \dCorrle{k}{s}(C_0,C_1)
  \coloneq
  \dkles
  \bigl(
    \bar\Gamma_k^{-1}(C_0),
    \bar\Gamma_k^{-1}(C_1)
  \bigr),
  \qquad
  C_0,C_1\in\Crle{s}{n}.
\]
Similarly, define the fixed-rank Riemannian distance on the smooth stratum
\(\Crk{s}{n}\) by
\[
  \dCorr{k}{s}(C_0,C_1)
  \coloneq
  \dQks
  \bigl(
    \bar\Gamma_k^{-1}(C_0),
    \bar\Gamma_k^{-1}(C_1)
  \bigr),
  \qquad
  C_0,C_1\in\Crk{s}{n}.
\]

\begin{corollary}[Correlation-side form]
  \label{cor:no-rank-drop-correlation}
  Let \(2\le s\le \min(n,k)\) and let $C_0,C_1\in\Crk{s}{n}$. Then
  \[
    \dCorrle{k}{s}(C_0,C_1)
    =
    \dCorr{k}{s}(C_0,C_1).
  \]
  Moreover, every $\dCorrle{k}{s}$-minimizing segment between
  \(C_0\) and \(C_1\) has rank \(s\) at all interior times.
\end{corollary}

\begin{proof}
  The descended Gram map identifies \(\Qkles\) with \(\Crle{s}{n}\) and
  \(\Qks\) with \(\Crk{s}{n}\).  The result is exactly
  \cref{thm:no-interior-rank-drop-rank-bounded} transported through
  this identification.
\end{proof}

\begin{remark}[Computing rank-constrained geodesics]
\label{rem:computing-rank-constrained-geodesics}
By \cref{thm:no-interior-rank-drop-rank-bounded},
rank-constrained minimizing geodesics between rank-\(s\) endpoints
remain in \(\Qks\). After the fixed-rank reduction of
\cref{thm:fixed-rank-riemannian-reduction}, they admit the
horizontal-lift representation described in
\cref{subsec:fixed-rank-geodesics-horizontal-lifts}.
Their correlation-matrix representation is given by
\cref{thm:fixed-rank-correlation-geodesic}.
\end{remark}

\section{Geodesics constrained to hit lower-rank loci}
\label{sec:geodesics-constrained-hit-lower-rank}

We next consider a different type of constraint. In
\cref{sec:intrinsic-geodesics-rank-constraints}, curves were forced
to remain inside a rank-bounded set \(\Qkles\). Here curves are
allowed to move in the ambient quotient \((Q_k,\dQ{k})\), but they are
required to meet a prescribed lower-rank locus. These are distinct
variational problems.

\begin{definition}[Admissible curves with lower-rank contact]
\label{def:admissible-contact-curves}
Let \(r\coloneq\min(n,k)\), let \(1\le\ell<r\), and let
\(x_0,x_1\in \Qrk{k}{r}\). The admissible curves joining \(x_0\)
to \(x_1\) and meeting \(\Qle{k}{\ell}\) form the class
\[
  \begin{aligned}
    \mathcal A_\ell(x_0,x_1)\coloneq\Bigl\{&
      \gamma:[0,1]\to Q_k\text{ continuous and }\dQ{k}\text{-rectifiable}:\\
      &\gamma(0)=x_0,\quad \gamma(1)=x_1,\quad
      \gamma([0,1])\cap \Qle{k}{\ell}\ne\varnothing\Bigr\}.
  \end{aligned}
\]
The constrained problem minimizes \(\Length_{\dQ{k}}\) over this
class; curves may otherwise move throughout \(Q_k\).
\end{definition}

\begin{proposition}[One-contact property]
  \label{prop:one-contact-lower-rank-locus}
  Let \(1\le \ell<r\), and let \(x_0,x_1\in \Qrk{k}{r}\). Suppose that
  \(\gamma\in\mathcal A_\ell(x_0,x_1)\) minimizes \(\Length_{\dQ{k}}\) among
  all curves in \(\mathcal A_\ell(x_0,x_1)\). Then there is a
  unique contact point \(p_*\) in the path image:
  \[
    \gamma([0,1])\cap \Qle{k}{\ell}=\{p_*\}.
  \]
  Every other point of the path image lies in \(\Qrk{k}{r}\).  The contact-time
  set \(\gamma^{-1}(p_*)\) may be a nontrivial closed interval if the
  parametrization pauses at \(p_*\); it is a singleton for a constant-speed,
  non-stalling parametrization.
\end{proposition}

\begin{figure}[htbp]
  \centering
  \begin{subfigure}[t]{.48\linewidth}
    \centering
    \includegraphics[width=\linewidth]{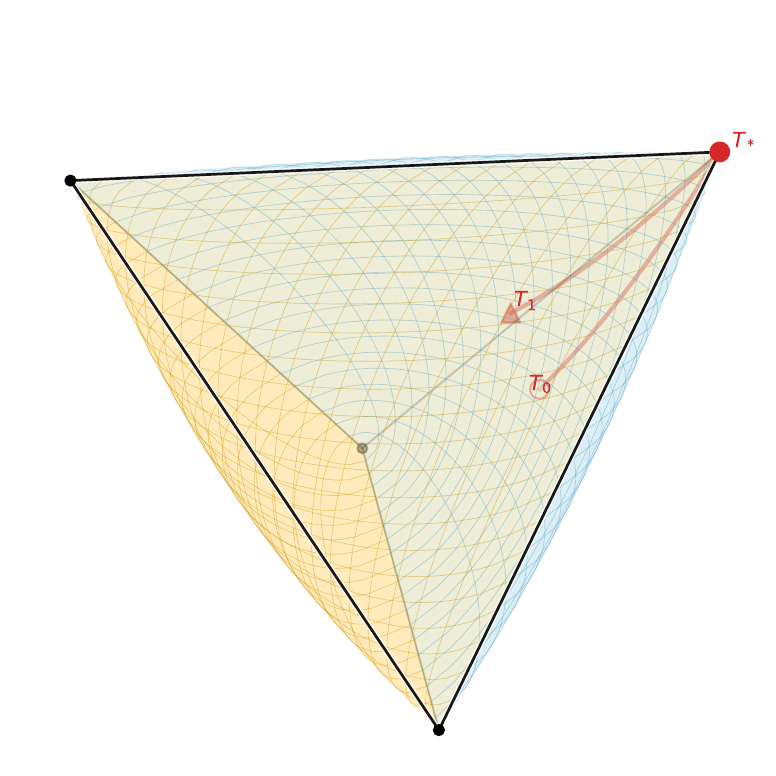}
    \caption{Broken path through the rank-one vertex \(T_*\).}
  \end{subfigure}\hfill
  \begin{subfigure}[t]{.48\linewidth}
    \centering
    \includegraphics[width=\linewidth]{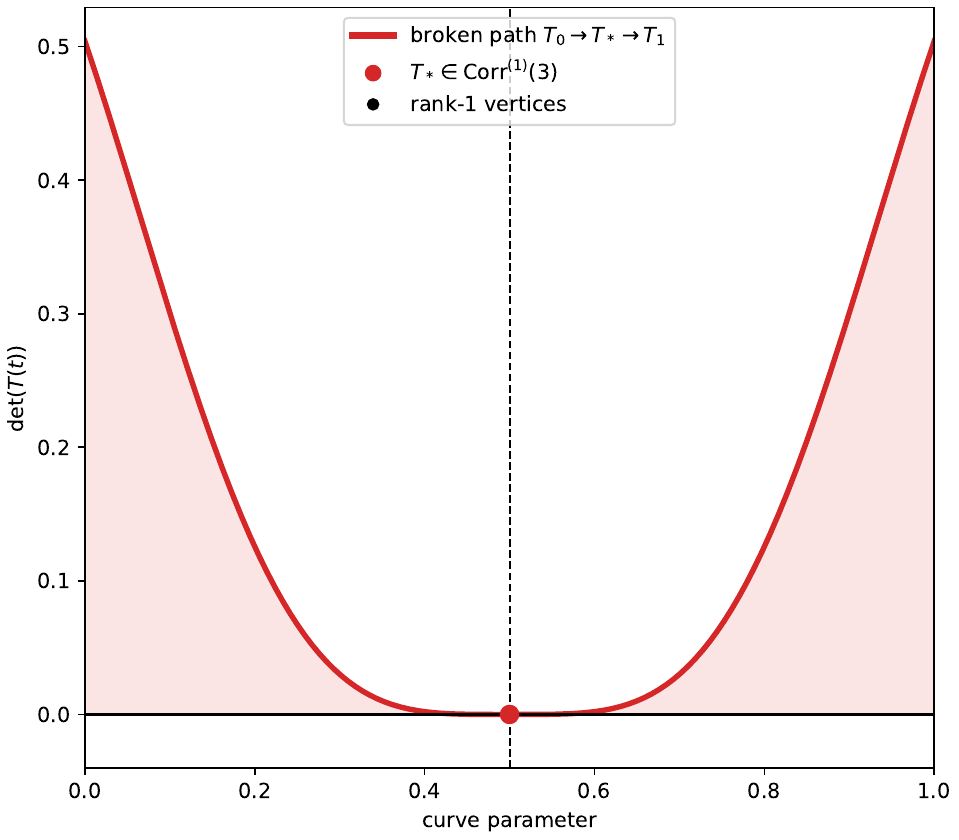}
    \caption{Determinant \(\det T(t)\) along the displayed path.}
  \end{subfigure}
  \caption{Illustration of the single-contact geometry in
    \cref{prop:one-contact-lower-rank-locus}.
    The red curve joins full-rank endpoints \(T_0,T_1\) through
    \(T_*\in\Crk{1}{3}\).
    }
  \label{fig:corr3-rank1-contact}
\end{figure}

\begin{proof}
    Reparametrize the path by arclength after deleting stationary
    subintervals; this preserves its image, length, and admissibility.  Choose
    \(s_*\) with \(\gamma(s_*)\in \Qle{k}{\ell}\), and put
    \(p_*\coloneq\gamma(s_*)\).  Since the endpoints are regular, \(s_*\) is an
    interior parameter.  Each restriction on either side of \(s_*\) is an
    ambient \(\dQ{k}\)-minimizing segment: replacing a nonminimizing restriction
    by a shorter segment would preserve the contact at \(p_*\) and contradict
    constrained minimality.

    Orbit-type monotonicity for minimizing segments
    \cite[Lemma~3.5]{AKLM03}, applied from each regular endpoint, places the
    interiors of both segments in \(\Qrk{k}{r}\).  Hence \(p_*\) is the only
    lower-rank point in the arclength-parametrized image, and therefore in the
    original image.

    Finally, if two original parameters \(t_1<t_2\) map to \(p_*\), the
    restriction between them must have zero length; otherwise deleting that
    loop would shorten an admissible path.  Thus \(\gamma\) is constant on
    \([t_1,t_2]\), so the full preimage of \(p_*\) is a closed interval.  A
    non-stalling constant-speed parametrization cannot contain such a
    nontrivial interval.
\end{proof}

The previous proposition shows that the constrained curve is a broken
ambient geodesic, with a single contact point on the lower-rank locus. The
next statement identifies the contact point by a finite-dimensional
optimization problem on \(\Qle{k}{\ell}\).

\begin{proposition}[Reduction to a contact-point optimization]
  \label{prop:reduction-constrained-hit}
  Let \(1\le \ell<r\), and let \(x_0,x_1\in \Qrk{k}{r}\). Define the functional
  \[
    \Phi_\ell(p)
    \coloneq
    \dQ{k}(x_0,p)+\dQ{k}(p,x_1),
    \qquad
    p\in \Qle{k}{\ell}.
  \]
  Then
  \[
    \inf_{\gamma\in\mathcal A_\ell(x_0,x_1)}
    \Length_{\dQ{k}}(\gamma)
    =
    \min_{p\in \Qle{k}{\ell}}\Phi_\ell(p).
  \]
  Moreover:
  \begin{enumerate}
    \item if \(\gamma\in\mathcal A_\ell(x_0,x_1)\) is minimizing and
      \(p_*\) is its unique contact point, then \(p_*\) minimizes \(\Phi_\ell\) on
      \(\Qle{k}{\ell}\);
    \item conversely, if \(p_*\in \Qle{k}{\ell}\) minimizes
      \(\Phi_\ell\), then the
      concatenation of any ambient minimizing segment from \(x_0\) to
      \(p_*\) with any
      ambient minimizing segment from \(p_*\) to \(x_1\) is a
      minimizing element of
      \(\mathcal A_\ell(x_0,x_1)\).
  \end{enumerate}
\end{proposition}

\begin{proof}
The continuous function \(\Phi_\ell\) attains its minimum on the
compact set \(\Qle{k}{\ell}\).
For any \(\gamma\in\mathcal A_\ell(x_0,x_1)\), choose
\(t\in[0,1]\) such that \(p\coloneq\gamma(t)\in \Qle{k}{\ell}\).
Then
\[
  \begin{aligned}
    \Length_{\dQ{k}}(\gamma)
    &=
    \Length_{\dQ{k}}\bigl(\restr{\gamma}{[0,t]}\bigr)
    +\Length_{\dQ{k}}\bigl(\restr{\gamma}{[t,1]}\bigr)\\
    &\ge \dQ{k}(x_0,p)+\dQ{k}(p,x_1)
    =\Phi_\ell(p)
    \ge\min_{\Qle{k}{\ell}}\Phi_\ell.
  \end{aligned}
\]
Conversely, let \(p_*\) minimize \(\Phi_\ell\).
Since \((Q_k,\dQ{k})\) is geodesic, choose any minimizing segments
from \(x_0\) to \(p_*\) and from \(p_*\) to \(x_1\).
Their concatenation belongs to \(\mathcal A_\ell(x_0,x_1)\)
and has length
\[
  \dQ{k}(x_0,p_*)+\dQ{k}(p_*,x_1)
  =\Phi_\ell(p_*)
  =\min_{\Qle{k}{\ell}}\Phi_\ell.
\]
This proves the equality of optimal values.
The stated characterization of minimizers follows from equality
in the preceding inequalities and the concatenation construction.
\end{proof}

The contact-point formulation reduces the constrained problem to the
minimization of the functional
\[
  \Phi_\ell : p\in \Qle{k}{\ell} \longmapsto
  \dQ{k}(x_0,p)+\dQ{k}(p,x_1) \in \R.
\]
This minimization sits in a stratified space. Nevertheless, once the
minimizing contact point is known to lie in a fixed smooth rank
stratum, the usual Riemannian first-order condition applies along
that stratum. More precisely, suppose that \(p_*\in \Qkj\) for some \(1\le j\le \ell\) is a minimizing contact point. Then \(p_*\) minimizes the restricted function \(\Phi_\ell^{(j)} \coloneq \restr{\Phi_\ell}{\Qkj}\) on the smooth Riemannian manifold \(\Qkj\). Therefore, at every point where the two restricted distance functions
\[
  p\longmapsto \dQ{k}(x_0,p),
  \qquad
  p\longmapsto \dQ{k}(p,x_1)
\]
are differentiable along \(\Qkj\), the differential of
\(\Phi_\ell^{(j)}\) must vanish:
\[
  \dd\Phi_\ell^{(j)}(p_*)[v]=0
  \qquad
  \text{for every }v\in\Tz{p_*}\Qkj.
\]
Equivalently,
\[
  \grad^{(j)} \dQ{k}(x_0,\cdot)
  +
  \grad^{(j)} \dQ{k}(\cdot,x_1)
  =
  0
  \qquad
  \text{at }p_*,
\]
where the Riemannian gradients are taken with respect to the quotient Riemannian metric on the rank-\(j\) stratum \(\Qkj\). Geometrically, this says that a minimizing broken geodesic
\[
  x_0\longrightarrow p_*\longrightarrow x_1
\]
satisfies a stationarity condition with respect to variations of the contact point inside the smooth stratum containing \(p_*\). The tangential components of the incoming and outgoing first-variation vectors cancel along \(\Tz{p_*}\Qkj\). Thus the stratified constraint determines which smooth stratum contains the contact point. On that stratum, the broken length is stationary under all tangent variations of
the contact point.

\begin{theorem}[Density of the top allowed stratum]
\label{thm:top-allowed-rank-determines-value}
Let \(2\le\ell\le r\coloneq\min(n,k)\). Then \(\Qrk{k}{\ell}\) is dense in \(\Qle{k}{\ell}\). Consequently, if \(x_0,x_1\in \Qrk{k}{r}\) with \(\ell<r\), and if \(\Phi_\ell(p) \coloneq \dQ{k}(x_0,p)+\dQ{k}(p,x_1)\), \(p\in \Qle{k}{\ell}\) then
\[
  \min_{p\in \Qle{k}{\ell}}\Phi_\ell(p)
  =
  \inf_{p\in \Qrk{k}{\ell}}\Phi_\ell(p).
\]
\end{theorem}

\begin{proof}
By \cref{thm:global-rank-bounded-metric-reduction}, the map
\[
  \padXi{\ell}{k}:
  (Q_\ell,\dQ{\ell})
  \longrightarrow
  (\Qle{k}{\ell},\dle{\mathrm{int},k}{\ell})
\]
is an isometry, in particular a homeomorphism, and it restricts to a homeomorphism $\Qrk{\ell}{\ell}\longrightarrow \Qrk{k}{\ell}$. Now \(\Mrk{\ell}{\ell}\) is open and dense in
\(M_\ell=(\mathbb S^{\ell-1})^n\). Since the quotient map $\pi_\ell:M_\ell\to Q_\ell$ is continuous and surjective, its image \(\Qrk{\ell}{\ell}=\pi_\ell(\Mrk{\ell}{\ell})\) is dense in \(Q_\ell\). Applying the homeomorphism \(\padXi{\ell}{k}\), we obtain
that \(\Qrk{k}{\ell}\) is dense in \(\Qle{k}{\ell}\). Finally, \(\Phi_\ell\) is continuous on the compact set \(\Qle{k}{\ell}\). Hence it
attains its minimum, and density of \(\Qrk{k}{\ell}\subset \Qle{k}{\ell}\) gives
\[
  \min_{p\in \Qle{k}{\ell}}\Phi_\ell(p)
  =
  \inf_{p\in \Qrk{k}{\ell}}\Phi_\ell(p).
\]
\end{proof}

\begin{remark}[Rank-adaptive interpretation]
\label{rem:rank-adaptive-hit-interpretation}
Assume \(2\le\ell<r\). The minimum of \(\Phi_\ell\)
may be attained on a lower-rank stratum \(\Qkj\),
with \(j<\ell\). By
\cref{thm:top-allowed-rank-determines-value}, its value
can nevertheless be obtained by minimizing
\[
  F_\ell(\bar z)
  \coloneq
  \dQ{k}\bigl(x_0,\padXi{\ell}{k}(\bar z)\bigr)
  +
  \dQ{k}\bigl(\padXi{\ell}{k}(\bar z),x_1\bigr),
  \qquad
  \bar z\in \Qrk{\ell}{\ell}.
\]
This formula extends continuously to \(Q_\ell\).
Indeed, the padding map $\padXi{\ell}{k}:(Q_\ell,\dQ{\ell})\longrightarrow(Q_k,\dQ{k})$ is \(1\)-Lipschitz, so $F_\ell$ is
\(2\)-Lipschitz with respect to \(\dQ{\ell}\). The smoothness of the stratum does not by itself imply
smoothness of \(F_\ell\). On any open subset of
\(\Qrk{\ell}{\ell}\) where both restricted distance terms
are smooth, \(F_\ell\) is smooth. At a minimizing point
\(\bar z_*\) where \(F_\ell\) is differentiable, the
usual first-order condition holds:
\[
  \dd F_\ell(\bar z_*)[v]=0
  \qquad
  \text{for every }
  v\in \Tz{\bar z_*}\Qrk{\ell}{\ell}.
\]
Let \((\bar z_m)\) be a minimizing sequence for \(F_\ell\),
and choose representatives \(Z_m\in \Mrk{\ell}{\ell}\).
If there exists $\varepsilon>0$ such that
  $\lambda_{\min}(Z_m^\top Z_m)\ge\varepsilon$ for all $m$,
  then compactness of \(M_\ell\)
gives a convergent subsequence whose limit still has
rank \(\ell\). By continuity, its image under
\(\padXi{\ell}{k}\) is a minimizing rank-\(\ell\)
contact point. Otherwise, a subsequence satisfies $\lambda_{\min}(Z_m^\top Z_m)\longrightarrow0$. After passing to a further subsequence, compactness
gives a limit of rank strictly less than \(\ell\).
Continuity of the extended objective then shows that
its image is a minimizing contact point on the
lower-rank boundary.
\end{remark}

\begin{corollary}[Correlation-side constrained hitting problem]
  \label{cor:correlation-constrained-hit}
  The statements of
  \cref{prop:one-contact-lower-rank-locus,prop:reduction-constrained-hit,thm:top-allowed-rank-determines-value}
  transport verbatim to the rank strata of correlation matrices
  through the descended
  Gram identification
  \[
    \bar\Gamma_k:Q_k\longrightarrow \Crle{k}{n}.
  \]
  Thus, for two full-rank points in \(\Crk{r}{n}\), a shortest
  ambient path
  constrained to hit \(\Crle{\ell}{n}\) has a unique contact point with
  that lower-rank
  locus, and the optimal contact value is already determined by the
  top allowed
  stratum \(\Crk{\ell}{n}\).
\end{corollary}

\section{Excess length of rank-constrained geodesics}
\label{sec:excess-rank-constrained-geodesics}

Throughout this section, let \(2\le s\le\min(n,k)\).
We quantify the excess length caused by requiring curves to
remain in \(\Qkles\).

\begin{definition}[Excess of the rank constraint]
\label{def:rank-constraint-excess}
For \(\bar p,\bar q\in\Qks\), define
\[
  \Delta_s(\bar p,\bar q)
  \coloneq
  \dkles(\bar p,\bar q)-\dQ{k}(\bar p,\bar q)
  \ge0.
\]
\end{definition}

Set
\[
  \bar p_s\coloneq\Xks^{-1}(\bar p),
  \qquad
  \bar q_s\coloneq\Xks^{-1}(\bar q).
\]
The global rank-bounded reduction
\cref{thm:global-rank-bounded-metric-reduction} gives
\begin{equation}
\label{eq:excess-reduced-model}
  \Delta_s(\bar p,\bar q)
  =
  \dQ{s}(\bar p_s,\bar q_s)-\dQ{k}(\bar p,\bar q).
\end{equation}
The estimates below use this identity, the Procrustes comparison,
and chord--arc inequalities on the sphere.

\subsection{Chordal quotient comparison}
\label{subsec:excess-chordal-comparison}

Recall the chordal quotient distance introduced in
\cref{lem:procrustes-chordal}
\[
  \dChord{s}([A],[B])
  \coloneq
  \min_{R\in O(s)}\|A-BR\|_F,
  \qquad A,B\in M_s.
\]
By the orthogonal Procrustes formula, $  \dChord{s}([A],[B])^2 =
2n-2\|A^\top B\|_*$, where \(\|\cdot\|_*\) is the nuclear norm. We
also use the following two consequences already proved in
\cref{sec:quotient-metric-sylvester-connection}. First, the
Procrustes formula is compatible with padding; indeed this is the
padding identity used in the proof of
\cref{lem:chordal-comparison-padded-factors}.  Thus, for \(Z_0,Z_1\in M_s\),
\begin{equation}
  \label{eq:chordal-padding-invariance-excess}
  \dChord{k}(\Xks([Z_0]),\Xks([Z_1]))
  =
  \dChord{s}([Z_0],[Z_1]).
\end{equation}
Second, \cref{lem:chordal-comparison-padded-factors} gives
\begin{equation}
  \label{eq:chordal-dominated-by-orbit-excess}
  \dChord{s}([Z_0],[Z_1])
  \le
  \dQ{k}(\Xks([Z_0]),\Xks([Z_1])).
\end{equation}

\subsection{Arcsine and cubic bounds}
\label{subsec:universal-cubic-excess-bound}

We first record the exact chord--arc identity upstairs. The resulting
arcsine bound immediately yields the local cubic estimate.

\begin{lemma}[Exact product chord--arc formula]
  \label{lem:exact-product-chord-arc-excess}
  Let \(X,Y\in M_s=(\mathbb S^{s-1})^n\). Write their rows as \(x_i\coloneq X_{i\cdot}\), \(y_i\coloneq Y_{i\cdot}\) and set \(c_i\coloneq\|x_i-y_i\|\), \(\theta_i\coloneq\dSphere{s-1}(x_i,y_i)\). Then
  \[
    \theta_i=2\arcsin(c_i/2),
    \quad \textnormal{and hence} \quad 
    \dM{s}(X,Y)=\left(\sum_{i=1}^n \bigl(2\arcsin(c_i/2)\bigr)^2\right)^{1/2}.
  \]
\end{lemma}

\begin{proof}
  On a unit sphere, \(c_i=2\sin(\theta_i/2)\) with
  \(0\le\theta_i\le\pi\), giving the first identity. The product
  metric on \(M_s\) is the \(\ell^2\)-product of the rowwise
  spherical metrics, giving the second identity.
\end{proof}

\begin{lemma}[Sharp radial chord--arc bound on the product]
  \label{lem:sharp-product-arcsine-bound-excess}
  Let \(X,Y\in M_s\), and set \(C\coloneq\|X-Y\|_F\). If \(0\le C\le2\), then \(\dM{s}(X,Y) \le 2\arcsin(C/2)\). This bound is sharp as a function of \(C\) on the product space: equality holds whenever all nonzero chordal displacement is contained in a single row.
\end{lemma}

\begin{proof}
  Let \(f(t)\coloneq2\arcsin(t/2)\).  For \(t>0\), the quotient \(f(t)/t\)
  is increasing
  on \((0,2]\).  Indeed, with \(t=2\sin u\),
  \[
    \frac{\dd}{\dd t}\left(\frac{f(t)}{t}\right)
    =
    \frac{2(\tan u-u)}{t^2}
    \ge0.
  \]
  Since \(c_i\le C\), \cref{lem:exact-product-chord-arc-excess} gives
  \[
    \dM{s}(X,Y)^2
    =
    \sum_i f(c_i)^2
    \le
    \left(\frac{f(C)}{C}\right)^2\sum_i c_i^2
    =
    f(C)^2,
  \]
  with the case \(C=0\) being trivial. The equality follows by taking
  exactly one nonzero \(c_i\).
\end{proof}

The equality case is already present on one sphere: choose
\(\|x-y\|=C\) in one row and keep every other row fixed. For \(C=1\),
this gives \(\dSphere{s-1}(x,y)=\pi/3\), hence equality in the
cubic constant \(\pi/3-1\).

\begin{theorem}[Arcsine bound for quotient excess]
  \label{thm:arcsine-bound-rank-excess}
  Let \(\bar p,\bar q\in\Qks\), and suppose $\dChord{k}(\bar p,\bar q)\le2$. Then
  \[
    0\le
    \Delta_s(\bar p,\bar q)
    \le
    2\arcsin\!\left(\frac{\dChord{k}(\bar p,\bar q)}2\right)
    -\dQ{k}(\bar p,\bar q).
  \]
  In particular,
  \[
    \Delta_s(\bar p,\bar q)
    \le
    2\arcsin\!\left(\frac{\dChord{k}(\bar p,\bar q)}2\right)
    -\dChord{k}(\bar p,\bar q).
  \]
  If \(\dChord{k}(\bar p,\bar q)\le1\), this gives the cubic envelope
  inherited from the chord--arc inequality
  \[
    \Delta_s(\bar p,\bar q)
    \le
    \left(\frac\pi3-1\right)\dChord{k}(\bar p,\bar q)^3.
  \]
\end{theorem}

\begin{proof}
  Let \(\bar p_s,\bar q_s\in\Qss\) be the reduced endpoints. By
  \cref{eq:excess-reduced-model}, \(\Delta_s(\bar p,\bar q)=\dQ{s}(\bar p_s,\bar q_s)-\dQ{k}(\bar p,\bar q)\). Choose a Procrustes minimizer \(R_*\in O(s)\) for representatives \(Z_p,Z_q\in M_s\), so that \(\|Z_p-Z_qR_*\|_F=\dChord{s}(\bar p_s,\bar q_s)\). The quotient distance satisfies \(\dQ{s}(\bar p_s,\bar q_s) \le \dM{s}(Z_p,Z_qR_*)\). By \cref{lem:sharp-product-arcsine-bound-excess} and the padding identity
  \cref{eq:chordal-padding-invariance-excess},
  \[
    \dQ{s}(\bar p_s,\bar q_s)
    \le
    2\arcsin\!\left(\frac{\dChord{s}(\bar p_s,\bar q_s)}2\right)
    =
    2\arcsin\!\left(\frac{\dChord{k}(\bar p,\bar q)}2\right).
  \]
  This proves the first upper bound. The second follows from
  \cref{eq:chordal-dominated-by-orbit-excess}. The cubic estimate
  follows from the one-dimensional bound on \([0,1]\), with equality at \(c=1\).
\end{proof}

\begin{corollary}[Cubic bound for rank-constrained excess]
  \label{thm:universal-cubic-rank-excess}
  Let $\bar p,\bar q\in\Qks$. If $\dQ{k}(\bar p,\bar q)\le1$, then
  \[
    0\le\Delta_s(\bar p,\bar q)
    \le\left(\frac\pi3-1\right)\dChord{k}(\bar p,\bar q)^3
    \le\left(\frac\pi3-1\right)\dQ{k}(\bar p,\bar q)^3.
  \]
\end{corollary}

\begin{proof}
By \cref{eq:chordal-padding-invariance-excess,eq:chordal-dominated-by-orbit-excess},
$\dChord{k}(\bar p,\bar q)\le \dQ{k}(\bar p,\bar q)\le1$.
The cubic conclusion of \cref{thm:arcsine-bound-rank-excess} gives the
first bound; the same comparison gives the second. Nonnegativity follows
from \cref{def:rank-constraint-excess}.
\end{proof}

\begin{corollary}[Procrustes rowwise sharpening]
\label{cor:procrustes-rowwise-sharpening-excess}
Let \(2\le s\le\min(n,k)\), let \(\bar p,\bar q\in\Qks\),
and let \(\bar p_s,\bar q_s\in\Qss\) be their reduced
endpoints. Choose representatives \(A,B\in\Mss\), with rows
\(a_i,b_i\), and define
\[
  \mathcal R_{\mathrm{Proc}}(A,B)
  \coloneq
  \argmin_{R\in O(s)}\|A-BR\|_F
\]
and
\[
  \dProc{s}(\bar p_s,\bar q_s)
  \coloneq
  \min_{R\in\mathcal R_{\mathrm{Proc}}(A,B)}
  \left(
    \sum_{i=1}^n
    \left[
      2\arcsin\!\left(\frac{\|a_i-b_iR\|}{2}\right)
    \right]^2
  \right)^{1/2}.
\]
Then, without any restriction on
\(\dChord{k}(\bar p,\bar q)\),
\begin{equation}
\label{eq:procrustes-rowwise-distance-bound-excess}
  \dQ{s}(\bar p_s,\bar q_s)
  \le
  \dProc{s}(\bar p_s,\bar q_s),
\end{equation}
and consequently
\[
  0\le\Delta_s(\bar p,\bar q)
  \le
  \dProc{s}(\bar p_s,\bar q_s)
  -\dQ{k}(\bar p,\bar q).
\]
If, in addition, \(\dChord{k}(\bar p,\bar q)\le2\), then
\[
  \dProc{s}(\bar p_s,\bar q_s)
  \le
  2\arcsin\!\left(\frac{\dChord{k}(\bar p,\bar q)}{2}\right).
\]
\end{corollary}

\begin{proof}
The set \(\mathcal R_{\mathrm{Proc}}(A,B)\) is nonempty
and compact, since \(O(s)\) is compact and the Frobenius
distance is continuous. Moreover, each row has unit norm,
so \(\|a_i-b_iR\|\le2\) for every \(R\in O(s)\).
Thus the rowwise expression is defined globally, and
the sphere chord--arc identity gives
\[
  \dProc{s}(\bar p_s,\bar q_s)
  =
  \min_{R\in\mathcal R_{\mathrm{Proc}}(A,B)}
  \dM{s}(A,BR).
\]
Orthogonal invariance shows that this value is independent
of the chosen representatives. Since
\[
  \dQ{s}(\bar p_s,\bar q_s)
  =
  \min_{R\in O(s)}\dM{s}(A,BR),
\]
minimizing over the smaller set
\(\mathcal R_{\mathrm{Proc}}(A,B)\) proves the first bound.
The excess estimate follows from
{\cref{eq:excess-reduced-model}}. Now assume \(\dChord{k}(\bar p,\bar q)\le2\).
For every \(R\in\mathcal R_{\mathrm{Proc}}(A,B)\),
padding invariance of the chordal quotient distance gives $\|A-BR\|_F=\dChord{s}(\bar p_s,\bar q_s)=\dChord{k}(\bar p,\bar q)$. Hence \cref{lem:sharp-product-arcsine-bound-excess} yields
\[
  \dM{s}(A,BR)
  \le
  2\arcsin\!\left(\frac{\dChord{k}(\bar p,\bar q)}{2}\right).
\]
Taking the minimum over the Procrustes minimizers proves
the bound.
\end{proof}

\begin{remark}[Sharpness]
  For \(n=3\) and \(s=2\), let \(A,B\) have rows
  \((\cos\alpha_i,\sin\alpha_i)\) and
  \((\cos\beta_i,\sin\beta_i)\), respectively, where
  \(\alpha=(-\pi/3,0,\pi/3)\) and
  \(\beta=(-5\pi/6,0,5\pi/6)\).
  This parametrization gives
  \(\dQ{2}([A],[B])=\dProc{2}([A],[B])=\pi/\sqrt{2}\),
  so the bound in \cref{eq:procrustes-rowwise-distance-bound-excess}
  is sharp.
  Additionally, set
  \[
    \widetilde A_{i\cdot}=(\cos\alpha_i,\sin\alpha_i,0),
    \qquad
    \widetilde B_{\varepsilon,i\cdot}
      =(\cos\beta_i\cos\varepsilon,\sin\beta_i,
        \cos\beta_i\sin\varepsilon).
  \]
  For small \(\varepsilon\), the squared product-sphere distance
  \(F(\varepsilon)\coloneq
    \dM{3}(\widetilde A,\widetilde B_\varepsilon)^2\) satisfies
  \[
    F(\varepsilon)
    =2\arccos^2\!\left(\frac{\sqrt3}{4}(1-\cos\varepsilon)\right)
      +\varepsilon^2
    =\frac{\pi^2}{2}
      +\left(1-\frac{\pi\sqrt3}{4}\right)\varepsilon^2
      +O(\varepsilon^4).
  \]
  Since the quadratic coefficient is negative, the padded endpoints
  have \(\dQ{3}<\pi/\sqrt2=\dQ{2}=\dProc{2}\), as illustrated in
  \cref{sfig:corr3-quotient-small-shortcut}.
\end{remark}

\begin{figure}[htbp]
  \centering
  \begin{subfigure}[t]{.48\linewidth}
    \centering
    \includegraphics[width=\linewidth]
      {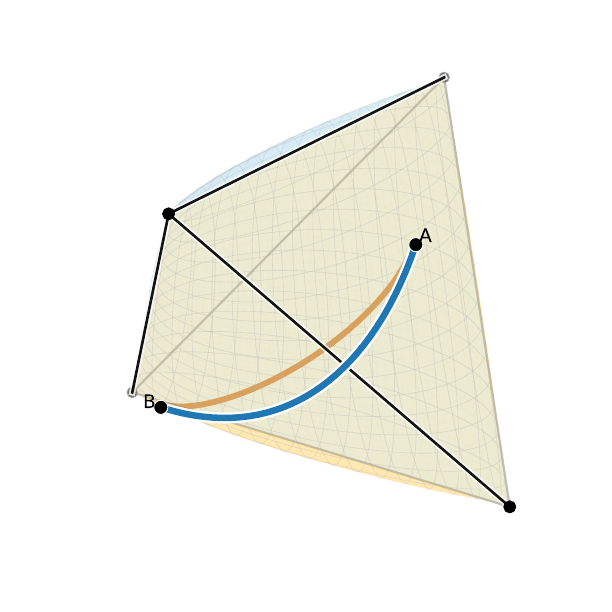}
    \caption{\(\dQ{3}<\dQ{2}=\dProc{2}\)}
    \label{sfig:corr3-quotient-small-shortcut}
  \end{subfigure}\hfill
  \begin{subfigure}[t]{.48\linewidth}
    \centering
    \includegraphics[width=\linewidth]
      {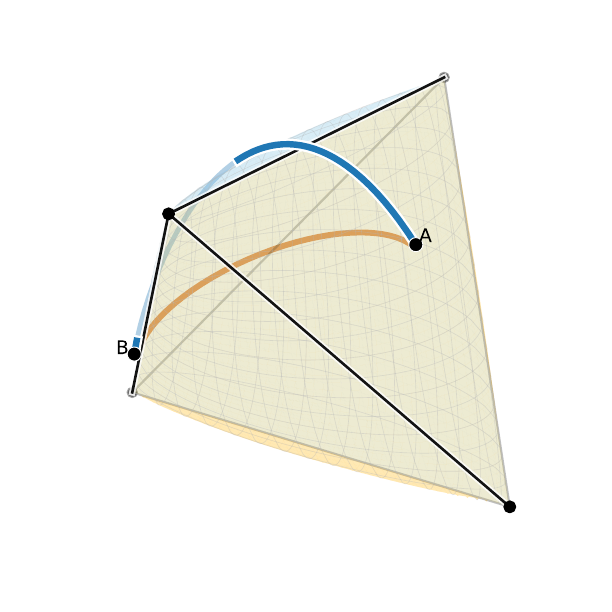}
    \caption{\(\dQ{3}<\dQ{2}<\dProc{2}\)}
    \label{sfig:corr3-quotient-large-shortcut}
  \end{subfigure}
  \par\medskip
  \definecolor{shortcutboundary}{HTML}{1F77B4}
  \definecolor{shortcutambient}{HTML}{D95F02}
  {\small
    \tikz[baseline=-.5ex]{\draw[shortcutboundary,line width=1.4pt] (0,0)--(.7,0);}
    \(\dQ{2}\)\quad
    \tikz[baseline=-.5ex]{\draw[shortcutambient,line width=1pt] (0,0)--(.7,0);}
    \(\dQ{3}\)\par}
  \caption{Two quotient-excess examples in \(\Corr(3)\).}
  \label{fig:corr3-quotient-small-excess-example}
\end{figure}

\subsection{A global bound}
\label{subsec:global-excess-bound}

The cubic bound is local in the sense that it is stated for points
with ambient distance less than one. Globally, compactness gives a
coarse but useful uniform bound based on the diameter of the product of spheres.

\begin{theorem}[Global uniform bound for rank-constrained excess]
  \label{thm:global-uniform-rank-excess-bound}
  For every \(\bar p,\bar q\in \Qks\),
  \[
    0\le
    \Delta_s(\bar p,\bar q)
    \le
    \pi\sqrt n.
  \]
\end{theorem}

\begin{proof}
Set \(\bar p_s\coloneq\Xks^{-1}(\bar p)\) and
\(\bar q_s\coloneq\Xks^{-1}(\bar q)\).  Then
\cref{eq:excess-reduced-model,prop:reduced-quotient-global-metric-geometry}
gives
\[
  0
  \le
  \Delta_s(\bar p,\bar q)
  \le
  \dkles(\bar p,\bar q)
  =
  \dQ{s}(\bar p_s,\bar q_s)
  \le
  \diam(Q_s)
  \le
  \pi\sqrt n.
\]
  \end{proof}

The same estimate may be read directly on correlation matrices. Let
\(C_0,C_1\in\Crk{s}{n}\). We denote by
\[
  \bar\Gamma_k:Q_k\longrightarrow \Crle{k}{n},
  \qquad
  \bar\Gamma_k([X])=XX^\top,
\]
the descended Gram identification, and by
\[
  \bar\Gamma_k^{\le s}:\Qkles\longrightarrow \Crle{s}{n}
\]
its restriction to the rank-bounded locus.  We retain the notation
\[
  \dQ{k}(C_0,C_1)
  \coloneq
  \dQ{k}\bigl(
    \bar\Gamma_k^{-1}(C_0),
    \bar\Gamma_k^{-1}(C_1)
  \bigr),
\]
for the transported ambient correlation distance, and write the intrinsic rank-bounded correlation distance as
\[
  \dCorrle{k}{s}(C_0,C_1)
  \coloneq
  \dkles\bigl(
    (\bar\Gamma_k^{\le s})^{-1}(C_0),
    (\bar\Gamma_k^{\le s})^{-1}(C_1)
  \bigr).
\]
Define the excess
\[
  \Delta_s^{\Corr}(C_0,C_1)
  \coloneq
  \dCorrle{k}{s}(C_0,C_1)-\dQ{k}(C_0,C_1) \ge0.
\]
If $\dQ{k}(C_0,C_1)\le1$, then
\[
  0\le
  \Delta_s^{\Corr}(C_0,C_1)
  \le
  {\left(\frac\pi3-1\right)}\dQ{k}(C_0,C_1)^3.
\]
Indeed, this is exactly \cref{thm:universal-cubic-rank-excess}
transported through the Gram identification of \(\Qks\) with the
rank-$s$ locus \(\Crk{s}{n}\).

\section{Volume measures and uniform laws}
\label{sec:volume-measures-correlation-strata}

The purpose of this section is to identify uniform probability laws on
fixed- and bounded-rank correlation strata and to express them as Gram pushforwards of
explicit probability densities on the factor space. Uniformity depends on
the chosen volume measure: the quotient metric and the metric induced by
off-diagonal Euclidean coordinates generally give different laws.
Sampling independent uniform rows on a sphere provides a natural starting
point, but its Gram pushforward need not be uniform for either metric.
Indeed, a region of the quotient receives mass from all the factors in
the orthogonal-group orbits above it, and these orbits can have different
volumes.

The coarea formula makes this dependence precise. With the quotient metric
on the Gram target, the normal Jacobian is one, and the pushforward of
factor-space volume has density equal to the volume of the fiber.
Weighting the factors by the reciprocal of that volume therefore produces
quotient-uniform Gram matrices. To obtain the ambient-uniform law, one
must also account for the change from quotient to ambient volume. This
second correction is the Jacobian of the Gram diffeomorphism on the
quotient. We compute it using horizontal lifts, first without the unit-row
constraints and then on the corresponding unit-diagonal level sets.

Throughout this section, let \(2\le s\le r=\min(n,k)\). By
\cref{thm:fixed-rank-riemannian-reduction}, the rank-\(s\) geometry can be
computed in the reduced model
\[
  \Mss\xrightarrow{\ \pi_s\ }\Qss
  \xrightarrow{\ \bar\Gamma_s\ }\Crk{s}{n},
  \qquad X\longmapsto[X]\longmapsto C=XX^\top,
\]
where
\[
  M_s=(\mathbb S^{s-1})^n,\qquad
  \Mss=\{X\in M_s:\rank X=s\},\qquad
  \Qss=\Mss/O(s).
\]
We retain the dimension notation used in previous sections:
\[
  m_s\coloneq \dim\Qss=\dim\Crk{s}{n}
      =n(s-1)-\binom s2.
\]
As in the connection and curvature formulas, write
\[
  E_X\coloneq X^\top X\in\SPD(s),\qquad
  x_i\coloneq X_{i\cdot}^{\top},\qquad
  P_C\coloneq XE_X^{-1}X^\top=CC^\dagger.
\]
Here \(P_C\) is the orthogonal projector onto \(\im C\), and
\(C^\dagger\) denotes the Moore--Penrose inverse.
For the spectral estimates in this section, we list the \(s\) positive
eigenvalues in \emph{nonincreasing} order:
\begin{equation}
\label{eq:volume-section-eigenvalue-convention}
  \lambda_1(C)\ge\cdots\ge\lambda_s(C)>0,
  \qquad \lambda_a(C)=\lambda_a(E_X),\quad 1\le a\le s.
\end{equation}
This ordering convention is local to this section. Repeated eigenvalues
are allowed; \(\lambda_s(C)\) is the smallest positive eigenvalue, even
when \(s<n\) and the remaining \(n-s\) eigenvalues of \(C\) vanish.

The terms \emph{source} and \emph{target} always refer to the domain and
codomain of the map under consideration, equipped with the stated metrics.
Thus \(\Gamma_s=\bar\Gamma_s\circ\pi_s\) has factor-space source
\(\Mss\), whereas \(\bar\Gamma_s\) has quotient source \(\Qss\);
both have Gram-matrix target \(\Crk{s}{n}\). We also use these symbols
on \(M_s\) and \(Q_s=M_s/O(s)\), respectively, with target
\(\Crle{s}{n}\). The constraint maps introduced below have their
own domains and take values in Euclidean space.

The argument is organized as follows.
\Cref{subsec:three-volume-forms-correlation-strata} introduces the
additional ambient metrics and volume measures, including their
Hausdorff interpretation on the stratified closure.
\Cref{subsec:ambient-volume-relation-quotient-density} recalls the normal
Jacobian and coarea formula, computes orbit volumes through the Sylvester
operator, and establishes a general rule for transferring weighted
factor measures to the quotient.
\Cref{subsec:exact-volume-comparison} computes the exact density of
ambient volume with respect to quotient volume, proves finiteness of
the total volumes, and describes the density near lower rank.
Finally, \cref{subsec:volume-section-uniform-laws-lifts} combines these
results into normalized factor densities for the two uniform laws and
characterizes when independent uniform sphere rows already give the
quotient-uniform law. The rank-one discrete law is recalled at the end.

\subsection{Quotient and ambient volume measures}
\label{subsec:three-volume-forms-correlation-strata}

Recall the product metric $g_{M_s}$ and the quotient metric
$\gQss$ from \cref{eq:product-sphere-metric,eq:fixed-rank-quotient-metric}.
The horizontal spaces, computed in
\cref{eq:tangent-space-product-spheres,eq:vertical-space-reduced-model,lem:horizontality-criterion},
are
\begin{equation}
\label{eq:volume-section-horizontal-space}
  \begin{aligned}
    \Tz{X}\Mss
      &=\{H\in\R^{n\times s}:\diag(HX^\top)=0\},\\
    \Vz{X}&=\{X\Omega:\Omega\in\so(s)\},\\
    \Hz{X}&=\{H\in\Tz{X}\Mss:X^\top H=H^\top X\}.
  \end{aligned}
\end{equation}
Transporting $\gQss$ through the Gram diffeomorphism
{(\cref{prop:quotient-model-low-rank-correlations})}
defines a metric
$g_Q$ on $\Crk{s}{n}$ by $\bar\Gamma_s^*g_Q=\gQss$.
For $C=XX^\top$, if $H_U,H_V\in\Hz{X}$ project to tangent vectors
$U,V\in\Tz{C}\Crk{s}{n}$ under
$\dd\Gamma_s(X)[H]=XH^\top+HX^\top$, then
\begin{equation}
\label{eq:volume-section-quotient-metric}
  g_{Q,C}(U,V)=\tr(H_U^\top H_V).
\end{equation}
Thus $\Gamma_s:(\Mss,g_{M_s})\to(\Crk{s}{n},g_Q)$ is a Riemannian
submersion. The Riemannian distance for $g_Q$ is the transport of
$\dQss=\dQ{s}$ on $\Qss$; by the reduction theorems it also equals
$\dQks=\dkles$ on rank-$s$ endpoints in $Q_k$
{(\cref{eq:rank-bounded-distance-on-top-stratum})}.

The two ambient Riemannian metrics on \(\Crk{s}{n}\) are
\begin{equation}
\label{eq:volume-section-ambient-metrics}
  g_{\mathrm E,C}(U,V)\coloneq \sum_{i<j}U_{ij}V_{ij},
  \qquad
  g_{\mathrm F,C}(U,V)\coloneq \tr(UV).
\end{equation}
The first is induced by the embedding
\(C\mapsto(C_{ij})_{i<j}\) into \(\R^{\binom n2}\); the second is induced
by the Frobenius inner product on \(\Sym(n)\). We write
\(\|U\|_{\mathrm E}^2\coloneq \sum_{i<j}U_{ij}^2\).
The ambient distances \(\|C-C'\|_{\mathrm E}\) and
\(\|C-C'\|_F\) measure chords, whereas the corresponding Riemannian
distances minimize lengths of curves remaining in the stratum. These
distances need not agree when the stratum is not convex. Moreover,
\(g_Q\) uses Frobenius inner products of horizontal \emph{factor}
variations, while \(g_{\mathrm F}\) uses Frobenius inner products of
\emph{Gram-matrix} variations.

\begin{definition}[Riemannian volume measures]
\label{def:volume-measures}
We use Riemannian volume \emph{measures}, which require no choice of
orientation. 
For a Riemannian metric \(g\) on an \(m\)-dimensional
manifold, its volume density in local coordinates \(u^1,\ldots,u^m\) is
\begin{equation}
\label{eq:volume-section-riemannian-density}
  \dd\Vol_g
  \coloneq \sqrt{\det\bigl(g(\partial_{u^\alpha},\partial_{u^\beta})\bigr)_{
      \alpha,\beta=1}^m}\,|\dd u^1\cdots\dd u^m|.
\end{equation}
We denote the measures associated with \(g_Q,g_{\mathrm E},g_{\mathrm F}\)
by \(\Vol_Q,\Vol_{\mathrm E},
\Vol_{\mathrm F}\), respectively. The measure
\(\Vol_{M_s}\) is that of \(g_{M_s}\).
\end{definition}

Every tangent
matrix to the correlation stratum is symmetric with zero diagonal, so
\begin{equation}
\label{eq:volume-section-frobenius-scaling}
  g_{\mathrm F}=2g_{\mathrm E},\qquad
  \dd\Vol_{\mathrm F}
  =2^{m_s/2}\dd\Vol_{\mathrm E}.
\end{equation}
Indeed, their metric matrices in any tangent basis differ by a factor
of \(2\), and their determinants by \(2^{m_s}\). The ambient volume
measures therefore give the same law after normalization.

\begin{definition}[Normalized volume laws]
\label{def:normalized-volume-laws}
Anticipating the finiteness and positivity proved in
\cref{lem:finite-volume-comparison-section}, define
\begin{equation}
\label{eq:volume-section-normalized-laws}
  \mu_{n,s}^{Q}\coloneq 
  \frac{\Vol_Q}{\Vol_Q(\Crk{s}{n})},
  \qquad
  \mu_{n,s}^{\mathrm E}\coloneq 
  \frac{\Vol_{\mathrm E}}
       {\Vol_{\mathrm E}(\Crk{s}{n})}.
\end{equation}
A law called \emph{uniform} on the stratum will always mean one of
these normalized volume measures. We refer to \(\mu_{n,s}^{Q}\) as the
quotient-uniform law and to \(\mu_{n,s}^{\mathrm E}\) as the
ambient-uniform law; the latter also equals normalized Frobenius volume.
\end{definition}

\begin{example}[Uniform laws on the two-by-two correlation stratum]
\label{ex:volume-section-two-by-two}
For \(n=s=2\), put
\[
  C(c)\coloneq \begin{pmatrix}1&c\\c&1\end{pmatrix},
  \qquad -1<c<1,\qquad c=\cos\theta,\quad 0<\theta<\pi.
\]
A factor curve satisfying \(X(\theta)X(\theta)^\top=C(\cos\theta)\) is
\[
  X(\theta)=
  \begin{pmatrix}
    \cos(\theta/2)&\sin(\theta/2)\\
    \cos(\theta/2)&-\sin(\theta/2)
  \end{pmatrix}.
\]
Writing \(a\coloneq \cos(\theta/2)\), \(b\coloneq \sin(\theta/2)\), one obtains
\[
  X'(\theta)=\frac12\begin{pmatrix}-b&a\\-b&-a\end{pmatrix},
  \qquad
  X(\theta)^\top X'(\theta)
     =\begin{pmatrix}-ab&0\\0&ab\end{pmatrix},
  \qquad \|X'(\theta)\|_F^2=\frac12.
\]
Since \(X^\top X'\) is symmetric, \(X'(\theta)\) is the horizontal
lift of the coordinate tangent vector \(\partial_\theta\).
The definition of the quotient metric therefore gives
\[
  g_Q(\partial_\theta,\partial_\theta)
  =\|X'(\theta)\|_F^2=\frac12.
\]
On this one-dimensional stratum, this determines
\(g_Q=\frac12(\dd\theta)^2\), where
\((\dd\theta)^2\coloneq \dd\theta\otimes\dd\theta\).

For the ambient metrics, the coordinate tangent vector
\[
  \partial_c=\frac{\dd C}{\dd c}
  =\begin{pmatrix}0&1\\1&0\end{pmatrix}
\]
satisfies
\(g_{\mathrm E}(\partial_c,\partial_c)=1\) and
\(g_{\mathrm F}(\partial_c,\partial_c)=2\).
Using \(\dd c=-\sin\theta\,\dd\theta\) and
\(\sin^2\theta=1-c^2\), we consequently obtain
\[
  g_Q=\frac12(\dd\theta)^2
     =\frac{(\dd c)^2}{2(1-c^2)},
  \qquad
  g_{\mathrm E}=(\dd c)^2,
  \qquad
  g_{\mathrm F}=2(\dd c)^2.
\]
The volume density of a one-dimensional metric \(h(c)(\dd c)^2\)
is \(\sqrt{h(c)}\,\dd c\). Hence
\[
  \dd\Vol_Q=\frac{\dd c}{\sqrt{2(1-c^2)}},
  \qquad
  \dd\Vol_{\mathrm E}=\dd c.
\]
Their total masses on \((-1,1)\) being respectively
\(\pi/\sqrt2\) and \(2\), normalization gives
\[
  \dd\mu_{2,2}^{Q}(c)=\frac{\dd c}{\pi\sqrt{1-c^2}},
  \qquad
  \dd\mu_{2,2}^{\mathrm E}(c)=\frac12\dd c.
\]
Thus quotient volume is uniform in the angle \(\theta\), whereas ambient
volume is uniform in the off-diagonal coordinate \(c\).
\end{example}

As $\Crle{s}{n}$ is a stratified space, we will need the Hausdorff measure to compute volumes.
For a metric space with distance \(\rho\),
we denote by \(\mathcal H_\rho^m\) the \(m\)-dimensional Hausdorff
measure, normalized to agree with Lebesgue measure on \(\R^m\). For an
integer \(m\ge1\), this means
\begin{equation}
\label{eq:volume-section-hausdorff-normalization}
  \mathcal H_\rho^m(B)
  \coloneq \frac{\omega_m}{2^m}\lim_{\varepsilon\downarrow0}
  \inf_{\substack{B\subset\bigcup_jU_j\\
                  \diam_\rho U_j<\varepsilon}}
  \sum_{j=1}^\infty(\diam_\rho U_j)^m,
\end{equation}
where \(\omega_m\) is the volume of the Euclidean unit ball. In words,
one covers the set by small sets and minimizes their total
\(m\)-dimensional diameter cost. For \(m=0\), Hausdorff measure is
counting measure. On a smooth \(m\)-dimensional Riemannian manifold,
\(\mathcal H_{\dg{g}}^m=\Vol_g\); on a smooth Euclidean
submanifold, the same surface volume is obtained from the ambient chord
distance \cite{Federer1969,BuragoBuragoIvanov2001}.

In particular, for a measurable (Borel) set \(B\subset\Crk{s}{n}\),
\begin{equation}
\label{eq:volume-section-hausdorff-identifications}
  \begin{aligned}
    \Vol_Q(B)
      &=\mathcal H_{\dQ{s}}^{m_s}\bigl(\bar\Gamma_s^{-1}(B)\bigr),\\
    \Vol_{\mathrm E}(B)
      &=\mathcal H_{\|\cdot\|_{\mathrm E}}^{m_s}(B),\\
    \Vol_{\mathrm F}(B)
      &=\mathcal H_{\|\cdot\|_F}^{m_s}(B).
  \end{aligned}
\end{equation}
The first Hausdorff measure is computed on \((Q_s,\dQ{s})\); the norm
subscripts in the last two lines denote the distances
\((C,C')\mapsto\|C-C'\|_{\mathrm E}\) and
\((C,C')\mapsto\|C-C'\|_F\). When \(s<n\), ambient volume therefore
means \(m_s\)-dimensional surface volume, rather than Lebesgue measure in
the full space of off-diagonal coordinates. Lower-rank strata have
smaller dimensions and will have zero \(m_s\)-dimensional measure; their
own intrinsic volumes need not vanish. This observation is made precise
in \cref{lem:finite-volume-comparison-section}.

\subsection{Coarea and the volume of the fibers}
\label{subsec:ambient-volume-relation-quotient-density}

We first explain the factor that appears when volume is pushed through
the Gram map. The coarea formula decomposes integration over the source
into integration along each fiber and integration over the target.
For example, if \(h:(0,1)\to(0,\infty)\) is continuous, projection of
the region \(\{(u,v):0<u<1,\ 0<v<h(u)\}\) onto the \(u\)-axis gives,
for nonnegative measurable \(f\),
\[
  \int_0^1\int_0^{h(u)} f(u)\dd v\dd u
  =\int_0^1 f(u)h(u)\dd u.
\]
The projected measure has density \(h(u)\): the measure of the fiber.
The same principle explains the role of orbit volumes in Gram sampling.

Let \(F:(\mathcal M,g_{\mathcal M})\to
(\mathcal N,g_{\mathcal N})\) be a smooth submersion. Its
\emph{normal Jacobian} is defined by
\begin{equation}
\label{eq:volume-section-normal-jacobian}
  J_F(x)\coloneq \sqrt{\det\bigl(\dd F(x)\,\dd F(x)^*\bigr)},
\end{equation}
where the determinant is taken on \(\Tz{F(x)}\mathcal N\).
The adjoint
\(\dd F(x)^*:\Tz{F(x)}\mathcal N\to\Tz{x}\mathcal M\)
depends on both metrics through its defining identity
\[
  g_{\mathcal M,x}(U,\dd F(x)^*V)
  =g_{\mathcal N,F(x)}(\dd F(x)[U],V),
\]
for \(U\in\Tz{x}\mathcal M\) and \(V\in\Tz{F(x)}\mathcal N\).

Since \(\ker\dd F(x)\) is tangent to the fiber \(F^{-1}(F(x))\),
its \(g_{\mathcal M}\)-orthogonal complement
\(\Hz{x}\coloneq (\ker\dd F(x))^\perp\) is the normal space to the fiber.
The restriction
\[
  \restr{\dd F(x)}{\Hz{x}}:\Hz{x}\longrightarrow\Tz{F(x)}\mathcal N
\]
is an isomorphism. In orthonormal bases adapted to
\(\Tz{x}\mathcal M=\Hz{x}\oplus\ker\dd F(x)\), the differential
has matrix \((B\;\;0)\), where the matrix \(B\) represents the restriction $\restr{\dd F(x)}{\Hz{x}}$.
The adjoint is represented by the transpose, and hence
\[
  J_F(x)=\sqrt{\det(BB^\top)}=|\det B|.
\]
Thus \(J_F(x)\) measures the factor by which the differential scales
volumes normal to the fiber. 
If $F$ is a diffeomorphism, then
\(\Hz{x}=\Tz{x}\mathcal M\), giving the usual Riemannian Jacobian.
For a Riemannian submersion, the restriction is an isometry,
so \(B\) is orthogonal and \(J_F\equiv1\).
We use the same notation \(J_F\) for the usual Jacobian in the
diffeomorphism case. These conventions, and their formulation in terms
of volume densities, are discussed in \cite{cibotaruNoteAreaCoarea2016}.

For every nonnegative measurable function \(\psi\), the smooth coarea
formula \cite{Federer1969,cibotaruNoteAreaCoarea2016} gives
\begin{equation}
\label{eq:volume-section-smooth-coarea}
  \int_{\mathcal M}\psi J_F\,\dd\Vol_{\mathcal M}
  =\int_{\mathcal N}\left(
    \int_{F^{-1}(y)}\psi\,\dd\Vol_{F^{-1}(y)}
  \right)\dd\Vol_{\mathcal N}(y).
\end{equation}
Each fiber carries the metric induced from \(g_{\mathcal M}\). In
particular, for a submersion \(a=(a_1,\ldots,a_q):\mathcal M\to\R^q\), with the
standard Euclidean metric on the target, one has
\begin{equation}
\label{eq:coarea-formula-a}
  \int_{\mathcal M}\psi J_a\,\dd\Vol_{\mathcal M}
  =\int_{\R^q}\left(
    \int_{a^{-1}(u)}\psi\,\dd\Vol_{a^{-1}(u)}
  \right)\dd u,
\end{equation}
and the normal Jacobian can be computed from the gradients of the
constraint functions:
\begin{equation}
\label{eq:volume-section-constraint-gradient-jacobian}
  J_a(x)^2
  =\det\bigl(g_{\mathcal M}(\grad a_i,\grad a_j)(x)\bigr)_{i,j=1}^q.
\end{equation}
Indeed, the adjoint of \(\dd a(x)\) sends the \(i\)-th standard basis
vector to \(\grad a_i(x)\), so the displayed Gram matrix represents
\(\dd a(x)\dd a(x)^*\). We will use this identity to impose the
unit-diagonal constraints.

For \(C=XX^\top\in\Crk{s}{n}\), let
\[
  v_s(C)\coloneq \Vol(X\cdot O(s)),
\]
be the measure of the fiber, where the orbit carries the metric induced from \(g_{M_s}\). Applying
the coarea formula (\cref{eq:volume-section-smooth-coarea}) to the Riemannian submersion
\(\Gamma_s:(\Mss,g_{M_s})\to(\Crk{s}{n},g_Q)\), with normal
Jacobian one, gives
\begin{equation}
\label{eq:volume-section-fiber-integration}
  \int_{\Mss} f(\Gamma_s(X))\,\dd\Vol_{M_s}(X)
  =\int_{\Crk{s}{n}} f(C)v_s(C)\,\dd\Vol_Q(C)
\end{equation}
for every nonnegative measurable function \(f:\Crk{s}{n}\to[0,\infty]\). 
Thus, the pushforward of
factor-space volume has density \(v_s\) with respect to quotient volume.
Orbit-volume factors also appear in Weyl-type integration formulas for
isometric group actions; see \cite{Magata2010} for formulas based on
minimal sections and their specialization to polar actions. In the present
setting, \cref{eq:volume-section-fiber-integration} follows directly
from coarea on the smooth quotient.

We now express \(v_s\) through the Sylvester operator from
\cref{subsec:sylvester-mechanical-connection}.
For any full-column-rank factor \(Y\in\R^{n\times s}\), set
\begin{equation}
\label{eq:orbit-volume-factor}
  D_s(Y)\coloneq \prod_{1\le a<b\le s}
    \bigl(\lambda_a(Y^\top Y)+\lambda_b(Y^\top Y)\bigr)^{1/2}.
\end{equation}
The positive eigenvalues of \(C=XX^\top\) are those of \(E_X\), and
we use the same invariant on the correlation stratum:
\[
  D_s(C)\coloneq \prod_{1\le a<b\le s}
    \bigl(\lambda_a(C)+\lambda_b(C)\bigr)^{1/2}=D_s(X).
\]
These products are symmetric in the positive eigenvalues and are
independent of the chosen factor representative.

\begin{lemma}[Orbit volume and the Sylvester operator]
\label{lem:volume-section-orbit-sylvester}
For \(X\in\Mss\), recall
\[
  \mathcal L_{E_X}:\so(s)\to\so(s),
  \qquad \mathcal L_{E_X}(\Omega)=E_X\Omega+\Omega E_X.
\]
Then
\begin{equation}
\label{eq:volume-section-sylvester-determinant}
  D_s(C)^2
  =\det
  \mathcal L_{E_X},
  \qquad C=XX^\top,
\end{equation}
and
\begin{equation}
\label{eq:volume-section-orbit-volume}
  v_s(C)=2^{-\binom s2/2}D_s(C)\Vol(O(s)).
\end{equation}
Here \(O(s)\) carries the bi-invariant Riemannian metric whose value at
the identity is the Frobenius inner product on \(\so(s)\), and its
volume includes both connected components.
\end{lemma}

\begin{proof}
Pick an orthonormal eigenbasis \(e_1,\ldots,e_s\) of \(E_X\), with
\(E_Xe_a=\lambda_a e_a\) and \(\lambda_a>0\).
The matrices
\[
  \Omega_{ab}\coloneq \frac{e_ae_b^\top-e_be_a^\top}{\sqrt2},
  \qquad 1\le a<b\le s,
\]
form a Frobenius-orthonormal basis of \(\so(s)\).
Since left and right multiplication by \(E_X\) multiply
\(e_ae_b^\top\) by \(\lambda_a\) and \(\lambda_b\), respectively,
\[
  \mathcal L_{E_X}(\Omega_{ab})
  =(\lambda_a+\lambda_b)\Omega_{ab}.
\]
Thus \(\mathcal L_{E_X}\) is diagonal in this basis, hence
\[
  \det\mathcal L_{E_X}
  =\prod_{a<b}(\lambda_a+\lambda_b)
  =D_s(C)^2,
\]
which proves \cref{eq:volume-section-sylvester-determinant}. To compute the orbit volume, consider
\[
  \phi_X:O(s)\longrightarrow XO(s),
  \qquad \phi_X(R)=XR,
\]
with the reference metric on \(O(s)\) and the metric induced from
\(g_{M_s}\) on the orbit. 

By linearity, its differential at the identity is $\dd\phi_X(\I_s)\colon \Omega \mapsto  X\Omega$.
We therefore compare the Frobenius inner products of
\(\Omega,\Lambda\in\so(s)\) with those of their images
\(X\Omega,X\Lambda\). Skew-symmetry, symmetry of \(E_X\), and
cyclicity of the trace give
\[
  \begin{aligned}
    \tr(\Omega^\top\Lambda E_X)
    =\tr\bigl((\Omega^\top\Lambda E_X)^\top\bigr)
    =\tr(E_X\Lambda^\top\Omega)
    =\tr(\Omega^\top E_X\Lambda).
  \end{aligned}
\]
Consequently, the orbit metric evaluated on the images of
\(\Omega\) and \(\Lambda\) satisfies
\[
  \begin{aligned}
    \Frob{X\Omega}{X\Lambda}
    =\tr(\Omega^\top E_X\Lambda)
    =\frac12
      \Frob{\Omega}{\mathcal L_{E_X}(\Lambda)}.
  \end{aligned}
\]
In particular, for \(a<b\) and \(c<d\),
\[
  \Frob{X\Omega_{ab}}{X\Omega_{cd}}
  =\frac{\lambda_c+\lambda_d}{2}\,
    \delta_{ac}\delta_{bd}.
\]
Thus the images of the reference-orthonormal basis are mutually
orthogonal, and for any $a<b$
\[\|X\Omega_{ab}\|_F^2=\frac{\lambda_a+\lambda_b}{2}.
\]
The Jacobian is the product of their lengths. Since there are
\(\binom s2\) basis vectors, we obtain
\[
  J_{\phi_X}(\I_s)
  =\prod_{a<b}
    \left(\frac{\lambda_a+\lambda_b}{2}\right)^{1/2}
  =2^{-\binom s2/2}D_s(C).
\]

Finally, full column rank of \(X\) makes \(\phi_X\) a diffeomorphism
onto its orbit, with smooth inverse given by
\[
  XO(s)\longrightarrow O(s),
  \qquad Y\longmapsto E_X^{-1}X^\top Y.
\]
The change-of-variables formula, integrated over both connected
components of \(O(s)\) then yields
\begin{align*}
  v_s(C)
    &=\int_{O(s)}J_{\phi_X}(R)\,
      \dd\Vol_{O(s)}(R)\\
    &=\int_{O(s)}2^{-\binom s2/2}D_s(C)\,
      \dd\Vol_{O(s)}(R)\\
    &=2^{-\binom s2/2}D_s(C)\Vol(O(s)),
\end{align*}
ensuring \cref{eq:volume-section-orbit-volume}.
\end{proof}

\begin{definition}[Uniform factor law and raw Gram law]
\label{def:raw-gram-law}
The \emph{uniform factor law} is normalized product-sphere volume:
\[
  \mathcal U_{M_s}
  \coloneq \frac{\Vol_{M_s}}{\Vol(M_s)},
  \qquad \Vol(M_s)\coloneq \Vol_{M_s}(M_s).
\]
Under this law, the rows are independent and uniform on
\(\mathbb S^{s-1}\). The \emph{raw Gram law} is its pushforward
\[
  \eta_{n,s}\coloneq (\Gamma_s)_\#\mathcal U_{M_s}.
\]
Explicitly, for every Borel set \(B\subset\Crle{s}{n}\),
\[
  \eta_{n,s}(B)
  =\mathcal U_{M_s}\bigl(\{X\in M_s:XX^\top\in B\}\bigr).
\]
In particular, \(\eta_{n,s}\) is the probability distribution of \(C=XX^\top\) when
\(X\sim\mathcal U_{M_s}\). Uniformity here refers to the factor
rows; the target volume-uniform laws are those of
\cref{def:normalized-volume-laws}.
\end{definition}

The rank-deficient locus is
\(\{X\in M_s:\det(X^\top X)=0\}\).
Its defining function is real analytic and not identically zero: a
factor whose first \(s\) rows form the standard basis of \(\R^s\)
has full column rank. Since \(M_s\) is connected for \(s\ge2\), the
analytic zero-set theorem \cite{Mityagin2020}, applied in coordinate
charts, shows that this locus has zero Riemannian volume.
Since \(\rank(XX^\top)=\rank X\), the law \(\eta_{n,s}\) is
almost surely supported on \(\Crk{s}{n}\).
Consider now
\[
  \kappa_{n,s}\coloneq 
  \frac{2^{-\binom s2/2}\Vol(O(s))}
       {\Vol(M_s)}.
\]

\begin{proposition}[Weighted quotient-volume transfer]
\label{prop:weighted-quotient-volume-transfer}
For every nonnegative measurable function \(h\) on \(\Crk{s}{n}\),
\begin{equation}
\label{eq:volume-section-weighted-transfer}
  (\Gamma_s)_\#\left(
    \ind{\Mss}\frac{h(\Gamma_s(X))}{D_s(X)}
    \mathcal U_{M_s}(\dd X)
  \right)
  =\kappa_{n,s}\,h(C)\Vol_Q(\dd C).
\end{equation}
The source density is understood to be zero outside \(\Mss\).
\end{proposition}

\begin{proof}
Let \(\varphi\) be a nonnegative measurable test function on
\(\Crk{s}{n}\). The functions \(\varphi\circ\Gamma_s\),
\(h\circ\Gamma_s\), and \(D_s\) are constant on each fiber. By
\cref{eq:volume-section-fiber-integration,eq:volume-section-orbit-volume},
\[
  \begin{aligned}
    &\int_{\Mss}
      \varphi(\Gamma_s(X))\frac{h(\Gamma_s(X))}{D_s(X)}
      \,\dd\mathcal U_{M_s}(X)\\
    &\qquad=
      \frac{1}{\Vol(M_s)}
      \int_{\Crk{s}{n}}\varphi(C)h(C)
      \frac{v_s(C)}{D_s(C)}\,\dd\Vol_Q(C)\\
    &\qquad=
      \kappa_{n,s}\int_{\Crk{s}{n}}
      \varphi(C)h(C)\,\dd\Vol_Q(C),
  \end{aligned}
\]
which concludes 
\cref{eq:volume-section-weighted-transfer}.
\end{proof}

Taking \(h(C)=D_s(C)\) gives the density of the raw Gram law with respect to quotient volume
\begin{equation}
\label{eq:volume-section-raw-gram-density}
  \dd\eta_{n,s}(C)
  =\kappa_{n,s}D_s(C)\,\dd\Vol_Q(C).
\end{equation}
Taking \(h\equiv1\) shows that weighting factors by \(D_s^{-1}\)
compensates for their fiber volumes. To obtain a lift of ambient volume,
we first compute its density with respect to quotient volume.

\subsection{Exact volume comparison and boundary behavior}
\label{subsec:exact-volume-comparison}

We now determine the additional correction required for ambient
uniformity. Throughout the remainder of the section, write
\[
  w(C)\coloneq \frac{\dd\Vol_{\mathrm E}}
                {\dd\Vol_Q}(C).
\]
In other words, for every measurable set \(B\subset\Crk{s}{n}\),
\[
  \Vol_{\mathrm E}(B)
  =\int_B w(C)\,\dd\Vol_Q(C).
\]
The density \(w(C)\) is
the local volume distortion between the two metrics. If
\(U_1,\ldots,U_{m_s}\) is a \(g_Q\)-orthonormal basis at \(C\), then
\[
  w(C)=\sqrt{\det\bigl(g_{\mathrm E,C}(U_\alpha,U_\beta)\bigr)
       _{\alpha,\beta=1}^{m_s}}.
\]
Equivalently, \(w\) is the Jacobian of \(\bar\Gamma_s\) when its
source \(\Qss\) carries \(\gQss\) and its target \(\Crk{s}{n}\)
carries \(g_{\mathrm E}\). To compute it, we first work with
unconstrained full-column-rank factors and rank-\(s\) positive
semidefinite matrices. The calculation is simplest with the Frobenius
metric on the target; \cref{eq:volume-section-frobenius-scaling} then
converts the result to \(g_{\mathrm E}\).

\subsubsection{The unconstrained Gram map and the diagonal constraints}
\label{subsubsec:volume-section-gram-constraints}

Let
\[
  \mathcal R_{n,s}\coloneq \{Y\in\R^{n\times s}:\rank Y=s\},\qquad
  \mathcal P_{n,s}\coloneq \{G\in\Sym(n):G\succeq0,\ \rank G=s\}.
\]
The unconstrained Gram map \(\Phi(Y)\coloneq YY^\top\) descends to a
diffeomorphism
\[
  \bar\Phi:\mathcal R_{n,s}/O(s)\longrightarrow\mathcal P_{n,s}.
\]
The source carries the quotient metric induced by the Frobenius metric on
\(\mathcal R_{n,s}\), while the target carries the metric induced by the
Frobenius metric on \(\Sym(n)\). We also write
\(\pi_s:\mathcal R_{n,s}\to\mathcal R_{n,s}/O(s)\),
\(\pi_s(Y)=[Y]\), for the quotient projection; its restriction to
\(\Mss\) agrees with the earlier notation. Define
\[
  \operatorname{Diag}_{n,s}
  \coloneq \{\diag(G):G\in\mathcal P_{n,s}\}
  \subset\R_{\ge0}^n,
\]
and
\[
  \bar a([Y])\coloneq \diag(YY^\top),
  \qquad b(G)\coloneq \diag(G).
\]
The identities \(\Phi=\bar\Phi\circ\pi_s\) and
\(\bar a=b\circ\bar\Phi\) give the two commuting triangles
\[
\begin{tikzcd}[column sep=huge, row sep=large]
  \mathcal R_{n,s}
    \arrow[r, "\Phi"]
    \arrow[d, "\pi_s"']
  &
  \mathcal P_{n,s}
    \arrow[d, "b"]
  \\
  \mathcal R_{n,s}/O(s)
    \arrow[ur, "\bar\Phi"]
    \arrow[r, "\bar a"']
  &
  \operatorname{Diag}_{n,s}
\end{tikzcd}
\]
with corresponding levels
\[
  \bar a^{-1}(\mathbf1_n)=\Qss,
  \qquad
  b^{-1}(\mathbf1_n)=\Crk{s}{n}.
\]
We view \(\bar a\) and \(b\) as maps into \(\R^n\) with its standard
Euclidean metric when computing their normal Jacobians.

Let \(H\) be the horizontal lift to \(\mathcal R_{n,s}\) of a
tangent vector \(\xi\) to \(\bar a^{-1}(\mathbf1_n)\) at \([X]\).
Horizontality and tangency to the level give, respectively,
\[
  X^\top H=H^\top X,
  \qquad
  \dd\bar a([X])[\xi]=2\diag(HX^\top)=0.
\]
These are exactly the conditions defining the horizontal space
\(\Hz{X}\) of \(\Mss\). The induced metric on
\(\bar a^{-1}(\mathbf1_n)=\Qss\) is therefore \(\gQss\),
since it is computed from the Frobenius inner products of these lifts.
On the target side, restricting the Frobenius metric of
\(\mathcal P_{n,s}\) to
\(b^{-1}(\mathbf1_n)=\Crk{s}{n}\) gives \(g_{\mathrm F}\).
For \(C=XX^\top\in\Crk{s}{n}\), set
\[
  \pdet(C)\coloneq \prod_{a=1}^s\lambda_a(C)=\det E_X,
  \qquad
  A_C\coloneq 2\diag(\diag P_C)-P_C\odot P_C,
\]
where \(\odot\) denotes the
Hadamard (entrywise) product, and $\pdet$ is the \emph{pseudo-determinant}, i.e.,  \(\pdet(C)\) is the product
of the positive eigenvalues of \(C\); it remains positive when
\(0<s<n\), although \(\det C=0\). We use \(\diag\) both to extract a
matrix diagonal as a vector and to form a diagonal matrix from a vector.
Thus \(\diag(\diag P_C)\) has diagonal entries \((P_C)_{ii}\), and
\[
  (A_C)_{ij}=2\delta_{ij}(P_C)_{ii}-(P_C)_{ij}^2.
\]
 As we will see in \cref{lem:unit-diagonal-normal-jacobians}, \(A_C\) is the
Gram matrix of the gradients of \(b_i(G)=G_{ii}\) on the 
manifold \(\mathcal P_{n,s}\). Consequently,
\(J_b(C)=\sqrt{\det A_C}\): this factor measures the normal Jacobian
of the diagonal constraint on the Gram-matrix side.

\begin{theorem}[Exact comparison of quotient and ambient volume]
\label{thm:quotient-versus-ambient-volume}
On \(\Crk{s}{n}\),
\begin{equation}
\label{eq:volume-section-exact-euclidean-density}
  \frac{\dd\Vol_{\mathrm E}}
       {\dd\Vol_Q}(C)
  =2^{\frac{s-n}{2}}\pdet(C)^{\frac{n-s+1}{2}}
   D_s(C)\sqrt{\det A_C},
\end{equation}
and
\begin{equation}
\label{eq:volume-section-exact-frobenius-density}
  \frac{\dd\Vol_{\mathrm F}}
       {\dd\Vol_Q}(C)
  =2^{\frac{s(2n-s+3)}4-n}
   \pdet(C)^{\frac{n-s+1}{2}}
   D_s(C)\sqrt{\det A_C}.
\end{equation}
The three volume measures are mutually absolutely continuous on the
fixed-rank stratum.
\end{theorem}

\subsubsection{Jacobian calculations}
\label{subsubsec:volume-section-jacobian-calculations}

The proof has three ingredients. We first relate the Jacobian of a
diffeomorphism to that of its restriction to corresponding regular levels.
We then compute \(J_{\bar\Phi}\) on the unconstrained quotient using
horizontal lifts. Finally, we compute the normal Jacobians of
\(\bar a\) and \(b\), which remove the normal distortion and leave
the Jacobian along the unit-diagonal levels.
Each constraint map uses the metric on its own domain and the standard
Euclidean metric on its codomain: \(\R^q\) in the following lemma
and \(\R^n\) in the application.

\begin{lemma}[Jacobian on corresponding regular level sets]
\label{lem:jacobian-corresponding-level-sets}
Let \(F:(\mathcal M,g_{\mathcal M})\to(\mathcal N,g_{\mathcal N})\) be a
smooth diffeomorphism between Riemannian manifolds, and let
\(a:\mathcal M\to\R^q\), \(b:\mathcal N\to\R^q\) be smooth maps
satisfying \(a=b\circ F\). If \(t\) is a regular value of both maps, then
\[
  J_{\restr{F}{a^{-1}(t)}}(x)
  =J_F(x)\frac{J_b(F(x))}{J_a(x)},\qquad x\in a^{-1}(t),
\]
where the level sets carry their induced metrics.
\end{lemma}

\begin{proof}
Let $t$ be a regular value of $a$ and $b$, and fix $x\in a^{-1}(t)$. Consider \(y\coloneq F(x)\), \(K_a\coloneq \ker\dd a(x)\), and
\(K_b\coloneq \ker\dd b(y)\). These are the tangent spaces to the
corresponding level sets. Since
\(\dd a(x)=\dd b(y)\circ\dd F(x)\) and \(\dd F(x)\) is
invertible, it maps \(K_a\) onto \(K_b\).
In orthonormal bases adapted to the direct sums
 \(\dd F(x) \colon K_a\oplus K_a^\perp \to K_b\oplus K_b^\perp\), we can thus write
\[
  \dd F(x)=\begin{pmatrix}B&T\\0&D\end{pmatrix}.
\]
Similarly, consider $N_a$ and $N_b$ such that 
\[
  \dd a(x)=\begin{pmatrix}0&N_a\end{pmatrix} \colon K_a\oplus K_a^\perp \to \R^q,
  \qquad
  \dd b(y)=\begin{pmatrix}0&N_b\end{pmatrix}
  \colon K_b\oplus K_b^\perp \to \R^q.
\]
The chain rule ensures
\[
  \begin{pmatrix}0&N_a\end{pmatrix}
  =
  \begin{pmatrix}0&N_b\end{pmatrix}
  \begin{pmatrix}B&T\\0&D\end{pmatrix}
  =
  \begin{pmatrix}0&N_bD\end{pmatrix},
\]
so \(N_a=N_bD\). Since all bases are orthonormal,
\[
  \begin{aligned}
    J_{\restr{F}{a^{-1}(t)}}(x)&=|\det B|,\\
    J_F(x)&=|\det B|\,|\det D|,\\
    J_a(x)&=|\det N_a|
           =|\det N_b|\,|\det D|
           =J_b(y)\,|\det D|.
  \end{aligned}
\]
Eliminating \(|\det D|\) yields
\[
  J_{\restr{F}{a^{-1}(t)}}(x)
  =J_F(x)\frac{J_b(F(x))}{J_a(x)}.
\]
\end{proof}

In this decomposition, \(J_a/J_b=|\det D|\) is the contribution of
the normal block. Multiplication by its reciprocal leaves the tangential
Jacobian \(|\det B|\). The identity holds at every point of the
specified regular level, so it applies directly to the unit-diagonal
constraints.

\begin{lemma}[Unconstrained Gram-quotient Jacobian]
\label{lem:unconstrained-gram-quotient-jacobian}
For \(Y\in\mathcal R_{n,s}\), with the source and target metrics
specified above,
\[
  J_{\bar\Phi}([Y])
  =2^{\frac{s(2n-s+3)}4}
   \det(Y^\top Y)^{\frac{n-s+1}{2}}D_s(Y).
\]
\end{lemma}

\begin{proof}
We compute the Jacobian in three steps: identify the quotient
differential through horizontal lifts, reduce \(Y\) by a singular
value decomposition, and construct an orthonormal horizontal basis
with orthogonal images. The Jacobian will then be the product of
their image lengths.

By definition of the quotient metric, \(\dd\pi_s(Y)\) restricts to
an isometry from the horizontal space
\(\{H:Y^\top H=H^\top Y\}\) onto the quotient tangent space.
Differentiating \(\Phi=\bar\Phi\circ\pi_s\) gives
\[
  \dd\bar\Phi([Y])[\dd\pi_s(Y)[H]]
  =\dd\Phi(Y)[H]
  =YH^\top+HY^\top.
\]
Consequently, for any Frobenius-orthonormal horizontal basis
\((H_\alpha)_\alpha\),
\[
  J_{\bar\Phi}([Y])
  =\sqrt{\det\Bigl(
    \bigl[
      \Frob{\dd\Phi(Y)[H_\alpha]}{\dd\Phi(Y)[H_\beta]}
    \bigr]_{\alpha,\beta}
    \Bigr)}.
\]

Choose a singular value decomposition
\[
  Y=U\begin{pmatrix}\Sigma\\0\end{pmatrix}V^\top,
  \qquad
  \Sigma=\diag(\sigma_1,\ldots,\sigma_s),
  \qquad
  \lambda_a\coloneq \sigma_a^2>0,
\]
where \(U\in O(n)\) and \(V\in O(s)\).
Under the orthonormal changes of basis
\(\widetilde Y\coloneq U^\top YV\), \(\widetilde H\coloneq U^\top HV\),
Frobenius inner products are preserved, and
\[
  \begin{aligned}
    \widetilde Y^\top\widetilde H
      &=V^\top(Y^\top H)V,\\
    \dd\Phi(\widetilde Y)[\widetilde H]
      &=U^\top\dd\Phi(Y)[H]U.
  \end{aligned}
\]
Thus horizontality and the Jacobian are preserved, so we may
assume \(Y=\binom{\Sigma}{0}\).

Write \(H=\binom{B}{K}\), with
\(B=(b_{ab})\in\R^{s\times s}\) and
\(K\in\R^{(n-s)\times s}\). Then
\(\|H\|_F^2=\|B\|_F^2+\|K\|_F^2\), and block multiplication gives
\[
  Y^\top H=\Sigma B,
  \qquad
  \dd\Phi(Y)[H]
  =\begin{pmatrix}
      \Sigma B^\top+B\Sigma&\Sigma K^\top\\
      K\Sigma&0
    \end{pmatrix}.
\]
Horizontality is therefore equivalent to
\[
  \Sigma B=B^\top\Sigma
  \quad\Longleftrightarrow\quad
  \sigma_a b_{ab}=\sigma_b b_{ba}\qquad(a<b).
\]
The diagonal entries of \(B\) and all entries of \(K\) are free,
meaning they may be chosen arbitrarily: the conditions above
involve only off-diagonal pairs of \(B\).
For each \(a<b\), choosing \(b_{ab}\) determines
\(b_{ba}=(\sigma_a/\sigma_b)b_{ab}\), so each pair has one
free parameter.

Let \(E_{ij} \coloneq \left(\delta_{ik}\delta_{jl}\right)_{k,l}\) denote the matrix unit of the appropriate size.
For each \(a<b\), the horizontal direction supported on that pair
is proportional to \(\sigma_bE_{ab}+\sigma_aE_{ba}\).
Its squared Frobenius norm is \(\lambda_a+\lambda_b\), so define
the unit direction
\[
  B_{ab}
  \coloneq \frac{\sigma_bE_{ab}+\sigma_aE_{ba}}
           {\sqrt{\lambda_a+\lambda_b}}.
\]
Together with the diagonal and \(K\)-directions, this gives the
following three families. The rows are indexed respectively by
\(1\le a\le s\), \(1\le a<b\le s\), and
\(1\le\ell\le n-s,\ 1\le a\le s\):
\[
  \begin{array}{c c c c}
    B & K & \dd\Phi(Y)[H] & \|\dd\Phi(Y)[H]\|_F\\[2pt]
    \hline
    E_{aa} & 0
      & 2\sigma_aE_{aa}
      & 2\sqrt{\lambda_a}\\[4pt]
    B_{ab} & 0
      & \sqrt{\lambda_a+\lambda_b}(E_{ab}+E_{ba})
      & \sqrt{2(\lambda_a+\lambda_b)}\\[4pt]
    0 & E_{\ell a}
      & \sigma_a(E_{a,s+\ell}+E_{s+\ell,a})
      & \sqrt{2\lambda_a}
  \end{array}
\]
Every source direction is horizontal and has unit Frobenius norm.
The image lengths follow from
\(\|E_{ii}\|_F=1\) and
\(\|E_{ij}+E_{ji}\|_F=\sqrt2\) for \(i\ne j\).
Distinct source directions have disjoint supports, as do their
images. Thus the source directions are orthonormal and their images
are mutually orthogonal.

Spanning follows explicitly from the horizontal constraints.
For \(K=(k_{\ell a})\), put
\[
  t_{ab}\coloneq \frac{b_{ab}}{\sigma_b}
         =\frac{b_{ba}}{\sigma_a},
  \qquad a<b.
\]
Every horizontal \(H=\binom{B}{K}\) admits the expansions
\[
  \begin{aligned}
    B
      &=\sum_{a=1}^s b_{aa}E_{aa}
        +\sum_{1\le a<b\le s}
          t_{ab}\sqrt{\lambda_a+\lambda_b}\,B_{ab},\\
    K
      &=\sum_{\ell=1}^{n-s}\sum_{a=1}^s
          k_{\ell a}E_{\ell a}.
  \end{aligned}
\]
The number of basis vectors is therefore
\[
  s+\binom{s}{2}+s(n-s)
  =ns-\binom{s}{2},
\]
which also agrees with the dimension of the unconstrained quotient:
\[
  \dim\bigl(\mathcal R_{n,s}/O(s)\bigr)
  =ns-\dim O(s)
  =ns-\binom{s}{2}.
\]

The image Gram matrix is diagonal, so multiplying the lengths in
the last column gives
\[
  \begin{aligned}
    J_{\bar\Phi}([Y])
      &=
        \underbrace{\prod_{a=1}^s2\sqrt{\lambda_a}}
                   _{\text{diagonal directions}}
        \underbrace{\prod_{a<b}\sqrt{2(\lambda_a+\lambda_b)}}
                   _{\text{off-diagonal pairs}}
        \underbrace{\prod_{a=1}^s\prod_{\ell=1}^{n-s}
                    \sqrt{2\lambda_a}}
                   _{\text{entries of }K}\\[4pt]
      &=\left(\prod_{a=1}^s2\sqrt{\lambda_a}\right)
        \left(\prod_{a<b}\sqrt{2(\lambda_a+\lambda_b)}\right)
        \left(\prod_{a=1}^s(2\lambda_a)^{(n-s)/2}\right)\\
      &=2^{s+\binom{s}{2}/2+s(n-s)/2}
        \left(\prod_{a=1}^s\lambda_a\right)^{(n-s+1)/2}
        \prod_{a<b}\sqrt{\lambda_a+\lambda_b}\\
      &=2^{s(2n-s+3)/4}
        \det(Y^\top Y)^{(n-s+1)/2}D_s(Y).
  \end{aligned}
\]
\end{proof}

\begin{lemma}[Normal Jacobians of the unit-diagonal constraints]
\label{lem:unit-diagonal-normal-jacobians}
At \([X]\in\Qss\) and \(C=XX^\top\),
\[
  J_{\bar a}([X])=2^n,\qquad J_b(C)=\sqrt{\det A_C},
\]
and
\[
  \frac1n\I_n\preceq A_C\preceq\I_n,\qquad
  n^{-n/2}\le\sqrt{\det A_C}\le1.
\]
In particular, \(\mathbf1_n\) is a regular value of both constraint
maps.
\end{lemma}

\begin{proof}
Define \(a\coloneq\bar a\circ\pi_s\), so that
\(a(Y)=\diag(YY^\top)\). Writing \(y_i\coloneq Y_{i\cdot}^\top\),
we have \(a_i(Y)=\|y_i\|^2\) and
\[
  \dd a_i(Y)[H]
  =2\Frob{y_i}{H_{i\cdot}^\top}
  =\Frob{2e_i y_i^\top}{H},
\]
where \(e_i\) is the \(i\)-th standard basis vector of \(\R^n\).
Thus the Frobenius gradient on \(\mathcal R_{n,s}\) is
\(\grad a_i(Y)=2e_i y_i^\top\). It is horizontal because
\[
  Y^\top\grad a_i(Y)=2y_i y_i^\top
\]
is symmetric. By the definition of the quotient metric, its
projection is the quotient gradient:
\[
  \grad\bar a_i([Y])
  =\dd\pi_s(Y)[2e_i y_i^\top].
\]
Thus, at a unit-row factor \(X\),
\[
  \Frob{2e_ix_i^\top}{2e_jx_j^\top}
  =4\delta_{ij}\|x_i\|^2=4\delta_{ij}.
\]
By \cref{eq:volume-section-constraint-gradient-jacobian},
\(J_{\bar a}([X])=\sqrt{\det(4\I_n)}=2^n\).

For the target, let \(P\coloneq P_C\) be the orthogonal projector onto
\(\im C\), and put \(P^\perp\coloneq \I_n-P\), the orthogonal
projector onto \(\ker C\). The tangent
space to \(\mathcal P_{n,s}\) and its Frobenius-orthogonal projection
are
\[
  \Tz{C}\mathcal P_{n,s}
  =\{S\in\Sym(n):P^\perp SP^\perp=0\},
  \qquad \Pi_C(S)=S-P^\perp SP^\perp.
\]
For \(S\in\Tz{C}\mathcal P_{n,s}\),
\[
  \dd b_i(C)[S]
  =S_{ii}
  =\Frob{e_ie_i^\top}{S}
  =\Frob{\Pi_C(e_ie_i^\top)}{S} + \Frob*{P^\perp e_i e_i^\top P^\perp}{S}
  =\Frob{\Pi_C(e_ie_i^\top)}{S},
\]
since \(S\) is tangent. Hence
\(\grad b_i(C)=\Pi_C(e_ie_i^\top)\).
The projection \(\Pi_C\) is self-adjoint and satisfies
\(\Pi_C^2=\Pi_C\). Using also
\(\Frob{e_ie_i^\top}{T}=T_{ii}\), we obtain
\[
  \begin{aligned}
    \Frob{\grad b_i(C)}{\grad b_j(C)}
      &=\Frob{e_ie_i^\top}{\Pi_C(e_je_j^\top)}\\
      &=\delta_{ij}-(P^\perp e_je_j^\top P^\perp)_{ii}\\
      &=\delta_{ij}-P^\perp_{ij}P^\perp_{ji}\\
      &=\delta_{ij}-(\delta_{ij}-P_{ij})^2\\
      &=2\delta_{ij}P_{ii}-P_{ij}^2.
  \end{aligned}
\]
Here we used the symmetry of \(P^\perp\), the identity
\(P^\perp_{ij}=\delta_{ij}-P_{ij}\), and
\(\delta_{ij}P_{ij}=\delta_{ij}P_{ii}\).
The gradient Gram matrix is therefore
\(2\diag(\diag P)-P\odot P=A_C\).
By \cref{eq:volume-section-constraint-gradient-jacobian}, \(J_b(C)=\sqrt{\det A_C}\).

Since \(P\) is an orthogonal projector, both \(P\) and
\(\I_n-P\) are positive semidefinite. Recall also that
\[
  P\odot\I_n=\diag(\diag P),
  \qquad \I_n\odot\I_n=\I_n.
\]
Starting from the definition of \(A_C\), the Schur product theorem
\cite[Theorem~7.5.3]{HornJohnsonMatrixAnalysis} therefore gives
\[
  \begin{aligned}
    A_C
      &=2\diag(\diag P)-P\odot P\\
      &=\diag(\diag P)+P\odot(\I_n-P)
        \succeq\diag(\diag P),\\
    \I_n-A_C
      &=\I_n-2\diag(\diag P)+P\odot P\\
      &=(\I_n-P)\odot(\I_n-P)\succeq0.
  \end{aligned}
\]
Thus \(\diag(\diag P)\preceq A_C\preceq\I_n\).
Recall that \(P=XE_X^{-1}X^\top\) and \(x_i=X^\top e_i\).
Since \(E_X\) is positive definite,
\(E_X^{-1}\succeq\lambda_1(E_X)^{-1}\I_s\), where
\(\lambda_1(E_X)\) is its largest eigenvalue.
Moreover, the unit-row constraints give
\[
  \lambda_1(E_X)\le\tr E_X
  =\sum_{j=1}^n\|x_j\|^2=n.
\]
Consequently,
\[
  \begin{aligned}
    P_{ii}
      =e_i^\top XE_X^{-1}X^\top e_i
        =x_i^\top E_X^{-1}x_i
      \ge\frac{\|x_i\|^2}{\lambda_1(E_X)}
        =\frac1{\lambda_1(E_X)}
        \ge\frac1n.
  \end{aligned}
\]
This proves \(n^{-1}\I_n\preceq A_C\preceq\I_n\). Taking
determinants gives the bounds on \(J_b\), and the positivity of the two
normal Jacobians proves regularity at every point of the specified levels.
\end{proof}

\begin{proof}[Proof of \cref{thm:quotient-versus-ambient-volume}]
The restriction of \(\bar\Phi\) to the unit-diagonal level is
\(\bar\Gamma_s\), with source metric \(\gQss\) and target metric
\(g_{\mathrm F}\). Applying
\cref{lem:jacobian-corresponding-level-sets,lem:unconstrained-gram-quotient-jacobian,lem:unit-diagonal-normal-jacobians}
gives
\[
  \begin{aligned}
    \frac{\dd\Vol_{\mathrm F}}
         {\dd\Vol_Q}(C)
      =J_{\bar\Phi}([X])\frac{J_b(C)}{J_{\bar a}([X])}
      =2^{\frac{s(2n-s+3)}4-n}
        \det(E_X)^{\frac{n-s+1}{2}}D_s(X)\sqrt{\det A_C}.
  \end{aligned}
\]
By \cref{eq:volume-section-frobenius-scaling}, multiplying this density
by \(2^{-m_s/2}\) gives the coordinate-Euclidean density. The exponent
simplifies as
\[
  \frac{s(2n-s+3)}4-n-\frac{m_s}{2}
  =\frac{s(2n-s+3)}4-n
   -\frac{n(s-1)}2+\frac{s(s-1)}4
  =\frac{s-n}{2}.
\]
Since \(\det(E_X)=\pdet(C)\) and \(D_s(X)=D_s(C)\),
this proves both formulas. Each density is positive and finite everywhere
on \(\Crk{s}{n}\), so the three measures have the same null sets as the Hausdorff measure.
\end{proof}

\subsubsection{Finiteness and behavior near lower rank}
\label{subsubsec:volume-section-finiteness-boundary}

To normalize the volume measures, we need finite, positive total masses.
A direct estimate of the Gram differential bounds ambient volume by
quotient volume. Compactness of the reduced quotient then gives
finiteness, while the exact density describes degeneration near the
lower-rank boundary.

\begin{proposition}[Metric and volume distortion]
\label{prop:ambient-quotient-volume-bounds}
For every \(X\in\Mss\), \(C=XX^\top\), and \(H\in\Hz{X}\),
\[
  \lambda_s(C)\|H\|_F^2
  \le\|\dd\Gamma_s(X)[H]\|_{\mathrm E}^2
  \le2\lambda_1(C)\|H\|_F^2.
\]
Consequently,
\[
  \lambda_s(C)^{m_s/2}
  \le\frac{\dd\Vol_{\mathrm E}}
           {\dd\Vol_Q}(C)
  \le(2\lambda_1(C))^{m_s/2}
  \le(2n)^{m_s/2}.
\]
\end{proposition}

\begin{proof}
Let \(H\in\Hz{X}\) be the horizontal lift of
\(U\in\Tz{C}\Crk{s}{n}\), so that
\(U=\dd\Gamma_s(X)[H]\) and \(g_{Q,C}(U,U)=\|H\|_F^2\).
By \cref{lem:horizontality-criterion}, \(X^\top H\) is symmetric.
Since \(U\) has zero diagonal,
\[
  \begin{aligned}
    g_{\mathrm E,C}(U,U)
      &=\frac12\|XH^\top+HX^\top\|_F^2\\
      &=\tr(E_XH^\top H)+\tr\bigl((X^\top H)^2\bigr)\\
      &=\tr(E_XH^\top H)+\|X^\top H\|_F^2.
  \end{aligned}
\]
The spectral bounds for \(E_X\) and \(C\) give
\[
  \begin{aligned}
    \lambda_s(C)\|H\|_F^2
      &\le\tr(E_XH^\top H)
       \le\lambda_1(C)\|H\|_F^2,\\
    0&\le\|X^\top H\|_F^2
       =\tr(H^\top CH)
       \le\lambda_1(C)\|H\|_F^2.
  \end{aligned}
\]
Adding these inequalities proves
\[
  \lambda_s(C)g_{Q,C}(U,U)
  \le g_{\mathrm E,C}(U,U)
  \le 2\lambda_1(C)g_{Q,C}(U,U).
\]

Now choose a \(g_Q\)-orthonormal basis
\(U_1,\ldots,U_{m_s}\) and let
\(G_{\mathrm E}
\coloneq (g_{\mathrm E,C}(U_\alpha,U_\beta))_{\alpha,\beta}\).
The metric bounds place all its eigenvalues in
\([\lambda_s(C),2\lambda_1(C)]\).
The basis spans a tangent parallelepiped of \(g_Q\)-volume \(1\)
and \(g_{\mathrm E}\)-volume \(\sqrt{\det G_{\mathrm E}}\).
Their ratio is \(w(C)\), so
\[
  \lambda_s(C)^{m_s/2}
  \le w(C)=\sqrt{\det G_{\mathrm E}}
  \le \bigl(2\lambda_1(C)\bigr)^{m_s/2}.
\]
Finally, \(C\) is positive semidefinite with unit diagonal, so
\(\lambda_1(C)\le\tr C=n\).
\end{proof}

\begin{lemma}[Finiteness of stratum volumes]
\label{lem:finite-volume-comparison-section}
\label{lem:finite-principal-quotient-volume}
For \(n\ge s\ge2\), the stratum \(\Crk{s}{n}\) has finite, positive
quotient, coordinate-Euclidean, and matrix-Frobenius volumes. The union
of the lower-rank strata in its closure has zero \(m_s\)-dimensional
Hausdorff measure for each of the distances specified in
\cref{eq:volume-section-hausdorff-identifications}, using \(\dQ{s}\) on
the reduced quotient \(Q_s\).
\end{lemma}

\begin{proof}
By \cref{prop:reduced-quotient-global-metric-geometry}, the reduced
quotient \((Q_s,\dQ{s})\) is a compact Alexandrov space of dimension
\(m_s\). Its top-dimensional Hausdorff measure is finite
\cite{BuragoBuragoIvanov2001}. On the principal stratum this normalized
Hausdorff measure agrees with the quotient Riemannian volume, by
\cref{eq:volume-section-hausdorff-identifications}. The volume is
positive because the principal stratum is a nonempty smooth manifold.
Integrating \cref{prop:ambient-quotient-volume-bounds} gives
\[
  \Vol_{\mathrm E}(\Crk{s}{n})
  \le(2n)^{m_s/2}\Vol_Q(\Crk{s}{n})<\infty.
\]
Since \(w(C)>0\) throughout the stratum, its coordinate-Euclidean
volume is also positive. The identity
\(\Vol_{\mathrm F}=2^{m_s/2}\Vol_{\mathrm E}\)
gives the same conclusions for Frobenius volume.

For \(1\le j<s\), let \(m_j=n(j-1)-\binom j2\), in accordance
with the earlier stratum-dimension notation. Then
\[
  m_s-m_j
  =(s-j)\left(n-\frac{s+j-1}{2}\right)>0.
\]
For \(j\ge2\), the inclusion of the smooth rank-\(j\) quotient,
equipped with its intrinsic Riemannian distance, into \((Q_s,\dQ{s})\)
is \(1\)-Lipschitz. Indeed, the stratum metrics agree by
\cref{cor:fixed-rank-reduction-correlation}, and restricting the
admissible curves cannot decrease distance.
Its image consequently has Hausdorff dimension at most \(m_j\),
and hence zero \(\mathcal H_{\dQ{s}}^{m_s}\)-measure. The rank-one
piece is finite and has the same zero-measure property. In ambient
coordinates the rank-\(j\) piece is a smooth submanifold of dimension
\(m_j\), so it also has zero \(m_s\)-dimensional Hausdorff measure
for either ambient distance. Taking the finite union over \(j<s\)
proves the claim.
\end{proof}

The laws in \cref{eq:volume-section-normalized-laws} are therefore
well-defined. They may also be viewed as laws on \(\Crle{s}{n}\)
by assigning zero mass to the lower-rank strata. Those strata still belong
to the topological support: the principal stratum is dense in the closure,
and every nonempty open subset of it has positive volume. Thus a
top-dimensional uniform law on the rank-bounded set has rank exactly
\(s\) almost surely; it does not prescribe a mixture of the intrinsic
volume laws on different ranks.

\begin{corollary}[Normalized density and boundary decay]
\label{cor:normalized-ambient-quotient-density}
For \(C\in\Crk{s}{n}\), set
\[
  W_{n,s}(C)\coloneq 
  \pdet(C)^{\frac{n-s+1}{2}}
  D_s(C)\sqrt{\det A_C}.
\]
Then
\[
  \frac{\dd\mu_{n,s}^{\mathrm E}}{\dd\mu_{n,s}^{Q}}(C)
  =\frac{W_{n,s}(C)}
  {\displaystyle\int_{\Crk{s}{n}}W_{n,s}\,\dd\mu_{n,s}^{Q}}.
\]
Moreover, for every \(\delta>0\), there exist positive constants
\(c_-,c_+\), depending only on \(n,s,\delta\), such that
\[
  c_-\,\lambda_s(C)^{\frac{n-s+1}{2}}
  \le
  \frac{\dd\Vol_{\mathrm E}}
       {\dd\Vol_Q}(C)
  \le
  c_+\,\lambda_s(C)^{\frac{n-s+1}{2}}
\]
whenever \(\lambda_{s-1}(C)\ge\delta\).
The same bounds hold for the normalized density
\(\dd\mu_{n,s}^{\mathrm E}/\dd\mu_{n,s}^{Q}\), with possibly
different constants. In particular, both densities tend to zero
when \(\lambda_s(C)\to0\) while \(\lambda_{s-1}(C)\ge\delta\).
\end{corollary}

\begin{proof}
By \cref{thm:quotient-versus-ambient-volume},
\[
  \dd\Vol_{\mathrm E}
  =2^{(s-n)/2}W_{n,s}\,\dd\Vol_Q.
\]
Both total volumes are finite and positive by
\cref{lem:finite-volume-comparison-section}. Normalizing the two
measures gives the first formula, with the constant factor
\(2^{(s-n)/2}\) canceling.

Suppose now that \(\lambda_{s-1}(C)\ge\delta\).
The first \(s-1\) positive eigenvalues lie in \([\delta,n]\).
Every pair \(\lambda_a(C)+\lambda_b(C)\), with \(a<b\),
therefore contains an eigenvalue at least \(\delta\), while its
sum is at most \(\tr C=n\). Consequently,
\[
  \begin{aligned}
    \delta^{s-1}\lambda_s(C)
      &\le\pdet(C)
       \le n^{s-1}\lambda_s(C),\\
    \delta^{\binom{s}{2}/2}
      &\le D_s(C)
       \le n^{\binom{s}{2}/2}.
  \end{aligned}
\]
Moreover, \cref{lem:unit-diagonal-normal-jacobians} gives
\(n^{-n/2}\le\sqrt{\det A_C}\le1\).
Substituting these bounds into the density formula proves the
claimed inequalities. Normalization multiplies the density by
a positive constant depending only on \(n,s\), so the same
bounds hold, with adjusted constants, for the normalized density.
\end{proof}

The density \(w\) is therefore globally bounded above, but tends to
zero when one positive eigenvalue vanishes and the others stay bounded
away from zero. Its reciprocal
\[
  w(C)^{-1}
  =\frac{\dd\Vol_Q}{\dd\Vol_{\mathrm E}}(C)
\]
is unbounded along such approaches to rank \(s-1\). On each set
\(\{C\in\Crk{s}{n}:\lambda_s(C)\ge\varepsilon\}\),
\(\varepsilon>0\), the two volume measures bound one another by
positive constant factors. No such comparison holds uniformly on the
whole stratum, despite their mutual absolute continuity there.

\subsection{Uniform laws and their lifts}
\label{subsec:volume-section-uniform-laws-lifts}

We now combine the orbit-volume correction with the comparison of
metrics to obtain probability laws on factors. Recall from
\cref{def:raw-gram-law} the reference law and its raw Gram pushforward:
\[
  \mathcal U_{M_s}
    =\frac{\Vol_{M_s}}{\Vol(M_s)},
  \qquad
  \eta_{n,s}=(\Gamma_s)_\#\mathcal U_{M_s}.
\]
Here \(\Vol_{M_s}\) is the volume of the restricted
Frobenius metric \(g_{M_s}\), so \(\mathcal U_{M_s}\) gives
independent uniform rows on \(\mathbb S^{s-1}\).
Recall from \cref{def:normalized-volume-laws} that
\(\mu_{n,s}^{Q}\) and \(\mu_{n,s}^{\mathrm E}\) are the normalized
quotient and ambient volume laws, respectively. The raw law \(\eta_{n,s}\) instead
has density \(\kappa_{n,s}D_s\) with respect to quotient volume.

By a \emph{lift} of a target law \(\mu\), we mean a probability law
\(\nu\) on \(M_s\) satisfying \((\Gamma_s)_\#\nu=\mu\).
All factor densities below are taken relative to
\(\mathcal U_{M_s}\) and are concentrated on \(\Mss\).

To cancel the factor \(D_s\) in
\cref{eq:volume-section-raw-gram-density}, reweight
\(\mathcal U_{M_s}\) by \(D_s(X)^{-1}\). This is the inverse
orbit-volume weight, up to a constant independent of \(X\).
It assigns less weight to factors belonging to larger orbits. Set
\begin{equation}
  \label{eq:quotient-volume-lifted-density}
  \begin{aligned}
    \dd\nu_{n,s}(X)
      &\coloneq Z_{n,s}^{-1}
      \ind{\Mss}(X)
        D_s(X)^{-1}\dd\mathcal U_{M_s}(X),\\
    Z_{n,s}
      &\coloneq \int_{\Mss}D_s(X)^{-1}\dd\mathcal U_{M_s}(X),
  \end{aligned}
\end{equation}

Note that as the rank deficient strata are of $\mathcal U_{M_s}$-measure 0, the $\ind{\Mss}(X)$ term is superfluous.
Furthermore, $\nu_{n,s}$ is a probability measure on $M_s$.
Since \(D_s(XR)=D_s(X)\) for \(R\in O(s)\), and the reference law $\mathcal U_{M_s}$
is \(O(s)\)-invariant, the $\nu_{n,s}$ is also \(O(s)\)-invariant.

\begin{theorem}[Quotient-volume lift]
\label{thm:uniform-rank-r-sampler}
For \(n\ge s\ge2\),
\[
  (\Gamma_s)_\#\bigl(\ind{\Mss}D_s^{-1}\mathcal U_{M_s}\bigr)
  =\kappa_{n,s}\Vol_Q,
\]
and
\[
  Z_{n,s}=\kappa_{n,s}\Vol_Q(\Crk{s}{n})\in(0,\infty),
  \qquad (\Gamma_s)_\#\nu_{n,s}=\mu_{n,s}^{Q}.
\]
\end{theorem}

\begin{proof}
Take \(h\equiv1\) in \cref{prop:weighted-quotient-volume-transfer}.
Evaluating both sides on \(\Crk{s}{n}\) gives the identity for
\(Z_{n,s}\). It is positive and finite by
\cref{lem:finite-principal-quotient-volume}. Dividing by \(Z_{n,s}\)
gives the normalized pushforward identity.
\end{proof}

The preceding results give explicit lifts of the two uniform laws,
while the raw Gram law is obtained directly from
\(\mathcal U_{M_s}\). We now summarize these lifts and the densities
of \(\eta_{n,s}\) with respect to the two uniform laws. 
See \cref{fig:volume-three-natural-laws} for an illustration.

\begin{corollary}[Lifted densities for the three natural laws]
\label{cor:ambient-euclidean-volume-lift}
The following source densities, given up to normalization relative
to \(\mathcal U_{M_s}\) on \(\Mss\), yield the indicated target
laws under \(\Gamma_s\):
\[
  \begin{array}{c|c}
    \text{target law on }\Crk{s}{n}
      &\text{unnormalized source density at }X\\ \hline
    \mu_{n,s}^{Q}
      &D_s(X)^{-1}\\[1ex]
    \mu_{n,s}^{\mathrm E}
      &\det(E_X)^{\frac{n-s+1}{2}}
       \sqrt{\det A_{\Gamma_s(X)}}\\[1ex]
    \eta_{n,s}
      &1.
  \end{array}
\]
The densities of the raw Gram law with respect to the two uniform
laws are
\[
  \frac{\dd\eta_{n,s}}{\dd\mu_{n,s}^{Q}}(C)
  =\frac{D_s(C)}
         {\displaystyle\int_{\Crk{s}{n}}D_s\,\dd\mu_{n,s}^{Q}},
\]
and
\[
  \frac{\dd\eta_{n,s}}{\dd\mu_{n,s}^{\mathrm E}}(C)
  =\frac{R_{n,s}(C)}
         {\displaystyle\int_{\Crk{s}{n}}
          R_{n,s}\,\dd\mu_{n,s}^{\mathrm E}},
  \qquad
  R_{n,s}(C)\coloneq 
  \pdet(C)^{-\frac{n-s+1}{2}}
  \det(A_C)^{-1/2}.
\]
\end{corollary}

\begin{figure}[H]
  \centering
  \begin{subfigure}[t]{.31\linewidth}
    \centering
    \includegraphics[width=\linewidth]{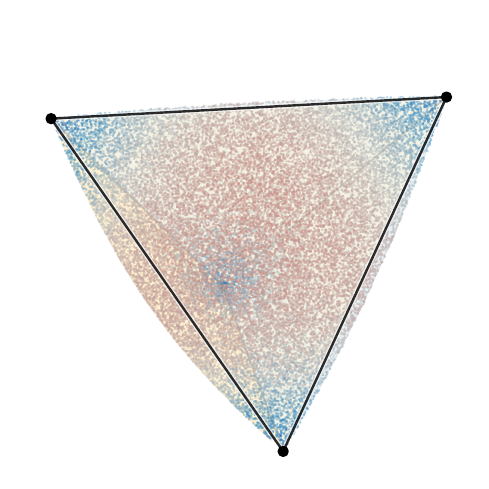}
    \caption{Product-sphere Gram law \(\eta_{3,3}\)}
  \end{subfigure}\hfill
  \begin{subfigure}[t]{.31\linewidth}
    \centering
    \includegraphics[width=\linewidth]{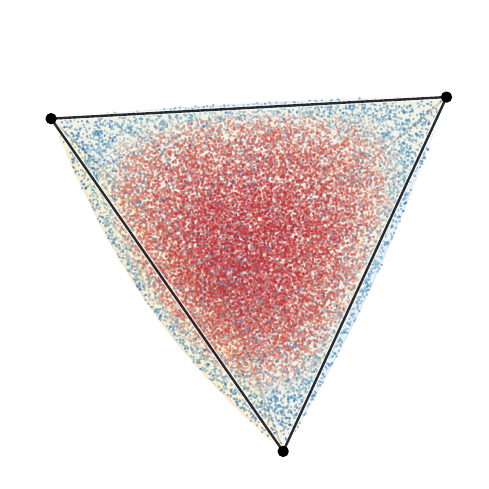}
    \caption{Euclidean-volume law \(\mu_{3,3}^{\mathrm E}\)}
  \end{subfigure}\hfill
  \begin{subfigure}[t]{.31\linewidth}
    \centering
    \includegraphics[width=\linewidth]{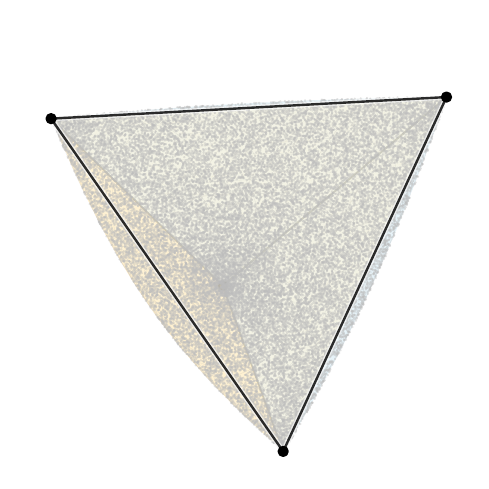}
    \caption{Quotient-volume law \(\mu_{3,3}^{Q}\)}
  \end{subfigure}
  \par\medskip
  \includegraphics[width=.75\linewidth]{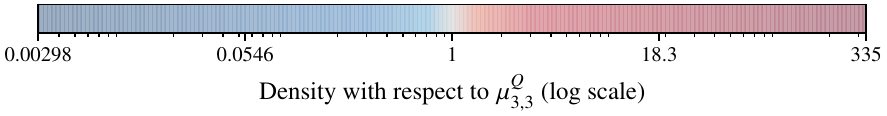}
  \caption{Three independent sampling of $\Crk 33$ with distributions densities given in \cref{cor:ambient-euclidean-volume-lift}. Colors correspond to theoretical density derivative with respect to $\mu^Q_{3,3}$. }
  \label{fig:volume-three-natural-laws}
\end{figure}

\begin{proof}
The quotient lift follows from \cref{thm:uniform-rank-r-sampler}.
For the ambient lift, apply
\cref{prop:weighted-quotient-volume-transfer} with \(h=w\), where
\(w=\dd\Vol_{\mathrm E}/\dd\Vol_Q\).
By \cref{thm:quotient-versus-ambient-volume},
\[
  \frac{w(\Gamma_s(X))}{D_s(X)}
  =2^{(s-n)/2}\det(E_X)^{\frac{n-s+1}{2}}
    \sqrt{\det A_{\Gamma_s(X)}}.
\]
More precisely, before normalization,
\[
  (\Gamma_s)_\#\left(
    \ind{\Mss}\det(E_X)^{\frac{n-s+1}{2}}
    \sqrt{\det A_{\Gamma_s(X)}}\,\mathcal U_{M_s}
  \right)
  =2^{(n-s)/2}\kappa_{n,s}\Vol_{\mathrm E}.
\]
Normalizing both sides gives the ambient lift in the table.

The constant source density gives \(\eta_{n,s}\) by definition,
and coarea yields
\[
  \dd\eta_{n,s}
  =\kappa_{n,s}D_s\,\dd\Vol_Q.
\]
Dividing this density by \(w\), the density of ambient volume
with respect to quotient volume, gives
\[
  \frac{\dd\eta_{n,s}}{\dd\Vol_{\mathrm E}}(C)
  =\frac{\kappa_{n,s}D_s(C)}{w(C)}
  =2^{(n-s)/2}\kappa_{n,s}R_{n,s}(C).
\]
Passing from each volume measure to its normalized law changes
only the constant prefactor. Thus the densities of \(\eta_{n,s}\)
with respect to \(\mu_{n,s}^{Q}\) and
\(\mu_{n,s}^{\mathrm E}\) are constant multiples of \(D_s\) and
\(R_{n,s}\), respectively. Since \(\eta_{n,s}\) has total mass
one, integrating each density against its reference law determines
the constant and gives the stated formulas.
\end{proof}

The final corollary determines when the inverse orbit-volume correction
is unnecessary.

\begin{corollary}[Uniformity of the raw Gram law]
\label{cor:raw-gram-not-uniform}
For \(n\ge s\ge2\), the raw Gram law \(\eta_{n,s}\) coincides
with the quotient-uniform law \(\mu_{n,s}^{Q}\) if and only if
\(s=2\).
\end{corollary}

\begin{proof}
By \cref{cor:ambient-euclidean-volume-lift}, the density of
\(\eta_{n,s}\) with respect to \(\mu_{n,s}^{Q}\) is proportional
to \(D_s\). Since \(D_s\) is continuous and quotient volume is
positive on every nonempty open subset of the stratum, the two
laws coincide if and only if \(D_s\) is constant.
For \(s=2\), the trace identity gives
\[
  D_2(C)=\sqrt{\lambda_1(C)+\lambda_2(C)}=\sqrt n.
\]

For \(s\ge3\), let \(X_\theta\) have row transposes
\[
  x_1=\cos\theta\,e_1+\sin\theta\,e_2,\qquad
  x_2=\cos\theta\,e_1-\sin\theta\,e_2,
\]
with \(x_i=e_i\) for \(3\le i\le s\), \(x_i=e_1\) for \(i>s\),
and \(0<\theta<\pi/2\). Setting \(\beta\coloneq 2\sin^2\theta\) gives
\[
  E_{X_\theta}
  =\diag(n-s+2-\beta,\beta,1,\ldots,1).
\]
All diagonal entries are positive for \(0<\beta<2\), so
\(X_\theta\in\Mss\). Grouping the pair sums defining \(D_s\) yields
\[
  D_s(X_\theta)^2
  =2^{(s-2)(s-3)/2}(n-s+2)
    \bigl((n-s+3-\beta)(1+\beta)\bigr)^{s-2}.
\]
The quadratic factor
\[
  (n-s+3-\beta)(1+\beta)
  =n-s+3+(n-s+2)\beta-\beta^2
\]
is positive and nonconstant on \(0<\beta<2\).
Since \(s-2>0\), \(D_s\) is therefore nonconstant on
\(\Crk{s}{n}\), proving the claim.
\end{proof}

For \(s=1\), the construction in \cref{rem:rank-one-stratum}
gives the uniform discrete law on the rank-one correlation matrices.

\section*{Acknowledgements}
This work has been supported by the U.S. Department of Energy, Office
of Science, Office of Advanced Scientific Computing Research,
Mathematical Multifaceted Integrated Capability Centers (MMICCs)
program, under Contract No. DEAC02-05CH11231 at Lawrence Berkeley
National Laboratory.
This work was supported by ERC grant \#786854 \emph{G-Statistics} from the
European Research Council under the European Union's Horizon 2020 research
and innovation program, and by the French government through the
\emph{3IA C\^ote d'Azur} Investments ANR-23-IACL-0001 managed by the National
Research Agency.

OpenAI's ChatGPT was used to assist with manuscript organization,
language, and the drafting and checking of mathematical derivations.
The authors assume
responsibility for all content.

\bibliographystyle{plain}
\phantomsection
\addcontentsline{toc}{section}{References}

\bibliography{references}

\end{document}